\documentclass{article}
\usepackage{graphicx}
\usepackage{amsmath}
\usepackage{authblk}
\usepackage{amsmath,amssymb,amsthm}
\usepackage[utf8]{inputenc}
\usepackage[english]{babel}
 \usepackage{epigraph}
\usepackage[usenames, dvipsnames]{color}
\usepackage{graphicx}
\usepackage{wrapfig}
\usepackage{caption, subcaption}
\usepackage{etoolbox, environ}
\usepackage{adjustbox}
\usepackage{cite}
\usepackage{tikz}
\usepackage{scalerel}
\usepackage{mathtools}
\usepackage[numbers]{natbib}
\usetikzlibrary{matrix}

\newcommand{\leg}{L_{R}(E_{\mathcal{G}})}
\newcommand{\ta}{t_{\alpha}t_{\beta^{*}}}

\newtheorem{theorem}{Theorem}[section]
\newtheorem{prop}{Proposition}[theorem]
\newtheorem{lemma}[theorem]{Lemma}
\newtheorem{corr}[theorem]{Corollary}
\theoremstyle{definition}
\newtheorem{definition}[theorem]{Definition}
\theoremstyle{remark}
\newtheorem{remark}[theorem]{Remark}
\newtheorem{example}[theorem]{Example}
\newtheorem{claim}[theorem]{Claim}

\begin{document}

\title{Morita Equivalence of Graph and Ultragraph Leavitt Path Algebras}
\author[1]{Michael Mekonen Firrisa}
\date{May, 2020}
\affil[1]{\small{Dartmouth College, Hanover, New Hampshire}}
\maketitle

\begin{abstract}
The primary purpose of this thesis is to show every ultragraph Leavitt path algebra is Morita equivalent, as a ring, to a graph Leavitt path algebra. Takeshi Katsura, Paul Muhly, Aidan Sims, and Mark Tomforde showed every ultragraph $C^{*}$-algebra is Morita equivalent, in the $C^{*}$-sense, to a graph $C^{*}$-algebra; our result is an algebraic analog of this fact. Further, we will use our result to give an alternate proof for established conditions which guarantee the simplicity of an ultragraph Leavitt path algebra over a field. 
\end{abstract}

\section{Historical Background}

Before proceeding with the task at hand, we would like the reader to know the work presented here is the author's Ph.D. thesis. The history of Leavitt path algebras, as with most areas of mathematics, is the confluence of ideas from different fields brought about by the work of numerous mathematicians. The interested reader can find a great description of the history in \cite{Gene-Abrams-Pere-Ara-Mercedes-Siles-Molina:2017aa}, but we will give a comparatively brief treatment here.

Our story starts with an investigation into the \textit{invariant basis number} (IBN) property. A unital ring $R$ satisfies the IBN property if the free modules $R^{n}$ and $R^{m}$ being isomorphic implies $m=n$. Some rings (e.g., non-trivial commutative rings, matrix rings over a field) always satisfy the IBN property, but not all do. In the early 60's, the eponymous William G. Leavitt proved that given any two positive integers $n$, $m$, with $n\neq m$, there exists a ring $R$ such that $R^{n}\cong R^{m}$ \cite[Theorem 8]{Leavitt:1962aa}; for $k=|m-n|$, one can show 
$$R^{i}\cong R^{j} \iff i\equiv j\ (mod\ k).$$
If $m$ is the smallest integer such that there exists an integer $l>m$ with $R^{m}\cong R^{l}$, and $n>m$ is the smallest such integer for which $R^{m}\cong R^{n}$, we say $R$ has \textit{module type $(m,n)$}. Leavitt achieved his result in part by showing, given any field $K$, one can construct a unital $K$-algebra (see Definition \ref{Ralg}), $L_{K}(m,n)$, with module type $(m,n)$. The algebra $L_{K}(m,n)$ is at times referred to as the \textit{Leavitt algebra of type $(m,n)$}. Leavitt goes on to prove interesting results regarding these algebras. Of particular interest is the universal property they possess: if $A$ is a unital $K$-algebra with module type $(m,n)$, then there exists a unital $K$-algebra homomorphism
$$\phi: L_{K}(m,n)\to A.$$
We will see later on that Leavitt path algebras have a similar universal property.

While Leavitt path algebras do not live in the world of analysis, an important component of their story comes from \textit{$C^{*}$-algebras}. To briefly describe what a $C^{*}$-algebra is, let $A$ be a $\mathbb{C}$-algebra. Suppose $A$ is equipped with an involution map, $*:A\to A$, with $a^{*}$ denoting the image of $a\in A$,  such that 
$$(a+b)^{*}=a^{*}+b^{*},\ (ab)^{*}=b^{*}a^{*},\ \text{and }(ka)^{*}=\overline{k}a^{*}$$
for all $a,b\in A$ and $k\in\mathbb{C}$. Lastly, suppose $A$ is also equipped with a norm, $\parallel\cdot\parallel$, such that $\parallel ab \parallel \leq \parallel a \parallel \parallel b \parallel$, and 
$$\parallel aa^{*}\parallel=\parallel a \parallel^{2}=\parallel a^{*} \parallel^{2} (C^{*}\text{-identity}),$$
for all $a,b \in A$. If $A$ is a complete topological space with respect to the topology induced by the norm, $A$ is a $C^{*}$-algebra. The study of these objects has its roots in quantum mechanics. It has since grown into a vast area of functional analysis in its own right. Mathematicians have worked on classifying them, and giving explicit constructions of $C^{*}$-algebras with desired properties, for decades since. In line with this endeavor, Joachim Cuntz gave an explicit construction for obtaining unital, simple, infinite, separable $C^{*}$-algebras (\textit{Cuntz algebras}) in the late 70's \cite{Cuntz:1977aa}---a $C^{*}$-algebra $A$ is separable if it contains a dense countable subset, it is infinite if there exits $a\in A$ such that $aa^{*}=1$ and $a^{*}a\neq1$, and it is simple if it doesn't contain any non-trivial closed two-sided ideals (i.e., $A$ and $\{0\}$ are the only closed two-sided ideals). Mathematicians, including Cuntz himself, then took on the enterprise of generalizing Cuntz algebras, giving us \textit{Cuntz-Krieger algebras}. By 1982, Yasuo Watatani had noticed each Cuntz-Krieger algebra has a generating set whose elements and relations can be encoded by a directed graph with a finite number of vertices, where each vertex receives and emits a non-zero finite amount of edges. This laid the foundation for graph algebras. Of course, as time went on, graph algebras came to include more arbitrary directed graphs than those introduced by Watatani. Although it didn't immediately attract much attention, many had taken notice of the utility in the graph approach before the 90's were through. In showing one can deduce results regarding $C^{*}$-algebras by studying the properties of the much more tractable directed graphs, papers such as \cite{Alex-Kumjian-David-Pask-Iain-Raeburn--Jean-Renault:1997aa}, \cite{arghh}, and \cite{Alex-Kumjian-David-Alan-Pask-Iain-Raeburn:1998aa} helped to illumine the power and elegance of graph algebras. 

Motivated by purely infinite simple $C^{*}$-algebras, the mathematicians Pere Ara, K. R. Goodearl, and Enrique Pardo introduced an algebraic analog to these objects in the early 2000's: purely infinite simple rings \cite{Pere-Ara-K.-R.-Goodearl-Enrique-Pardo:2002aa}. The same group of mathematicians, along with M. A. Gonz\'alez-Barroso, set out to find ways to construct explicit examples of such rings. In \cite{Pere-Ara-M.A.-Gonzalez-Barroso-K.-R.-Goodearl--Enrique-Pardo:2003aa}, they introduced a way to construct a class of rings called \textit{fractional skew monoid rings}, which are themselves purely infinite simple rings. In investigating the algebraic $K$-theory of fractional skew monoid rings, the aforementioned authors introduced Leavitt path algebras in the same paper---having been inspired by the work of $C^{*}$-algebraists, the authors used the term ``graph algebras'' instead of ``Leavitt path algebras.'' As it turned out, another group of mathematicians had happened upon Leavitt path algebras as well.

During a conference held in 2004 on graph $C^{*}$-algebras, which some ring theorists were invited to attend, one of the ring theorists present, Gene Abrams, noticed the algebraic data of graph $C^{*}$-algebras is quite similar to objects known well by algebraists as \textit{path algebras}. Inspired by what he saw at the conference, Abrams, together with Aranda Pino, published \cite{Gene-Abrams-and-Gonzalo-Aranda-Pino:2005aa}. In their paper, Abrams and Aranda Pino define Leavitt path algebras to stand as an algebraic analog to graph $C^{*}$-algebras. Much like the initial results established for graph $C^{*}$-algebras, the main result of \cite{Gene-Abrams-and-Gonzalo-Aranda-Pino:2005aa} is a simplicity theorem for Leavitt path algebras; graph $C^{*}$-algebras are much older than Leavitt path algebras and have more established results, a lot of the research in Leavitt path algebras is thus dedicated to proving analogs of these results (e.g., this thesis). Now, Abrams had been studying Leavitt algebras prior to his joint work with Aranda Pino. His particular focus had been the Leavitt algebra $L_{K}(1,n)$. Through the course of their work, Abrams and Aranda Pino realized the Leavitt algebra $L_{K}(1,n)$ is an example of a Leavitt path algebra. It is this realization which imparted the ``Leavitt'' to Leavitt path algebras. The name ``Leavitt path algebra'' became ubiquitous due to the fact \cite{Gene-Abrams-and-Gonzalo-Aranda-Pino:2005aa} appeared in print before \cite{Pere-Ara-M.A.-Gonzalez-Barroso-K.-R.-Goodearl--Enrique-Pardo:2003aa}. Regardless, both \cite{Gene-Abrams-and-Gonzalo-Aranda-Pino:2005aa} and \cite{Pere-Ara-M.A.-Gonzalez-Barroso-K.-R.-Goodearl--Enrique-Pardo:2003aa} are taken to be the foundation of Leavitt path algebras. It's worth noting that Leavitt path algebras were always taken to be over a field in their original definition. It was not until Mark Tomforde's work in \cite{Tomforde:2011aa} that they were defined for algebras over unital commutative rings.

The main aim of this thesis is to establish the Morita equivalence of graph and ultragraph Leavitt path algebras as rings. As mentioned before, this endeavor mirrors a result from the world of $C^{*}$-algebras as well; a result which showed how graph $C^{*}$-algebras are related to another class of $C^{*}$-algebras: \textit{Exel-Laca} algebras. In their original formulation, Cuntz-Krieger algebras are $C^{*}$-algebras associated to finite matrices with entries in $\{0,1\}$. Watatani's work led to their generalization as graph algebras. However, by dropping the finite restriction on the associated matrices, with the only criteria being they don't contain a zero row, Exel-Laca algebras were introduced as another generalization of Cuntz-Krieger algebras \cite{Ruy-Exel-Marcelo-Laca:1999aa}. So, naturally, one is compelled to ask exactly how these two class of $C^{*}$-algebras are related, seeing they both generalize Cuntz-Krieger algebras. It was worked out relatively quickly that they are not the same---there are Exel-Laca algebras which cannot be realized as graph $C^{*}$-algebras \cite[Example 4.2]{Iain-Raeburn-Wojciech-Szymanski:2004aa}, and there are graph $C^{*}$-algebras which cannot be realized as Exel-Laca algebras \cite[Proposition A.16.2]{Tomforde:2002aa}. Then exactly how are they related? In \cite{Takeshi-Katsura-Paul-Muhly-Aidan-Sims--Mark-Tomforde:2010aa}, the authors showed the two classes are the same up to Morita Equivalence. In \cite{Tomforde:2003aa}, Tomforde introduced the notion of ultragraphs and ultragarph $C^{*}$-algebras as a way of unifying graph algebras and Exel-Laca algebras. The class of ultragraph $C^{*}$-algebras properly contains graph algebras and Exel-Laca algebras \cite[Section 5]{Tomforde:2003aaa}. Thus, to establish the Morita equivalence of graph algebras and Exel-Laca algebras, it suffices to show any ultragraph $C^{*}$-algebra is Morita equivalent to a graph algebra, and vice versa. That is precisely what was established in \cite{Takeshi-Katsura-Paul-Muhly-Aidan-Sims--Mark-Tomforde:2010aa}, an algebraic analog of which we wish to prove in this thesis.


\section{Graphs and Ultragraphs}

\begin{definition}\label{graph}
A \textit{graph}, $E=(E^{0},E^{1},r_{E},s_{E})$, consists of a set of countable vertices $E^{0}$, a set of countable edges $E^{1}$, and range and source maps $r_{E},\ s_{E}:E^{1}\to E^{0}$.
\end{definition}

A \textit{finite path} in $E$, $\alpha=e_{1}\ e_{2}\dots e_{n}$, is a sequence of edges such that $s_{E}(e_{i+1})=r_{E}(e_{i})$ for $1\leq i \leq n-1$. We set $|\alpha|:=n$ to be the \textit{length} of $\alpha$; the set of such paths of finite length is denoted by $E^{*}$ ($E^{0}$ is included in $E^{*}$ as paths of length 0). The range and source maps are extended to $E^{*}$ by $r_{E}(\alpha)=r_{E}(e_{n})$ and $s_{E}(\alpha)=s_{E}(e_{1})$; we set $r_{E}(v)=s_{E}(v)=v$ for $v\in E^{0}$. Given $\alpha,\beta\in E^{*}$, with $r_{E}(\alpha)=s_{E}(\beta)$, we get an element of $E^{*}$ by concatenating $\alpha$ with $\beta$ which is denoted by $\alpha\beta$. A path $\alpha$ such that $s_{E}(\alpha)=r_{E}(\alpha)$ is called a \textit{cycle}. Finally, we call $v\in E^{0}$ a \textit{singular vertex} if $|s_{E}^{-1}(v)|=0 \text{ or } \infty$.

\begin{remark}\label{whatitdo}
The reader might find, when reading about graph $C^{*}$-algebras in particular, a slightly different definition of a finite path: a sequence of edges $\alpha=e_{1}\ e_{2}\dots e_{n}$ such that $s_{E}(e_{i})=r_{E}(e_{i+1})$ for all $1\leq i\leq n-1$; further, $r_{E}(\alpha):=r_{E}(e_{1})$ and $s_{E}(\alpha):=s_{E}(e_{n})$. This difference in convention doesn't change any of the established results so long as one accounts for the reversal in roles of the range and source maps. The author has yet to see this convention used within the context of Leavitt path algebras.  
\end{remark}

Pictorially, we depict graphs in terms of dots and arrows: a dot for each vertex, an arrow for each edge, and the placement of the arrow indicates the source and range vertices of the associated edge. 

\clearpage

\begin{example}\label{gr1}
\end{example}
\begin{figure}[h!]
\begin{center}
\begin{tikzpicture}[shorten >=1pt,auto, scale=0.5]   
\tikzset{vertex/.style = {shape=circle, draw=black!100,fill=black!100, thick, inner sep=0pt, minimum size=2 mm}}
   \node[vertex](foo) [label=below:$v$] {$v$};
  \path[->,every loop/.style={>=latex, line width=1pt, looseness=50}] (foo)
         edge  [in=60,out=120,loop] node {$e_{1}$} ();
   \path[->,every loop/.style={>=latex, line width=1pt, looseness=50}] (foo)
         edge  [in=100,out=160,loop] node {$e_{n}$} ();
   \path[->,every loop/.style={>=latex, line width=1pt, looseness=50}] (foo)
         edge  [in=20,out=80,loop] node {$e_{2}$} ();
    \path[->,every loop/.style={>=latex, line width=1pt, looseness=50}] (foo)
         edge  [in=-20,out=40,loop] node {$e_{3}$} ();
    \node [shape=circle,minimum size=1.5em] (d2) at (1,-.2) {};
     \node [shape=circle,minimum size=1.5em] (d3) at (-1,.7) {};
     \path (d2) edge [bend left=90, loosely dotted, line width=1pt] node[pos=0.2]{} (d3);
  \end{tikzpicture}
  \caption{$E^{0}=\{v\},\ E^{1}=\{e_{1},e_{2},\dots,e_{n}\},\ \text{and } s_{E}(e_{i})=r_{E}(e_{i})=v\text{ for each } i.$}
  \label{im1}
  \end{center}
 \end{figure}
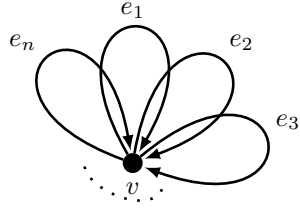

\begin{example}\label{gr2}
\end{example}
\begin{figure}[h!]
\begin{center}
\begin{tikzpicture}
\tikzset{vertex/.style = {shape=circle, draw=black!100,fill=black!100, thick, inner sep=0pt, minimum size=2 mm}}
\tikzset{edge/.style = {->, line width=1pt}}
\tikzset{v/.style = {shape=rectangle, dashed, draw, inner sep=0pt, minimum size=2em, minimum width=3em}}
    \node[vertex] (a) at (-9,1) [label=above:$v_{1}$]{};
    \node[vertex] (b) at (-6,1) [label=above:$v_{2}$]{};
    \node[vertex] (c) at (-3,1) [label=above:$v_{3}$]{};
    \node[vertex] (d) at (2,1) [label=above:$v_{n-1}$]{};
     \node[vertex] (e) at (5,1) [label=above:$v_{n}$]{};   
       
   \path (a) edge [->, >=latex, line width=.75pt, shorten <= 2pt, shorten >= 2pt, right] node[above]{$e_{1}$} (b);
   
   \path (b) edge [->, >=latex, line width=.75pt, shorten <= 2pt, shorten >= 2pt, right] node[above]{$e_{2}$} (c);

    \node [shape=circle,minimum size=1.5em] (d1) at (0,1) {};
    
     \path (c) edge [->, >=latex, line width=.75pt, shorten <= 2pt, shorten >= 2pt, right] (d1);

   \path (d1) to node {\dots} (d);

   \path (d) edge [->, >=latex, line width=.75pt, shorten <= 2pt, shorten >= 2pt, right] node[above]{$e_{n-1}$} (e);
   
   \end{tikzpicture}
   \caption{$E^{0}=\{v_{0},v_{1},\dots, v_{n}\},\ E^{1}=\{e_{1},e_{2},\dots,e_{n}\},\ s_{E}(e_{i})=v_{i-1} \text{ and } r_{E}(e_{i})=v_{i}\text{ for each } i.$}
  \label{im2}
\end{center}
\end{figure}
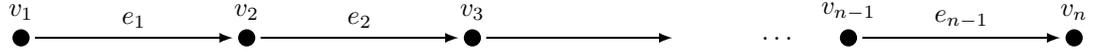

\begin{definition}\label{ultgraph}\cite{Tomforde:2003aa}   
An \textit{ultragraph} $\mathcal{G}=(G^{0},\mathcal{G}^{1},r,s)$ consists of a set of countable vertices $G^{0}$, a set of countable edges $\mathcal{G}^{1}$, a source map $s:\mathcal{G}^{1}\to G^{0}$, and a range map $r:\mathcal{G}^{1}\to \mathcal{P}(G^{0})\setminus\{\emptyset\}$ ($\mathcal{P}(G^{0})$ denotes the power set of $G^{0}$). Further, $\mathcal{G}^{0}$ denotes the smallest subset of $\mathcal{P}(G^{0})$ containing $\{v\}$ for each $v\in G^{0}$ and $r(e)$ for each $e\in\mathcal{G}^{1}$, and is closed under finite unions, finite intersections, and relative complements. 
\end{definition}

Similar to graphs, a \textit{finite path} in $\mathcal{G}$, $\alpha=e_{1}e_{2}\dots e_{n}$, is a sequence of edges such that $s(e_{i+1})\in r(e_{i})$ for $1\leq i\leq n-1$; we define its \textit{length} to be $|\alpha|:=n$. The set of finite paths in $\mathcal{G}$ is denoted by $\mathcal{G}^{*}$. We extend the range and source maps to $r:\mathcal{G}^{*}\to\mathcal{G}^{0}$ and $s:\mathcal{G}^{*}\to \mathcal{G}^{0}$ by $s(\alpha)=s(e_{1})$ and $r(\alpha)=r(e_{n})$; further, since $\mathcal{G}^{0}$ sits in $\mathcal{G}^{*}$ as paths of length 0, for $A\in\mathcal{G}^{0}$, we set $r(A)=s(A)=A$. Given $\alpha,\beta\in \mathcal{G}^{*}$, with $s(\beta)\in r(\alpha)$, we get an element of $\mathcal{G}^{*}$ by concatenating $\alpha$ with $\beta$; we denote this element by $\alpha\beta$. If $\alpha$ is a path such that $s(\alpha)\in r(\alpha)$, it is called a \textit{cycle}. Also, we call $v\in G^{0}$ a \textit{singular vertex} if $|s^{-1}(v)|=0 \text{ or } \infty$.

Again, similar to graphs, we depict ultragraphs in terms of dots and arrows: a dot for each vertex, an arrow for each edge, and the placement of the arrow indicates the source vertex and the range ``generalized'' vertex of the associated edge. The following is a simple example but one which will be useful later.

\clearpage

\begin{example}\label{ugr1}
\end{example}

\begin{figure}[h!]
\begin{center}
\begin{tikzpicture}
\tikzset{vertex/.style = {shape=circle, draw=black!100,fill=black!100, thick, inner sep=0pt, minimum size=2 mm}}
\tikzset{edge/.style = {->, line width=1pt}}
\tikzset{v/.style = {shape=rectangle, dashed, draw, inner sep=0pt, minimum size=2em, minimum width=3em}}
    \node[vertex] (a) at (0,-1) [label=below:$v_{0}$]{};
    \node[vertex] (b) at (-2,1) [label=above:$v_{1}$]{};
    \node[vertex] (c) at (0,1) [label=above:$v_{2}$]{};
     \node[vertex] (d) at (3,1) [label=above:$v_{n}$]{};
     
      \node (e) at (.5,1) {};
    \node (f) at (2.5,1) {};
    
     \path (e) -- (f) node [midway, sloped] {$\dots$};
      
    \path (a) edge [->, >=latex, line width=1pt] node[left]{$e$} (b);
   \path (a) edge [->, >=latex, line width=1pt] node[left]{$e$} (c);
  \path (a) edge [->, >=latex, line width=1pt] node[sloped, above]{$e$} (d);
   
    \node (g) at (3.5,1) {};
    \node (h) at (4,1) {};
    
    \path (g) -- (h) node [midway, sloped] {$\dots$};
    
      \node (i) at (.5,0) {};
    \node (j) at (1,0) {};
    
    \path (i) -- (j) node [midway, sloped] {$\dots$};
    
      \node (k) at (2.5,0) {};
    \node (l) at (3,0) {};
    
    \path (k) -- (l) node [midway, sloped] {$\dots$};
              
   \end{tikzpicture}
   \caption{$G^{0}=\{v_{0},v_{1},v_{2},\dots, v_{n},\dots\},\ \mathcal{G}^{1}=\{e\},\ \mathcal{G}^{0}=\{\mathcal{S}\in\mathcal{P}(G^{0}):\ \mathcal{S}, \text{ or } G^{0}\setminus\mathcal{S},\text{ is finite}\},\ s(e)=v_{0},\  r(e)=\{v_{1},v_{2},\dots,v_{n},\dots\}.$}
   \label{im3}
\end{center}
\end{figure}
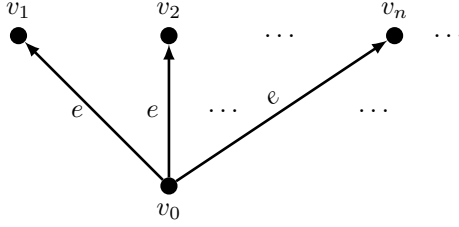
\vspace{.5cm}


\subsection{Constructing a Graph from an Ultragraph}

The following construction, from \cite{Takeshi-Katsura-Paul-Muhly-Aidan-Sims--Mark-Tomforde:2010aa}, plays a fundamental role in our work. For $n\in\mathbb{N}$, let $\{0,1\}^{n}$ denote the set of words of length $n$ whose entry are 0's and 1's. Given $\omega\in\{0,1\}^{n}$, and $1\leq i\leq n$, $\omega_{i}:=\text{$i$-th entry of } \omega$. We will at times write $\omega =(\omega_{1}\omega_{2}...\omega_{n})$ and set its length to be $|\omega|:=n$. Further, we define elements $(\omega,0),(\omega,1)\in\{0,1\}^{n+1}$ by $(\omega,0):=(\omega_{1}\omega_{2}...\omega_{n}0) \text{ and } (\omega,1):=(\omega_{1}\omega_{2}...\omega_{n}1);$ and for $m\leq n$, we define $\omega|_{m}\in\{0,1\}^{m}$ by $\omega|_{m}:=(\omega_{1}\omega_{2}...\omega_{m}).$ Finally, we let $0^{n}$ (similarly, $1^{n}$) denote the sequence of $n$ 0's (and $n$ 1's). We will also write, for example, $(0^{n-1},1)$ to denote the sequence of $(n-1)$ 0's followed by a 1.

With this in mind, let $\mathcal{G}$ be an ultragraph and $\{e_{1}, e_{2},...\}$ an enumeration of $\mathcal{G}^{1}$. For $\omega\in\{0,1\}^{n}\setminus\{0^{n}\}$, set 
$$r(\omega):=\Bigg(\bigcap\limits_{\{i:\ \omega_{i}=1\}}r(e_{i})\Bigg) \setminus \Bigg(\bigcup\limits_{\{j:\ \omega_{j}=0\}}r(e_{j})\Bigg)\subseteq G^{0},$$
also set
\begin{align*}
\hspace{1cm} & \Delta_{n}:=\{\omega\in\{0,1\}^{n}\setminus\{0^{n}\}:\ |r(\omega)|=\infty\},\ \ \ \Delta:=\bigcup\limits_{i=1}^{\infty}\Delta_{n},\\
& \Gamma_{0}:=\{(0^{n},1):\ n\geq0,\ |r((0^{n},1))|=\infty\},\ \ \ \ \Gamma_{+}:=\Delta\setminus\Gamma_{0},\\
&  W_{+}:=\bigcup\limits_{\omega\in\Delta}r(\omega)\subseteq G^{0},\ \ \ \ \text{ and } \ W_{0}:=G^{0}\setminus W_{+}.
\end{align*}

\begin{lemma}\label{sig}
There exists a function $\sigma: W_{+}\to \Delta$ such that $v\in r(\sigma(v))$ for each $v\in W_{+}$, and $\sigma^{-1}(\omega)$ is finite (possibly empty) for each $\omega\in\Delta$.
\end{lemma}

\begin{proof}
\cite[Lemma 3.7]{Takeshi-Katsura-Paul-Muhly-Aidan-Sims--Mark-Tomforde:2010aa}. 
\end{proof}

As for $\sigma$ in Lemma \ref{sig}, it can be extended  to a function $\overline{\sigma}:G^{0}\to\Delta\cup\{\emptyset\}$ by setting $\overline{\sigma}|_{W_{+}}:=\sigma$ and $\overline{\sigma}|_{W_{0}}:=\emptyset$. Abusing notation, we will write ``$\sigma$'' to also denote $\overline{\sigma}$. Now, for each $n\in\mathbb{N}$, set 
$$X(e_{n}):=\{v\in r(e_{n}):\ |\sigma(v)|<n\}\sqcup\{\omega\in\Delta_{n}:\ \omega_{n}=1\}\subseteq G^{0}\sqcup\Delta;$$
we have the following useful lemma.

 \begin{lemma}\label{lemxen}
For each $n\in\mathbb{N}$, $X(e_{n})$ is nonempty and finite.
\end{lemma}

\begin{proof}
\cite[Lemma 3.11]{Takeshi-Katsura-Paul-Muhly-Aidan-Sims--Mark-Tomforde:2010aa}.      
\end{proof}

\begin{definition}\label{gfromug}
Let $\mathcal{G}=(G^{0},\mathcal{G}^{1}, r,s)$ be an ultragraph. Fix an enumeration of $\mathcal{G}^{1}$ and let $\sigma$ be the function described following Lemma \ref{sig}, we can then construct a graph $E_{\mathcal{G}}=(E_{\mathcal{G}}^{0},E_{\mathcal{G}}^{1},r_{E},s_{E})$ associated to $\mathcal{G}$. First, set
\begin{align*}
&E_{\mathcal{G}}^{0}:=\{v_{\iota}\}_{\iota\in G^{0}\sqcup\Delta}, \ \ E_{\mathcal{G}}^{1}:=\{e_{\kappa}\}_{\kappa\in\{W_{+}\sqcup\Gamma_{+}\}\sqcup\{(e_{n},x):\ e_{n}\in\mathcal{G}^{1},\ x\in X(e_{n})\}};
\end{align*}
then, define the range and source maps as follows:
\begin{align*}
&\text{if $\kappa=v\in W_{+}$, then } r_{E}(e_{v})=v_{v} \text{ and } s_{E}(e_{v})=v_{\sigma(v)},\\
&\text{if $\kappa=\omega\in\Gamma_{+}$, then } r_{E}(e_{\omega})=v_{\omega} \text{ and } s_{E}(e_{\omega})=v_{\omega|_{|\omega|-1}},\\
&\text{and finally, } r_{E}(e_{(e_{n},x)})=v_{x} \text{ and } s_{E}(e_{(e_{n},x)})=v_{s(e_{n})}.
\end{align*}
\end{definition}
For the sake of convenient notation, we will write ``$(e_{n},x)$'' for $e_{(e_{n},x)}$. We will also usually write the other elements of $E_{\mathcal{G}}^{0}$ as ``$v$'', or ``$\omega$.'' It's important to note $E_{\mathcal{G}}$ depends on our choice of ordering of $\mathcal{G}^{1}$ as well as our choice of $\sigma$ (please see Figures \ref{im4} and \ref{im5} in Example \ref{grugrex} for an illustration).

\begin{definition}\label{subgraph}
Let $E=(E^{0},E^{1},r_{E},s_{E})$ be a graph. A \textit{subgraph} $F=(F^{0},F^{1},r_{F},s_{F})$ of $E$ is a graph such that: $F^{0}\subseteq E^{0},\ F^{1}\subseteq E^{1},\ r_{F}=r_{E}|_{F^{1}},\ \text{and}\ s_{F}=s_{E}|_{F^{1}},$ where, for each $e\in F^{1}$, $s_{E}(e),r_{E}(e)\in F^{0}$.
\end{definition}

\begin{remark}\label{FsubE}
There is an important subgraph $F$ of $E_{\mathcal{G}}$ which we will call upon in later chapters; it is the subgraph where $F^{0}=E_{\mathcal{G}}^{0}\ \text{and}\ F^{1}=\{e_{\kappa}\}_{\kappa\in W_{+}\sqcup\Gamma_{+}}.$ For more details, see Section 4 of \cite{Takeshi-Katsura-Paul-Muhly-Aidan-Sims--Mark-Tomforde:2010aa}.
\end{remark}

\begin{figure}[h!]

\begin{example}\label{grugrex}
\raggedright

\vspace{1cm}

\begin{tikzpicture}
\tikzset{vertex/.style = {shape=circle, draw=black!100,fill=black!100, thick, inner sep=0pt, minimum size=1 mm}}
\tikzset{edge/.style = {->, line width=1pt}}
\tikzset{v/.style = {shape=rectangle, dashed, draw, inner sep=0pt, minimum size=2em, minimum width=3em}}

    \node[vertex] (v0) at (0,-2) [label=below:{$v_{0}$}]{};
    \node[vertex] (v1) at (-6,2) [label=left:$v_{1}$]{};
    \node[vertex] (v2) at (-4,2) [label=left:$v_{2}$]{};
    \node[vertex] (v3) at (-2,2) [label=left:$v_{3}$]{};
    \node[vertex] (v4) at (0,2) [label=left:$v_{4}$]{};   
    \node[vertex] (vk) at (6,2) [label=right:$v_{k}$]{};

    \node[vertex] (w1) at (-6,4) [label=left:$w_{1}$]{}; 
    \node[vertex] (w3) at (-6,6) [label=left:$w_{3}$]{}; 
    \node[vertex] (w4) at (-6,8) [label=left:$w_{4}$]{};
     \node[vertex] (w10) at (-6,10) [label=left:$w_{10}$]{};

      \node[vertex] (w2) at (-4,4) [label=left:$w_{2}$]{}; 
    \node[vertex] (w5) at (-4,6) [label=left:$w_{5}$]{}; 
    \node[vertex] (w9) at (-4,8) [label=left:$w_{9}$]{};
     \node[vertex] (w11) at (-4,10) [label=left:$w_{11}$]{};   
     
      \node[vertex] (w6) at (-2,4) [label=left:$w_{6}$]{}; 
    \node[vertex] (w8) at (-2,6) [label=left:$w_{8}$]{}; 
    \node[vertex] (w12) at (-2,8) [label=left:$w_{12}$]{};
     \node[vertex] (a1) at (-2,10) {} ; 
     
      \node[vertex] (w7) at (0,4) [label=left:$w_{7}$]{}; 
    \node[vertex] (w13) at (0,6) [label=left:$w_{13}$]{}; 
    \node[vertex] (a2) at (0,8) {};
     \node[vertex] (a3) at (0,10) {} ; 
     
      \node[vertex] (a4) at (6,4) {}; 
    \node[vertex] (a5) at (6,6) {}; 
    \node[vertex] (a6) at (6,8) {};
     \node[vertex] (a7) at (6,10) {} ;

     \path (w1) edge [->, >=latex, color=gray, bend left, dashed, line width=.5pt] (w2);
     \path (w2) edge [->, >=latex, color=gray, dashed, line width=.5pt] (w3);
     \path (w3) edge [->, >=latex, color=gray, bend right, dashed, line width=.5pt] (w4);
     \path (w4) edge [->, >=latex, color=gray, dashed, line width=.5pt] (w5);
     \path (w5) edge [->, >=latex, color=gray, dashed, line width=.5pt] (w6);
     \path (w6) edge [->, >=latex, color=gray, bend left, dashed, line width=.5pt] (w7);
     \path (w7) edge [->, >=latex, color=gray, dashed, line width=.5pt] (w8);
     \path (w8) edge [->, >=latex, color=gray, dashed, line width=.5pt] (w9);
     \path (w9) edge [->, >=latex, color=gray, dashed, line width=.5pt] (w10);
     \path (w10) edge [->, >=latex, color=gray, bend left, dashed, line width=.5pt] (w11);
     \path (w11) edge [->, >=latex, color=gray, dashed, line width=.5pt] (w12);
     \path (w12) edge [->, >=latex, color=gray, dashed, line width=.5pt] (w13);
     
      \node [shape=circle,minimum size=1.5em] (d1) at (2,4) {};
      
     \path (w13) edge [->, >=latex, color=gray, dashed, line width=.5pt] (d1);

       \path (v4) to node {\dots} (vk);
       
       \node [shape=circle,minimum size=1.5em] (d2) at (10,2) {};
       \path (vk) to node {\dots} (d2);

       \node [shape=circle,minimum size=1.5em] (d3) at (-6,12) {};
       \path (w10) -- (d3) node [midway, sloped] {$\dots$};

        \node [shape=circle,minimum size=1.5em] (d4) at (-4,12) {};
        \path (w11) -- (d4) node [midway, sloped] {$\dots$};

          \node [shape=circle,minimum size=1.5em] (d5) at (-2,12) {};
          \path (a1) -- (d5) node [midway, sloped] {$\dots$};

          \node [shape=circle,minimum size=1.5em] (d6) at (0,12) {};
          \path (a3) -- (d6) node [midway, sloped] {$\dots$};

          \node [shape=circle,minimum size=1.5em] (d7) at (6,12) {};
           \path (a7) -- (d7) node [midway, sloped] {$\dots$}; 
         
           \path (v0) edge [->, >=latex, line width=.5pt, shorten <= 2pt, shorten >= 2pt, right] node[sloped, above] {\scriptsize{$e_{1}$}} (v1);
           
            \path (v0) edge [->, >=latex, line width=.5pt, shorten <= 2pt, shorten >= 2pt, right] node[sloped, above] {\scriptsize{$e_{1}$}} (v2);

             \path (v0) edge [->, >=latex, line width=.5pt, shorten <= 2pt, shorten >= 2pt, right] node[sloped, above] {\scriptsize{$e_{1}$}} (v3);
             
              \path (v0) edge [->, >=latex, line width=.5pt, shorten <= 2pt, shorten >= 2pt, right] node[right] {\scriptsize{$e_{1}$}} (v4);
              
               \path (v0) edge [->, >=latex, line width=.5pt, shorten <= 2pt, shorten >= 2pt, right] node[sloped, above] {\scriptsize{$e_{1}$}} (vk);

                \path (v1) edge [->, >=latex, line width=.5pt, shorten <= 2pt, shorten >= 2pt, right] node[left] {\scriptsize{$e_{2}$}} (w1);
                
                 \path (w1) edge [->, >=latex, line width=.5pt, shorten <= 2pt, shorten >= 2pt, right] node[left] {\scriptsize{$e_{4}$}} (w3);
                 
                 \path (w3) edge [->, >=latex, line width=.5pt, shorten <= 2pt, shorten >= 2pt, right] node[left] {\scriptsize{$e_{5}$}} (w4);
                 
                 \path (w4) edge [->, >=latex, line width=.5pt, shorten <= 2pt, shorten >= 2pt, right] node[left] {\scriptsize{$e_{11}$}} (w10);

                 \path (v2) edge [->, >=latex, line width=.5pt, shorten <= 2pt, shorten >= 2pt, right] node[left] {\scriptsize{$e_{3}$}} (w2);
                 
                 \path (w2) edge [->, >=latex, line width=.5pt, shorten <= 2pt, shorten >= 2pt, right] node[left] {\scriptsize{$e_{6}$}} (w5);
                 
                 \path (w5) edge [->, >=latex, line width=.5pt, shorten <= 2pt, shorten >= 2pt, right] node[left] {\scriptsize{$e_{10}$}} (w9);
                 
                 \path (w9) edge [->, >=latex, line width=.5pt, shorten <= 2pt, shorten >= 2pt, right] node[left] {\scriptsize{$e_{12}$}} (w11);

                 \path (v3) edge [->, >=latex, line width=.5pt, shorten <= 2pt, shorten >= 2pt, right] node[left] {\scriptsize{$e_{7}$}} (w6);
                 
                 \path (w6) edge [->, >=latex, line width=.5pt, shorten <= 2pt, shorten >= 2pt, right] node[left] {\scriptsize{$e_{9}$}} (w8);
                 
                 \path (w8) edge [->, >=latex, line width=.5pt, shorten <= 2pt, shorten >= 2pt, right] node[left] {\scriptsize{$e_{13}$}} (w12);
                 
                 \path (w12) edge [->, >=latex, line width=.5pt, shorten <= 2pt, shorten >= 2pt, right] node[sloped, below] {$$} (a1);

                 \path (v4) edge [->, >=latex, line width=.5pt, shorten <= 2pt, shorten >= 2pt, right] node[left] {\scriptsize{$e_{8}$}} (w7);
                 
                 \path (w7) edge [->, >=latex, line width=.5pt, shorten <= 2pt, shorten >= 2pt, right] node[left] {\scriptsize{$e_{14}$}} (w13);
                 
                 \path (w13) edge [->, >=latex, line width=.5pt, shorten <= 2pt, shorten >= 2pt, right] node[sloped, below] {$$} (a2);
                 
                 \path (a2) edge [->, >=latex, line width=.5pt, shorten <= 2pt, shorten >= 2pt, right] node[sloped, below] {$$} (a3);

                 \path (vk) edge [->, >=latex, line width=.5pt, shorten <= 2pt, shorten >= 2pt, right] node[sloped, below] {$$} (a4);
                 
                 \path (a4) edge [->, >=latex, line width=.5pt, shorten <= 2pt, shorten >= 2pt, right] node[sloped, below] {$$} (a5);
                 
                 \path (a5) edge [->, >=latex, line width=.5pt, shorten <= 2pt, shorten >= 2pt, right] node[sloped, below] {$$} (a6);
                 
                 \path (a6) edge [->, >=latex, line width=.5pt, shorten <= 2pt, shorten >= 2pt, right] node[sloped, below] {$$} (a7);

             \node [shape=circle,minimum size=1.5em] (d8) at (1,0) {};
              \node [shape=circle,minimum size=1.5em] (d9) at (2,0) {};
              
            \path (d8) -- (d9) node [midway, sloped] {$\dots$};

            \node [shape=circle,minimum size=1.5em] (d10) at (4,0) {};
              \node [shape=circle,minimum size=1.5em] (d11) at (5,0) {};
              
            \path (d10) -- (d11) node [midway, sloped] {$\dots$};
         
   \end{tikzpicture}

\caption{Ultragraph $\mathcal{G}$: The dashed gray arrows illustrate the ordering mechanism for $\mathcal{G}^{1}$.}
\label{im4}
\end{example}
\end{figure}

\begin{center}
\begin{figure}[h!]
\begin{tikzpicture}[scale=0.9]

\tikzset{vertex/.style = {shape=circle, draw=black!100,fill=black!100, thick, inner sep=0pt, minimum size=1 mm}}
\tikzset{edge/.style = {->, line width=1pt}}
\tikzset{v/.style = {shape=rectangle, dashed, draw, inner sep=0pt, minimum size=2em, minimum width=3em}}
    \node[vertex] (v0) at (-8,0) [label=below:{$v_{v_{0}}$}]{};
    \node[vertex] (v1) at (-6,2) [label=left:$v_{v_{1}}$]{};
    \node[vertex] (v2) at (-4,2) [label=left:$v_{v_{2}}$]{};
    \node[vertex] (v3) at (-2,2) [label=left:$v_{v_{3}}$]{};
    \node[vertex] (v4) at (0,2) [label=left:$v_{v_{4}}$]{};   
    \node[vertex] (vk) at (6,2) [label=right:$v_{v_{k}}$]{}; 
    
      \node[vertex] (1) at (-6,0) [label=below:$v_{(1)}$]{};
    \node[vertex] (10) at (-4,0) [label=below:$v_{(10)}$]{};
    \node[vertex] (100) at (-2,0) [label=below:$v_{(100)}$]{};
    \node[vertex] (1000) at (0,0) [label=below:$v_{(1000)}$]{};   
    \node[vertex] (10k) at (6,0) [label=below:$v_{(1,0^{k-1})}$]{};

    \node[vertex] (w1) at (-6,4) [label=left:$v_{w_{1}}$]{}; 
    \node[vertex] (w3) at (-6,6) [label=left:$v_{w_{3}}$]{}; 
    \node[vertex] (w4) at (-6,8) [label=left:$v_{w_{4}}$]{};
     \node[vertex] (w10) at (-6,10) [label=left:$v_{w_{10}}$]{};

      \node[vertex] (w2) at (-4,4) [label=left:$v_{w_{2}}$]{}; 
    \node[vertex] (w5) at (-4,6) [label=left:$v_{w_{5}}$]{}; 
    \node[vertex] (w9) at (-4,8) [label=left:$v_{w_{9}}$]{};
     \node[vertex] (w11) at (-4,10) [label=left:$v_{w_{11}}$]{};   
     
      \node[vertex] (w6) at (-2,4) [label=left:$v_{w_{6}}$]{}; 
    \node[vertex] (w8) at (-2,6) [label=left:$v_{w_{8}}$]{}; 
    \node[vertex] (w12) at (-2,8) [label=left:$v_{w_{12}}$]{};
     \node[vertex] (a1) at (-2,10) {} ; 
     
      \node[vertex] (w7) at (0,4) [label=left:$v_{w_{7}}$]{}; 
    \node[vertex] (w13) at (0,6) [label=left:$v_{w_{13}}$]{}; 
    \node[vertex] (a2) at (0,8) {};
     \node[vertex] (a3) at (0,10) {} ; 
     
      \node[vertex] (a4) at (6,4) {}; 
    \node[vertex] (a5) at (6,6) {}; 
    \node[vertex] (a6) at (6,8) {};
     \node[vertex] (a7) at (6,10) {} ;

      \node [shape=circle,minimum size=1.5em] (d1) at (2,4) {};

       \path (v4) to node {\dots} (vk);
       
       \node [shape=circle,minimum size=1.5em] (d2) at (10,2) {};
       \path (vk) to node {\dots} (d2);

       \node [shape=circle,minimum size=1.5em] (d3) at (-6,12) {};
       \path (w10) -- (d3) node [midway, sloped] {$\dots$};

        \node [shape=circle,minimum size=1.5em] (d4) at (-4,12) {};
        \path (w11) -- (d4) node [midway, sloped] {$\dots$};

          \node [shape=circle,minimum size=1.5em] (d5) at (-2,12) {};
          \path (a1) -- (d5) node [midway, sloped] {$\dots$};

          \node [shape=circle,minimum size=1.5em] (d6) at (0,12) {};
          \path (a3) -- (d6) node [midway, sloped] {$\dots$};

          \node [shape=circle,minimum size=1.5em] (d7) at (6,12) {};
           \path (a7) -- (d7) node [midway, sloped] {$\dots$};

                \path (v1) edge [->, >=latex, line width=.5pt, shorten <= 2pt, shorten >= 2pt, right] node[left] {\tiny{$(e_{2},w_{2})$}} (w1);
                
                 \path (w1) edge [->, >=latex, line width=.5pt, shorten <= 2pt, shorten >= 2pt, right] node[left] {\tiny{$(e_{4},w_{4})$}} (w3);
                 
                 \path (w3) edge [->, >=latex, line width=.5pt, shorten <= 2pt, shorten >= 2pt, right] node[left] {\tiny{$(e_{5},w_{5})$}} (w4);
                 
                 \path (w4) edge [->, >=latex, line width=.5pt, shorten <= 2pt, shorten >= 2pt, right] node[left] {\tiny{$(e_{11},w_{11})$}} (w10);

                 \path (v2) edge [->, >=latex, line width=.5pt, shorten <= 2pt, shorten >= 2pt, right] node[left] {\tiny{$(e_{3},w_{3})$}} (w2);
                 
                 \path (w2) edge [->, >=latex, line width=.5pt, shorten <= 2pt, shorten >= 2pt, right] node[left] {\tiny{$(e_{6},w_{6})$}} (w5);
                 
                 \path (w5) edge [->, >=latex, line width=.5pt, shorten <= 2pt, shorten >= 2pt, right] node[left] {\tiny{$(e_{10},w_{10})$}} (w9);
                 
                 \path (w9) edge [->, >=latex, line width=.5pt, shorten <= 2pt, shorten >= 2pt, right] node[left] {\tiny{$(e_{12},v_{12})$}} (w11);

                 \path (v3) edge [->, >=latex, line width=.5pt, shorten <= 2pt, shorten >= 2pt, right] node[left] {\tiny{$(e_{7},w_{7})$}} (w6);
                 
                 \path (w6) edge [->, >=latex, line width=.5pt, shorten <= 2pt, shorten >= 2pt, right] node[left] {\tiny{$(e_{9},w_{9})$}} (w8);
                 
                 \path (w8) edge [->, >=latex, line width=.5pt, shorten <= 2pt, shorten >= 2pt, right] node[left] {\tiny{$(e_{13},w_{13})$}} (w12);
                 
                 \path (w12) edge [->, >=latex, line width=.5pt, shorten <= 2pt, shorten >= 2pt, right] node[sloped, below] {$$} (a1);

                 \path (v4) edge [->, >=latex, line width=.5pt, shorten <= 2pt, shorten >= 2pt, right] node[left] {\tiny{$(e_{8},w_{8})$}} (w7);
                 
                 \path (w7) edge [->, >=latex, line width=.5pt, shorten <= 2pt, shorten >= 2pt, right] node[left] {\tiny{$(e_{14},w_{14})$}} (w13);
                 
                 \path (w13) edge [->, >=latex, line width=.5pt, shorten <= 2pt, shorten >= 2pt, right] node[sloped, below] {$$} (a2);
                 
                 \path (a2) edge [->, >=latex, line width=.5pt, shorten <= 2pt, shorten >= 2pt, right] node[sloped, below] {$$} (a3);

                 \path (vk) edge [->, >=latex, line width=.5pt, shorten <= 2pt, shorten >= 2pt, right] node[sloped, below] {$$} (a4);
                 
                 \path (a4) edge [->, >=latex, line width=.5pt, shorten <= 2pt, shorten >= 2pt, right] node[sloped, below] {$$} (a5);
                 
                 \path (a5) edge [->, >=latex, line width=.5pt, shorten <= 2pt, shorten >= 2pt, right] node[sloped, below] {$$} (a6);
                 
                 \path (a6) edge [->, >=latex, line width=.5pt, shorten <= 2pt, shorten >= 2pt, right] node[sloped, below] {$$} (a7);

                  \path (1000) to node {\dots} (10k);
                  
                   \node [shape=circle,minimum size=1.5em] (d8) at (10,0) {};
       \path (10k) to node {\dots} (d8);

        \path (v0) edge [->, >=latex, line width=.5pt, shorten <= 2pt, shorten >= 2pt, right] node[sloped, above] {\tiny{$(e_{1},(1))$}} (1);
        
         \path (1) edge [->, >=latex, line width=.5pt, shorten <= 2pt, shorten >= 2pt, right] node[sloped, above] {\tiny{$e_{(10)}$}} (10);
         
          \path (10) edge [->, >=latex, line width=.5pt, shorten <= 2pt, shorten >= 2pt, right] node[sloped, above] {\tiny{$e_{(100)}$}} (100);
          
           \path (100) edge [->, >=latex, line width=.5pt, shorten <= 2pt, shorten >= 2pt, right] node[sloped, above] {\tiny{$e_{(1000)}$}} (1000);
           
            \node [shape=circle,minimum size=1.5em] (d9) at (2,0) {};
            
             \node [shape=circle,minimum size=1.5em] (d10) at (4,0) {};
             
             \path (1000) edge [->, >=latex, line width=.5pt, shorten <= 2pt, shorten >= 2pt, right] node[sloped, below] {$$} (d9);
             
              \path (d10) edge [->, >=latex, line width=.5pt, shorten <= 2pt, shorten >= 2pt, right] node[sloped, above] {\tiny{$e_{(1,0^{k-1})}$}} (10k);

               \path (1) edge [->, >=latex, line width=.5pt, shorten <= 2pt, shorten >= 2pt, right] node[left] {\tiny{$e_{v_{1}}$}} (v1);
               
              \path (10) edge [->, >=latex, line width=.5pt, shorten <= 2pt, shorten >= 2pt, right] node[left] {\tiny{$e_{v_{2}}$}} (v2);
                
               \path (100) edge [->, >=latex, line width=.5pt, shorten <= 2pt, shorten >= 2pt, right] node[left] {\tiny{$e_{v_{3}}$}} (v3);
                 
                \path (1000) edge [->, >=latex, line width=.5pt, shorten <= 2pt, shorten >= 2pt, right] node[left] {\tiny{$e_{v_{4}}$}} (v4);

                \path (10k) edge [->, >=latex, line width=.5pt, shorten <= 2pt, shorten >= 2pt, right] node[left] {\tiny{$e_{v_{k}}$}} (vk);

   \end{tikzpicture}
\caption{Graph $E_{\mathcal{G}}$: $\sigma: W_{+}\to\Delta$ is given by $v_{i}\mapsto(1,0^{i-1})$.}
\label{im5}
\end{figure}
\end{center}


\clearpage


\section{Leavitt Path Algebras}

Before we begin with the meat of this section, we would like to let our reader know algebras are taken to be over unital commutative rings for the remainder of this thesis. 

\begin{definition}\label{Ralg}
Let $R$ be a commutative unital ring. An \textit{$R$-algebra} $\mathcal{A}$ is a ring (not necessarily unital) equipped with a map
$$\cdot:R\times\mathcal{A}\to\mathcal{A}$$
such that:\\ 
1) $r\cdot(x+y)=r\cdot x+ r\cdot y$, and $r\cdot(xy)=(r\cdot x)y=x(r\cdot y)$, for all $r\in R$ and all $x,y\in \mathcal{A}$,\\
2) $(r_{1}+r_{2})\cdot x=r_{1}\cdot x+r_{2}\cdot x$ for all $r_{1},r_{2}\in R$ and for all $x\in \mathcal{A}$,\\
3) $(r_{1}r_{2})\cdot x=r_{1}\cdot(r_{2}\cdot x)$ for all $r_{1},r_{2}\in R$ and for all $x\in \mathcal{A}$,\\
4) $1\cdot x=x$ for all $x\in \mathcal{A}$.\\
Moving forward, we will simply write ``$rx$'' to denote $r\cdot x$. Also, if $\mathcal{A}$ happens to be a unital ring, we say $\mathcal{A}$ is a \textit{unital $R$-algebra}.
\end{definition} 
 
\begin{definition}\label{ralghomo}
Let $\mathcal{A}$ and $\mathcal{B}$ be $R$-algebras. A map
$$f:\mathcal{A}\to\mathcal{B}$$
is an \textit{$R$-algebra homomorphism} if 
$$f(r_{1}x+r_{2}y)=r_{1}f(x)+r_{2}f(y)\text{ and } f(xy)=f(x)f(y)$$
for all $r_{1},r_{2}\in R$ and $x,y \in\mathcal{A}$. If $\mathcal{A}$ and $\mathcal{B}$ are unital $R$-algebras, with units $1_{A}$ and $1_{B}$ respectively, we say $f$ is a \textit{unital $R$-algebra homomorphism} if it also satisfies the condition $f(1_{A})=1_{B}$. 
\end{definition}

\subsection{Graph Leavitt Path Algebras}
Given a graph $E=(E^{0},E^{1}, r_{E},s_{E})$, we define a set of ``ghost edges'' by associating an element $e^{*}$ to each edge $e$,
$$(E^{1})^{*}:=\{e^{*}\}_{e\in E^{1}}.$$
And, for each $\alpha=e_{1}e_{2}\dots e_{n}\in E^{*}$, we define an associated ``ghost path'' by 
$$\alpha^{*}:=e_{n^{*}}e_{(n-1)^{*}}\dots e_{1^{*}}.$$

\begin{definition}\label{glgf}
Let $E$ be a graph.  A \textit{Leavitt $E$-family} in an $R$-algebra $\mathcal{A}$ is a set $\{Q_{v}, T_{e}, T_{e^{*}}\}_{v\in E^{0}, e\in E^{1}}\subseteq \mathcal{A}$ such that:\\
(\textbf{LP1}) $Q_{v}Q_{w}=\delta_{v,w}Q_{v}$ for all $v,w\in E^{0}$,\\
(\textbf{LP2}) $Q_{s_{E}(e)}T_{e}=T_{e}Q_{r_{E}(e)}=T_{e},\ Q_{r_{E}(e)}T_{e^{*}}=T_{e^{*}}Q_{s_{E}(e)}=T_{e^{*}}$ for all $e\in E^{1}$,\\
(\textbf{LP3}) $T_{e^{*}}T_{f}=\delta_{e,f}Q_{r_{E}(e)}$ for all $e,f\in E^{1}$,\\
(\textbf{LP4}) $Q_{v}=\sum\limits_{e\in s_{E}^{-1}(v)}T_{e}T_{e^{*}}$ whenever $0<|s_{E}^{-1}(v)|<\infty$.\\
For the sake of avoiding cluttered notation, we will write ``$\{Q_{v}, T_{e}, T_{e^{*}}\}$'' to denote $\{Q_{v}, T_{e}, T_{e^{*}}\}_{v\in E^{0}, e\in E^{1}}$. 
\end{definition}

\begin{definition}\label{glpa}\cite{Tomforde:2011aa, Gene-Abrams-and-Gonzalo-Aranda-Pino:2005aa,Pere-Ara-M.A.-Gonzalez-Barroso-K.-R.-Goodearl--Enrique-Pardo:2003aa}
A \textit{Leavitt path algebra} of a graph $E$ over $R$, $L_{R}(E)$, is an $R$ algebra generated by a Leavitt $E$- family $\{q_{v}, t_{e}, t_{e^{*}}\}\subseteq L_{R}(E)$ having the following universal mapping property: given an $R$-algebra $\mathcal{A}$ and a Leavitt $E$-family $\{Q_{v}, T_{e}, T_{e^{*}}\}\subseteq \mathcal{A}$, there exists an $R$-algebra homomorphism
$$\phi:L_{R}(E)\to \mathcal{A}$$
with $\phi(q_{v})=Q_{v},\ \phi(t_{e})=T_{e},\ \text{and}\ \phi(t_{e^{*}})=T_{e^{*}}$. We will write ``$L_{R}(\{q,t\})$,'' or ``$L_{R}(E)$,'' to denote such an algebra.
\end{definition}

We will briefly mention how to construct such an algebra later on in this section. It's worth noting that, due to its universal mapping property, it's unique up to isomorphism. For this reason, once the existence of such an algebra is established, we will say ``the'' Leavitt path algebra of a graph. The following construction (see proof of  \cite[Proposition 3.4]{Tomforde:2011aa}) is of important utility in determining some properties of $L_{R}(E)$. For the reader unfamiliar with $R$-modules, please see Definitions \ref{rmod} and \ref{rmodhom} before proceeding. Let $R$ be a unital commutative ring and $E$ a graph. Set $Z:=\bigoplus\limits_{n\in\mathbb{N}}R$; then, for each $e\in E^{1}$, let $Z_{e}:=Z.$ For each $v\in E^{0}$, let

\begin{center}
$Z_{v}:=
\begin{cases}
\bigoplus\limits_{e\in s_{E}^{-1}(v)} Z_{e} & \text{if } 0<|s_{E}^{-1}(v)|<\infty, \\
  Z\oplus\bigg(\bigoplus\limits_{e\in s_{E}^{-1}(v)} Z_{e} \bigg) & \text{if } |s_{E}^{-1}(v)|=\infty,\\
  Z & \text{if } |s_{E}^{-1}(v)|=0.
\end{cases}$
\end{center}
Lastly, let $\mathcal{X}:=\bigoplus\limits_{v\in E^{0}}Z_{v}$.

Now, for each $v\in E^{0}$, take the identity map $\text{Id}_{Z_{v}}:Z_{v}\to Z_{v}$. We extend $\text{Id}_{Z_{v}}$ to an $R$-module map $Q_{v}:\mathcal{X}\to\mathcal{X}$, where $Q_{v}$ is such that, given the inclusion map
$$\bigoplus\limits_{v\neq w\in E^{0}}Z_{w}\xhookrightarrow{i}\mathcal{X},$$
$Q_{v}\circ i=0$ (and $Q_{v}\circ\text{Id}_{Z_{v}}:Z_{v}\to\mathcal{X}$ is the inclusion map of $Z_{v}$ into $\mathcal{X}$). For each $e\in E^{1}$, note that $Z_{e}$ and $Z_{r_{E}(e)}$ are both free $R$-modules on a countable set. Thus, by the universal mapping property of free $R$-modules, we can define an isomorphism $T'_{e}: Z_{r_{E}(e)}\to Z_{e}$---note that $Z_{e}$ is a summand of $Z_{s_{E}(e)}$, and so a summand of $\mathcal{X}$ as well. We then extend $T'_{e}$ to a map $T_{e}:\mathcal{X}\to\mathcal{X}$, where, identifying $Z_{e}$ with its inclusion in $\mathcal{X}$, $\text{im}T_{e}\subseteq Z_{e}$, and given the inclusion map
$$\bigoplus\limits_{r_{E}(e)\neq w\in E^{0}}Z_{w}\xhookrightarrow{i}\mathcal{X},$$
$T_{e}\circ i=0$, and for the inclusion map $i_{e}:Z_{r_{E}(e)}\to\mathcal{X}$ and the projection map $p_{e}:\mathcal{X}\to Z_{e}$, we have $i_{e}\circ T_{e}\circ p_{e}=T'_{e}.$ We can similarly extend $(T'_{e})^{-1}$ to a map $T_{e^{*}}:\mathcal{X}\to\mathcal{X}$. It is easy to check the set $\{Q_{v},T_{e},T_{e^{*}}\}$ forms a Leavitt $E$-family in $\text{End}_{R}(\mathcal{X})$ (the ring of  $R$-module endomorphisms of $\mathcal{X}$). 

What is particularly useful about the above construction is the fact $rQ_{v}=0\iff r=0$, and that $T_{e}\circ T_{f}$ (similarly, $T_{f^{*}}\circ T_{e^{*}}$) is non-zero if and only if $ef\in E^{*}$. Thus, by the universal property of $L_{R}(E)$, we have 
$$rq_{v}=0\iff r=0,$$
and 
$$t_{e}t_{f}\neq0\ (\text{similarly, }t_{f^{*}}t_{e^{*}}\neq0)\iff ef\in E^{*},$$
in $L_{R}(E)$. These facts are useful when figuring out how multiplication works in $L_{R}(E)$. Even more important, they play a crucial role in establishing some isomorphism theorems for $L_{R}(E)$.

Given $\alpha=e_{1}e_{2}\dots\ e_{n}\in E^{*}$, we set $t_{\alpha}:=t_{e_{1}}t_{e_{2}}\cdots\ t_{e_{n}}$; similarly, $t_{\alpha^{*}}:=t_{e_{n}^{*}}t_{e_{(n-1)}^{*}}\cdots t_{e_{1}^{*}}$. With this in mind, we will now give a description of multiplication in $L_{R}(E)$. For $\alpha, \beta\in E^{*}$, we have

\begin{center}
$t_{\alpha}t_{\beta}=
\begin{cases}
t_{\alpha\beta} & \text{if } \alpha\beta\in E^{*}, \\
0 & \text{otherwise},\hspace{1cm}
\end{cases}$
$t_{\beta^{*}}t_{\alpha^{*}}=
\begin{cases}
t_{(\alpha\beta)^{*}} & \text{if } \alpha\beta\in E^{*}, \\
0 & \text{otherwise},
\end{cases}$
\end{center}
and
$t_{\alpha}t_{\beta^{*}}\neq0\iff r_{E}(\alpha)=r_{E}(\beta)$---this follows from \textbf{LP2} and, as witnessed by the construction above, the universal property of $L_{R}(E)$; further, by \cite[Equation 2.1]{Tomforde:2011aa}, we have

\begin{center}
$t_{\beta^{*}}t_{\alpha}=
\begin{cases}
t_{\gamma} & \text{if } \alpha=\beta\gamma\text{ for } \gamma\in E^{*}, \\
t_{\gamma^{*}} & \text{if } \beta=\alpha\gamma\text{ for } \gamma\in E^{*}, \\
q_{r_{E}(\alpha)} & \text{if } \beta=\alpha, \\
0 & \text{otherwise}.
\end{cases}$
\end{center}
To put all of this more concisely, we have the following lemma.

\begin{lemma}\label{gspan}
For a given graph $E$, 
$$L_{R}(E)=\text{span}_{R}\{t_{\alpha}t_{\beta^{*}}: \alpha,\ \beta\in E^{*}\text{ and } r_{E}(\alpha)=r_{E}(\beta).$$
\end{lemma}

\begin{proof}
\cite[Proposition 3.4]{Tomforde:2011aa}
\end{proof}

Before we proceed to give examples of Leavitt path algebras, and continue our discussion of them, we will need the following remark regarding free $R$-algebras.

\begin{remark}\label{freealg}
Let $R$ be a unital ring. For a set $X$, let $w(X)$ denote the set of nonempty words in $X$, i.e.,
$$w(X):=\{w=x_{1}x_{2}\dots x_{n}: n\in\mathbb{N},\ \ x_{i}\in X\}.$$
Let $\mathbb{F}_{R}(w(X))$ denote the free $R$-module over $w(X)$; we may think of the elements of $\mathbb{F}_{R}(w(X))$ as formal sums of the form
$$\sum\limits_{i=1}^{n}r_{i}\cdot w_{i},\ \text{with}\ r_{i}\in R,\ w_{i}\in w(X).$$
Note that since $R$ is unital, we can embed $X$ in $\mathbb{F}_{R}(w(X))$ by the map $i:X\to \mathbb{F}_{R}(w(X))$ where $i(x)=1\cdot x$. Further, we can define multiplication on $\mathbb{F}_{R}(w(X))$ by
$$\big(\sum\limits_{i=1}^{n}r_{i}\cdot w_{i}\big)\big(\sum\limits_{j=1}^{m}r_{j}\cdot w_{j}\big):=\sum\limits_{i=1}^{n}\sum\limits_{j=1}^{m}r_{i}r_{j}\cdot w_{i}w_{j},$$
where $w_{i}w_{j}$ is given by concatenation. Then, with the multiplication defined above, $\mathbb{F}_{R}(w(X))$ is the \textit{associative free $R$-algebra on $X$}. Further, it has the following universal property (see \cite[Proposition 2.6]{Gonzalo-Aranda-Pino-John-Clark-Astrid-an-Huef--Iain-Raeburn:2013aa}): let $\mathcal{A}$ be any $R$-algebra and $f:X\to \mathcal{A}$ a set map, then, there exists a unique $R$-algebra homomorphism $\phi:\mathbb{F}_{R}(w(X))\to\mathcal{A}$ making the following diagram commute:

\begin{figure}[h!]
\begin{center}
\begin{tikzpicture}
  \matrix (m) [matrix of math nodes,row sep=4em,column sep=2em,minimum width=2em]
  {
      & \mathbb{F}_{R}(w(X)) \\
     X & \mathcal{A} \\};
  \path[-stealth]
  
    (m-2-1) edge node [below] {$f$}
            node [above] {} (m-2-2)
             edge node [above] {$i$} (m-1-2)
    (m-1-2) edge [dashed,->] node [right] {$\phi$} (m-2-2);
\end{tikzpicture}
\end{center}
\end{figure}
\noindent Using the same construction as above, if we allow for the empty word $\emptyset$ in $w(X)$ (with $\emptyset w:=w$ for all $w\in w(X)$), we get the \textit{unital associative free $R$-algebra}, where $1\cdot\emptyset$ is the unit, which we will denote by ${F}_{R}(w(X))$. So, if $\mathcal{A}$ in the above diagram is a unital $R$-algebra, ${F}_{R}(w(X))$ has the universal property that $\phi$ is the unique unital $R$-algebra homomorphism making the diagram commute.
\end{remark}

Recall from Chapter 1 the Leavitt algebra $L_{K}(1,n)$, with $n\geq 2$. We will give its construction and show, as our first example, it is a Leavitt path algebra. Given the set $X:=\{x_{1},\dots, x_{n},y_{1},\dots,y_{n}\}$, let $F_{K}(w(X))$ be the unital associative free $K$-algebra as in Remark \ref{freealg}. Further, let $I\lhd F_{K}(w(X))$ be the ideal generated by the set
$$\{1-\sum\limits_{i=1}^{n}x_{i}y_{i},\ \delta_{i,j}1-y_{i}x_{j}: i,j\in\{1,\dots,n\}\}.$$
Then, 
$$L_{K}(1,n):=F_{K}(w(X))/I.$$

\begin{example}\label{grex1}
Let $K$ be field, and $E$ the graph in Example \ref{gr1}. Then, $L_{K}(E)\cong L_{K}(1,n).$ To see this, first note that $q_{v}$ is the identity of $L_{K}(E)$; this fact follows from Lemma \ref{gspan}. Now, take the set $\{Q_{v},T_{e},T_{e^{*}}\}\subseteq L_{K}(1,n)$ to be
$$Q_{v}=\overline{1}, T_{e_{i}}=\overline{x_{i}}, \text{ and } T_{e_{i}^{*}}=\overline{y_{i}}.$$ 
One can quickly check $\{Q_{v},T_{e},T_{e^{*}}\}$ is a Leavitt $E$-family in $L_{K}(1,n)$. And so, by the universal mapping property of  $L_{K}(E)$, we have a $K$-algebra homomorphism $\phi:L_{K}(E)\to L_{K}(1,n)$ such that $\phi(q_{v})=\overline{1},\ \phi(t_{e_{i}})=\overline{x_{i}}, \text{ and }\phi(t_{e_{i}^{*}})=\overline{y_{i}}$. On the other hand, since $L_{K}(E)$ is a unital $K$-algebra, by the universal mapping property of $F_{K}(w(X))$, there exists a unital $K$-algebra homomorphism $\varphi:F_{K}(w(X))\to L_{K}(E)$ such that $\varphi(1)=q_{v},\ \varphi(x_{i})=t_{e_{i}}, \text{ and } \varphi(y_{i})=t_{e_{i}^{*}}$. Moreover, by the properties of a Leavitt $E$-family, we have $I\subseteq \text{ker}\varphi$. Thus, there exists a unital $K$-algebra homomorphism $\overline{\varphi}:L_{K}(1,n)\to L_{K}(E)$ where $\overline{\varphi}=\varphi\circ\pi$, meaning $\overline{\varphi}(\overline{1})=q_{v},\ \overline{\varphi}(\overline{x_{i}})=t_{e_{i}}, \text{ and } \overline{\varphi}(\overline{y_{i}})=t_{e_{i}^{*}}$. It is clear $\phi$ and $\varphi$ are inverses of each other, hence $L_{K}(E)\cong L_{K}(1,n)$. In the case where $E$ consists of a single vertex and a single looped edge, note a similar argument as above shows $L_{K}(E)\cong K[x,x^{-1}]$---the Laurent polynomials over $K$. 
\end{example}

\begin{example}\label{grex2}
Let $R$ be a unital commutative ring, and $E$ the graph in Example \ref{gr2}. We will show $L_{R}(E)\cong M_{n}(R).$ Let $E_{i,j}$ denote the $n\times n$ matrix whose $(ij)$-th entry is  1, and all other entries are $0$. Let $\{Q_{v},T_{e}, T_{e^{*}}\}\subseteq M_{n}(R)$ be given by
 $$Q_{v_{i}}=E_{i,i},\ T_{e_{i}}=E_{i,i+1},\text{ and } T_{e_{i}^{*}}=E_{i+1,i}.$$
One can quickly verify $\{Q_{v},T_{e}, T_{e^{*}}\}$ is a Leavitt $E$-family in $M_{n}(R)$; thus, by the universal mapping property of $L_{R}(E)$, there exists an $R$-algebra homomorphism $\phi:L_{R}(E)\to M_{n}(R).$ Given $E_{i,j}$, suppose w.l.o.g. that $i<j$. Then, we can easily work out that
$$E_{i,j}=E_{i,i+1}E_{i+1,i+2}\cdots E_{j-1,j}.$$
A similar argument shows $E_{i,j}$ can be written as the product of elements in $\{Q_{v},T_{e}, T_{e^{*}}\}$ when $i>j$; since $M_{n}(R)=\text{span}_{R}\{E_{i,j}: i,j\in \{1,\dots, n\}\},$ $\phi$ must be surjective. The fact each $E_{i,j}$ can be written as the product of elements in $\{Q_{v},T_{e}, T_{e^{*}}\}$ means that $E_{i,j}=\phi(t_{\alpha_{i}}t_{\beta_{j}^{*}})$ for some $\alpha_{i},\beta_{j}\in E^{*}$ such that $r_{E}(\alpha_{i})=r_{E}(\beta_{j}).$ And so, 
$$M_{n}(R)=\text{span}_{R}\{\phi(t_{\alpha_{i}}t_{\beta_{j}^{*}}): i,j\in \{1,\dots, n\},\ \alpha_{i},\beta_{j}\in E^{*}\ \text{such that} \ r_{E}(\alpha_{i})=r_{E}(\beta_{j}) \},$$
which means $\phi$ is injective as well since
$$0=\phi\bigg(\sum\limits_{i,j}r_{i,j}t_{\alpha_{i}}t_{\beta_{j}^{*}}\bigg)\iff r_{i,j}=0\text{ for all } i,j.$$
Thus, $L_{R}(E)\cong M_{n}(R)$.
\end{example}

One of the main endeavors within the study of Leavitt path algebras is to see how certain properties of $E$ translate to properties of $L_{R}(E)$. While this pursuit has resulted in some important theorems, it's worth noting that very different looking graphs can produce the same Leavitt path algebra, as Example \ref{grex2} and the following example show.

\begin{example}\label{grex3}
Let $R$ be a unital commutative ring, and $E$ the graph illustrated below. Then, $L_{R}(E)\cong M_{n}(R).$
\end{example}

\vspace{.5cm}

\begin{figure}[h!]
\begin{center}
\begin{tikzpicture}
\tikzset{vertex/.style = {shape=circle, draw=black!100,fill=black!100, thick, inner sep=0pt, minimum size=2 mm}}
\tikzset{edge/.style = {->, line width=1pt}}
\tikzset{v/.style = {shape=rectangle, dashed, draw, inner sep=0pt, minimum size=2em, minimum width=3em}}
    \node[vertex] (a) at (0,0) [label=left:$v$]{};
    \node[vertex] (b) at (3,0) [label=right:$w$]{};
      
    \path (a) edge [->, >=latex, bend left=90, line width=1pt] node[above]{$e_{1}$} (b);
   \path (a) edge [->, >=latex, bend left=20, line width=1pt] node[above]{$e_{2}$} (b);
   \path (a) edge [->, >=latex, bend right=90, line width=1pt] node[below]{$e_{n-1}$} (b);
   
    \node (c) at (1.5,.2) {};
    \node (d) at (1.5,-.8) {};
     \path (c) edge [loosely dotted, line width=1.5pt] node[pos=0.2]{} (d);
     
       \end{tikzpicture}
       \caption{}
       \label{im6}
\end{center}
\end{figure}
Let $E_{i,j}\in M_{n}(R)$ be as in Example \ref{grex2}, let $\text{Id}_{n-1}\in M_{n}(R)$ denote the matrix whose first $n-1$ diagonal entries are 1's and all other entries are 0's (i.e., the top left $(n-1)\times (n-1)$ sub-matrix is the identity on $M_{n-1}(R)$ and the $n$-th row and column are 0's). Let $\{Q_{v},T_{e}, T_{e^{*}}\}\subseteq M_{n}(R)$ be given by
$$Q_{v}=\text{Id}_{n-1},\ Q_{w}=E_{n,n},\ T_{e_{i}}=E_{i,n},\text{ and } T_{e_{i}^{*}}=E_{n,i}.$$
Then, $\{Q_{v},T_{e}, T_{e^{*}}\}$ is a Leavitt $E$-family in $M_{n}(R)$, and so there exists an $R$-algebra homomorphism
$$\phi:L_{R}(E)\to M_{n}(R).$$
Again, an argument similar to the one used in Example \ref{grex2} shows $\phi$ is an isomorphism; thus, $L_{R}(E)\cong M_{n}(R)$.

While we will not go into great detail here, since we will see this more explicitly when constructing Exel-Laca algebras, but the usual method of constructing $L_{R}(E)$ is to take $\mathbb{F}_{R}(w(X))$ (see Remark \ref{freealg}), where $X=\{v\}_{v\in E^{0}}\cup\{e, e^{*}\}_{e\in E^{1}}$, and quotient by $I\lhd\mathbb{F}_{R}(w(X))$, where $I$ is the ideal generated by the union of sets of the form:\\
$\bullet$ $\{vw-\delta_{v,w}v:v,w\in E^{0}\}$,\\
$\bullet$ $\{s_{E}(e)e-e,\ er_{E}(e)-e,\ e^{*}s_{E}(e)-e^{*},\ r_{E}(e)e^{*}-e^{*}:e\in E^{1}\}$,\\
$\bullet$ $\{e^{*}f-\delta_{e,f}r_{E}(e):e,f\in E^{1}\}$,\\
$\bullet$ $\{\sum\limits_{e\in s_{E}^{-1}(v)}ee^{*}-v:v\in E^{0}\text{ such that }0<| s_{E}^{-1}(v)|<\infty\}$.\\ 
The reason we take $\mathbb{F}_{R}(w(X))$, and not $F_{R}(w(X))$, is due to the fact $L_{R}(E)$ need not be unital. Lastly, the Leavitt $E$-family $\{q_{v},t_{e},t_{e^{*}}\}$ is given by $q_{v}:=\pi(v)$, $t_{e}:=\pi(e)$, and $t_{e^{*}}:=\pi(e^{*})$, where
$$\pi:\mathbb{F}_{R}(w(X))\to\mathbb{F}_{R}(w(X))/I$$
is the projection map. The following lemma is mentioned in the paragraph after \cite[Definition 4.2]{Tomforde:2011aa}.

\begin{lemma}\label{glpalgunital}
For a given graph $E$, $L_{R}(E)$ is unital if and only if $E^{0}$ is finite.
\end{lemma}

\begin{proof}
Suppose $E^{0}$ is finite. Then, using the properties of a Leavitt $E$-family and Lemma \ref{gspan}, we can show $\sum\limits_{v\in E^{0}}q_{v}$ is a unit for $L_{R}(K)$. On the other hand, suppose $E^{0}$ is infinite and $L_{R}(E)$ is unital. Let $\{v_{1},\dots,v_{n},\dots\}$ be an enumeration of $E^{0}$. Set
$$u_{k}:=\sum\limits_{i=1}^{k}q_{v_{i}}.$$
By applying Lemma \ref{gspan} and \textbf{LP2}, we can show
$$L_{R}(E)=\bigcup\limits_{k\in\mathbb{N}}u_{k}L_{R}(E)u_{k}.$$
Further, it is easy to check $u_{k}$ is an idempotent element, which means $u_{k}$ is the identity for $u_{k}L_{R}(E)u_{k}$; and, $u_{k}u_{k+1}=u_{k+1}u_{k}=u_{k}$, which means $u_{k}L_{R}(E)u_{k}\subseteq\ u_{k+1}L_{R}(E)u_{k+1}$. And so for $1\in L_{R}(E)$, there exists $k_{0}\in\mathbb{N}$ such that $1\in u_{k_{0}}L_{R}(E)u_{k_{0}}$; but, since $u_{k_{0}}$ is the identity for $u_{k_{0}}L_{R}(E)u_{k_{0}}$, it must mean $1=u_{k_{0}}$. However, taking $v_{i}$ such that $i>k_{0}$, \textbf{LP1} implies $q_{v_{i}}=q_{v_{i}}1=q_{v_{i}}u_{k_{0}}=0. \Rightarrow\Leftarrow$---to see why, refer to the paragraph after Definition \ref{glpa}. Thus, $E^{0}$ infinite $\implies$ $L_{R}(E)$ is nonunital; all together, we have $L_{R}(E)$ is unital $\iff$ $E^{0}$ is finite.   
\end{proof}

While there is a lot more which can be said regarding Leavitt path algebras (there's an entire book on the topic, see \cite{Gene-Abrams-Pere-Ara-Mercedes-Siles-Molina:2017aa}), we will close our discussion of Leavitt path algebras by stating two important results. But before we do so, we will quickly discuss $\mathbb{Z}$-graded rings.

\begin{definition}\label{dgr}
A ring $R$ is said to be a \textit{$\mathbb{Z}$-graded ring} if there are subgroups $R_{i}\subseteq R$ such that  
$$R=\bigoplus\limits_{i\in\mathbb{Z}}R_{i}$$
as an internal direct sum, and $R_{n}R_{m}\subseteq R_{n+m}$ for each $n,m \in\mathbb{Z}$. The elements of $R_{i}$ are called \textit{homogeneous elements of degree $i$}. If we want to say an element belongs to one of the factor subgroups, we will call it a \textit{homogeneous element}.
\end{definition}

\begin{definition}\label{dgrh}
Let $R=\bigoplus\limits_{i\in\mathbb{Z}}R_{i}$ and $S=\bigoplus\limits_{i\in\mathbb{Z}}S_{i}$ be two $\mathbb{Z}$-graded rings. A ring homomorphism 
$$f:R\to S$$
is called a \textit{graded homomorphism} if $f(R_{i})\subseteq S_{i}$ for each $i$. 
\end{definition}

\begin{definition}\label{dgi}
Given a $\mathbb{Z}$-graded $R$, an ideal $I\subseteq R$ is a \textit{homogeneous (or graded) ideal} if it is generated by its homogeneous elements. 
\end{definition}

The following remark will be of use to us in later sections.

\begin{remark}\label{rgi}
An ideal $I$ being homogeneous is equivalent to saying $I=\bigoplus\limits_{i\in\mathbb{Z}}(I\cap R_{i}).$ Further, if $I$ is a homogeneous ideal, then $R/I$ is a $\mathbb{Z}$-graded ring with
$$R/I=\bigoplus\limits_{i\in\mathbb{Z}}(R_{i}+I)/I.$$
For specific details, see \cite[Section 1.1.5.]{Hazrat:2016aa}.
\end{remark}

\begin{prop}\label{gzgrade}
For any graph $E$, $L_{R}(E)$ is a $\mathbb{Z}$-graded ring with
$$L_{R}(E)_{i}=\text{span}_{R}\{t_{\alpha}t_{\beta^{*}}: \alpha,\ \beta\in E^{*},\ r_{E}(\alpha)=r_{E}(\beta), \text{ and } |\alpha|-|\beta|=i\}.$$
\end{prop}

\begin{proof}
\cite[Proposition 4.7]{Tomforde:2011aa}
\end{proof}

\begin{theorem}\label{ggrdhomom}
Let $E$ be a graph, and $S$ a $\mathbb{Z}$-graded ring. Suppose there exists a graded ring homomorphism 
$$\pi:L_{R}(E)\to S$$
such that $\pi(rq_{v})\neq 0$ for all $v\in E^{0}$ and $r\in R\setminus\{0\}$. Then, $\pi$ is injective.
\end{theorem}

\begin{proof}
\cite[Theorem 5.3]{Tomforde:2011aa}
\end{proof}

\subsection{Ultragraph Leavitt Path Algebras}

Similar to graphs, given an ultragraph $\mathcal{G}=(G^{0},\mathcal{G}^{1}, r,s)$, we define a set of ``ghost edges'' by associating an element $e^{*}$ to each edge $e$,
$$(\mathcal{G}^{1})^{*}:=\{e^{*}\}_{e\in \mathcal{G}^{1}}.$$
And, for each $\alpha=e_{1}e_{2}\dots e_{n}\in \mathcal{G}^{*}$, we define an associated ``ghost path'' by 
$$\alpha^{*}:=e_{n^{*}}e_{(n-1)^{*}}\dots e_{1^{*}}.$$

\begin{definition}\label{dulfg}
Let $\mathcal{G}$ be an ultragraph. A \textit{Leavitt $\mathcal{G}$-family} in an $R$-algebra $\mathcal{A}$ is a set $\{P_{A},\ S_{e},\ S_{{e}^{*}}\}_{A\in\mathcal{G}^{0},e\in\mathcal{G}^{1}}\subseteq \mathcal{A}$ such that:\\
(\textbf{uLP1}) $P_{\emptyset}=0,\ P_{A\cap B}=P_{A}P_{B},\ P_{A\cup B}=P_{A}+P_{B}-P_{A\cap B}$ for all $A,B\in\mathcal{G}^{0}$,\\
(\textbf{uLP2}) $P_{s(e)}S_{e}=S_{e}P_{r(e)}=S_{e}$ and $P_{r(e)}S_{e^{*}}=S_{e^{*}}P_{s(e)}=S_{e^{*}}$ for all $e\in\mathcal{G}^{1}$,\\
(\textbf{uLP3}) $S_{e^{*}}S_{f}=\delta_{e,f}P_{r(e)}$ for all $e,f\in\mathcal{G}^{1}$,\\
(\textbf{uLP4}) $P_{v}=\sum\limits_{e\in s^{-1}(v)}S_{e}S_{e^{*}}$ for $v\in G^{0}$ such that $0<|s^{-1}(v)|<\infty$.\\
For $v\in G^{0}$, we take $P_{v}:=P_{\{v\}}$.
\end{definition}

Notice that, for each $A\in\mathcal{G}^{0}$, \textbf{uLP1} implies $P_{A}$ is idempotent: $P_{A}=P_{A\cap A}=P_{A}P_{A}$. Also, the following lemma expands on \textbf{uLP2}. Its statement and proof is exactly that of \cite[Lemma 2.8]{Tomforde:2003aa}.

\begin{lemma}\label{ulve}
Given a Leavitt $\mathcal{G}$-family $\{P_{A},\ S_{e},\ S_{e^{*}}\}\subseteq \mathcal{A}$, for $A\in\mathcal{G}^{0}$ and $e\in\mathcal{G}^{1}$:

\begin{center}
$P_{A}S_{e}=
\begin{cases}
S_{e} & \text{if } s(e)\in A, \\
0 & \text{otherwise},
\end{cases}$
\end{center}

and

\begin{center}
$S_{e^{*}}P_{A}=
\begin{cases}
S_{e^{*}} & \text{if } s(e)\in A, \\
0 & \text{otherwise}.
\end{cases}$
\end{center}
\end{lemma}

\begin{proof}
Let $A\in\mathcal{G}^{0}$ and $e\in\mathcal{G}^{1}$. By \textbf{uLP2}, we have $P_{s(e)}S_{e}=S_{e}$, and \textbf{uLP1} implies $P_{A}P_{s(e)}=P_{A\cap s(e)}$. And so

\begin{center}
$P_{A}S_{e}=P_{A}P_{s(e)}S_{e}=P_{A\cap s(e)}S_{e}=
\begin{cases}
P_{s(e)}S_{e}=S_{e} & \text{if } s(e)\in A, \\
P_{\emptyset}=0 & \text{otherwise}.
\end{cases}$
\end{center}
A similar argument shows
\begin{center}
$S_{e^{*}}P_{A}=
\begin{cases}
S_{e^{*}} & \text{if } s(e)\in A, \\
0 & \text{otherwise}.
\end{cases}$
\end{center}

\end{proof}

\begin{definition}\label{uglpa}
A \textit{Leavitt path algebra} of an ultragraph $\mathcal{G}$ over $R$, $L_{R}(\mathcal{G})$, is an $R$ algebra generated by a Leavitt $\mathcal{G}$- family $\{p_{A}, s_{e}, s_{e^{*}}\}\subseteq L_{R}(\mathcal{G})$ having the following universal property: given an $R$-algebra $\mathcal{A}$ and a Leavitt $\mathcal{G}$-family $\{P_{A}, S_{e}, S_{e^{*}}\}\subseteq \mathcal{A}$, there exists an $R$-algebra homomorphism
$$\phi:L_{R}(\mathcal{G})\to \mathcal{A}$$
with $\phi(p_{A})=P_{A},\ \phi(s_{e})=S_{e},\ \text{and}\ \phi(s_{e^{*}})=S_{e^{*}}$. We will sometimes denote the Leavitt path algebra of $\mathcal{G}$ by $L_{R}(\{p,s\})$.
\end{definition}

As with the Leavitt path algebra of a graph, the universal mapping property guarantees $L_{R}(\mathcal{G})$ is unique up to isomorphism. We construct $L_{R}(\mathcal{G})$ by taking the associative free algebra $\mathbb{F}_{R}(w(X))$, where $X=\{A\}_{A\in\mathcal{G}^{0}}\cup\{e,e^{*}\}_{e\in\mathcal{G}^{1}}$, and taking its quotient by the ideal $I\lhd\mathbb{F}_{R}(w(X))$, where $I$ is the ideal generated by the union of sets of the form:\\
$\bullet$ $\{A\cap B-AB, A\cup B-(A+B)+A\cap B:A,B\in\mathcal{G}^{0}\text{ such that } A\neq\emptyset,B\neq\emptyset,A\cap B\neq\emptyset\}$,\\
$\bullet$ $\{s(e)e-e,\ er(e)-e,\ r(e)e^{*}-e^{*},\ e^{*}s(e)-e^{*}:e\in\mathcal{G}^{1}\}$,\\
$\bullet$ $\{e^{*}f-\delta_{e,f}r(e):e,f\in\mathcal{G}^{1}\}$,\\
$\bullet$ $\{\sum\limits_{e\in s^{-1}(v)}ee^{*}-v:v\in G^{0}\text{ such that } 0<|s^{-1}(v)|<\infty\}$;\\
given the projection map
$$\pi:\mathbb{F}_{R}(w(X))\to\mathbb{F}_{R}(w(X))/I,$$
the Leavitt $\mathcal{G}$-family $\{p_{A},s_{e},s_{e^{*}}\}$ is given by $s_{e}:=\pi(e)$, $s_{e^{*}}:=\pi(e^{*})$, $p_{A}:=\pi(A)$ (when $A\neq\emptyset$), and $p_{\emptyset}:=0$. 

Suppose $\mathcal{A}$ is an $R$-algebra generated by a Leavitt $\mathcal{G}$ and satisfies the hypothesis of Definition \ref{uglpa}. We can use the universal mapping properties of $L_{R}(\mathcal{G})$ and $\mathcal{A}$ to define maps 
$$\phi:L_{R}(\mathcal{G})\to\mathcal{A} \text{ and }\varphi:\mathcal{A}\to L_{R}(\mathcal{G})$$
which are inverses of each other. Thus, as with $L_{R}(E)$, $L_{R}(\mathcal{G})$ is unique up to isomorphism. Much like the construction described in the paragraph after Definition \ref{glpa}, there is an analogous construction (see \cite[Theorem 2.6]{M.-Imanfar:2017aa}) which shows
$$rp_{A}\neq0 \text{ for all } A\in\mathcal{G}^{0}\setminus\{\emptyset\} \text{ and } r\in R\setminus\{0\};$$
and
$$s_{\alpha}s_{\beta}\neq0,\ s_{\beta^{*}}s_{\alpha^{*}}\neq0\iff \alpha\beta\in\mathcal{G}^{*}.$$

We will now give a description of how multiplication works in $L_{R}(\mathcal{G})$. Unsurprisingly, it is similar to that of $L_{R}(E)$. Let $\alpha,\beta\in\mathcal{G}$, then:
\vspace{.5cm}
if $|\alpha|,|\beta|>0$,
\begin{center}
$s_{\alpha}s_{\beta}=
\begin{cases}
s_{\alpha\beta} & \text{if } \alpha\beta\in E^{*}, \\
0 & \text{otherwise},\hspace{1cm}
\end{cases}$
and \hspace{.5cm} $s_{\beta^{*}}s_{\alpha^{*}}=
\begin{cases}
s_{(\alpha\beta)^{*}} & \text{if } \alpha\beta\in E^{*}, \\
0 & \text{otherwise};
\end{cases}$
\end{center}
\vspace{.5cm}
if $|\alpha|>0, |\beta|=0$, then for $\beta=A\in\mathcal{G}^{0}$, \textbf{uLP1} and \textbf{uLP2} imply $s_{\alpha}p_{A}=s_{\alpha}p_{r(\alpha)\cap A}$, which is $0$ if $r(\alpha)\cap A=\emptyset$, similarly, $p_{A}s_{\alpha^{*}}=p_{r(\alpha)\cap A}s_{\alpha^{*}}$; the case where $|\alpha|=|\beta|=0$ is addressed in \textbf{uLP1}, and the case where $|\alpha|=0$ is Lemma \ref{ulve}; note,
$$s_{\alpha}s_{\beta^{*}}=s_{\alpha}(p_{r(\alpha)\cap r(\beta)})s_{\beta^{*}}\neq0\iff r(\alpha)\cap r(\beta)\neq\emptyset;$$
finally, by \cite[Lemma 2.4]{M.-Imanfar:2017aa}, we have

\begin{center}
$s_{\beta^{*}}s_{\alpha}=
\begin{cases}
s_{\gamma} & \text{if } \alpha=\beta\gamma\text{ for } \gamma\in \mathcal{G}^{*}, \\
s_{\gamma^{*}} & \text{if } \beta=\alpha\gamma\text{ for } \gamma\in \mathcal{G}^{*}, \\
p_{r(\alpha)} & \text{if } \beta=\alpha, \\
0 & \text{otherwise}.
\end{cases}$
\end{center}
Putting it all together, we have the following theorem which is analogous to the combination of Lemma \ref{gspan} and Proposition \ref{gzgrade}.

\begin{theorem}\label{uggrade}
 For any ultragraph $\mathcal{G}$:\\
 1) $L_{R}(\mathcal{G})=\text{span}_{R}\{s_{\alpha}p_{A}s_{\beta^{*}}:\alpha,\beta\in\mathcal{G}^{*},\ A\in\mathcal{G}^{0}, \text{ and } r(\alpha)\cap\ A\cap r(\beta)\neq\emptyset\}$,\\
 2) $L_{R}(\mathcal{G})$ is a $\mathbb{Z}$-graded ring with\\
 $L_{R}(\mathcal{G})_{i}=\text{span}_{R}\{s_{\alpha}p_{A}s_{\beta^{*}}:\alpha,\beta\in\mathcal{G}^{*},\ A\in\mathcal{G}^{0},\ r(\alpha)\cap\ A\cap r(\beta)\neq\emptyset \text{ and } |\alpha|-|\beta|=i\}.$
\end{theorem}

\begin{proof}
\cite[Theorem 2.5]{M.-Imanfar:2017aa}
\end{proof}
We also have the following theorem which is analogous to Theorem \ref{ggrdhomom}.
\begin{theorem}\label{uggrdhomom}
For a given ultragraph $\mathcal{G}$, and a $\mathbb{Z}$-graded ring $S$, suppose
$$\pi:L_{R}(\mathcal{G})\to S$$
is a graded ring homomorphism such that $\pi(rp_{A})\neq0$ for all $A\in\mathcal{G}^{0}\setminus\{\emptyset\}$ and $r\in R\setminus\{0\}$. Then, $\pi$ is injective.
\end{theorem}

\begin{proof}
\cite[Corollary 3.4]{M.-Imanfar:2017aa}
\end{proof}

Just like the Leavitt path algebras of graphs, the Leavitt path algebras of ultragraphs need not be unital. The following result is a restatement of \cite[Lemma 6.11]{M.-Imanfar:2017aa}.

\begin{lemma}\label{ugnunit}
For a given ultragraph $\mathcal{G}$, $L_{R}(\mathcal{G})$ is unital $\iff$ $G^{0}\in\mathcal{G}^{0}$.
\end{lemma}

\begin{proof}
Suppose $G^{0}\in\mathcal{G}^{0}$. Then, by part 1) of Theorem \ref{uggrade}, $p_{G^{0}}$ is a unit for $L_{R}(\mathcal{G})$. On the other hand, suppose $L_{R}(\mathcal{G})$ is unital. Again, Theorem \ref{uggrade} 1) implies
$$1=\sum\limits_{i=1}^{n}r_{i}s_{\alpha_{i}}p_{A_{i}}s_{\beta_{i}^{*}}.$$
Let $A=\bigcup\limits_{i=1}^{n}s(\alpha_{i})$. Since $A$ is the finite union of elements in $G^{0}$, we have $A\in\mathcal{G}^{0}$. If $G^{0}\notin\mathcal{G}^{0}$, it must mean there exists a $v\in G^{0}\setminus A$. But then, by \text{uLP1},
$$p_{v}=p_{v}1=p_{v}\Bigg(\sum\limits_{i=1}^{n}r_{i}s_{\alpha_{i}}p_{A_{i}}s_{\beta_{i}^{*}}\Bigg)=\sum\limits_{i=1}^{n}r_{i}p_{v}s_{\alpha_{i}}p_{A_{i}}s_{\beta_{i}^{*}}=0 \ \ \ \Rightarrow\Leftarrow.$$
Thus, $L_{R}(\mathcal{G})$ is unital $\iff$ $G^{0}\in\mathcal{G}^{0}$.
\end{proof}

\begin{example}\label{glugl}
Every graph Leavitt path algebra is an ultragraph Leavitt path algebra. This follows from the fact that a graph is an ultragraph. To See this, for $E=\mathcal{G}$, note that $\{p_{v},s_{e},s_{e^{*}}\}$ is a Leavitt $E$-family in $L_{R}(\mathcal{G})$. Thus, there exists an $R$-algebra homomorphism
$$\phi:L_{R}(E)\to L_{R}(\mathcal{G}).$$
On the other hand, for each $A\in\mathcal{G}^{0}=\mathcal{G}^{0}$, define:
$$P_{A}:=\sum\limits_{v\in A}q_{v},\ S_{e}:=t_{e}, \text{ and } S_{e^{*}}:=t_{e^{*}}.$$
One can quickly show $\{P_{A},S_{e},S_{e^{*}}\}$ is a Leavitt $\mathcal{G}$-family in $L_{R}(E)$. And so there exists an $R$-algebra homomorphism 
$$\varphi:L_{R}(\mathcal{G})\to L_{R}(E);$$
it is easily verified $\phi$ and $\varphi$ are inverses of each other, hence, $L_{R}(\mathcal{G})\cong L_{R}(E)$.
\end{example}

\begin{remark}\label{uglbgl}
Let $\mathcal{G}$ be the ultragraph from Example \ref{ugr1}. It turns out $L_{\mathbb{Z}_{2}}(\mathcal{G})$ is not isomorphic to the Leavitt path algebra of any graph. So, just like the $C^{*}$-algebra case, the class of ultragraph Leavitt path algebras properly contains the class of graph Leavitt path algebras. Showing this fact is a bit involved and would require further discussion of Leavitt path algebras; in particular, we would need to know a lot more about their ideal structure. See \cite[Example 6.12]{M.-Imanfar:2017aa} for further details.
\end{remark}

While the class of ultragraph Leavitt path algebras is strictly larger, we do have the following:

\begin{prop}\label{gegiso}
For an ultragraph $\mathcal{G}$, let $E_{\mathcal{G}}$ be the graph from Definition \ref{gfromug}. Suppose $|E_{\mathcal{G}}^{0}|<\infty$, then,
$$L_{R}(\mathcal{G})\cong L_{R}(E_{\mathcal{G}}).$$
\end{prop}

\begin{proof}
 Note that $|E_{\mathcal{G}}^{0}|<\infty \implies |G^{0}|<\infty$. But, by the construction of $E_{\mathcal{G}}$, $|G^{0}|<\infty\implies\Delta=\emptyset$. And so we can identify $E_{\mathcal{G}}^{0}$ with $G^{0}$, and $E_{\mathcal{G}}^{1}$ with the set 
 $$\{(e_{n},v):e_{n}\in\mathcal{G}^{1}, v\in r(e_{n})\}.$$
Define a set $\{Q_{v},T_{e},T_{e^{*}}\}\subseteq L_{R}(\mathcal{G})$ by 
$$Q_{v}:=p_{v},\ T_{(e_{n},v)}:=s_{e_{n}}p_{v} \text { and } T_{(e_{n},v)^{*}}:=p_{v}s_{e_{n}^{*}}.$$
It can quickly  be verified that $\{Q_{v},T_{e},T_{e^{*}}\}$ is a Leavitt $E_{\mathcal{G}}$-family in $L_{R}(\mathcal{G})$; thus, there exists an $R$-algebra homomorphism
$$\phi: L_{R}(E_{\mathcal{G}})\to L_{R}(\mathcal{G})$$
such that $\phi(q_{v})=Q_{v},\ \phi(t_{(e_{n},v)})=T_{(e_{n},v)}, \text{ and } \phi(t_{(e_{n},v)^{*}})=T_{(e_{n},v)^{*}}$. By Proposition \ref{gzgrade}, Theorem \ref{uggrade}, and Theorem \ref{uggrdhomom}, $\phi$ is an injective, $\mathbb{Z}$-graded, homomorphism. Note that $|G^{0}|<\infty$ implies
$|A|<\infty$ for each $A\in\mathcal{G}^{0}$ and $|r(e_{n})|<\infty$ for each $e_{n}\in\mathcal{G}^{1}$. Further, for each $A\in\mathcal{G}^{0}$ and each $e_{n}\in\mathcal{G}^{1}$,
$$p_{A}=\sum\limits_{v\in A}Q_{v},\ s_{e_{n}}=\sum\limits_{v\in r(e_{n})} T_{(e_{n},v)}, \text{ and } s_{e_{n}^{*}}=\sum\limits_{v\in r(e_{n})} T_{(e_{n},v)^{*}};$$
since $L_{R}(\mathcal{G})$ is generated by $\{p_{A},s_{e},s_{e^{*}}\}$, this means $\phi$ is surjective as well. Thus, $L_{R}(\mathcal{G})\cong L_{R}(E_{\mathcal{G}})$.
 \end{proof}
 
Because Morita equivalence is strictly weaker than being isomorphic, Proposition \ref{gegiso} implies $L_{R}(\mathcal{G})$ is Morita equivalent to $L_{R}(E_{\mathcal{G}})$ when $|E_{\mathcal{G}}^{0}|<\infty$. And so it only remains to establish Morita equivalence when $|E_{\mathcal{G}}^{0}|=\infty$; notice that, by Lemma \ref{glpalgunital}, this necessarily means $L_{R}(E_{\mathcal{G}})$ is nonunital. 

While Morita equivalence of rings has many formulations, it is at heart an equivalence of categories. To clearly understand what this means and how the different formulations of Morita equivalence arise, it is useful to review some foundational category theory. 

\section{Category Theory}

\begin{definition}\label{cat}
A \textit{category}, $\mathcal{C}$, consists of:\\
1) a class of \textit{objects}, $\text{ob }\mathcal{C}$;\\
2) for each pair of objects, $(A,B)\in \text{ob }\mathcal{C}\times\text{ob }\mathcal{C}$, a set of \textit{morphisms}, $\text{Hom}_{\mathcal{C}}(A,B)$, we can visualize $f\in\text{Hom}_{\mathcal{C}}(A,B)$ as an arrow $A \xrightarrow{f} B$ (indeed, morphisms are sometimes called ``arrows''), $A$ is the \textit{domain} of $f$, and $B$ is the \textit{codomain};\\
3) further, for each $(A,B), (B,C)\in \text{ob } \mathcal{C}\times\text{ob }\mathcal{C}$, there is a set map (the composition map),  
$$\text{Hom}_{\mathcal{C}}(A,B)\times\text{Hom}_{\mathcal{C}}(B,C)\to\text{Hom}_{\mathcal{C}}(A,C);$$
the image of $(f,g)\in\text{Hom}_{\mathcal{C}}(A,B)\times\text{Hom}_{\mathcal{C}}(B,C)$ is denoted by ``$g\circ f$;''\\
 finally,  the objects and morphisms must satisfy the conditions that,\\
  (\textbf{C1}) for $(A,B),(C,D)\in \text{ob }\mathcal{C}\times\text{ob }\mathcal{C}$,
  $$(A,B)\neq(C,D)\implies \text{Hom}_{\mathcal{C}}(A,B)\cap\text{Hom}_{\mathcal{C}}(C,D)=\emptyset,$$
  (\textbf{C2}) given $f\in\text{Hom}_{\mathcal{C}}(A,B)$, $g\in\text{Hom}_{\mathcal{C}}(B,C)$, and $h\in\text{Hom}_{\mathcal{C}}(C,D)$,
  $$h\circ(g\circ f)=(h\circ g)\circ f,$$
(\textbf{C3}) for each $A\in \text{ob }\mathcal{C}$, there exists a unique element $1_{A}\in \text{Hom}_{\mathcal{C}}(A,A)$ such that,\\
i) for each $B\in\text{ob }\mathcal{C},$ and each $f\in\text{Hom}_{\mathcal{C}}(A,B),$ $f\circ1_{A}=f,$\\
ii) for each $B\in\text{ob }\mathcal{C},$ and each $g\in\text{Hom}_{\mathcal{C}}(B,A),$ $1_{A}\circ g=g$.
\end{definition}

\begin{example}[Category of Groups, $\textbf{Grp}$]\label{group}
Here, $\text{ob }\textbf{Grp}$ is the class of groups. For $(A,B)\in\text{ob }\textbf{Grp}\times\text{ob }\textbf{Grp}$, $\text{Hom}_{\textbf{Grp}}(A,B)$ is the set of group homomorphisms from $A$ to $B$. For $(f,g)\in \text{Hom}_{\textbf{Grp}}(A,B)\times\text{Hom}_{\textbf{Grp}}(B,C)$, $g\circ f\in\text{Hom}_{\textbf{Grp}}(A,C)$ is given by the usual composition of group homomorphisms.
\end{example}

\begin{example}[Category of Rings, $\textbf{Rng}$]\label{ring}
In this case, $\text{ob }\textbf{Rng}$ is the class of rings.  $\text{Hom}_{\textbf{Rng}}(A,B)$ is the set of ring homomorphisms from $A$ to $B$.\\
 For $(f,g)\in \text{Hom}_{\textbf{Rng}}(A,B)\times\text{Hom}_{\textbf{Rng}}(B,C)$, $g\circ f\in\text{Hom}_{\textbf{Rng}}(A,C)$ is given by composition of ring homomorphisms. 
\end{example}

\begin{example}[Category of Pointed Topological Spaces, $\textbf{Top}^{*}$]\label{ptop}
Here, $\text{ob }\textbf{Top}^{*}$ consists of pointed topological spaces. Given $\big((X,x_{0}),(Y,y_{0})\big)\in\text{ob }\textbf{Top}^{*}\times\text{ob }\textbf{Top}^{*}$, $\text{Hom}_{\textbf{Top}^{*}}\big((X,x_{0}),(Y,y_{0})\big)$ consists of continuous maps $f:X\to Y$ such that $f(x_{0})=y_{0}$.\\ For $(f,g)\in \text{Hom}_{\textbf{Top}^{*}}\big((X,x_{0}),(Y,y_{0})\big)\times\text{Hom}_{\textbf{Top}^{*}}\big((Y,y_{0}),(Z,z_{0})\big)$,\\
 $g\circ f\in\text{Hom}_{\textbf{Top}^{*}}\big((X,x_{0}),(Z,z_{0})\big)$ is given by the usual composition of continuous maps. 
\end{example}

To get some idea of how general categories can be, consider the following example.

\begin{example}\label{top}
Let $X$ be a topological space. We can define a category $\mathcal{C}$ as follows:\\
Let $\text{ob }\mathcal{C}$ consist of the points of $X$; that is, $\text{ob }\mathcal{C}:=\{x:x\in X\}$. For $x,y\in \text{ob }\mathcal{C}$, let $\text{Hom}_{\mathcal{C}}(x,y)$ be the set of equivalence classes of paths from $x$ to $y$. For $([f],[g])\in \text{Hom}_{\mathcal{C}}(x,y)\times\text{Hom}_{\mathcal{C}}(y,z)$, $[g]\circ [f]\in\text{Hom}_{\mathcal{C}}(x,z)$ is given by $[g]\circ [f]:=[fg]$, where $[fg]$ is the equivalence class of the concatenation of $f$ by $g$. Finally, for $x\in \text{ob }\mathcal{C}$, $1_{x}$ is the equivalence class of the constant path at $x$. 
\end{example}

\begin{definition}\label{subcat}
Let $\mathcal{C}$ be a category. A \textit{subcategory} of $\mathcal{C}$ is a category, $\mathcal{D}$, such that:\\
1) $\text{ob }\mathcal{D}$ is a subclass of $\text{ob }\mathcal{C}$.\\
2) For each $(A,B)\in \text{ob }\mathcal{D}\times \text{ob }\mathcal{D}$, $\text{Hom}_{\mathcal{D}}(A,B)\subseteq\text{Hom}_{\mathcal{C}}(A,B)$.\\
3) For each $A\in \text{ob }\mathcal{D}$, $1_{A}\in \text{Hom}_{\mathcal{D}}(A,A)$ is the same as $1_{A}\in\text{Hom}_{\mathcal{C}}(A,A)$; and, the composition map 
$$\text{Hom}_{\mathcal{D}}(A,B)\times\text{Hom}_{\mathcal{D}}(B,C)\to\text{Hom}_{\mathcal{D}}(A,C)$$
is the restriction of the composition map
$$\text{Hom}_{\mathcal{C}}(A,B)\times\text{Hom}_{\mathcal{C}}(B,C)\to\text{Hom}_{\mathcal{C}}(A,C).$$
\end{definition}

\begin{example}[Category of Abelian Groups, $\textbf{Ab}$]\label{ab}
$\text{ob }\textbf{Ab}$ is the class of abelian groups, and the morphisms are group homomorphisms. It is straightforward to check $\textbf{Ab}$ is a subcategory of $\textbf{Grp}$.
\end{example}

Recall that every ring has an abelian group structure. Moreover, every ring homomorphism is a group homomorphism as well. Thus, $\textbf{Rng}$ is a subcategory of $\textbf{Ab}$. As a result, \textbf{Rng} is subcategory of $\textbf{Grp}$ as well. However, since not every group homomorphism is a ring homomorphism, $\text{Hom}_{\textbf{Rng}}(A,B)$ is a strict subset of $\text{Hom}_{\textbf{Ab}}(A,B)$. If $\mathcal{D}$ is a subcategory of $\mathcal{C}$ such that, for every $A,B\in\text{ob }\mathcal{D}$, $\text{Hom}_{\mathcal{D}}(A,B)=\text{Hom}_{\mathcal{C}}(A,B)$, we say $\mathcal{D}$ is a \textit{full} subcategory. For example, $\textbf{Ab}$ is a full subcategory of $\textbf{Grp}$.

Much like maps between sets, we have a notion of mapping between categories.

\begin{definition}\label{func}
 Let $\mathcal{C}$ and $\mathcal{D}$ be categories. A \textit{(covariant) functor} from $\mathcal{C}$ to $\mathcal{D}$, 
 $$F:\mathcal{C}\to\mathcal{D},$$
 consists of:\\
 1) a mapping, $A\mapsto F(A)$, from $\text{ob }\mathcal{C}$ to $\text{ob }\mathcal{D}$,\\
 2) a mapping, $f\mapsto F(f)$, from $\text{Hom}_{\mathcal{C}}(A,B)$ to $\text{Hom}_{\mathcal{D}}(F(A), F(B))$.\\
 Further, the following conditions must be satisfied.\\
 \textbf{F1:} For every $A\in\text{ob }\mathcal{C}$, 
 $$F(1_{A})=1_{F(A)}\in \text{Hom}_{\mathcal{D}}(F(A), F(A)).$$\
 \textbf{F2:} For every $(f,g)\in\text{Hom}_{\mathcal{C}}(A,B)\times\text{Hom}_{\mathcal{C}}(B,C)$, and $g\circ f\in\text{Hom}_{\mathcal{C}}(A,C)$,
 $$F(g\circ f)=F(g)\circ F(f)\in\text{Hom}_{\mathcal{D}}(F(A),F(C)).$$
\end{definition}

We can also compose functors. Let $\mathcal{C}$, $\mathcal{D}$, and $\mathcal{E}$ be categories with functors $F:\mathcal{C}\to\mathcal{D}$ and $G:\mathcal{D}\to\mathcal{E}$. Then, we get a functor, $G\circ F:\mathcal{C}\to\mathcal{E}$, where $G\circ F(A)=G(F(A))$ for each $A\in\text{ob }\mathcal{C}$, and $G\circ F(f)=G(F(f))$ for each $f\in\text{Hom}_{\mathcal{C}}(A,B)$. 

\begin{example}
For any category $\mathcal{C}$, the \textit{indentity functor}, $1_{\mathcal{C}}:\mathcal{C}\to\mathcal{C}$, consists of: the identity mapping on $\text{ob }\mathcal{C}$; the identity mapping on $\text{Hom}_{\mathcal{C}}(A,B)$ for each $A,B\in\text{ob }\mathcal{C}$.
\end{example}

\begin{example}
Given a pointed topological space $(X,x_{0})$, let $\pi_{1}(X,x_{0})$ denote the fundamental group of $X$ with base point $x_{0}$. There is a functor $F:\textbf{Top}^{*}\to\textbf{Grp}$ consisting of a map from $\text{ob }\textbf{Top}^{*}$ to $\text{ob }\textbf{Grp}$, given by $(X,x_{0})\mapsto\pi_{1}(X,x_{0})$, and a map from $\text{Hom}_{\textbf{Top}^{*}}\big((X,x_{0}),(Y,y_{0})\big)$ to $\text{Hom}_{\textbf{Grp}}\big(\pi_{1}(X,x_{0}),\pi_{1}(Y,y_{0})\big)$, given by $f\mapsto F(f)$ where $F(f):\pi_{1}(X,x_{0})\to\pi_{1}(Y,y_{0})$ is the group homomorphism given by $[\gamma]\mapsto[f\circ\gamma]$.
\end{example}

As we stated earlier, Morita equivalence is a statement regarding the equivalence of categories. We need the following definition in order to precisely say what an equivalence of categories is.

\begin{definition}\label{nattran}
Let $\mathcal{C}$ and $\mathcal{D}$ be categories, and $F,G:\mathcal{C}\to\mathcal{D}$ functors. A \textit{natural transformation}, $\eta$, from $F$ to $G$ is a map which assigns a morphism, $\eta_{A}\in \text{Hom}_{\mathcal{D}}(F(A),G(A))$, for each $A\in\text{ob }\mathcal{C}$ such that, given any two $A,B\in\text{ob }\mathcal{C}$ and $f\in\text{Hom}_{\mathcal{C}}(A,B)$, the following diagram commutes:

\begin{figure}[h!]
\begin{center}
\begin{tikzpicture}

 \node [shape=circle,minimum size=1.5em] (d1) at (0,1) {$F(A)$};
 \node [shape=circle,minimum size=1.5em] (d2) at (4,1) {$G(A)$};
 \node [shape=circle,minimum size=1.5em] (d3) at (0,-1) {$F(B)$};
 \node [shape=circle,minimum size=1.5em] (d4) at (4,-1) {$G(B)$};
 
\path (d1) edge [->, >=latex, shorten <= 2pt, shorten >= 2pt, right] node[above]{$\eta_{A}$} (d2);

\path (d3) edge [->, >=latex, shorten <= 2pt, shorten >= 2pt, right] node[above]{$\eta_{B}$} (d4);

 \node [shape=circle,minimum size=1.5em] (d6) at (0,-1) {};

 \path (d1) edge [->, >=latex, shorten <= -10pt, right] node[pos=0.2]{$F(f)$} (d6);
 
 \node [shape=circle,minimum size=1.5em] (d7) at (4,-1) {};
 
  \path (d2) edge [->, >=latex, shorten <= -10pt, right] node[pos=0.2]{$G(f)$} (d7);

\end{tikzpicture}
\end{center}
\end{figure}
\end{definition}

Given a morphism $f\in\text{Hom}_{\mathcal{C}}(A,B)$, $f$ is an \textit{isomorphism} if there exists a morphism $g\in\text{Hom}_{\mathcal{C}}(B,A)$ such that $g\circ f=1_{A}$ and $f\circ g=1_{B}$; note this means $g$ is an isomorphism as well. If $\eta_{A}$ in Definition \ref{nattran} is an isomorphism for each $A\in\text{ob} \mathcal{C}$, we say $\eta$ is a \textit{natural isomorphism}. If $F,G:\mathcal{C}\to\mathcal{D}$  are functors for which there is a natural isomorphism, we say they are \textit{naturally isomorphic}, which we will denote this by writing ``$F\simeq G$.''

\begin{definition}\label{catequiv}
Two categories, $\mathcal{C}$ and $\mathcal{D}$, are $\textit{equivalent}$ if there exist functors\\
 $F:\mathcal{C}\to\mathcal{D}$ and $G:\mathcal{D}\to\mathcal{C}$ such that $G\circ F\simeq 1_{\mathcal{C}} \text{ and } F\circ G\simeq 1_{\mathcal{D}}$.
\end{definition}

\section{Morita Equivalence of Nonunital Rings}

Morita equivalence of rings is the equivalence of module categories. The theory was first developed within the context of unital rings. If the reader wishes to learn more about Morita equivalence for unital rings, section 3.12 of Nathan Jacobson's \textit{Basic Algebra II}, \cite{Jacobson:2009aa}, is an excellent source. The theory was then extended to certain nonunital rings by numerous mathematicians over several decades. In our discussion of Morita equivalence, we will solely rely on Jacobson, \cite{Abrams:1983aa}, \cite{P.N.-Anh:1987aa}, and \cite{Jose-Luis-Garcia-and-Juan-Jacobo-Simon:1991aa}. We will start our brief foray into this vast topic by laying the basic foundation.

\begin{definition}\label{rmod}
Let $R$ be a ring. An abelian group, $M$, is a \textit{left $R$-module} if there exists a map $\cdot:R\times M\to M$ such that:\\
1) $r\cdot(x+y)=r\cdot x+ r\cdot y$ for all $r\in R$ and all $x,y\in M$,\\
2) $(r_{1}+r_{2})\cdot x=r_{1}\cdot x+r_{2}\cdot x$ for all $r_{1},r_{2}\in R$ and for all $x\in M$,\\
3) $(r_{1}r_{2})\cdot x=r_{1}\cdot(r_{2}\cdot x)$ for all $r_{1},r_{2}\in R$ and for all $x\in M$.\\
If $R$ is unital, then we also require\\
4) $1\cdot x=x$ for all $x\in M$.
\end{definition}

\begin{definition}\label{rmodhom}
Let $M$ and $N$ be left $R$-modules. A map, $f:M\to N$, is a \textit{left $R$-module homomorphism} if
$$f(r_{1}\cdot x+r_{2}\cdot y)=r_{1}\cdot f(x)+r_{2}\cdot f(y)$$
for all $r_{1},r_{2}\in R$ and for all $x,y\in M.$
\end{definition}

We denote the set of $R$-module homomorphisms from $M$ to $N$ by ``$\text{Hom}_{R}(M,N)$,'' and ``$\text{End}_{R}(M)$'' denotes the set of $R$-modules endomorphisms of $M$. One can check $\text{End}_{R}(M)$ is a ring under pointwise addition and with multiplication given by composition. Given a ring $R$, the category of left $R$-modules, \textbf{R-Mod}, consists of left $R$-modules as objects and $R$-module homomorphisms as morphisms. 

We similarly define a \textit{right $R$-module} to be an abelian group $M$ equipped with a map 
$$\cdot:M\times R\to M$$
satisfying similar axioms to those in Definition \ref{rmod}, of course, the difference being elements of $R$ are now acting on the right. A \textit{right $R$-module homomorphism} is then a map, $f:M\to N$, where $M$ and $N$ are right $R$-modules such that 
$$f(x\cdot r_{1}+y\cdot r_{2})=f(x)\cdot r_{1}+f(y)\cdot r_{2} \text{ for all } r_{1},r_{2}\in R \text{ and for all } x,y\in M.$$
Much like \textbf{R-Mod}, the category of right $R$-modules, \textbf{Mod-R}, consists of right $R$-modules as objects and right $R$-module homomorphisms as morphisms. 

Throughout this paper, however, we will take $R$-modules to be left $R$-modules unless specified otherwise. Further, given an $R$-module $M$, we will write ``$rx$'' to denote $r\cdot x$.

\begin{example}
Let $R$ be a ring. Then, $R$ is a left, and right, $R$-module over itself where $\cdot: R\times R\to R$ is given by the ring multiplication. 
\end{example}

\begin{example}
A vector space, $V$, over a field, $K$, is just a $K$-module. In fact, modules are a generalization of vector spaces.
\end{example}

\begin{example}
Let $R=M_{n}(\mathbb{Z})$, the ring of $n\times n$ matrices over $\mathbb{Z}$, and let $M=M_{n\times 1}(\mathbb{Z})$, the group of $n\times1$ matrices over $\mathbb{Z}$. Then, $M$ is a left $R$-module with $\cdot: R\times M\to M$ given by matrix multiplication. Note that $M$ is not a right $R$-module under matrix multiplication. In general, having a left $R$-module structure doesn't automatically give a right $R$-module structure. If $R$ is a commutative ring, a left $R$-module $M$ can be made into a right $R$-module but setting $m\cdot r:=rm$; this fails when $R$ is noncommutative since 3) of Definition \ref{rmod} won't hold. However, taking $m\cdot r:=rm$ does give $M$ a right $R^{op}$-module structure (see Definition \ref{rop} for an explanation of ``$R^{op}$'').
\end{example}

If $R$ is a unital ring, then 4) of Definition \ref{rmod} implies
$$RM:=\{rx:r\in R,\ x\in M\}=M.$$
If $R$ is a nonunital ring, then it isn't always the case $RM=M$. The following simple example illustrates this point.

\begin{example}\label{rmodex}
Let $R=2\mathbb{Z}$. Viewed as a $2\mathbb{Z}$-module over itself,
$$(2\mathbb{Z})(2\mathbb{Z})=4\mathbb{Z}\neq2\mathbb{Z}.$$
\end{example}

Moreover, 4) of Definition \ref{rmod} again implies that, in the unital case,
$$Rx:=\{rx:r\in R\}=\{0\} \implies x=0;$$
this also need not hold when $R$ is nonunital (e.g., take $R=2\mathbb{Z}$, $M=\mathbb{Z}/2\mathbb{Z}$, and $x=\overline{1}$).

For a nonunital ring $R$, an $R$-module $M$ is said to be \textit{unital} if $RM=M$; it is said to be \textit{nondegenerate} if for any $x\in M$, $Rx=\{0\}\implies x=0$. Throughout this paper, $R-MOD$ denotes the full subcategory of \textbf{$R$-Mod} whose objects are unital nondegenrate left $R$-modules. Note if $R$ is unital, then $R-MOD$ is the same category as \textbf{$R$-Mod}. We define $MOD-R$ in the same way. Now, finally,

\begin{definition}\label{Moequiv}
Two rings, $R$ and $S$, are \textit{Morita equivalent} if $R-MOD$ is equivalent to $S-MOD$. It can be shown, certainly in the cases we are interested in, $R-MOD$ is equivalent to $S-MOD$ if and only if $MOD-R$ is equivalent to $MOD-S$, \cite[Corollary 2.3]{P.N.-Anh:1987aa}
\end{definition}

There are several ways Morita equivalence can be characterized for certain types of rings. We are particularly interested in the case where $R$ is a ring with local units.

\begin{definition}\label{ring:def2}
$R$ is a \textit{ring with local units} if, for each finite subset $S\subseteq R$, there exists an idempotent $e\in R$ such that $S\subseteq eRe$. Note that $eRe$ is a subring of $R$. 
\end{definition}

It's worth clarifying that the results in \cite{Jose-Luis-Garcia-and-Juan-Jacobo-Simon:1991aa} are proved for idempotent rings; a ring is \textit{idempotent} if 
$$R^{2}:=\{rs: r,s\in R\}=R.$$
It is straightforward to see that a ring with local units is also an idempotent ring. And so the results in \cite{Jose-Luis-Garcia-and-Juan-Jacobo-Simon:1991aa} apply to rings with local units as well. With that, we work toward the different characterizations of Morita equivalence. 

\begin{definition}\label{balprod}
Let $R$ be a ring, $M$ a right $R$-module, and $N$ a left $R$-module. An \textit{$R$-balanced product of $M$ and $N$} is a pair, $(P,f)$, where $P$ is an abelian group, and $f:M\times N\to P$ a map such that:\\
1) $f(x_{1}+x_{2},y)=f(x_{1},y)+f(x_{2},y)$ for all $x_{1},x_{2}\in M$ and $y\in N$,\\
2) $f(x,y_{1}+y_{2})=f(x,y_{1})+f(x,y_{2})$ for all $x\in M$ and $y_{1},y_{2}\in N$,\\
3) $f(xr,y)=f(x,ry)$ for all $x\in M$, $y\in N$, and $r\in R$.
\end{definition}

\begin{definition}\label{tensprod}
Let $R$ be a ring, $M$ a right $R$-module, and $N$ a left $R$-module. The \textit{tensor product of $M$ and $N$ over $R$} is an  $R$-balanced product of $M$ and $N$, 
$$(M\otimes_{R} N, \otimes),$$
such that, given any other $R$-balanced product, $(P,f)$, there exists a unique group homomorphism, $\phi:M\otimes_{R} N\to P$, making the following diagram commute:

\begin{figure}[h!]
\begin{center}
\begin{tikzpicture}
  \matrix (m) [matrix of math nodes,row sep=4em,column sep=2em,minimum width=2em]
  {
      & M\otimes_{R} N \\
     M\times N & P \\};
  \path[-stealth]
  
    (m-2-1) edge node [below] {$f$}
            node [above] {} (m-2-2)
             edge node [above] {$\otimes$} (m-1-2)
    (m-1-2) edge [dashed,->] node [right] {$\phi$} (m-2-2);
\end{tikzpicture}
\end{center}
\end{figure}
\end{definition}

To construct $M\otimes_{R} N$, we start with the \textit{free abelian group on $M\times N$}, which we will denote by ``$F_{ab}(M\times N)$.'' We can think of the elements of $F_{ab}(M\times N)$ as commuting formal sums of the form
$$\sum\limits_{i=1}^{n}(x_{i},y_{i})$$
where $(x_{i},y_{i})\in M\times N$ for each $i$. Further, $F_{ab}(M\times N)$ has the universal property that, given any abelian group, $P$, and a set map, $f:M\times N\to P$, there exists a unique group homomorphism $\widetilde{\phi}: F_{ab}(M\times N)\to P$ making the following diagram commute:

\clearpage

\begin{figure}[h!]
\begin{center}
\begin{tikzpicture}
  \matrix (m) [matrix of math nodes,row sep=4em,column sep=2em,minimum width=2em]
  {
      & F_{ab}(M\times N) \\
     M\times N & P \\};
  \path[-stealth]
  
    (m-2-1) edge node [below] {$f$}
            node [above] {} (m-2-2)
             edge node [above] {$i$} (m-1-2)
    (m-1-2) edge [dashed,->] node [right] {$\widetilde{\phi}$} (m-2-2);
\end{tikzpicture}
\end{center}
\end{figure}
In the diagram above, $i:M\times N\to F_{ab}(M\times N)$ is the inclusion of $M\times N$ into $F_{ab}(M\times N)$. Taking $G$ to be the subgroup of $F_{ab}(M\times N)$ generated by elements of the form $(x_{1}+x_{2},y)-(x_{1},y)-(x_{2},y)$, $(x,y_{1}+y_{2})-(x,y_{1})-(x,y_{2})$, and $(xr,y)-(x,ry)$, where $x_{1},x_{2}\in M$, $y_{1},y_{2}\in N$, and $r\in R$, we set $M\otimes_{R} N:=F_{ab}(M\times N)/G,$ and, given the projection $p:F_{ab}(M\times N)\to M\otimes_{R} N$, we set $\otimes:M\times N\to M\otimes_{R} N$ to be the group homomorphism $p\circ i$. For $(x,y)\in M\times N$, we write ``$x\otimes y$'' to denote $\otimes(x,y)$. Elements of the form $x\otimes y$ are referred to as \textit{simple tensors}. Each element of $M\otimes_{R} N$ can be expressed as the sum of simple tensors; inconveniently, this expression need not be unique. From here on out, we will simply write ``$M\otimes_{R} N$'' instead of ``$(M\otimes_{R} N,\otimes)$.'' If it's clear which ring we are taking the tensor product over, we will simply write ``$M\otimes N$.''

While $M\otimes_{R} N$ is an abelian group, it can be endowed with a module structure under certain circumstances.

\begin{definition}\label{bimod}
Let $R$ and $S$ be rings. An \textit{$R$-$S$-bimodule} is an abelian group $M$ such that:\\
1) $M$ is a left $R$-module,\\
2) $M$ is a right $S$-module,\\
3) $r(xs)=(rx)s$ for all $r\in R$, $s\in S$, and $x\in M$.
\end{definition}

$M$ is \textit{unital} if $RM=MS=M$. Further, given two $R$-$S$-bimodules $M$ and $N$, a map $f:M\to N$ is an \textit{$R$-$S$-bimodule homomorphism} if it is simultaneously a left $R$-module homomorphism and a right $S$-module homomorphism. We have the following standard proposition relating tensor products, modules, and bimodules:

\begin{prop}\label{tensmod}
Let $R$, $S$, and $T$ be rings. Then:
1) $M\otimes_{R} N$ is a left $S$-module, with $s\cdot(x\otimes y)=(sx)\otimes y$ on simple tensors, if $M$ is an $S$-$R$-bimodule,\\
2) $M\otimes_{R} N$ is a right $T$-module, with $(x\otimes y)\cdot t=x\otimes (yt)$ on simple tensors, if $N$ is an $R$-$T$-bimodule. 
\end{prop}

The first different formulation of Morita equivalence we will give relies on the existence of a \textit{Morita context}.

\begin{definition}\label{mocont}
A \textit{Morita context} is a sextuple, $(R,S,M,N,\tau,\mu)$, where $R$ and $S$ are rings, $M$ is a unital $S$-$R$-bimodule, $N$ is a unital $R$-$S$-bimodule, $\tau:N\otimes_{S} M\to R$ is an $R$-$R$-bimodule homomorphism, and $\mu:M\otimes_{R} N\to S$ is an $S$-$S$-bimodule homomorphism such that:\\
1) $\mu(x\otimes y)x'=x\tau(y\otimes x')$ for all $x,x'\in M$ and $y\in N$,\\
2) $y\mu(x\otimes y')=\tau(y\otimes x)y'$ for all $x\in M$ and $y,y'\in N$.\\
\end{definition}

A Morita context $(R,S,M,N,\tau,\mu)$ is of particular interest when $\tau$ and $\mu$ are surjective; in fact, it isn't uncommon to find the surjectivity condition included in the definition of a Morita context. 

\begin{theorem}\label{Mor1}
Let $R$ and $S$ be idempotent rings. Then, $R$ is Morita equivalent to $S$ if and only if there exists a Morita context, $(R,S,M,N,\tau,\mu)$, with $\tau$ and $\mu$ surjective. 
\end{theorem}

\begin{proof}
\cite[Propositions 2.5 and 2.7]{Jose-Luis-Garcia-and-Juan-Jacobo-Simon:1991aa}
\end{proof}

Consider the following example, \cite{Jacobson:2009aa}.

\begin{example}\label{moreqex}
Consider a unital commutative ring $R$ and $M_{n}(R)$. Under the usual rules of matrix multiplication, and multiplying matrices by scalars, we can see $M_{n\times 1}(R)$ is an $M_{n}(R)$-$R$-bimodule; likewise, $M_{1\times n}(R)$ is an $R$-$M_{n}(R)$-bimodule. Now, define maps $\mu: M_{n\times 1}(R)\times M_{1\times n}(R) \to M_{n}(R)$ and $\tau: M_{1\times n}(R)\times M_{n\times 1}(R) \to R$ by

\begin{equation*}
(\begin{bmatrix}x_{1}\\
\vdots\\
 x_{n} \end{bmatrix},
 \begin{bmatrix}y_{1} & \ldots & y_{n} \end{bmatrix})\mapsto
 \begin{bmatrix} x_{1}y_{1} & \ldots & x_{1}y_{n}\\
  \vdots & \ddots &\vdots \\
  x_{n}y_{1} & \ldots & x_{n}y_{n} \end{bmatrix},
\end{equation*}

and

\begin{equation*}
(\begin{bmatrix}y_{1} & \ldots & y_{n} \end{bmatrix},
\begin{bmatrix}x_{1}\\
\vdots\\
 x_{n} \end{bmatrix})\mapsto \sum\limits_{i=1}^{n}x_{i}y_{i}.
\end{equation*}

One can easily check $(M_{n}(R),\mu)$ and $(R,\tau)$ are $R$-balanced products with each map being surjective. So, by exploiting the universal mapping property of tensor products, we have surjective group homomorphisms 
$$\mu: M_{n\times 1}(R)\otimes_{R} M_{1\times n}(R)\to M_{n}(R) \text{ and }\tau: M_{1\times n}(R)\otimes_{M_{n}(R)} M_{n\times 1}(R) \to R.$$
In fact, $\mu$ and $\tau$ are, respectively, $M_{n}(R)$-$M_{n}(R)$ and $R$-$R$-bimodule homomorphisms satisfying the conditions listed in Definition \ref{mocont}. Thus,\\ 
$\big(R,M_{n}(R),M_{n\times 1}(R),M_{1 \times n}(R),\tau,\mu\big)$ is a Morita context with $\mu$ and $\tau$ surjective; by Theorem \ref{Mor1}, this means $R$ and $M_{n}(R)$ are Morita equivalent.  
\end{example}

It may be tedious but it isn't difficult to show being isomorphic implies Morita equivalence---this can be done working directly from definitions, without having to establish a Morita context. On the other hand, Example \ref{moreqex} illustrates Morita equivalence is strictly weaker than being isomorphic. Still, Morita equivalence is of interest in part because it preserves certain ideal structures.

\begin{definition}\label{poset}
A \textit{partially ordered set}, $(P,\leq)$, is a set, $P$, together with a binary relation, $\leq$, such that:\\
1) $x\leq x$ for all $x\in P$,\\
2) for all $x,y\in P$, $x\leq y\text{ and }y\leq x\implies x=y$,\\
3) for all $x,y,z\in P$, $x\leq y \text{ and } y\leq z\implies x\leq z$.
\end{definition}

Given two partially ordered sets, $(P,\leq_{p})$ and $(Q,\leq_{q})$, a map $f:P\to Q$ is \textit{order preserving}  if $x\leq_{p}y\implies f(x)\leq_{q} f(y)$ for all $x,y\in P$. Let $(P,\leq)$ be a partially ordered set and $S\subseteq P$. Then, $x\in P$ is a \textit{lower bound} of $S$ if $x\leq y$ for all $y\in S$; $x$ is the \textit{greatest lower bound} if, for any other lower bound $z$ of $S$, $z\leq x$. Similarly, $x\in P$ is an \textit{upper bound} of $S$ if $y\leq x$ for all $y\in S$; it is the \textit{least upper bound} if, for any other upper bound $z$ of $S$, $x\leq z$. Note that 2) of Definition \ref{poset} implies least upper bounds, and greatest lower bounds, are unique should they exist. 

\begin{definition}\label{lattice}
A \textit{lattice} is a partially ordered set, $(P,\leq)$, such that every two-element subset, $\{x,y\}\subseteq P$, has a greatest lower bound and a least upper bound. A \textit{lattice homomorphism} is an order preserving map between two lattices.   
\end{definition}

Let $R$ be a ring and $\mathcal{I}:=\{I: I\text{ is an ideal of }R\}$. Under inclusion, $(\mathcal{I},\subseteq)$ is a partially order set. Moreover, for any two $I,J\in\mathcal{I}$, $I\cap J$ is the greatest lower bound, and $I+J$ is the least upper bound, of $\{I,J\}$; $\mathcal{I}$ is the \textit{lattice of ideals} of $R$. 

\begin{prop}\label{Morid}
Let $R$ and $S$ be two Morita equivalent rings, and let\\
$(R,S,M,N,\tau,\mu)$ be a Morita context. Set $\mathcal{I}_{R}:=\{I\lhd R: RIR=I\}$, and $\mathcal{I}_{S}:=\{I\lhd S: SIS=I\}$. Then, the map  $I\mapsto \mu(MI\otimes_{R} N)$ defines a lattice isomorphism from $\mathcal{I}_{R}$ to $\mathcal{I}_{S}$ with the inverse given by $I\mapsto \tau(NI\otimes_{S} M)$.
\end{prop}

\begin{proof}
\cite[Proposition 4.12]{Tomforde:2011aa}
\end{proof}

In Proposition \ref{Morid}, $MI\otimes_{R} N:=\text{span}_{S}\{xr\otimes_{R} y: r\in I\lhd R, x\in M, \text{ and } y\in N\}$; similarly $NI\otimes_{S} M:=\text{span}_{R}\{ys\otimes_{S} x: s\in I\lhd S, x\in M, \text{ and } y\in N\}$. Also, note that if $R$ is a ring with local units, every ideal $I$ in $R$ satisfies the condition $RIR=I$. While the study of Morita equivalence is extensive in its own right, we will conclude our discussion of Morita equivalence by setting up and stating one last characterization of it. A crucial component of our last characterization is the categorical notion of direct limits. 

\begin{definition}\label{directsys}
Let $(I,\leq)$ be a partially ordered set, and let $\mathcal{C}$ be a category. A \textit{direct system in $\mathcal{C}$ over $I$} is a set of objects, $\{A_{\alpha}\}_{\alpha\in I}\subseteq \text{ob }\mathcal{C}$, and a set of morphisms, 
$$\{\varphi_{\alpha,\beta}:A_{\alpha}\to A_{\beta}\}_{\alpha,\beta\in I: \alpha\leq\beta},$$
such that:\\
1) $\varphi_{\alpha,\gamma}=\varphi_{\beta,\gamma}\circ\varphi_{\alpha,\beta}$ for all $\alpha\leq\beta\leq\gamma$,\\
2) $\varphi_{\alpha,\alpha}=1_{A_{\alpha}}$ for each $\alpha$.\\
We will write ``$\big\langle A_{\alpha}, \varphi_{\alpha,\beta}\big\rangle$'' to denote direct systems.
\end{definition}

\begin{definition}\label{directlimit}
Let $(I,\leq)$ be a partially ordered set, and let $\mathcal{C}$ be a category. Given a direct system, $\big\langle A_{\alpha}, \varphi_{\alpha,\beta}\big\rangle$, its \textit{direct limit}, should it exist, is an object, $\varinjlim\limits_{\alpha\in I}A_{\alpha}\in \text{ob }\mathcal{C}$, together with a set of morphisms, 
$$\{\eta_{\alpha}:A_{\alpha}\to\varinjlim\limits_{\alpha\in I}A_{\alpha}\}_{\alpha\in I},$$
such that:\\
1) $\eta_{\alpha}=\eta_{\beta}\circ\varphi_{\alpha,\beta}$ for all $\alpha\leq\beta$,\\
2) Given any object $B\in\text{ob }\mathcal{C}$ and a set of morphisms $\{\zeta_{\alpha}:A_{\alpha}\to B\}_{\alpha\in I}$ satisfying $\zeta_{\alpha}=\zeta_{\beta}\circ\varphi_{\alpha,\beta}$ for all $\alpha\leq\beta$, there exists a unique morphism
$$\theta:\varinjlim\limits_{\alpha\in I}A_{\alpha}\to B$$
such that $\zeta_{\alpha}=\theta\circ\eta_{\alpha}$ for all $\alpha$. 
\end{definition}
Part 2) of Definition \ref{directlimit} ensures that, when it exists, $\varinjlim\limits_{\alpha\in I}A_{\alpha}$ is unique up to isomorphism. Fortunately for us, the direct limit always exists in the categories we will consider.

\begin{definition}\label{compset}
Let $(I,\leq)$ be a partially ordered set, and let $\mathcal{C}$ be a category. A \textit{compatible set in $\mathcal{C}$ over $I$}, $\{A_{\alpha},\varphi_{\alpha,\beta},\psi_{\beta,\alpha},I\}$, consists of a set of objects, $\{A_{\alpha}\}_{\alpha\in I}$, and a set of morphisms, 
$$\{\varphi_{\alpha,\beta}:A_{\alpha}\to A_{\beta},\ \psi_{\beta,\alpha}:A_{\beta}\to A_{\alpha}\}_{\alpha,\beta\in I},$$
such that:\\
1) $\varphi_{\alpha,\alpha}=\psi_{\alpha,\alpha}=1_{A_{\alpha}}$ for all $\alpha$,\\
2) $\varphi_{\beta,\gamma}\circ\varphi_{\alpha,\beta}=\varphi_{\alpha,\gamma}$ and $\psi_{\beta,\alpha}\circ\psi_{\gamma,\beta}=\psi_{\gamma,\alpha}$ for all $\alpha\leq\beta\leq\gamma$,\\
3) $\psi_{\beta,\alpha}\circ\varphi_{\alpha,\beta}=1_{A_{\alpha}}$ for all $\alpha\leq\beta$,\\
4) For all $\alpha,\beta,\gamma\in I$ such that $\alpha,\beta\leq\gamma$,
 $$\varphi_{\beta,\gamma}\circ\psi_{\gamma,\beta}\circ\varphi_{\alpha,\gamma}\circ\psi_{\gamma,\alpha}=\varphi_{\alpha,\gamma}\circ\psi_{\gamma,\alpha}\circ\varphi_{\beta,\gamma}\circ\psi_{\gamma,\beta}.$$
\end{definition}

Note that, given a compatible set $\{A_{\alpha},\varphi_{\alpha,\beta},\psi_{\beta,\alpha},I\}$, $\big\langle A_{\alpha}, \varphi_{\alpha,\beta}\big\rangle$ forms a direct system. Our next, and last, formulation of Morita equivalence relies on the existence of certain direct limits and compatible sets.  

\begin{definition}\label{fingen}
Let $R$ be a ring. An $R$-module, $M$, is a \textit{finitely generated} module if there exists a finite set, $X\subseteq M$, such that 
 $$M=\text{span}\{rx:r\in R,\ x\in X\}.$$
 \end{definition}

\begin{definition}\label{projmod}
An $R$-module, $P$, is \textit{projective} if for every surjective $R$-module homomorphism, $f:N\to M$, and any $R$-module homomorphism, $g:P\to M$, there exists an $R$-module homomorphism, $h:P\to N$ making the diagram 

\begin{center}
\begin{tikzpicture}
  \matrix (m) [matrix of math nodes,row sep=4em,column sep=2em,minimum width=2em]
  {
      & N \\
     P & M \\};
  \path[-stealth]
  
    (m-2-1) edge node [below] {$g$}
            node [above] {} (m-2-2)
             edge [dashed,->] node [left] {$h$} (m-1-2)
    (m-1-2) edge  node [right] {$f$} (m-2-2);
\end{tikzpicture}
\end{center}
commute.
\end{definition}

\begin{definition}\label{rmodgen}
An $R$-module, $M$, is a \textit{generator of $R-MOD$}, if for any $R$-module, $N$, there exists a set $I$ and a surjective $R$-module homomorphism
$$\rho:\bigoplus\limits_{i\in I}M\to N.$$
\end{definition}

The notion of a generator is actually categorical. Given a category $\mathcal{C}$, we say $G\in\text{ob }\mathcal{C}$ is a \textit{generator of $\mathcal{C}$} if, for any $f,g\in\text{Hom}_{\mathcal{C}}(A,B)$ with $f\neq g$, there exists $h\in\text{Hom}_{\mathcal{C}}(G,A)$ such that $f\circ h\neq g\circ h$. Generators may not exist within a category in general. However, for \textit{abelian} categories (see Definition 6.2 in Jacobson's \textit{Basic Algebra II}), which admit \textit{products} and \textit{coproducts}, not only do generators exist, but one finds the following equivalent definition: let $\mathcal{C}$ be an abelian category, then $G\in\text{ob }\mathcal{C}$ is a \textit{generator of $\mathcal{C}$} if, given any $A\in\text{ob }\mathcal{C}$, there exists a set $I$ such that the coproduct, $\bigoplus\limits_{i\in I}G$, admits an \textit{epimorphism}
$$\rho:\bigoplus\limits_{i\in I}G\to A;$$
which is to say, the morphism $\rho\in\text{Hom}_{\mathcal{C}}\Big(\bigoplus\limits_{i\in I}G,A\Big)$ is such that, given any $f,g\in\text{Hom}_{\mathcal{C}}(A,B)$, 
$$f\circ\rho=g\circ\rho\implies f=g.$$
Since, as a \textit{Grothendieck} category, $R-MOD$ is abelian \cite{Jose-Luis-Garcia-and-Juan-Jacobo-Simon:1991aa}, Definition \ref{rmodgen} is appropriate. Because we will make use of the equivalence of the two definitions only once in this thesis, we will not go through the trouble of giving the rather laborious definition of an abelian category. We will, however, give a standard proof showing the equivalence of the two definitions. The following proof is from a lecture note posted online which the author cannot seem to relocate for proper citation.

\begin{claim}\label{defabequiv}
Let $\mathcal{C}$ be an abelian category. For $G\in\text{ob }\mathcal{C}$, the following definitions are equivalent:\\
 1) for any $f,g\in\text{Hom}_{\mathcal{C}}(A,B)$ with $f\neq g$, there exists $h\in\text{Hom}_{\mathcal{C}}(G,A)$ such that $f\circ h\neq g\circ h$;\\
 2) given any $A\in\text{ob }\mathcal{C}$, there exists a set $I$ with the coproduct, $\bigoplus\limits_{i\in I}G$, admitting an \textit{epimorphism} $\rho:\bigoplus\limits_{i\in I}G\to A$.
\end{claim}

\begin{proof}
$2) \Rightarrow 1)$ Suppose $G\in\text{ob }\mathcal{C}$ is such that, for any $A\in\text{ob }\mathcal{C}$, there exists a set $I$ and an epimorphism $\rho\in\text{Hom}_{\mathcal{C}}\Big(\bigoplus\limits_{i\in I}G,A\Big)$. Let $f,g\in\text{Hom}_{\mathcal{C}}(A,B)$ be such that $f\neq g$. Suppose $f\circ\rho=g\circ\rho$. Since we are in an abelian category, we can carry out the following calculations: $f\circ\rho=g\circ\rho\implies (f-g)\circ\rho=0\circ\rho;$ but, because $\rho$ is an epimorphism, $(f-g)\circ\rho=0\circ\rho\implies f-g=0\implies f=g. \Rightarrow\!\Leftarrow$ Thus,  $f\circ\rho\neq g\circ\rho$.

$1) \Rightarrow 2)$ On the other hand, suppose $G\in\text{ob }\mathcal{C}$ is such that, for any $f,g\in\text{Hom}_{\mathcal{C}}(A,B)$, with $f\neq g$, there exists $h\in\text{Hom}_{\mathcal{C}}(G,A)$ satisfying $f\circ h\neq g\circ h$. Let $I:=\text{Hom}_{\mathcal{C}}(G,A).$ Since coproducts exist in abelian categories, we have a family of morphisms, $\{\phi_{j}:G\to \bigoplus\limits_{i\in I}G\}_{j\in I},$ and a morphism, $\rho:\bigoplus\limits_{i\in I}G\to A,$ such that $\rho\circ\phi_{i}=i$ for each $i\in I$. Now, suppose $f,g\in\text{Hom}_{\mathcal{C}}(A,B)$ such that $f\circ\rho=g\circ\rho$ but $f\neq g$. By our assumption, there exists $h\in\text{Hom}_{\mathcal{C}}(G,A)$ with $f\circ h\neq g\circ h$. But, $h=\rho\circ\phi_{h}$, and so $f\circ\rho=g\circ\rho\implies f\circ\rho\circ\phi_{h}= g\circ\rho\circ\phi_{h}\implies f\circ h=g\circ h.\Rightarrow\!\Leftarrow$ Thus, $f\circ\rho=g\circ\rho\implies f=g$, meaning $\rho$ is an epimorphism. 
\end{proof}

\begin{definition}\label{ring:def5}
An $R$-module, $M$, is \textit{locally projective} if there exists a compatible set $\{M_{\alpha},\varphi_{\alpha,\beta},\psi_{\beta,\alpha},I\}$ in $R-MOD$ such that each $M_{\alpha}$ is a finitely generated, projective, $R$-module, and, for the direct system $\big\langle M_{\alpha}, \varphi_{\alpha,\beta}\big\rangle$, $M\cong\varinjlim\limits_{\alpha\in I}M_{\alpha}$.
\end{definition}

\begin{prop}\label{ring:prop2}
Let $M$ be a locally projective $R$-module. Define
$$\overline{\varphi}_{\alpha,\beta}:\text{End}_{R}M_{\alpha}\to \text{End}_{R}M_{\beta}$$
by $\overline{\varphi}_{\alpha,\beta}(\phi)=\varphi_{\alpha,\beta}\circ\phi\circ\psi_{\beta,\alpha}$. Then, $\Big\langle \text{End}_{R}M_{\alpha},\overline{\varphi}_{\alpha,\beta}\Big\rangle$ is a direct system over $I$ in the category of rings.
\end{prop}

\begin{proof}
\cite[Proposition 2.5]{Haefner:1994aa}.   
\end{proof}

\begin{definition}\label{rop}
Let $R$ be a ring. Its \textit{opposite ring}, $R^{op}$, is defined to be the ring where $R$ and $R^{op}$ are the same as abelian groups, with multiplication, $\cdot\ $, in $R^{op}$ is given by $r_{1}\cdot r_{2}:=r_{2}r_{1} \text{ for all } r_{1},r_{2}\in R$.
\end{definition}

In both \cite{Abrams:1983aa} and \cite{P.N.-Anh:1987aa}, ``$\text{End}_{R}(M)$'' should be interpreted as ``$\Big(\text{End}_{R}(M)\Big)^{op}$.'' The reason for this is originally, in the case of two unital rings $R$ and $S$, Morita equivalence is characterized by the existence of an ``anti-isomorphism'' from $S$ onto $\text{End}_{R}P$, where $P$ is a projective generator of $R-MOD$; but this is precisely equivalent to saying there is an isomorpshim from $S$ onto $(\text{End}_{R}P)^{op}$. One of the main goals of both \cite{Abrams:1983aa} and \cite{P.N.-Anh:1987aa} is to generalize this characterization. For reasons unclear to the author, instead of defining the opposite ring, elements of $\text{End}_{R}(M)$ are thought of as ``right operators'' in the aforementioned papers. Further, it's worth noting Proposition \ref{ring:prop2} still holds if we replace $\Big\langle \text{End}_{R}M_{\alpha},\overline{\varphi}_{\alpha,\beta}\Big\rangle$ with $\Big\langle \Big(\text{End}_{R}M_{\alpha}\Big)^{op},\ \overline{\varphi}_{\alpha,\beta}\Big\rangle$.

\begin{theorem}\label{ring:theorem1}
Two rings with local units, $R$, $S$, are Morita equivalent if and only if there exists a locally projective $R$-module, $M$, such that $M$ is a generator of $R-MOD$ and 
$$S\cong\varinjlim\limits_{\alpha\in I}\Big(\mathrm{End}_{R}M_{\alpha}\Big)^{op}.$$
\end{theorem}

\begin{proof}
 \cite[Theorem 2.5]{P.N.-Anh:1987aa}. 
\end{proof}

\section{Morita Equivalence and Corner Rings}

Recall, given an ultragraph, $\mathcal{G}$, $E_{\mathcal{G}}$ denotes the graph constructed from $\mathcal{G}$ as in Definition \ref{gfromug}. In the case where $\mathcal{G}$ doesn't have any singular vertices, the contents of this section are essential to establishing the Morita equivalence of $L_{R}(\mathcal{G})$ and $L_{R}(E_{\mathcal{G}})$. The result we establish here need not hold in general, but it does hold for rings which share a certain property with Leavitt path algebras. As we have shown, Leavitt path algebras need not be unital. However, they do have ``the next best thing.''

\begin{definition}\label{ring:def9}
 A ring, $R$, is \textit{$\sigma$-unital} if there exists a collections of elements $\{u_{k}\}_{k\in\mathbb{N}}\subseteq R$ such that
 \begin{center}
$R=\bigcup\limits_{k\in\mathbb{N}}u_{k}Ru_{k}$ \hspace{.5cm} and \hspace{.5cm} $u_{k}=u_{k}u_{k+1}=u_{k+1}u_{k}$ for all $k\geq1$.
 \end{center}
The collection $\{u_{k}\}_{k\in\mathbb{N}}$ is called a \textit{$\sigma$-unit}.
\end{definition}

For any graph $E$ with $|E^{0}|=\infty$, the set $\{u_{k}\}_{k\in\mathbb{N}}$ from the proof of Lemma \ref{glpalgunital} is a $\sigma$-unit for $L_{R}(E)$. Although it requires a bit more work, one can show that, for $\mathcal{G}$ with $|G^{0}|=\infty$, $L_{R}(\mathcal{G})$ is $\sigma$-unital; we will prove it later on. In \cite{Takeshi-Katsura-Paul-Muhly-Aidan-Sims--Mark-Tomforde:2010aa}, the authors make use of multiplier algebras in order to show ultragraph $C^{*}$-algebras are Morita Equivalent to graph $C^{*}$-algebras. We will need to do something similar. But first, we need an algebraic analogue to the multiplier algebra found in \cite{Takeshi-Katsura-Paul-Muhly-Aidan-Sims--Mark-Tomforde:2010aa}.

\begin{definition}\label{ring:def8}
For a ring, $R$, let $\mathcal{M}(R)$ denote the set of pairs $(T,S)$ where $T,S:R\to R$ are additive homomorphisms such that:
\begin{center}
\textit{(i) $xT(y)=S(x)y$, \hspace{.5cm}  (ii) $T(xy)=T(x)y$, \hspace{.5cm}   (iii) $S(xy)=xS(y)$}
\end{center}
for all $x,y\in R$. $\mathcal{M}(R)$ is a ring with respect to the operations
\begin{center}
\textit{$(T_{1},S_{1})+(T_{2},S_{2}):=(T_{1}+T_{2},S_{1}+S_{2}),$ and $(T_{1},S_{1})(T_{2},S_{2}):=(T_{1}\circ T_{2}, S_{2}\circ S_{1});$}
\end{center}
it is called the \textit{multiplier ring} of $R$.
\end{definition}

Given a ring $R$, for each $x\in R$, we can define additive homomorphisms, $L_{x}$, $R_{x}$, by left and right multiplication by $x$, respectively; one can quickly verify $(L_{x}, R_{x})\in\mathcal{M}(R)$. Further, it is straightforward to check the map given by $x\mapsto (L_{x},R_{x})$ defines a ring homomorphism from $R$ into $\mathcal{M}(R)$, the kernel of which is 
 $$\text{ann}(R):=\{x\in R: xR=Rx=0\}.$$
 
 \begin{definition}\label{ring:def7}
Let $R$ be a ring. $I \lhd R$ is an \textit{essential ideal} if, for all $x\in R$, 
$$xI=Ix=0 \implies x=0.$$
$R$ is an \textit{essential ring} if $R$ is an essential ideal of itself.
\end{definition}

If $R$ is an essential ring, the ring homomorphism $x\mapsto (L_{x},R_{x})$ gives an embedding of $R$ into $\mathcal{M}(R)$; in this case, we identify $R$ with its image in $\mathcal{M}(R)$. How $R$ behaves in $\mathcal{M}(R)$ is of interest to us.  

\begin{prop}\label{ring:prop4}
Let $R$ be a $\sigma$-unital ring, and let $i:R\to\mathcal{M}(R)$ be the ring homomorphism given by $x\mapsto (L_{x},R_{x})$. Then,\\
1) $i(R)$ is  an essential ideal in $\mathcal{M}(R)$,\\
2) given a ring, $S$, and an injective ring homomorphism, $\phi:R\to S$, such that $\phi(R)$ is as an ideal in $S$, there exists a ring homomorphism, $\theta:S\to\mathcal{M}(R)$, such that the following diagram commutes:
\begin{center}
\begin{tikzpicture}
  \matrix (m) [matrix of math nodes,row sep=3em,column sep=4em,minimum width=3em]
  {
     \mathcal{M}(R) \\
     R & S; \\};
  \path[-stealth]
    (m-2-1) edge node [left] {$i$} (m-1-1)
            
        (m-2-1) edge node [above] {$\phi$} (m-2-2)
           
     (m-2-2) edge node [above] {$\theta$} (m-1-1);      

\end{tikzpicture}
\end{center}
further, if $\phi(R)$ is an essential ideal in $S$, $\theta$ is injective.
\end{prop}

\begin{proof}
\cite[Theorem 7.1.4]{Loring:1997aa}
\end{proof}

By 1) of Proposition \ref{ring:prop4}, since $L_{R}(E_{\mathcal{G}})$ is $\sigma$-unital, we can think of $L_{R}(E_{\mathcal{G}})$ as an essential ideal in $\mathcal{M}(L_{R}(E_{\mathcal{G}}))$. Following the blueprint set in \cite{Takeshi-Katsura-Paul-Muhly-Aidan-Sims--Mark-Tomforde:2010aa}, our result rests on showing the existence of a \textit{full idempotent} element, $Q\in \mathcal{M}(L_{R}(E_{\mathcal{G}}))$, such that $QL_{R}(E_{\mathcal{G}})Q$ is Morita Equivalent to $L_{R}(E_{\mathcal{G}})$. From there, in later sections, we will show $L_{R}(\mathcal{G})\cong QL_{R}(E_{\mathcal{G}})Q$, thereby establishing the Morita equivalence of $L_{R}(\mathcal{G})$ and $L_{R}(E_{\mathcal{G}})$. To that end:

\begin{definition}\label{ring:def10}
Let $R$ be a ring and $Q\in \mathcal{M}(R)$ an idempotent element. $Q$ is a \textit{full idempotent} if $RQR=R$.
\end{definition}

\noindent The technique used in proving the following theorem is modeled on \cite[Corollary 4.3]{Abrams:1983aa}.

\begin{theorem}\label{ring:theorem2}
Let $\{u_{k}\}_{k\in\mathbb{N}}$ be a $\sigma$-unit for $R$, and let $Q\in\mathcal{M}(R)$ be a full idempotent element such that $u_{k}Q=Qu_{k}$ for each $k$. Then, $QRQ$ and $R$ are Morita equivalent.
\end{theorem}

Before proving Theorem \ref{ring:theorem2}, we would like to first prove the following minor, but important, claim.

\begin{claim}\label{ring:remark2}
Let $R$ be a ring and $e\in R$ an idempotent. Then, $Re:=\{re: r\in R\}$ is a projective $R$-module. 
\end{claim}

\begin{proof}
Suppose $f:Re\to M$  and $g:N\twoheadrightarrow M$ are $R$-module homomorphisms. Since $g$ is surjective, we can fix an element $n\in N$ be such that $g(n)=f(e)$. Now, define $h:Re\to N$ by $h(re)=re\cdot n$. It's easy enough to check $h$ is an $R$-module map; moreover, $g(h(re))=g(re\cdot n)=re\cdot g(n)=re\cdot f(e)=f(re^{2})=f(re)$; meaning the diagram below commutes: 
\begin{center}
\begin{tikzpicture}
  \matrix (m) [matrix of math nodes,row sep=4em,column sep=2em,minimum width=2em]
  {
      & N \\
     Re & M \\};
  \path[-stealth]
  
    (m-2-1) edge node [below] {$f$}
            node [above] {} (m-2-2)
             edge [dashed,->] node [left] {$h$} (m-1-2)
    (m-1-2) edge  node [right] {$g$} (m-2-2);
\end{tikzpicture}
\end{center}
Hence, $Re$ is projective.
\end{proof}

\begin{proof}[Proof of Theorem \ref{ring:theorem2}] Set $M:=RQ$, $Q_{k}:=u_{k}Q$($=Qu_{k}$), and $M_{k}:=RQ_{k}$. We will prove our result by showing $M$ is a locally projective generator of $R-MOD$ such that $\varinjlim\limits\text{End}_{R}(M_{k})\cong QRQ$. To that end, recall from Proposition \ref{ring:prop4} that   
$$R \text{ a } \sigma\text{-unital ring} \implies R \text{ sits in $\mathcal{M}(R)$ as an essential ideal};$$
since $Q_{k}$ is an idempotent element in $R$ for each $k$, by Claim \ref{ring:remark2}, $M_{k}$ is a finitely generated, projective, $R$-module for each $k$. By the property of a $\sigma$-unit, for $i\leq k$, we have $u_{i}u_{k}=u_{i}$; this, along with the fact the elements of the $\sigma$-unit of $R$ commute with $Q$, implies $M_{i}\subseteq M_{k}$. Define $\varphi_{i,k}:M_{i}\to M_{k}$ to be the inclusion map. One can quickly check $\langle M_{i},\varphi_{i,k}\rangle$ is a direct system of $R$-modules over $\mathbb{N}$, let $\varinjlim\limits M_{i}$ be its direct limit.

Consider now the $R$-module $M$ together with the inclusion maps $\{\varPhi_{i}:M_{i}\to M\}_{i\in\mathbb{N}}$. Since $\varPhi_{i}=\varPhi_{k}\circ\varphi_{i,k}$, by the universal property of $\varinjlim\limits M_{i}$, there exists a unique $R$-module homomorphism $\theta: \varinjlim\limits M_{i}\to M$ such that $\varPhi_{i}=\theta\circ\eta_{i}$ for each $i$. For $\{u_{k}\}_{k\in\mathbb{N}}$ a $\sigma$-unit of $R$, we have $R=\bigcup Ru_{k}$, which in  turn means $M=\bigcup M_{k}$. From this, along with the fact each $\varPhi_{k}$ is an inclusion map, we can conclude $\theta$ is surjective, by \cite[24.3, 1) and 2)]{Wisbauer:1991aa}, $\theta$ is injective as well. Thus, $\varinjlim\limits M{i}\cong M.$ Finally, for each $i\leq k$, define $\psi_{k,i}:M_{k}\to M_{i}$ to be right multiplication by $u_{i}$. One can easily check $\{M_{i},\varphi_{i,k},\psi_{k,i},\mathbb{N}\}$ is a compatible set, allowing us to conclude $M$ is a locally projective $R$-module.

To see $M$ is a generator of $R-MOD$, let $N$ by any $R$-module. For each $n\in N$, define $\rho_{n}:M\to N$ to be right multiplication by $n$; then, extend $\rho_{n}$ to homomorphism from $\bigoplus\limits_{n\in N}M$ to $N$ in the usual way. After which, set 
$$\rho:=\bigoplus\limits_{n\in N}\rho_{n}:\bigoplus\limits_{n\in N}M\to N,$$
where the restriction of $\rho$ to the $n$-coordinate is $\rho_{n}$. Since $M=RQ$, note $\text{im}\rho=RQN$. Moreover, because $RN=N$, recall we are assuming all modules to be unital, and $Q$ is full, we have 
$$\text{im}\rho=RQN=RQRN=RN=N.$$
Thus, $\rho$ is surjective, implying $M$ is a generator of $R-Mod$.

Now, for a fixed $i$, consider the $R$-module $M_{i}$ and let $\phi\in\text{End}_{R}M_{i}^{op}$. For $rQ_{i}\in M_{i}$, note $\phi(rQ_{i})=\phi(rQ_{i}^{2})=rQ_{i}\phi(Q_{i}),$ and so $\phi$ is just right multiplication by $\phi(Q_{i})$; moreover, $\phi(Q_{i})=\phi(Q_{i}^{2})=Q_{i}\phi(Q_{i})$ implies $\phi(Q_{i})\in Q_{i}M_{i}:=Q_{i}RQ_{i}\subseteq M_{i}$. Let $\text{ev}_{Q_{i}}:\text{End}_{R}M_{i}^{op}\to Q_{i}RQ_{i};$ be the evaluation map at $Q_{i}$ (i.e., $\phi\mapsto \phi(Q_{i})$), it's straight forward to work out $\text{ev}_{Q_{i}}$ is a ring homomorphism. Since right multiplication by any element of $Q_{i}RQ_{i}$ defines an element of $\text{End}_{R}M_{i}$, $\text{ev}_{Q_{i}}$ is surjective. Further, if $\phi(Q_{i})=0$, $\text{im}\phi=\phi(RQ_{i})=R\phi(Q_{i})=\{0\}$; this means $\phi=0$, and so $\text{ev}_{Q_{i}}$ is injective as well. Hence, we have $\text{End}_{R}M_{i}^{op}\cong Q_{i}RQ_{i}.$ Lastly, the family of maps $\{\text{ev}_{Q_{i}}\}_{i\in\mathbb{N}}$ are such that the following diagram commutes for each $i\leq k\in\mathbb{N}$:

\begin{figure}[h!]
\begin{center}
\begin{tikzpicture}

 \node [shape=circle,minimum size=1.5em] (d1) at (0,1) {$\text{End}_{R}M_{i}^{op}$};
 \node [shape=circle,minimum size=1.5em] (d2) at (4,1) {$Q_{i}RQ_{i}$};
 \node [shape=circle,minimum size=1.5em] (d3) at (0,-1) {$\text{End}_{R}M_{k}^{op}$};
 \node [shape=circle,minimum size=1.5em] (d4) at (4,-1) {$Q_{k}RQ_{k}$};
 
\path (d1) edge [->, >=latex, shorten <= 2pt, shorten >= 2pt, right] node[above]{$\text{ev}_{Q_{i}}$} (d2);

\path (d3) edge [->, >=latex, shorten <= 2pt, shorten >= 2pt, right] node[above]{$\text{ev}_{Q_{k}}$} (d4);

 \node [shape=circle,minimum size=1.5em] (d6) at (0,-1) {};

 \path (d1) edge [->, >=latex, shorten <= -17pt, right] node[pos=-0.1]{$\overline{\varphi}_{i,k}$} (d6);
 
 \node [shape=circle,minimum size=1.5em] (d7) at (4,-1) {};
 
  \path (d2) edge [->, >=latex, shorten <= -10pt, right] node[pos=0.2]{$\varphi_{i,k}$} (d7);

\end{tikzpicture}
\end{center}
\end{figure}
\noindent where the set $\{\varphi_{i,k}:Q_{i}RQ_{i}\to Q_{k}RQ_{k}\}_{i,k\in\mathbb{N}}$ consists of inclusion maps, and the set
$$\{\overline{\varphi}_{i,k}:\text{End}_{R}M_{i}^{op}\to \text{End}_{R}M_{k}^{op}\}_{i,k\in\mathbb{N}}$$
consists of maps as in Proposition \ref{ring:prop2}. By \cite[24.4]{Wisbauer:1991aa}, the family of maps $\{\text{ev}_{Q_{i}}\}_{i\in\mathbb{N}}$ induce a ring isomorphism $\text{ev}_{Q}:\varinjlim\limits\text{End}_{R}M_{i}^{op}\to \varinjlim\limits Q_{i}RQ_{i},$ thus, $\varinjlim\limits\text{End}_{R}M_{i}^{op}\cong\varinjlim\limits Q_{i}RQ_{i}.$ Using the same argument as the one we used to show $\varinjlim\limits M{i}\cong M$, we can show $\varinjlim\limits Q_{i}RQ_{i}\cong QRQ,$ meaning $\varinjlim\limits\text{End}_{R}M_{i}^{op}\cong QRQ,$ and so, by\\
\noindent Theorem \ref{ring:theorem1}, $R$ and $QRQ$ are Morita equivalent. 
\end{proof}

At this juncture, we encourage the reader to see the paragraph preceding\\
\noindent Lemma \ref{sig} for the definitions of $W_{0}$ and $\Gamma_{0}$. With that, we will proceed with using the machinery we have built up to establish an important result regarding graph Leavitt path algebras---the existence of an algebraic analog to the projection constructed just before \cite[Proposition 5.20]{Takeshi-Katsura-Paul-Muhly-Aidan-Sims--Mark-Tomforde:2010aa}.

\begin{lemma}\label{ring:lemma1}
There exists a full idempotent $Q\in\mathcal{M}(L_{R}(E_{\mathcal{G}}))$ such that, for any $t_{\alpha}t_{\beta^{*}}\in L_{R}(E_{\mathcal{G}})$,

 \[
    Qt_{\alpha}t_{\beta^{*}}=\left\{
                \begin{array}{ll}
                  t_{\alpha}t_{\beta^{*}}\ \ \ if\ \  s(\alpha)\in W_{0}\sqcup\Gamma_{0},\\
                0\ \ \ \ \ \  \ otherwise.
                \end{array}
              \right.
  \]
  
  \end{lemma}

\begin{proof}
Let $X=W_{0}\sqcup\Gamma_{0}$. From \cite[Proposition 3.4]{Tomforde:2011aa}, we can deduce
\begin{align*}
L_{R}(E_{\mathcal{G}}) & \cong\Big(\bigoplus_{v\in X} L_{R}(E_{\mathcal{G}})q_{v}\Big)\bigoplus\Big( \bigoplus_{v\in E^{0}_{\mathcal{G}}\setminus X} L_{R}(E_{\mathcal{G}})q_{v} \Big)\\
&\cong \Big(\bigoplus_{v\in X} q_{v}L_{R}(E_{\mathcal{G}})\Big)\bigoplus\Big( \bigoplus_{v\in E^{0}_{\mathcal{G}}\setminus X} q_{v}L_{R}(E_{\mathcal{G}}) \Big)
\end{align*}
as abelian groups. Let ${}_{X}P$ be the projection of $\leg$ onto $\bigoplus_{v\in X} q_{v}\leg$, and let $P_{X}$ be the projection onto $\bigoplus_{v\in X} \leg q_{v}$, then set $Q:=({}_{X}P,P_{X})$.

Fix $t_{\alpha}t_{\beta^{*}}\in L_{R}(E_{\mathcal{G}})$. Since $L_{\ta}$ is left multiplication by $\ta$, and ${}_{X}P$ is the projection of $L_{R}(E_{\mathcal{G}})$ onto $\bigoplus_{v\in X} q_{v}\leg,$ we have 
    \[
    {}_{X}P\circ L_{\ta}=\left\{
                \begin{array}{ll}
                  L_{\ta}\ \ \ \text{if}\ \  s(\alpha)\in X,\\
                0\ \ \ \ \ \ \ \ \text{otherwise.}
                \end{array}
              \right.
  \]
On the other hand, by \textbf{LP1}, $R_{\ta}=0$ on 
 $$\bigoplus_{\{v\in E^{0}_{\mathcal{G}}:v\neq s(\alpha)\}} \leg q_{v},$$
 and since $P_{X}$ is the projection of $L_{R}(E_{\mathcal{G}})$ onto $\bigoplus_{v\in X} \leg q_{v}$, 
\[
    R_{\ta}\circ P_{X}=\left\{
                \begin{array}{ll}
                  R_{\ta}\ \ \ \text{if}\ \  s(\alpha)\in X,\\
                0\ \ \ \ \ \ \ \ \text{otherwise.}
                \end{array}
              \right.
  \]
In identifying $L_{R}(E_{\mathcal{G}})$ with $i(L_{R}(E_{\mathcal{G}}))$, ``$t_{\alpha}t_{\beta^{*}}$'' should be understood to mean 
 $$(L_{t_{\alpha}t_{\beta^{*}}},R_{t_{\alpha}t_{\beta^{*}}})\in\mathcal{M}(L_{R}(E_{\mathcal{G}})).$$
 In which case, ``$Qt_{\alpha}t_{\beta^{*}}$'' should be taken to mean
 $$({}_{X}P,P_{X})(L_{t_{\alpha}t_{\beta^{*}}},R_{t_{\alpha}t_{\beta^{*}}}):=({}_{X}P\circ L_{t_{\alpha}t_{\beta^{*}}},R_{t_{\alpha}t_{\beta^{*}}}\circ P_{X}),$$
which we have just shown is equal to $(L_{\ta},R_{\ta})$. And so, by notational abuse and all, we have
 \[
    Qt_{\alpha}t_{\beta^{*}}=\left\{
                \begin{array}{ll}
                  t_{\alpha}t_{\beta^{*}}\ \ \ \text{if}\ \  s(\alpha)\in X,\\
                 0\ \ \ \ \ \   \ \text{otherwise.}
                \end{array}
              \right.
  \]
The fact $Q$ is idempotent follows directly from the fact ${}_{X}P$ and $P_{X}$ are projections. To see $Q$ is full, fix $\ta\in\leg$. By \cite[Lemma 4.6]{Takeshi-Katsura-Paul-Muhly-Aidan-Sims--Mark-Tomforde:2010aa}, there exists a path, $\alpha'$, such that $s(\alpha')\in X$ and $r(\alpha')=s(\alpha)$. Note then $t_{\alpha'^{*}}t_{\alpha'\alpha}t_{\beta}=\ta$, moreover, since $s(\alpha')\in X$, $Qt_{\alpha'\alpha}t_{\beta}=t_{\alpha'\alpha}t_{\beta}$. And so we have $\ta=t_{\alpha'^{*}}t_{\alpha'\alpha}t_{\beta}=t_{\alpha'^{*}}Qt_{\alpha'\alpha}t_{\beta}$. Since $\leg$ is the span of elements of the form $\ta$, we can conclude $Q$ is full.
\end{proof}

\begin{corr}\label{meqcor}
Let $\mathcal{G}$ be an ultragraph, $E_{\mathcal{G}}$ the graph as in Definition \ref{gfromug}, and $Q\in\mathcal{M}(\leg)$ as in Lemma \ref{ring:lemma1}. Then, $\leg$ is Morita equivalent to $Q\leg Q$.
\end{corr}

\begin{proof}
We have $\{u_{k}\}_{k\in\mathbb{N}}$, as defined in Lemma \ref{glpalgunital}, is a $\sigma$-unit for $\leg$. Let $Q\in\mathcal{M}(\leg)$ be as in Lemma \ref{ring:lemma1}. For a given $v\in E_{\mathcal{G}}^{0}$, we have, by \ref{ring:lemma1}, 
\[
    Qq_{v}=\left\{
                \begin{array}{ll}
                  q_{v}\ \ \ \text{if}\ \  v\in X,\\
                0\ \ \ \   \ \text{otherwise.}
                \end{array}
              \right.
  \]
Further, a straight forward calculation shows 
 \[
   (L_{q_{v}},R_{q_{v}})({}_{X}P,P_{X}):=(L_{q_{v}}\circ{}_{X}P,P_{X}\circ R_{q_{v}})=\left\{
                \begin{array}{ll}
                  (L_{q_{v}},R_{q_{v}})\ \ \ \text{if}\ \  v\in X,\\
                0\ \ \ \   \ \text{otherwise.}
                \end{array}
              \right.
  \]
Which is to say, 
 \[
    q_{v}Q=\left\{
                \begin{array}{ll}
                  q_{v}\ \ \ \text{if}\ \  v\in X,\\
                0\ \ \ \   \ \text{otherwise.}
                \end{array}
              \right.
  \]
Thus, it must be $u_{k}Q=Qu_{k}$ for each $k\in\mathbb{N}$; and so, by Theorem \ref{ring:theorem2}, $\leg$ is Morita equivalent to $Q\leg Q$.
\end{proof}

Let $\mathcal{G}$ be an ultragraph without any singular vertices and $E_{\mathcal{G}}$ its associated graph. The first major result of this chapter will be showing $L_{R}(\mathcal{G})$ and $\leg$ are Morita equivalent. We will prove this fact by showing 
$$L_{R}(\mathcal{G})\cong Q\leg Q,$$
where $Q$ is as in Lemma \ref{ring:lemma1}. We will later show that even in the case where $\mathcal{G}$ contains singular veritces, $L_{R}(\mathcal{G})$ is still Morita equivalent to a graph Leavitt path algebra; in particular, $L_{R}(\mathcal{G})$ is Morita equivalent to $L_{R}(E_{\mathcal{F}})$, where $\mathcal{F}$ is the desingularization of $\mathcal{G}$.

An important component of our endeavor is to define an algebraic analog to the Exel-Laca algebras defined in \cite[Definition 3.3]{Takeshi-Katsura-Paul-Muhly-Aidan-Sims--Mark-Tomforde:2008aa}.
 
\section{Algebraic Exel-Laca Algebras}
 
For $\lambda,\ \mu$, finite subsets of $\mathcal{G}^{1}$, let 
$$r(\lambda,\mu):=\bigcap\limits_{e\in\lambda}r(e)\setminus\bigcup\limits_{f\in\mu}r(f)$$

\begin{definition}\label{dcel}
Let $\mathcal{G}$ be an ultragraph and $\mathcal{A}$ an $R$-algebra. A collection of idempotent elements $\{\mathsf{P}_{v},\ \mathsf{Q}_{e}:v\in G^{0},\ e\in\mathcal{G}^{1}\}$ in $\mathcal{A}$ satisfy \textit{condition (EL)} if:\\
(\textbf{EL1}) the elements of $\{\mathsf{P}_{v}\}$ are mutually orthogonal,\\
(\textbf{EL2}) the elements of $\{\mathsf{Q}_{e}\}$ pairwise commute,\\
(\textbf{EL3}) $\mathsf{P}_{v}\mathsf{Q}_{e} =\mathsf{Q_{e}}\mathsf{P}_{v}=
\begin{cases}
\mathsf{P}_{v} & \text{if }v\in r(e) \\
0 & \text{if }v\notin r(e)
\end{cases}$ \\
(\textbf{EL4}) $\prod\limits_{e\in\lambda}\mathsf{Q}_{e}\prod\limits_{f\in\mu}(1-\mathsf{Q}_{f})=\sum\limits_{v\in r(\lambda,\mu)}\mathsf{P}_{v}$ for all $\lambda,\ \mu$, finite subsets of $\mathcal{G}^{1}$ such that $\lambda\cap\mu=\emptyset$, $\lambda\neq\emptyset$, and $r(\lambda, \mu)$ is finite.
\end{definition}

If $\mathcal{A}$ is nonunital, we can take the identity in \textbf{EL4} to be the identity in $\mathcal{A}^{*}$. Here, $\mathcal{A}^{*}$ is the \textit{unitization} of $\mathcal{A}$. As an abelian group, $\mathcal{A}^{*}=R\oplus\mathcal{A}$; multiplication is given by $(r_{1},a_{1})(r_{2},a_{2})=(r_{1}r_{2},r_{1}a_{2}+r_{2}a_{1}+a_{1}a_{2})$; scalar multiplication is given by $r(r_{1},a_{1})=(rr_{1},ra_{1})$. It's easy enough to check $A^{*}$ is an $R$-algebra. More importantly, since $R$ is unital, $\mathcal{A}^{*}$ is unital with unit $(1,0)$. It's also easy to check $a\mapsto(0,a)$ embeds $\mathcal{A}$ into $\mathcal{A}^{*}$ as an ideal. 

\begin{definition}\label{dexg}
Let $\mathcal{G}$ be an ultragraph and $\mathcal{A}$ an $R$-algebra. An \textit{Exel-Laca $\mathcal{G}$-family} in $\mathcal{A}$ is a collection of idempotent elements $\{\mathsf{P}_{v}\}_{v\in G^{0}}$, and elements $\{\mathsf{S}_{e},\ \mathsf{S}_{e^{*}}\}_{e\in\mathcal{G}^{1}}$ such that:\\
(\textbf{ExL1}) the collection of elements $\{\mathsf{P}_{v}\}_{v\in G^{0}} \cup \{\mathsf{S}_{e^{*}}\mathsf{S}_{e}\}_{e\in\mathcal{G}^{1}}$ satisfy condition (EL),\\
(\textbf{ExL2}) $\mathsf{P}_{s(e)}\mathsf{S}_{e}=\mathsf{S}_{e}\mathsf{S}_{e^{*}}\mathsf{S}_{e}=\mathsf{S}_{e}$ and $\mathsf{S}_{e^{*}}\mathsf{S}_{e}\mathsf{S}_{e^{*}}=\mathsf{S}_{e^{*}}\mathsf{P}_{s(e)}=\mathsf{S}_{e^{*}}$ for all $e\in\mathcal{G}^{1}$,\\
(\textbf{ExL3}) $\mathsf{S}_{f^{*}}\mathsf{S}_{e}=0$ when $e\neq f$,\\
(\textbf{ExL4}) $\mathsf{P}_{v}=\sum\limits_{e\in s^{-1}(v)}\mathsf{S}_{e}\mathsf{S}_{e^{*}}$ for each $v\in G^{0}$ such that $0<|s^{-1}(v)|<\infty$.
\end{definition}

\begin{lemma}\label{llgex}
Let $\mathcal{A}$ be an $R$-algebra. Given a Leavitt $\mathcal{G}$-family $\{P_{A},\ S_{e},\ S_{e}^{*}\}$ in $\mathcal{A}$, the set $\{P_{v},\ S_{e},\ S_{e}^{*}\}$ forms an Exel-Laca $\mathcal{G}$-family in $\mathcal{A}$.
\end{lemma}

\begin{proof}
The proof given here is the same as the first part of the proof of\\
\noindent \cite[Lemma 2.10]{Takeshi-Katsura-Paul-Muhly-Aidan-Sims--Mark-Tomforde:2008aa}. Let $\{P_{A},\ S_{e},\ S_{e}^{*}\}$ be a Leavitt $\mathcal{G}$-family in $\mathcal{A}$. The fact $\{P_{v},\ S_{e},\ S_{e^{*}}\}$ satisfies \textbf{ExL2}, \textbf{ExL3}, and \textbf{ExL4}, as well as the fact $\{P_{v}, S_{e^{*}}S_{e}\}=\{P_{v},P_{r(e)}\}$ satisfies \textbf{EL1}, \textbf{EL2}, and \textbf{EL3}, follows directly from the properties of a Leavitt $\mathcal{G}$-family. By \textbf{uLP4}, we have
\begin{align*}
\prod\limits_{e\in r(\lambda)}P_{r(e)}\prod\limits_{f\in\mu}(1-P_{r(f)})&=P_{\scriptstyle\bigcap\limits_{e\in r(\lambda)}r(e)}-P_{\scriptstyle\bigcap\limits_{e\in r(\lambda)}r(e)}\ P_{\scriptstyle\bigcup\limits_{f\in r(\mu)}r(f)}\\
&=P_{\scriptstyle\bigcap\limits_{e\in r(\lambda)}\setminus\bigcup\limits_{f\in r(\mu)}}=\sum\limits_{v\in r(\lambda,\mu)}P_{v};
\end{align*}
meaning \textbf{EL4} is satisfied, which in turn means \textbf{ExL1} is satisfied. Thus, the set $\{P_{v},\ S_{e},\ S_{e}^{*}\}$ forms an Exel-Laca $\mathcal{G}$-family in $\mathcal{A}$.
\end{proof}

\begin{definition}\label{dexalg}
Given an ultragraph $\mathcal{G}$, we define the \textit{Exel-Laca algebra of $\mathcal{G}$}, $\mathcal{EL}_{R}(\mathcal{G})$, to be the $R$-algebra generated by an Exel-Laca $\mathcal{G}$-family $\{\mathsf{p}_{v},\ \mathsf{s}_{e},\ \mathsf{s}_{e^{*}}\}\subseteq \mathcal{EL}_{R}(\mathcal{G})$ such that, for any $R$-algebra $\mathcal{A}$ and an Exel-Laca $\mathcal{G}$-family $\{\mathsf{P}_{v},\ \mathsf{S}_{e},\ \mathsf{S}_{e^{*}}\}\subseteq\mathcal{A}$, there exists an $R$-algebra homomorphism 
$$\phi:\mathcal{EL}_{R}(\mathcal{G})\to \mathcal{A}$$
with $\phi(\mathsf{p}_{v})=\mathsf{P}_{v},\ \phi_{\mathsf{s}_{e}}=\mathsf{S}_{e},\ \phi(\mathsf{s}_{e^{*}})=\mathsf{S}_{e^{*}}$.
\end{definition}

\subsection{Constructing $\mathcal{EL}_{R}(\mathcal{G})$}

For $X=G^{0}\cup\mathcal{G}^{1}\cup(\mathcal{G}^{1})^{*}$, let $\mathbb{F}_{R}(w(X))$ be the nonunital, associative, free $R$-algebra as in Remark \ref{freealg}, and let $I\lhd \mathbb{F}_{R}(w(X))$ be the ideal generated by the union of the sets:\\
$\bullet$ $\{vw-\delta_{v,w}v:v,w\in G^{0}\}$, $\{e^{*}ef^{*}f-f^{*}fe^{*}e:e,f\in\mathcal{G}^{1}\}$,\\
 $\bullet$ $\{ve^{*}e-v, e^{*}ev-v:v\in G^{0},e\in\mathcal{G}^{1}\text{ such that }v\in r(e)\}$, $\{ve^{*}e, e^{*}ev:v\in G^{0},\\
 \hspace*{.5cm}e\in\mathcal{G}^{1}\text{ such that }v\notin r(e)\}$,\\
$\bullet$ $\{\prod\limits_{e\in\lambda}e^{*}e\prod\limits_{f\in\mu}(1-f^{*}f)-\sum\limits_{v\in r(\lambda,\mu)}v: \ \lambda,\ \mu\ \text{finite subsets of} \ \mathcal{G}^{1}$\\
\hspace*{.5cm}$\text{with}\ \lambda\cap\mu=\emptyset, \lambda\neq\emptyset,\ \text{and}\ r(\lambda, \mu)\ \text{is finite}\}$.\\
$\bullet$ $\{e^{*}f: e,f\in\mathcal{G}^{1}\text{ such that }e\neq f\}$, $\{e-s(e)e, ee^{*}e-e, e^{*}-e^{*}s(e), e^{*}ee^{*}-e^{*}:e\in\mathcal{G}^{1}\}$,\\
$\bullet$ $\{v-\sum\limits_{e\in s^{-1}(v)}ee^{*}:\  v\in G^{0}\ \text{such that}\ 0<|s^{-1}(v)|<\infty\}$.
Then, 
$$\mathcal{EL}_{R}(\mathcal{G}):=\mathbb{F}_{R}(w(X))/I;$$
given the projection 
$$\pi: \mathbb{F}_{R}(w(X))\to \mathbb{F}_{R}(w(X))/I,$$
the Exel-Laca $\mathcal{G}$-family, $\{\mathsf{s}_{e},\mathsf{s}_{e^{*}}, \mathsf{p}_{v}\}$, is taken to be the family $\{\pi(e),\pi(e^{*}),\pi(v)\}$. To see that $\mathcal{EL}_{R}(\mathcal{G})$ and $\{\mathsf{s}_{e},\mathsf{s}_{e^{*}}, \mathsf{p}_{v}\}$ have the desired univeral propery, let $\mathcal{A}$ be an $R$-algebra and $\{\mathsf{P}_{v},\mathsf{S}_{e},\mathsf{S}_{e^{*}}\}$ an Exel-Laca $\mathcal{G}$-family in $\mathcal{A}$. By the univerisal mapping property of $\mathbb{F}_{R}(w(X))$, we have an $R$-algebra homomorphism 
$$\phi: \mathbb{F}_{R}(w(X)) \to \mathcal{A}$$
such that $\phi(v)=\mathsf{P}_{v},\ \phi(e)=\mathsf{S}_{e},\ \phi(e^{*})=\mathsf{S}_{e^{*}}$. Further, since $\{\mathsf{P}_{v},\mathsf{S}_{e},\mathsf{S}_{e^{*}}\}$ is an Exel-Laca $\mathcal{G}$-family, $I \subset \text{ker}\phi$. Thus, there exists an $R$-algebra homomorphism
$$\overline{\phi}: \mathbb{F}_{R}(w(X))/I \to \mathcal{A}$$
such that $\phi=\overline{\phi}\circ\pi$. Put more desirably, there exists an $R$-algebra homomorphism 
$$\overline{\phi}:\mathcal{EL}_{R}(\mathcal{G})\to\mathcal{A}$$
such that $\overline{\phi}(\mathsf{p}_{v})=\mathsf{P}_{v},\ \overline{\phi}(\mathsf{s}_{e})=\mathsf{S}_{e},\ \overline{\phi}(\mathsf{s}_{e^{*}})=\mathsf{S}_{e^{*}}$.

Let $\mathcal{A}$ be an $R$-algebra with an Exel-Laca $\mathcal{G}$-family satisfying the properties of Definition \ref{dexalg}. Then, there exist $R$-homomorphims
$$\varphi:\mathcal{A}\to\mathcal{EL}_{R}(\mathcal{G}) \text{ and } \phi: \mathcal{EL}_{R}(\mathcal{G})\to\mathcal{A}$$
which are inverses of each other. Thus, $\mathcal{EL}_{R}(\mathcal{G})$ is unique up to isomorphism.

\section{Properties of $\mathcal{EL}_{R}(\mathcal{G})$}

\begin{lemma}\label{llgex2}
There exists a surjective $R$-algebra homomorphism 
$$\varphi:\mathcal{EL}_{R}(\mathcal{G})\to L_{R}(\mathcal{G})$$
such that $\varphi(\mathsf{p}_{v})=p_{v},\ \varphi(\mathsf{s}_{e})=s_{e},\ \varphi(\mathsf{s}_{e^{*}})=s_{e^{*}}$.
\end{lemma}

\begin{proof}
By Lemma \ref{llgex}, the set $\{p_{v},s_{e},s_{e^{*}}\}$ forms an Exel-Laca $\mathcal{G}$-family in $L_{R}(\mathcal{G})$. And so by the universal property of $\mathcal{EL}_{R}(\mathcal{G})$, there exists an $R$-algebra homomorphism $\varphi$ such that $\varphi(\mathsf{p}_{v})=p_{v},\ \varphi(\mathsf{s}_{e})=s_{e},\ \varphi(\mathsf{s}_{e^{*}})=s_{e^{*}}$. Further, for each $e\in\mathcal{G}^{1}$, $p_{r(e)}\in\text{im}\varphi$ since $p_{r(e)}=s_{e^{*}}s_{e}=\varphi(\mathsf{s}_{e^{*}}\mathsf{s}_{e})$; which, by Theorem \ref{uggrade}, means $p_{A}\in \text{im}\varphi$ for each $A\in\mathcal{G}^{0}$. Since $L_{R}(\mathcal{G})$ is the $R$-linear span of elements of the form $s_{\alpha}p_{A}s_\beta^{*}$, with $r(\alpha)\cap A\cap r(\beta)\neq \emptyset$ \cite[Theorem 2.5]{M.-Imanfar:2017aa}, it follows $\varphi$ is surjective. 
\end{proof}

We now need to establish the injectivity of $\varphi$. This endeavor is a bit more involved. We will start by giving a useful characterization of $\mathcal{EL}_{R}(\mathcal{G})$. Firstly, for $e,f\in\mathcal{G}^{1}$, \textbf{Exl2} implies $\mathsf{s}_{e}\mathsf{s}_{f}=(\mathsf{s}_{e}\mathsf{s}_{e^{*}}\mathsf{s}_{e})(\mathsf{p}_{s(f)}\mathsf{s}_{f})$. But, by \textbf{EL3},  this means $\mathsf{s}_{e}\mathsf{s}_{f}=\mathsf{s}_{e}(\mathsf{s}_{e^{*}}\mathsf{s}_{e}\mathsf{p}_{s(f)})\mathsf{s}_{f}=0$ if $s(f)\notin r(e)$. Which is to say $\mathsf{s}_{e}\mathsf{s}_{f}=0$ unless $ef$ is a path in $\mathcal{G}$. A similar argument also  shows $\mathsf{s}_{f^{*}}\mathsf{s}_{e^{*}}=0$ unless $ef$ is a path. On the other hand, from the construction of $\mathcal{EL}_{R}(\mathcal{G})$, we can see 
$$ef\ \in \mathcal{G}^{*} \implies \mathsf{s}_{e}\mathsf{s}_{f}\neq0,\ \mathsf{s}_{f^{*}}\mathsf{s}_{e^{*}}\neq0.$$
Thus, just as with $L_{R}(\mathcal{G})$, 
$$\mathsf{s}_{e_{1}}\cdots \mathsf{s}_{e_{n}}\neq 0 \iff e_{1}\cdots e_{n} \in \mathcal{G}^{*}.$$
In which case, for $\alpha:=e_{1}\cdots e_{n}$, we take $\mathsf{s}_{\alpha}:=\mathsf{s}_{e_{1}}\cdots \mathsf{s}_{e_{n}}$ and $\mathsf{s}_{\alpha^{*}}:=\mathsf{s}_{e_{n}^{*}}\cdots \mathsf{s}_{e_{1}^{*}}$. Recall that each vertex $v$ is considered a path of length 0; and so for $\alpha=v$, we take $\mathsf{s}_{\alpha}=\mathsf{s}_{\alpha^{*}}=\mathsf{p}_{v}$.

The following useful lemma shows results of the form Theorem \ref{uggrade} and Proposition \ref{gspan} hold for $\mathcal{EL}_{R}(\mathcal{G})$; its method of proof is also similar to how one obtains the aforementioned results. 

\begin{lemma}\label{lex}
Let $\mathcal{G}$ be an ultragraph. Then,
\begin{align}
\mathcal{EL}_{R}(\mathcal{G}) =\text{span}_{R}\Big\{& \{\mathsf{s}_{\alpha},\ \mathsf{s}_{\beta^{*}},\ \mathsf{s}_{\alpha}\mathsf{p}_{v},\ \mathsf{p}_{v}\mathsf{s}_{\beta^{*}}:\ v\in r(\alpha),\  v\in r(\beta)\}\ \cup \notag \\
& \{\mathsf{s}_{\alpha}\mathsf{p}_{v}\mathsf{s}_{\beta^{*}}:\ v\in r(\alpha)\cap r(\beta)\}\ \cup \notag \\
& \{\prod\limits_{e\in S}\mathsf{s}_{e^{*}}\mathsf{s}_{e},\  ,\ \mathsf{s}_{\alpha}\big(\prod\limits_{e\in S}\mathsf{s}_{e^{*}}\mathsf{s}_{e}\big)\mathsf{s}_{\beta^{*}}:\ S\ \text{is a finite subset of } \mathcal{G}^{1}, \notag \\ 
&r(\alpha)\cap r(\beta)\cap \big(\bigcap\limits_{e\in S}r(e)\big)\neq\emptyset\} \Big\}.\notag
\end{align}
\end{lemma}

\begin{proof}
We can see from its construction that $\mathcal{EL}_{R}(\mathcal{G})$ is generated as an $R$-algebra by the set $\{\mathsf{s}_{\alpha},\ \mathsf{s}_{\beta^{*}}\}_{\alpha,\beta\in\mathcal{G}^{*}}$; and so to prove our lemma, we only need to argue a non-zero product of elements in $\{\mathsf{s}_{\alpha},\ \mathsf{s}_{\beta^{*}}\}_{\alpha,\beta\in\mathcal{G}^{*}}$ can be reduced to one of the prescribed forms. To that end, suppose $0\neq x\in\mathcal{EL}_{R}(\mathcal{G})$ is such an element.

1) Suppose there is a term of the form $\mathsf{s}_{\alpha}\mathsf{s}_{\alpha'}$ (similarly $\mathsf{s}_{\beta^{*}}\mathsf{s}_{\beta'^{*}}$) in the expression of $x$. If both $\alpha$ and $\alpha'$ have lengths larger than 0, then by our assumption $x\neq0$, it must be $\alpha\alpha'\in\mathcal{G}^{*}$ and $\mathsf{s}_{\alpha}\mathsf{s}_{\alpha'}=\mathsf{s}_{\alpha\alpha'}$. If $|\alpha|=0$, then it must be that $\alpha=s(\alpha')$ and $\mathsf{s}_{\alpha}\mathsf{s}_{\alpha'}=\mathsf{p}_{s(\alpha')}\mathsf{s}_{\alpha'}=\mathsf{s}_{\alpha'}$. Lastly, if $|\alpha'|=0$, then we can conclude $\alpha'=v\in r(\alpha)$ and $\mathsf{s}_{\alpha}\mathsf{s}_{\alpha'}=\mathsf{s}_{\alpha}\mathsf{p}_{v}$. We can similarly deduce that $\mathsf{s}_{\beta^{*}}\mathsf{s}_{\beta'^{*}}$ will be in one of the following forms: $\mathsf{s}_{(\beta'\beta)^{*}}$, $\mathsf{s}_{\beta*}$, or $\mathsf{p}_{v}\mathsf{s}_{\beta'^{*}}$ with $v\in r(\beta')$. Note that the result stated here comes as a direct consequence of \textbf{EL3} and \textbf{ExL2}.

2) Suppose there is a term of the form $\mathsf{s}_{\alpha}\mathsf{s}_{\beta^{*}}$ in the expression of $x$. If $|\alpha|=0$, \textbf{EL3} and \textbf{ExL2} give us $\mathsf{s}_{\alpha}\mathsf{s}_{\beta^{*}}=\mathsf{p}_{v}\mathsf{s}_{\beta^{*}}$, where $\alpha=v\in r(\beta)$; similarly, if $|\beta|=0$, we get $\mathsf{s}_{\alpha}\mathsf{s}_{\beta^{*}}=\mathsf{s}_{\alpha}\mathsf{p}_{v}$ where $\beta=v\in r(\alpha)$. Alternatively, if $\alpha=e_{1}\cdots e_{n}$ and $\beta=f_{1}\cdots f_{m}$, with $n,m>0$, \textbf{ExL2} implies 
$$\mathsf{s}_{\alpha}\mathsf{s}_{\beta^{*}}=\mathsf{s}_{\alpha}\big((\mathsf{s}_{e_{n}^{*}}\mathsf{s}_{e_{n}})(\mathsf{s}_{f_{m}^{*}}\mathsf{s}_{f_{m}}) \big)\mathsf{s}_{\beta^{*}}.$$
The fact $x\neq 0$, along with \textbf{EL3} and \textbf{ExL2}, implies $r(e_{n})\cap r(f_{m})\neq\emptyset$.

3) Finally, suppose $\alpha=e_{1}\cdots e_{n},\ \beta=f_{1}\cdots f_{m}\in \mathcal{G}^{*}$. Then, 
$$\mathsf{s}_{\beta^{*}}\mathsf{s}_{\alpha}:=(\mathsf{s}_{f_{m}^{*}}\cdots \mathsf{s}_{f_{1}^{*}})(\mathsf{s}_{e_{1}}\cdots \mathsf{s}_{e_{n}})=\mathsf{s}_{f_{m}^{*}}\cdots \mathsf{s}_{f_{2}^{*}}(\mathsf{s}_{f_{1}^{*}}\mathsf{s}_{e_{1}})\mathsf{s}_{e_{2}}\cdots \mathsf{s}_{e_{n}}.$$
By \textbf{ExL3} , $\mathsf{s}_{f_{1}^{*}}\mathsf{s}_{e_{1}}=0$ if $f_{1}\neq e_{1}$; otherwise, $\mathsf{s}_{f_{1}^{*}}\mathsf{s}_{e_{1}}=\mathsf{s}_{e_{1}^{*}}\mathsf{s}_{e_{1}}$. Consequently, \textbf{EL3} and \textbf{Exl2} imply $\mathsf{s}_{e_{1}^{*}}\mathsf{s}_{e_{1}}\mathsf{s}_{e_{2}}=\mathsf{s}_{e_{1}^{*}}\mathsf{s}_{e_{1}}\mathsf{p}_{s(e_{2})}\mathsf{s}_{e_{2}}=\mathsf{s}_{e_{2}}$. This means $\mathsf{s}_{f_{m}^{*}}\cdots \mathsf{s}_{f_{2}^{*}}(\mathsf{s}_{f_{1}^{*}}\mathsf{s}_{e_{1}})\mathsf{s}_{e_{2}}\cdots \mathsf{s}_{e_{n}}=0$ if $f_{1}\neq e_{1}$, otherwise, 
$$\mathsf{s}_{f_{m}^{*}}\cdots \mathsf{s}_{f_{2}^{*}}(\mathsf{s}_{f_{1}^{*}}\mathsf{s}_{e_{1}})\mathsf{s}_{e_{2}}\cdots \mathsf{s}_{e_{n}}=\mathsf{s}_{f_{m}^{*}}\cdots \mathsf{s}_{f_{2}^{*}}\mathsf{s}_{e_{2}}\cdots \mathsf{s}_{e_{n}}.$$
Continuing in this manner, we can see that:
\begin{center}
$\mathsf{s}_{\beta^{*}}\mathsf{s}_{\alpha}=
\begin{cases}
\mathsf{s}_{e_{n}^{*}}\mathsf{s}_{e_{n}} & \text{if } \alpha=\beta, \\
\mathsf{s}_{\beta'^{*}} & \text{if } \beta=\alpha\beta', \text{ where } \beta'=f_{n+1}\cdots f_{m},\\
\mathsf{s}_{\alpha'} & \text{if } \alpha=\beta\alpha', \text{ where } \alpha'=e_{m+1}\cdots e_{n},\\
0 & \text{otherwise}.
\end{cases}$ 
\end{center}
Should $|\beta|=0$, or $|\alpha|=0$, note that \textbf{Exl2} implies the above formula still holds.

By using 1), 2), and 3), as needed, we can reduce $x$ to one of the desired forms. 
\end{proof}

An important component to showing the injectivity of $\varphi$ from Lemma \ref{llgex2} rests on the fact $\mathcal{EL}_{R}(\mathcal{G})$ is a $\mathbb{Z}$-graded ring. 

\begin{lemma}\label{lex2}
$\mathcal{EL}_{R}(\mathcal{G})$ is a $\mathbb{Z}$-graded ring with $\mathcal{EL}_{R}(\mathcal{G})\cong\bigoplus\limits_{i\in\mathbb{Z}}\mathcal{EL}_{R}(\mathcal{G})_{i}$, where
\begin{align}
\mathcal{EL}_{R}(\mathcal{G})_{i} =\text{span}_{R}\Big\{& \{\mathsf{s}_{\alpha},\ \mathsf{s}_{\beta^{*}},\ \mathsf{s}_{\alpha}\mathsf{p}_{v},\ \mathsf{p}_{v}\mathsf{s}_{\beta^{*}}:\ v\in r(\alpha),\  v\in r(\beta),\ |\alpha|=|\beta|=i\}\ \cup \notag \\
& \{\mathsf{s}_{\alpha}\mathsf{p}_{v}\mathsf{s}_{\beta^{*}}:\ v\in r(\alpha)\cap r(\beta),\ |\alpha|-|\beta|=i\}\ \cup \notag \\
& \{\prod\limits_{e\in S}\mathsf{s}_{e^{*}}\mathsf{s}_{e},\  ,\ \mathsf{s}_{\alpha}\big(\prod\limits_{e\in S}\mathsf{s}_{e^{*}}\mathsf{s}_{e}\big)\mathsf{s}_{\beta^{*}}:\ S\ \text{is a finite subset of } \mathcal{G}^{1}, \notag \\ 
&\bigcap\limits_{e\in S}r(e)\neq\emptyset,\  r(\alpha)\cap r(\beta)\cap \big(\bigcap\limits_{e\in S}r(e)\big)\neq\emptyset,\  |\alpha|-|\beta|=i\} \Big\}.\notag
\end{align}
\end{lemma}

\begin{proof}
Let $X=G^{0}\cup \mathcal{G}^{1}\cup (\mathcal{G}^{1})^{*}$ and let $\mathbb{F}_{R}(w(X))$ be the free $R$-algebra on $X$. For each $w\in w(X)$, take 
\begin{align*}
|w|&:=\{the\ number\ of\ elements\ in\ \mathcal{G}^{1}\ which\ appear\ in\ w\}\\
&-\{the\ number\ of\ elements\ in\ (\mathcal{G}^{1})^{*}\ which\ appear\ in\ w\}.
\end{align*}
Since $\mathbb{F}_{R}(w(X))$ is the free $R$-module on $w(X)$, we have
$$\mathbb{F}_{R}(w(X))=\bigoplus\limits_{i\in\mathbb{Z}}\mathbb{F}_{R}(w(X))_{i},$$
where $\mathbb{F}_{R}(w(X))_{i}:=\text{span}_{R}\{w\in w(X): |w|=i\}.$ Moreover, for $w_{1}, w_{2} \in w(X)$, we have $|w_{1}w_{2}|=|w_{1}|+|w_{2}|$; which means 
$$\mathbb{F}_{R}(w(X))_{n}\cdot\mathbb{F}_{R}(w(X))_{m}\subseteq\mathbb{F}_{R}(w(X))_{n+m}.$$ 
Thus, $\mathbb{F}_{R}(w(X))$ is a $\mathbb{Z}$-graded ring with respect to the grading given above. 

Now, recall $\mathcal{EL}_{R}(\mathcal{G}):=\mathbb{F}_{R}(w(X))/I$, where $I$ is generated by the union of the sets:\\
$\bullet$ $\{vw,\ f^{*}e: v\neq w,\ e\neq f\}$, $\{e^{*}ef^{*}f-f^{*}fe^{*}e\}$,\\
 $\bullet$ $\{ve^{*}e-v, e^{*}ev-v:\ v\in r(e)\}$, $\{ve^{*}e, e^{*}ev:\ v\notin r(e)\}$,\\
$\bullet$ $\{\prod\limits_{e\in\lambda}e^{*}e\prod\limits_{f\in\mu}(1-f^{*}f)-\sum\limits_{v\in r(\lambda,\mu)}v: \ \lambda,\ \mu\ \text{finite subsets of} \ \mathcal{G}^{1}$\\
\hspace*{.5cm} $\text{with}\ \lambda\cap\mu=\emptyset, \lambda\neq\emptyset,\ \text{and}\ r(\lambda, \mu)\ \text{is finite}\}$,\\
$\bullet$ $\{f^{*}e: e\neq f\}$, $\{e-s(e)e, ee^{*}e-e, e^{*}-e^{*}s(e), e^{*}ee^{*}-e^{*}\}$,\\
$\bullet$ $\{v-\sum\limits_{e\in s^{-1}(v)}ee^{*}:\  v\in G^{0}\ \text{with}\ 0<|s^{-1}(v)|<\infty\}$.\\
Note that $I$ is generated by homogeneous elements of degree 0; which means $\mathcal{EL}_{R}(\mathcal{G}):=\mathbb{F}_{R}(w(X))/I$ is a $\mathbb{Z}$-graded ring (see Remark \ref{rgi}). In particular, 
$$\mathbb{F}_{R}(w(X))/I=\bigoplus\limits_{i\in\mathbb{Z}}(\mathbb{F}_{R}(w(X))_{i}+I)/I$$
 as an internal direct sum. Since
$$(\mathbb{F}_{R}(w(X))_{i}+I)/I=\text{span}_{R}\{\overline{w}:w\in(\mathbb{F}_{R}(w(X))_{i}\},$$
Lemma \ref{lex} implies
\begin{align}
\mathcal{EL}_{R}(\mathcal{G})_{i} &:=(\mathbb{F}_{R}(w(X))_{i}+I)/I \notag \\ 
&=\text{span}_{R}\Big\{\{\mathsf{s}_{\alpha},\ \mathsf{s}_{\beta^{*}},\ \mathsf{s}_{\alpha}\mathsf{p}_{v},\ \mathsf{p}_{v}\mathsf{s}_{\beta^{*}}:\ v\in r(\alpha),\  v\in r(\beta),\ |\alpha|=|\beta|=i\}\ \cup \notag \\
& \{\mathsf{s}_{\alpha}\mathsf{p}_{v}\mathsf{s}_{\beta^{*}}:\ v\in r(\alpha)\cap r(\beta),\ |\alpha|-|\beta|=i\}\ \cup \notag \\
& \{\prod\limits_{e\in S}\mathsf{s}_{e^{*}}\mathsf{s}_{e},\  ,\ \mathsf{s}_{\alpha}\big(\prod\limits_{e\in S}\mathsf{s}_{e^{*}}\mathsf{s}_{e}\big)\mathsf{s}_{\beta^{*}}:\ S\ \text{is a finite subset of } \mathcal{G}^{1}, \notag \\ 
&\bigcap\limits_{e\in S}r(e)\neq\emptyset,\  r(\alpha)\cap r(\beta)\cap \big(\bigcap\limits_{e\in S}r(e)\big)\neq\emptyset,\  |\alpha|-|\beta|=i\} \Big\}.\notag
\end{align}
\end{proof}

The following theorem is analogous to Theorems \ref{ggrdhomom} and \ref{uggrdhomom}. 

\begin{theorem}\label{thm:keythm}
Suppose $\mathcal{G}$ is an ultragraph with no singular vertices and $S$ a $\mathbb{Z}$-graded ring. Let 
$$\pi:\mathcal{EL}_{R}(\mathcal{G})\to S$$
be a graded ring homomorphism such that $\pi(rp_{v})\neq 0$ and $\pi(rs_{e^{*}}s_{e})\neq 0$ for all $v\in G^{0}$, $e\in\mathcal{G}^{1}$, and $r\in R\setminus\{0\}$; then, $\pi$ is injective. 
\end{theorem}

We will prove Theorem \ref{thm:keythm} using similar techniques employed by the aforementioned papers. Namely, we will ``approximate" $\mathcal{EL}_{R}(\mathcal{G})$ by the Leavitt path algebras of finite graphs. To that end, we will show how one defines a finite graph $G_{F}$ from an ultragraph $\mathcal{G}$.

Let $\mathcal{G}$ be an ultragraph with no singular vertices, and let $F\subseteq \mathcal{G}^{1}$ be finite. Recall that for $\lambda,\ \mu$, finite subsets of $\mathcal{G}^{1}$, $r(\lambda,\mu):=\bigcap\limits_{e\in\lambda}r(e)\setminus\bigcup\limits_{f\in\mu}r(f)$. With that, we will define the finite graph $G_{F}$ as follows:
\begin{align}
&G_{F}^{0}:=F\cup\{X:\ \emptyset\neq X\subseteq F\ \text{for which}\ \{e\in\mathcal{G}^{1}:\ s(e)\in r(X,F\setminus X)\}\nsubseteq F\}, \notag \\
&G_{F}^{1}:=\{(e,f)\in F\times F\: s(f)\in r(e)\}\cup\{(e,X): e\in X\}; \notag
\end{align}
with the range and source maps given by 
\begin{align}
&s_{F}((e,f))=e,\ \ s_{F}((e,X))=e, \notag\\
&r_{F}((e,f))=f,\ \ r_{F}((e,X))=X.\notag
\end{align}

In order to prove Theorem \ref{thm:keythm}, we will need the following lemmas. They establish an analogous result to \cite[Lemma 3.1]{M.-Imanfar:2017aa} for $\mathcal{EL}_{R}(\mathcal{G})$. We will also take this time to state that, in general, Leavitt path algebras and Exel-Laca algebras will not be unital, but we will still find ourselves referring to a unit. We can interpret ``1'' to be the unit in their respective unitizations (see the note in Definition \ref{dcel}). 

\begin{lemma}\label{lci}
Let $\{P_{i}\}_{i=1}^{n}$ be a set of commuting idempotent elements in a ring $R$. Then, 
$$\sum\limits_{\emptyset\neq Y\subseteq\{1,\cdots,n\}}\Big(\prod\limits_{i\in Y}P_{i}\prod\limits_{i\notin Y}(1-P_{i})\Big)=1-\prod\limits_{i\in\{1,\cdots,n\}}(1-P_{i}).$$
Note: If $R$ is not unital, we again take 1 to be the identity in its unitization.
\end{lemma}

\begin{proof} \cite[Lemma 5.2]{Tomforde:2003aa}  states
$$\sum\limits_{Y\subseteq\{1,\cdots,n\}}\Big(\prod\limits_{i\in Y}P_{i}\prod\limits_{i\notin Y}(1-P_{i})\Big)=1.$$
Since $\prod\limits_{i\in Y}P_{i}\prod\limits_{i\notin Y}(1-P_{i})=\prod\limits_{i\in\{1,\cdots,n\}}(1-P_{i})$ for $Y=\emptyset$, we have
$$\sum\limits_{\emptyset\neq Y\subseteq\{1,\cdots,n\}}\Big(\prod\limits_{i\in Y}P_{i}\prod\limits_{i\notin Y}(1-P_{i})\Big)=1-\prod\limits_{i\in\{1,\cdots,n\}}(1-P_{i}).$$
\end{proof}

\begin{lemma}\label{lfg}
Let $\mathcal{G}$ be an ultragraph with no singular vertices, $F\subseteq\mathcal{G}^{1}$ finite, and $G_{F}$ the associated finite graph as constructed above. Then,
\begin{align}
&P_{e}:=\mathsf{s}_{e}\mathsf{s}_{e^{*}},\ \ \ \ \ \ P_{X}:=\Big(\prod\limits_{e\in X}\mathsf{s}_{e^{*}}\mathsf{s}_{e}\prod\limits_{e'\in F\setminus X}(1-\mathsf{s}_{e'^{*}}\mathsf{s}_{e'})\Big)\Big(1-\sum\limits_{f\in F}\mathsf{s}_{f}\mathsf{s}_{f^{*}}\Big),  \notag \\
&S_{(e,f)}:=\mathsf{s}_{e}P_{f},\ \ \ S_{(e,X)}:= \mathsf{s}_{e}P_{X},\ \ \ S_{(e,f)^{*}}:=P_{f}\mathsf{s}_{e^{*}},\ \ \ S_{(e,X)^{*}}:=P_{X}\mathsf{s}_{e^{*}}   \notag
\end{align}
forms a Leavitt $G_{F}$-family in $\mathcal{EL}_{R}(\mathcal{G})$.
\end{lemma}

\begin{proof} (\textbf{LP1}). We will first show that $\{P_{e}, P_{X}\}$ forms a set of pairwise orthogonal idempotents. For $e,f\in F$ such that $e\neq f$, we have 
$$P_{e}P_{f}=\mathsf{s}_{e}\mathsf{s}_{e^{*}}\mathsf{s}_{f}\mathsf{s}_{f^{*}}=\mathsf{s}_{e}(\mathsf{s}_{e^{*}}\mathsf{s}_{f})\mathsf{s}_{f^{*}}=0$$
since $\mathsf{s}_{e^{*}}\mathsf{s}_{f}=0$ by \textbf{ExL3}; further, by \textbf{ExL2}, $P_{e}$ is an idempotent for each $e$.
By \textbf{EL4}, we have 
$$\prod\limits_{e\in X}\mathsf{s}_{e^{*}}\mathsf{s}_{e}\prod\limits_{e'\in F\setminus X}(1-\mathsf{s}_{e'^{*}}\mathsf{s}_{e'})=\sum\limits_{v\in r(X,F\setminus X)} \mathsf{p}_{v}.$$
In particular, by applying \textbf{ExL2}, we can see that 
$$P_{X}=\Big(\sum\limits_{v\in r(X,F\setminus X)} \mathsf{p}_{v}\Big)\Big(1-\sum\limits_{f\in F}\mathsf{s}_{f}\mathsf{s}_{f^{*}}\Big)=\Big(\sum\limits_{v\in r(X,F\setminus X)} \mathsf{p}_{v}-\sum_{\substack{f\in F:\\ s(f)\in r(X,F\setminus X)}}\mathsf{s}_{f}\mathsf{s}_{f^{*}}\Big);$$
using this fact, and exploiting \textbf{ExL2} again, we can work out $P_{X}$ is an idempotent for each $X$. 

Of particular utility in showing the pairwise orthogonality of the set $\{P_{e}, P_{X}\}$ is the fact $\mathsf{s}_{e^{*}}\mathsf{p}_{v}=\mathsf{s}_{e^{*}}$ if $v=s(e)$, and $\mathsf{s}_{e^{*}}\mathsf{p}_{v}=0$ otherwise. With this in mind, for $e,X\in G_{F}^{0}$, note that $\mathsf{s}_{e}\mathsf{s}_{e^{*}}\Big(\sum\limits_{v\in r(X,F\setminus X)} \mathsf{p}_{v}\Big)=0$ if $s(e)\notin r(X,F\setminus X)$, meaning $P_{e}P_{X}=0$; if $s(e)\in r(X,F\setminus X)$, then by \textbf{ExL2} and \textbf{ExL3} we have
$$\mathsf{s}_{e}\mathsf{s}_{e^{*}}\Big(\sum\limits_{v\in r(X,F\setminus X)} \mathsf{p}_{v}\Big)=\mathsf{s}_{e}\mathsf{s}_{e^{*}}\text{ and }\mathsf{s}_{e}\mathsf{s}_{e^{*}}\Big(\sum_{\substack{f\in F:\\ s(f)\in r(X,F\setminus X)}}\mathsf{s}_{f}\mathsf{s}_{f^{*}}\Big)=\mathsf{s}_{e}\mathsf{s}_{e^{*}},$$
which again gives us 
$$P_{e}P_{X}=\mathsf{s}_{e}\mathsf{s}_{e^{*}}\Big(\sum\limits_{v\in r(X,F\setminus X)} \mathsf{p}_{v}-\sum_{\substack{f\in F:\\ s(f)\in r(X,F\setminus X)}}\mathsf{s}_{f}\mathsf{s}_{f^{*}}\Big)=0.$$
A similar argument shows $P_{X}P_{e}=0$ for each $e,X\in G_{F}^{0}$; thus, $P_{e}P_{X}=P_{X}P_{e}=0$ for each $e,X\in G_{F}^{0}$.

Now suppose $X,Y\in G_{F}^{0}$ such that $X\neq Y$. WLOG, let $e'\in X\setminus Y$. This means 
$$r(X,F\setminus X):=\bigcap\limits_{e\in X}r(e)\setminus\bigcup\limits_{f\in F\setminus X}r(f)\subseteq r(e')$$
since $e'\in X$; and $r(Y,F\setminus Y)\cap r(e')=\emptyset$ since $e'\in F\setminus Y$. And so we can conclude
$$r(X,F\setminus X)\cap r(Y,F\setminus Y)=\emptyset.$$
As a result, for 
$$P_{X}=\Big(\sum\limits_{v\in r(X,F\setminus X)} \mathsf{p}_{v}-\sum_{\substack{f\in F:\\ s(f)\in r(X,F\setminus X)}}\mathsf{s}_{f}\mathsf{s}_{f^{*}}\Big) \text{ and } P_{Y}=\Big(\sum\limits_{v\in r(Y,F\setminus Y)} \mathsf{p}_{v}-\sum_{\substack{f\in F:\\ s(f)\in r(Y,F\setminus Y)}}\mathsf{s}_{f}\mathsf{s}_{f^{*}}\Big),$$
we have $P_{X}P_{Y}=P_{Y}P_{X}=0$. Thus, $\{P_{e},P_{X}\}$ forms a set of pairwise orthogonal idempotents in $\mathcal{EL}_{R}(\mathcal{G})$.

(\textbf{LP2}). Recall $G_{F}^{1}:=\{(e,f)\in F\times F\: s(f)\in r(e)\}\cup\{(e,X): e\in X\}$. With a straight forward application of \textbf{ExL2}, we can deduce 
{\footnotesize
\begin{align*}
&P_{s_{F}((e,f))}S_{(e,f)}=S_{(e,f)}P_{r_{F}((e,f))}=S_{(e,f)},\ \ \ P_{s_{F}((e,X))}S_{(e,X)}=S_{(e,X)}P_{r_{F}((e,X))}=S_{(e,X)}\\
& P_{r_{F}((e,f))}S_{(e,f)^{*}}=S_{(e,f)^{*}}P_{s_{F}((e,f))}=S_{(e,f)^{*}}, \text{ and }  P_{r_{F}((e,X))}S_{(e,X)^{*}}=S_{(e,X)^{*}}P_{s_{F}((e,X))}=S_{(e,X)^{*}}
\end{align*}\par}

(\textbf{LP3}). By definition, $S_{(e,f)^{*}}S_{(e',f')}=P_{f}\mathsf{s}_{e^{*}}\mathsf{s}_{e'}P_{f'}$. If $e\neq e'$, then $\mathsf{s}_{e^{*}}\mathsf{s}_{e'}=0$ by \textbf{ExL2}, which in turn means $S_{(e,f)^{*}}S_{(e',f')}=0$. If $e=e'$, then $P_{f}\mathsf{s}_{e^{*}}\mathsf{s}_{e'}P_{f'}=P_{f}P_{f'}$  by \textbf{ExL2} and \textbf{EL3}. Moreover, by \textbf{LP1}, $P_{f}P_{f'}=0$ if $f\neq f'$. All in all, this means $S_{(e,f)^{*}}S_{(e',f')}=0$ if $(e,f)\neq (e'f')$, and $S_{(e,f)^{*}}S_{(e',f')}=P_{f}=P_{r_{F}((e,f))}$ if $(e,f)=(e',f')$. Similarly, $S_{(e,X)^{*}}S_{(e',X')}=P_{X}\mathsf{s}_{e^{*}}\mathsf{s}_{e'}P_{X'}$. As before, if $e\neq e'$, then $S_{(e,X)^{*}}S_{(e',X')}=0$. If $e=e'$, then for $e\in X'$, \textbf{EL2} implies 
$$\mathsf{s}_{e^{*}}\mathsf{s}_{e}\Big(\prod\limits_{g\in X'}\mathsf{s}_{g^{*}}\mathsf{s}_{g}\Big)=\prod\limits_{g\in X'}\mathsf{s}_{g^{*}}\mathsf{s}_{g}.$$
And so
$$\mathsf{s}_{e^{*}}\mathsf{s}_{e}P_{X'}=\mathsf{s}_{e^{*}}\mathsf{s}_{e}\Big(\prod\limits_{g\in X'}\mathsf{s}_{g^{*}}\mathsf{s}_{g}\prod\limits_{g'\in F\setminus X'}(1-\mathsf{s}_{g'^{*}}\mathsf{s}_{g'})\Big)\Big(1-\sum\limits_{f\in F}\mathsf{s}_{f}\mathsf{s}_{f^{*}}\Big)=P_{X'}.$$
Therefore, for $e=e'$, we have $S_{(e,X)^{*}}S_{(e',X')}=P_{X}P_{X'}$. Exploiting \textbf{LP1} again, we have $S_{(e,X)^{*}}S_{(e',X')}=0$ if $(e,X)\neq (e'X')$, and $S_{(e,X)^{*}}S_{(e',X')}=P_{X}=P_{r_{F}((e,X))}$ if $(e,X)=(e',X')$. Finally, for $(e,f),(e',X')\in G_{F}^{1}$, we can deduce $S_{(e,f)^{*}}S_{(e',X')}=S_{(e',X')^{*}}S_{(e,f)}=0$ due to the fact $P_{f}P_{X'}=P_{X'}P_{f}=0$.

(\textbf{LP4}). Since $X\in G_{F}^{0}$ is a sink, we only need to consider $e\in F$. To that end, suppose $X \subseteq F$ such that $X\notin G_{F}^{0}$; which is to say $X\subseteq F$ for which 
$$\{e\in\mathcal{G}^{1}:s(e)\in r(X, F\setminus X)\}\subseteq F.$$
Then, by \textbf{ExL4}, we have
{\footnotesize  
\begin{align*}
\Big(\prod\limits_{e\in X}\mathsf{s}_{e^{*}}\mathsf{s}_{e}\prod\limits_{e'\in F\setminus X}(1-\mathsf{s}_{e'^{*}}\mathsf{s}_{e'})\Big)&\Big(1-\sum\limits_{f\in F}\mathsf{s}_{f}\mathsf{s}_{f^{*}}\Big) = \Big(\sum\limits_{v\in r(X,F\setminus X)} \mathsf{p}_{v}-\sum\limits_{\{f\in F:s(f)\in r(X,F\setminus X)\}}\mathsf{s}_{f}\mathsf{s}_{f^{*}}\Big)\\
&=\sum\limits_{v\in r(X,F\setminus X)}\Big(\mathsf{p}_{v}-\sum\limits_{\{f\in F:s(f)=v\}}\mathsf{s}_{f}\mathsf{s}_{f^{*}}\Big)=0.
\end{align*}\par}
On the other hand, for a fixed $e\in F$, and $X\subseteq F$ such that $e\notin X$ (and so $e\in F\setminus X$), \textbf{EL2} implies
\begin{align*}
 &\mathsf{s}_{e^{*}}\mathsf{s}_{e}\Big(\prod\limits_{g\in X}\mathsf{s}_{g^{*}}\mathsf{s}_{g}\prod\limits_{g'\in F\setminus X}(1-\mathsf{s}_{g'^{*}}\mathsf{s}_{g'})\Big)\Big(1-\sum\limits_{f\in F}\mathsf{s}_{f}\mathsf{s}_{f^{*}}\Big)\\
&=\mathsf{s}_{e^{*}}\mathsf{s}_{e}(1-\mathsf{s}_{e^{*}}\mathsf{s}_{e})\Big(\prod\limits_{g\in X}\mathsf{s}_{g^{*}}\mathsf{s}_{g}\prod\limits_{g'\in F\setminus X,\ g'\neq e}(1-\mathsf{s}_{g'^{*}}\mathsf{s}_{g'})\Big)\Big(1-\sum\limits_{f\in F}\mathsf{s}_{f}\mathsf{s}_{f^{*}}\Big)\\
&=0.
\end {align*}
Combining the previous two calculations, we can deduce
\begin{align*}
\mathsf{s}_{e^{*}}\mathsf{s}_{e}\Big(\sum\limits_{\substack{X\in G_{F}^{0}:\\ e\in X}}P_{X}\Big)&=\mathsf{s}_{e^{*}}\mathsf{s}_{e}\Bigg(\bigg(\sum\limits_{\emptyset\neq X\subseteq F}\Big(\prod\limits_{g\in X}\mathsf{s}_{g^{*}}\mathsf{s}_{g}\prod\limits_{g'\in F\setminus X}(1-\mathsf{s}_{g'^{*}}\mathsf{s}_{g'})\Big)\bigg)\Big(1-\sum\limits_{f\in F}\mathsf{s}_{f}\mathsf{s}_{f^{*}}\Big)\Bigg),\\
&\text{which by Lemma \ref{lci},}\\
&=\mathsf{s}_{e^{*}}\mathsf{s}_{e}\Bigg(\bigg(1-\prod\limits_{g'\in F}(1-\mathsf{s}_{g'^{*}}\mathsf{s}_{g'})\bigg)\Big(1-\sum\limits_{f\in F}\mathsf{s}_{f}\mathsf{s}_{f^{*}}\Big)\Bigg)\\
&=\Big(\mathsf{s}_{e^{*}}\mathsf{s}_{e}-\mathsf{s}_{e^{*}}\mathsf{s}_{e}(1-\mathsf{s}_{e^{*}}\mathsf{s}_{e})\prod\limits_{g'\in F, g'\neq e}(1-\mathsf{s}_{g'^{*}}\mathsf{s}_{g'})\Big)\Big(1-\sum\limits_{f\in F}\mathsf{s}_{f}\mathsf{s}_{f^{*}}\Big)\\
&=\mathsf{s}_{e^{*}}\mathsf{s}_{e}\Big(1-\sum\limits_{f\in F}\mathsf{s}_{f}\mathsf{s}_{f^{*}}\Big),\\
&\text{ and since } \mathsf{s}_{e^{*}}\mathsf{s}_{e}\mathsf{s}_{f}\mathsf{s}_{f^{*}}=0 \text{ if } s(f)\notin r(e),\\
 &=\mathsf{s}_{e^{*}}\mathsf{s}_{e}\Big(1-\sum\limits_{\{f:s(f)\in r(e)\}}\mathsf{s}_{f}\mathsf{s}_{f^{*}}\Big)=\mathsf{s}_{e^{*}}\mathsf{s}_{e}\Big(1-\sum\limits_{\{f:s(f)\in r(e)\}}P_{f}\Big).
\end{align*}
Note that, by \textbf{ExL2}, 
\begin{align*}
\mathsf{s}_{e}\Big(\sum\limits_{\{X:e\in X\}}P_{X}\Big) &=\mathsf{s}_{e}\bigg(\mathsf{s}_{e^{*}}\mathsf{s}_{e}\Big(\sum\limits_{\{X:e\in X\}}P_{X}\Big)\bigg)\\ 
&=\mathsf{s}_{e}\bigg(\mathsf{s}_{e^{*}}\mathsf{s}_{e}\Big(1-\sum\limits_{\{f:s(f)\in r(e)\}}P_{f}\Big)\bigg)=\mathsf{s}_{e}\Big(1-\sum\limits_{\{f:s(f)\in r(e)\}}P_{f}\Big)\\
\end{align*}
Thus,
{\footnotesize
\begin{align*}
\sum\limits_{\{f: s(f)\in r(e)\}}S_{(e,f)}S_{(e,f)^{*}}+&\sum\limits_{\{X:e\in X\}}S_{(e,X)}S_{(e,X)^{*}}=\sum\limits_{\{f: s(f)\in r(e)\}}S_{(e,f)}S_{(e,f)^{*}}+\sum\limits_{\{X:e\in X\}}S_{(e,X)}S_{(e,X)^{*}}\\
&=\mathsf{s}_{e}\Big(\sum\limits_{\{f:s(f)\in r(e)\}}P_{f}\Big)\mathsf{s}_{e^{*}} +\mathsf{s}_{e}\Big(\sum\limits_{\{X:e\in X\}}P_{X}\Big)\mathsf{s}_{e^{*}}\\
&=\mathsf{s}_{e}\Big(\sum\limits_{\{f:s(f)\in r(e)\}}P_{f}\Big)\mathsf{s}_{e^{*}}+\mathsf{s}_{e}\Big(1-\sum\limits_{\{f:s(f)\in r(e)\}}P_{f}\Big)\mathsf{s}_{e^{*}}\\
&=\mathsf{s}_{e}\mathsf{s}_{e^{*}}=P_{e}.
\end{align*}\par}
\end{proof} 

\begin{lemma}\label{lfg2}
 Let $\mathcal{G}$ be an ultragraph with no singular vertices, $F\subseteq \mathcal{G}^{1}$ finite, and $G_{F}$ the associated finite graph. Then, there exists a $\mathbb{Z}$-graded injective $R$-algebra homomorphism
$$\pi_{F}: L_{R}(G_{F})\to \mathcal{EL}_{R}(\mathcal{G})$$
such that $\pi_{F}(p_{e})=P_{e},\ \pi_{F}(p_{X})=P_{X},\ \pi_{F}(s_{(e,f)})=S_{(e,f)},\ \text{and},\  \pi_{F}(s_{(e,X)})=S_{(e,X)}$.
\end{lemma}

\begin{proof}
By \cite[Proposition 3.4]{Tomforde:2011aa}, we have 
$$L_{R}(G_{F}):=L_{R}(s,p)=\text{span}_{R}\{s_{\alpha}s_{\beta^{*}}:r(\alpha)=r(\beta)\};$$
and is a $\mathbb{Z}$-graded ring. Specifically, $L_{R}(G_{F})=\bigoplus\limits_{i\in\mathbb{Z}}L_{R}(G_{F})_{i}$, where
$$L_{R}(G_{F})_{i}=\text{span}_{R}\{s_{\alpha}s_{\beta^{*}}:r(\alpha)=r(\beta) \text{ and } |\alpha|-|\beta|=i\}.$$
By Lemma \ref{lfg}, we have a Leavitt $G_{F}$-family $\{P,S\}\subseteq\mathcal{EL}_{R}(\mathcal{G})$; and so by the universal mapping property of $L_{R}(G_{F})$, we have an $R$-algebra homomorphism
$$\pi_{F}:L_{R}(G_{F})\to \mathcal{EL}_{R}(\mathcal{G})$$
such that $\pi_{F}(p_{e})=P_{e},\ \pi_{F}(p_{X})=P_{X},\ \pi_{F}(s_{(e,f)})=S_{(e,f)},\ \text{and},\  \pi_{F}(s_{(e,X)})=S_{(e,X)}$. Further, using Lemma \ref{lex2} and Tomforde's aforementioned result on the $\mathbb{Z}$-grading of $L_{R}(G_{F})$, we can deduce 
$$\pi_{F}(L_{R}(G_{F})_{i})\subseteq \mathcal{EL}_{R}(\mathcal{G})_{i};$$
thus, $\pi_{F}$ is a graded homomorphism. Finally, Lemma \ref{llgex2} and \cite[Theorem 2.6]{M.-Imanfar:2017aa} imply that, for $r\neq 0$: $r\mathsf{p}_{v}\neq 0$, $r\mathsf{s}_{e}\neq 0$, and $r\mathsf{s}_{e^{*}}\neq 0$. This in turn implies that, for $r\neq 0$,
$$rP_{e}\neq0,\ rP_{X}\neq 0,\ rS_{(e,f)}\neq0, \text{ and } rS_{(e,X)}\neq0.$$
And so, by Theorem \ref{ggrdhomom}, $\pi_{F}$ is injective. 
\end{proof}

We are now ready to prove Theorem \ref{thm:keythm}.

\begin{proof}[Proof of Theorem \ref{thm:keythm}] Let  $F_{1}\subseteq F_{2}\subseteq\cdots\subseteq F_{n} \cdots$ be an increasing sequence of finite subsets of $\mathcal{G}^{1}$ such that  $\bigcup\limits_{n\in\mathbb{N}}F_{n}=\mathcal{G}^{1}$. For each $F_{n}$, let $G_{F_{n}}$ be the associated finite graph. By Lemma \ref{lfg2}, for each $n$, there exists an injective $R$-algebra homomorphism
$$\pi_{F_{n}}:L_{R}(G_{F_{n}})\to\mathcal{EL}_{R}(\mathcal{G}).$$
This in turn gives us a $\mathbb{Z}$-graded ring homomorphism
$$\pi\circ\pi_{F_{n}}:L_{R}(G_{F_{n}})\to S$$
for each $n$. By our assumption on $\pi$, Theorem \ref{ggrdhomom}, and Lemma \ref{lfg2}, $\pi\circ\pi_{F_{n}}$ is injective for each $n$.

Finally, let $V \subseteq G^{0}$ and $F\subseteq\mathcal{G}^{1}$ be finite, and $s^{-1}(v)\subseteq F$ for each $v\in V$. Fix $n$ large enough such that $F\subseteq F_{n}$. Since we are assuming $\mathcal{G}$ doesn't have any singular vertices, we have 
$$\mathsf{p}_{v}=\sum\limits_{\{e:s(e)=v\}}\mathsf{s}_{e}\mathsf{s}_{e^{*}}=\sum\limits_{\{e:s(e)=v\}}\pi_{F_{n}}(p_{e})$$
for each $v\in V$; and so, for each $v\in V$, $\mathsf{p}_{v}\in\text{im}\pi_{F_{n}}$. For $e\in F\subseteq F_{n},$ we have 
\begin{align*}
\sum\limits_{\{f:s(f)\in r(e)\}}\pi_{F_{n}}(s_{(e,f)})&+\sum\limits_{\{X:e\in X\}}\pi_{F_{n}}(s_{(e,X)}) =\\
&\mathsf{s}_{e}\Big(\sum\limits_{\{f:s(f)\in r(e)\}}P_{f} +(1-\sum\limits_{\{f:s(f)\in r(e)\}}P_{f})\Big)=\mathsf{s}_{e};
\end{align*}
we can similarly show
$$\sum\limits_{\{f:s(f)\in r(e)\}}\pi_{F_{n}}(s_{(e,f)^{*}})+\sum\limits_{\{X:e\in X\}}\pi_{F_{n}}(s_{(e,X)^{*}})=\mathsf{s}_{e^{*}}.$$
Thus, for each $v\in V$ and $e\in F$, we have that $\mathsf{p}_{v},\mathsf{s}_{e},\mathsf{s}_{e^{*}}\in\text{im}\pi_{F_{n}}$. That is, the $R$-algebra generated by $\{\mathsf{p}_{v},\mathsf{s}_{e},\mathsf{s}_{e^{*}}\}$ is a subset of $\text{im}\pi_{F_{n}}$. In turn, this implies
$$\mathcal{EL}_{R}(\mathcal{G})=\bigcup\limits_{n\in\mathbb{N}}\text{im}\pi_{F_{n}}.$$
Combining this with the fact $\pi\circ\pi_{F_{n}}$ is injective for each $n$, we can conclude $\pi$ is injective as well.
\end{proof}

\begin{theorem}\label{texlg}
If $\mathcal{G}$ is an ultragraph with no singular vertices, then
$$\mathcal{EL}_{R}(\mathcal{G})\cong L_{R}(\mathcal{G}).$$ 
\end{theorem}

\begin{proof} We have, by Theorem \ref{uggrade}, 
$$L_{R}(\mathcal{G})=\text{span}_{R}\{s_{\alpha}p_{A}s_{\beta^{*}}:r(\alpha)\cap A\cap r(\beta)\neq\emptyset\}$$
and is a $\mathbb{Z}$-graded ring. Particularly, $L_{R}(\mathcal{G})=\bigoplus\limits_{i\in\mathbb{Z}}L_{R}(\mathcal{G})_{i}$ where 
$$L_{R}(\mathcal{G})_{i}=\text{span}_{R}\{s_{\alpha}p_{A}s_{\beta^{*}}:r(\alpha)\cap A\cap r(\beta)\neq\emptyset,\ |\alpha|-|\beta|=i\}.$$
By Lemma \ref{llgex2}, there exists a surjective map $\varphi:\mathcal{EL}_{R}(\mathcal{G})\to L_{R}(\mathcal{G})$ such that $\varphi(\mathsf{p}_{v})=p_{v},\ \varphi(\mathsf{s}_{e})=s_{e},\ \varphi(\mathsf{s}_{e^{*}})=s_{e^{*}}$. It is straightforward to see $\varphi(\mathcal{EL}_{R}(\mathcal{G})_{i})\subseteq L_{R}(\mathcal{G})_{i},$ meaning $\varphi$ is a $\mathbb{Z}$-graded $R$-algebra homomorphism. By Theorem \ref{uggrade} and Theorem \ref{thm:keythm}, $\varphi$ is also injective. Thus, $\mathcal{EL}_{R}(\mathcal{G})\cong L_{R}(\mathcal{G})$.
\end{proof}

\section{$L_{R}(\mathcal{G})\cong QL_{R}(E_{\mathcal{G}})Q$}

Our next crucial step is to show $L_{R}(\mathcal{G})\cong QL_{R}(E_{\mathcal{G}})Q$. But, before we do, there are some results we first need to establish analogous to \cite[Lemma 5.11, Proposition 5.15, Lemma 5.19, and Theorem 5.22]{Takeshi-Katsura-Paul-Muhly-Aidan-Sims--Mark-Tomforde:2010aa}. By \cite[Lemma 4.6 and Definition 4,7]{Takeshi-Katsura-Paul-Muhly-Aidan-Sims--Mark-Tomforde:2010aa}, for each $x\in E_{\mathcal{G}}^{0}$, there exists a unique path $\alpha_{x}$ in the subgraph $F$ of $E_{\mathcal{G}}$ such that $r(\alpha_{x})=x$ and $s(\alpha_{x})\in W_{0}\sqcup\Gamma_{0}$ (see Remark \ref{FsubE}). Given the enumeration of $\mathcal{G}^{1}$, and identifying $G^{0}$ as a subset of $E_{\mathcal{G}}^{0}$, let
\begin{align*}
&P_{v}:=t_{\alpha_{v}}t_{\alpha_{v}^{*}}\ \text{for}\ v\in G^{0},\\
&S_{e_{n}}:=\sum\limits_{x\in X(e_{n})}t_{\alpha_{s(e_{n})}}t_{(e_{n},x)}t_{\alpha_{x}^{*}}, \text{and}\\
&S_{e_{n}^{*}}:=\sum\limits_{x\in X(e_{n})}t_{\alpha_{x}}t_{(e_{n},x)^{*}}t_{\alpha_{s(e_{n})}^{*}}\ \text{for}\ e_{n}\in\mathcal{G}^{1}.
\end{align*}
The following quick claim is useful:
\begin{claim}\label{corth}
For $e_{n}\in\mathcal{G}^{1}$, let $x,y\in X(e_{n})$. If $x\neq y$, then 
$$t_{\alpha_{x}^{*}}t_{\alpha_{y}}=0.$$\
\end{claim}

\begin{proof}
Note that $t_{\alpha_{x}^{*}}t_{\alpha_{y}}\neq0$ if and only if $\alpha_{x}$ extends $\alpha_{y}$, or vice versa. If $x,y\in\Delta_{n}$, then \cite[Lemma 4.6 (2)]{Takeshi-Katsura-Paul-Muhly-Aidan-Sims--Mark-Tomforde:2010aa} implies $t_{\alpha_{x}^{*}}t_{\alpha_{y}}=0$. If $x,y\in G^{0}$, then \cite[Lemma 4.6 (1)]{Takeshi-Katsura-Paul-Muhly-Aidan-Sims--Mark-Tomforde:2010aa} implies $t_{\alpha_{x}^{*}}t_{\alpha_{y}}=0$.

 Finally, WLOG, suppose $x\in G^{0}$ and $y\in \Delta_{n}$. In this case, \cite[Lemma 4.6 (1)]{Takeshi-Katsura-Paul-Muhly-Aidan-Sims--Mark-Tomforde:2010aa} implies $\alpha_{y}$ can't extend $\alpha_{x}$. And so the only option left is for $\alpha_{y}$ to extend $\alpha_{x}$. But, for 
$$r'(y):=\{v\in G^{0}:|\sigma(v)|\geq n,\ \sigma(v)|_{n}=y\}\ \ \text{and}\ \ x\in G^{0}\cap X(e_{n}),$$
\cite[Lemma 4.6 (3)]{Takeshi-Katsura-Paul-Muhly-Aidan-Sims--Mark-Tomforde:2010aa} implies $t_{\alpha_{x}^{*}}t_{\alpha_{y}}=0$.
\end{proof}

\begin{lemma}\label{lex3}
The set $\{P_{v},\ S_{e_{n}},\ S_{e_{n}^{*}}\}$ forms an Exel-Laca $\mathcal{G}$-family in $L_{R}(E_{\mathcal{G}})=L_{R}(\{t,q\})$.
\end{lemma}

\begin{proof} (\textbf{ExL2}). Using the fact $t_{\alpha_{v}}=t_{\alpha_{v}}t_{\alpha_{v}^{*}}t_{\alpha_{v}}$ (similarly, $t_{\alpha_{v}^{*}}=t_{\alpha_{v}^{*}}t_{\alpha_{v}}t_{\alpha_{v}^{*}})$, a direct calculation shows
$$P_{s(e_{n})}S_{e_{n}}=S_{e_{n}}\ \ \ \text{and}\ \ \ S_{e_{n}^{*}}=S_{e_{n}^{*}}P_{s(e_{n})}.$$
Further, using the properties of a Leavitt $E_{\mathcal{G}}$-family, we can show
{\footnotesize
\begin{align*}
S_{e_{n}}S_{e_{n}^{*}}S_{e_{n}}&=\Bigg(\sum\limits_{x\in X(e_{n})}t_{\alpha_{s(e_{n})}}t_{(e_{n},x)}t_{\alpha_{x}^{*}}\Bigg) \Bigg(\sum\limits_{y\in X(e_{n})}t_{\alpha_{y}}t_{(e_{n},y)^{*}}t_{\alpha_{s(e_{n})}^{*}} \Bigg) \Bigg( \sum\limits_{z\in X(e_{n})}t_{\alpha_{s(e_{n})}}t_{(e_{n},z)}t_{\alpha_{z}^{*}}\Bigg)\\
&=\Bigg(\sum\limits_{x\in X(e_{n})}t_{\alpha_{s(e_{n})}}t_{(e_{n},x)}t_{\alpha_{x}^{*}}\Bigg) \Bigg(\sum\limits_{y\in X(e_{n})}\ \sum\limits_{z\in X(e_{n})}t_{\alpha_{y}}t_{(e_{n},y)^{*}}t_{\alpha_{s(e_{n})}^{*}} t_{\alpha_{s(e_{n})}}t_{(e_{n},z)}t_{\alpha_{z}^{*}} \Bigg)\\
&=\Bigg(\sum\limits_{x\in X(e_{n})}t_{\alpha_{s(e_{n})}}t_{(e_{n},x)}t_{\alpha_{x}^{*}}\Bigg) \Bigg(\sum\limits_{y\in X(e_{n})}t_{\alpha_{y}}t_{\alpha_{y}^{*}} \Bigg)\\
&=\sum\limits_{x\in X(e_{n})}\sum\limits_{y\in X(e_{n})}t_{\alpha_{s(e_{n})}}t_{(e_{n},x)}t_{\alpha_{x}^{*}}t_{\alpha_{y}}t_{\alpha_{y}^{*}}\\
&=\sum\limits_{x\in X(e_{n})}t_{\alpha_{s(e_{n})}}t_{(e_{n},x)}t_{\alpha_{x}^{*}}t_{\alpha_{x}}t_{\alpha_{x}^{*}} =\sum\limits_{x\in X(e_{n})}t_{\alpha_{s(e_{n})}}t_{(e_{n},x)}t_{\alpha_{x}^{*}}=S_{e_{n}};
\end{align*}\par}
a similar argument shows $S_{e_{n}^{*}}S_{e_{n}}S_{e_{n}^{*}}=S_{e_{n}^{*}}$.\\

(\textbf{ExL3}). First, suppose $n\neq m$. We have,
\begin{center}
$t_{\alpha_{s(e_{n})}^{*}}t_{\alpha_{s(e_{m})}}=
\begin{cases}
t_{\alpha} & \text{if } \alpha_{s(e_{m})}=\alpha_{s(e_{n})}\alpha \\
t_{\alpha'^{*}} & \text{if } \alpha_{s(e_{n})}=\alpha_{s(e_{m})}\alpha' \\
0 & \text{otherwise}.
\end{cases}$ 
\end{center}
If $t_{\alpha_{s(e_{n})}^{*}}t_{\alpha_{s(e_{m})}}=0$, we are done. Otherwise, suppose $t_{\alpha_{s(e_{n})}^{*}}t_{\alpha_{s(e_{m})}}=t_{\alpha}$. Then,
\begin{align*}
S_{e_{n}^{*}}S_{e_{m}}&=\sum\limits_{x\in X(e_{n})}\ \sum\limits_{y\in X(e_{m})}t_{\alpha_{x}}t_{(e_{n},x)^{*}}t_{\alpha_{s(e_{n})}^{*}} t_{\alpha_{s(e_{m})}}t_{(e_{m},y)}t_{\alpha_{y}^{*}}\\
&=\sum\limits_{x\in X(e_{n})}\ \sum\limits_{y\in X(e_{m})}t_{\alpha_{x}}t_{(e_{n},x)^{*}}t_{\alpha}t_{(e_{m},y)}t_{\alpha_{y}^{*}}\\
&=\sum\limits_{x\in X(e_{n})}\ \sum\limits_{y\in X(e_{m})}t_{\alpha_{x}}t_{(e_{n},x)^{*}}q_{s(e_{n})}t_{\alpha}t_{(e_{m},y)}t_{\alpha_{y}^{*}}.
\end{align*}
 Since $s(e_{n})\in G^{0}$, \cite[Lemma 4.6 (1)]{Takeshi-Katsura-Paul-Muhly-Aidan-Sims--Mark-Tomforde:2010aa} implies $|\alpha|=0$ (i.e., $t_{\alpha}=q_{s(e_{n})}$). Thus,
$$S_{e_{n}^{*}}S_{e_{m}}=\sum\limits_{x\in X(e_{n})}\ \sum\limits_{y\in X(e_{m})}t_{\alpha_{x}}t_{(e_{n},x)^{*}}q_{s(e_{n})}t_{(e_{m},y)}t_{\alpha_{y}^{*}}.$$
If $s(e_{n})\neq s(e_{m})$, then $q_{s(e_{n})}q_{s(e_{m})}=0 \implies S_{e_{n}^{*}}S_{e_{m}}=0$. If $s(e_{n})=s(e_{m})$, then
$$S_{e_{n}^{*}}S_{e_{m}}=\sum\limits_{x\in X(e_{n})}\ \sum\limits_{y\in X(e_{m})}t_{\alpha_{x}}t_{(e_{n},x)^{*}}t_{(e_{m},y)}t_{\alpha_{y}^{*}};$$
by \textbf{LP3}, we have $t_{(e_{n},x)^{*}}t_{(e_{m},y)}=0$ and so $S_{e_{n}^{*}}S_{e_{m}}=0$. We can use a similar argument to show this conclusion still holds for the choice $t_{\alpha_{s(e_{n})}^{*}}t_{\alpha_{s(e_{m})}}=t_{\alpha'^{*}}$.\\

(\textbf{ExL4}). For $E_{\mathcal{G}}=(E^{0},E^{1},s_{E},r_{E})$, note that, for each $n$,
$$s_{E}^{-1}(s(e_{n}))=\{(e_{n},x):x\in X(e_{n})\}.$$
And so,
\begin{align*}
\sum\limits_{v=s(e_{n})}S_{e_{n}}S_{e_{n}^{*}}&=\sum\limits_{v=s(e_{n})}\Bigg(\sum\limits_{x\in X(e_{n})}\sum\limits_{y\in X(e_{n})}t_{\alpha_{s(e_{n})}}t_{(e_{n},x)}t_{\alpha_{x}^{*}}t_{\alpha_{y}}t_{(e_{n},y)^{*}}t_{\alpha_{s(e_{n})}^{*}} \Bigg),\\
&\text{by Claim \ref{corth}, } x\neq y \implies \alpha_{x}^{*}\alpha_{y}=0, \text{and so}\\
&=\sum\limits_{v=s(e_{n})}\sum\limits_{x\in X(e_{n})}t_{\alpha_{s(e_{n})}}t_{(e_{n},x)}t_{(e_{n},x)^{*}}t_{\alpha_{s(e_{n})}^{*}}\\
&=t_{\alpha_{s(e_{n})}}\Bigg(\sum\limits_{v=s(e_{n})}\sum\limits_{x\in X(e_{n})}t_{(e_{n},x)}t_{(e_{n},x)^{*}} \Bigg)t_{\alpha_{s(e_{n})}^{*}},\\
&\text{by \textbf{LP2}, for $v=s(e_{n})$,}\\
&=t_{\alpha_{v}}(q_{v})t_{\alpha_{v}^{*}}=t_{\alpha_{v}}t_{\alpha_{v}^{*}}=P_{v}.
\end{align*}

(\textbf{EL1} \& \textbf{EL2}). For $v,w \in G^{0}$, utilizing \cite[Lemma 4.6 (1)]{Takeshi-Katsura-Paul-Muhly-Aidan-Sims--Mark-Tomforde:2010aa}, we have 
\begin{center}
$t_{\alpha_{v}^{*}}t_{\alpha_{w}}=
\begin{cases}
q_{v} & \text{if } v=w\\
0 & \text{if } v\neq w.
\end{cases}$ 
\end{center}
So, for $P_{v}:=t_{\alpha_{v}}t_{\alpha_{v}^{*}}$ and $P_{w}:=t_{\alpha_{w}}t_{\alpha_{w}^{*}}$,
\begin{center}
$P_{v}P_{w}=
\begin{cases}
P_{v} & \text{if } v=w\\
0 & \text{if } v\neq w;
\end{cases}$ 
\end{center}
meaning, the set $\{P_{v}\}_{v\in G^{0}}$ consists of pairwise orthogonal idempotent elements. In fact, elements of the form $t_{\alpha_{x}}t_{\alpha_{x}^{*}}$ are idempotents which pairwise commute (see \cite[Equation 2.1]{Tomforde:2011aa}); since a calculation we've previously done (see the proof of \textbf{ExL2}) shows 
$$S_{e_{n}^{*}}S_{e_{n}}=\sum\limits_{x\in X(e_{n})}t_{\alpha_{x}}t_{\alpha_{x}^{*}},$$
we can conclude $\{S_{e_{n}^{*}}S_{e_{n}}\}_{e_{n}\in\mathcal{G}^{1}}$ is a set of pairwise commuting idempotents.

(\textbf{EL3}). For $n\in\mathbb{N}$ and $\omega\in\{0,1\}^{n}\setminus\{0^{n}\}$, set
$$r'(\omega):=\{v\in G^{0}:|\sigma(v)|\geq n,\ \sigma(v)|_{n}=\omega\}.$$
Further, let $G_{n}^{0}$ denote $X(e_{n})\cap G^{0}$. Note that for each $\omega\in\{0,1\}^{n}\setminus\{0^{n}\}$, we have 
$$r'(\omega)\cap G_{n}^{0}=\emptyset$$
since $v\in G_{n}^{0}\implies \sigma(\omega)<n$. Given $\omega,\omega'\in\{0,1\}^{n}\setminus\{0^{n}\}$ such that $\omega\neq\omega'$, we have
$$r'(\omega)\cap r'(\omega')=\emptyset$$
due to the fact $v\in r'(\omega)\implies \sigma(\omega)|_{n}=\omega$. Further, \cite[Lemma 4.4]{Takeshi-Katsura-Paul-Muhly-Aidan-Sims--Mark-Tomforde:2010aa} states
$$r(e_{n})=G_{n}^{0}\cup\Bigg(\bigcup\limits_{\{\omega\in\{0,1\}^{n}:\omega_{n}=1\}}r'(\omega)\Bigg).$$
This means, for each $v\in r(e_{n})$, $v$ belongs exclusively to $G_{n}^{0}$, or $r'(\omega)$ for a unique $\omega$. Let $X_{\Delta,n}:=X(e_{n})\cap \Delta$. Then, 
$$S_{e_{n}^{*}}S_{e_{n}}P_{v}=\sum\limits_{x\in X(e_{n})}t_{\alpha_{x}}t_{\alpha_{x}^{*}}t_{\alpha_{v}}t_{\alpha_{v}^{*}}=\sum\limits_{v'\in G_{n}^{0}}t_{\alpha_{v'}}t_{\alpha_{v'}^{*}}t_{\alpha_{v}}t_{\alpha_{v}^{*}}+\sum\limits_{\omega\in X_{\Delta,n}}t_{\alpha_{\omega}}t_{\alpha_{\omega}^{*}}t_{\alpha_{v}}t_{\alpha_{v}^{*}}.$$
Now, $t_{\alpha_{\omega}^{*}}t_{\alpha_{v}}=0$ unless $\alpha_{v}$ extends $\alpha_{\omega}$, or $\alpha_{\omega}$ extends $\alpha_{v}$. However,
since $v\in G^{0}$, \cite[Lemma 4.6 (1)]{Takeshi-Katsura-Paul-Muhly-Aidan-Sims--Mark-Tomforde:2010aa} tells us $\alpha_{\omega}$ can't extend $\alpha_{v}$. Moreover,  \cite[Lemma 4.6 (3)]{Takeshi-Katsura-Paul-Muhly-Aidan-Sims--Mark-Tomforde:2010aa} tells us there exists a path in $F$ from $\omega\in\Delta$ to $v\in G^{0}$ if and only if $v\in r'(\omega)$. And so if $v\notin r(e_{n})$, then for $\omega\in X_{\Delta,n}$, we have $t_{\alpha_{\omega}}t_{\alpha_{\omega}^{*}}t_{\alpha_{v}}t_{\alpha_{v}^{*}}=0$. By \cite[Lemma 4.6 (1)]{Takeshi-Katsura-Paul-Muhly-Aidan-Sims--Mark-Tomforde:2010aa} and \cite[Equation 2.1]{Tomforde:2011aa}, for $v\notin r(e_{n})$ and $v'\in G_{n}^{0}$, $t_{\alpha_{v'}}t_{\alpha_{v'}^{*}}t_{\alpha_{v}}t_{\alpha_{v}^{*}}=0$. Thus,
$$v\notin r(e_{n})\implies S_{e_{n}^{*}}S_{e_{n}}P_{v}=0.$$
On the other hand, suppose $v\in r(e_{n})$. If $v\in G_{n}^{0}$, it means $v\notin r'(\omega)$ for any $\omega\in X_{\Delta,n}$. This gives us 
$$\sum\limits_{\omega\in X_{\Delta,n}}t_{\alpha_{\omega}}t_{\alpha_{\omega}^{*}}t_{\alpha_{v}}t_{\alpha_{v}^{*}}=0\ \ \text{and}\ \ \sum\limits_{v'\in G_{n}^{0}}t_{\alpha_{v'}}t_{\alpha_{v'}^{*}}t_{\alpha_{v}}t_{\alpha_{v}^{*}}=t_{\alpha_{v}}t_{\alpha_{v}^{*}};$$
that is, $S_{e_{n}^{*}}S_{e_{n}}P_{v}=P_{v}$. If $v\notin G_{n}^{0}$, then $v\in r'(\omega')$ for a unique $\omega'\in X_{\Delta,n}$, which in turn means
$$S_{e_{n}^{*}}S_{e_{n}}P_{v}=t_{\alpha_{\omega'}}t_{\alpha_{\omega'}^{*}}t_{\alpha_{v}}t_{\alpha_{v}^{*}}.$$
Because $\alpha_{v}$ must extend $\alpha_{\omega'}$, $t_{\alpha_{\omega'}}t_{\alpha_{\omega'}^{*}}t_{\alpha_{v}}t_{\alpha_{v}^{*}}=t_{\alpha_{v}}t_{\alpha_{v}^{*}}$; again, it must be $S_{e_{n}^{*}}S_{e_{n}}P_{v}=P_{v}$. A similar argument can be used to show $P_{v}S_{e_{n}^{*}}S_{e_{n}}=P_{v}$. Thus,
\begin{center}
$P_{v}S_{e_{n}^{*}}S_{e_{n}}=S_{e_{n}^{*}}S_{e_{n}}P_{v}=
\begin{cases}
P_{v} & \text{if } v\in r(e_{n}),\\
0 & \text{otherwise.}
\end{cases}$ 
\end{center}

(\textbf{ExL1} \& \textbf{EL4}). Finally, let $N,M\subseteq\mathbb{N}$ be finite. Further, suppose $N\neq\emptyset$, $N\cap M=\emptyset$, and 
$$r(N,M):=\bigcap\limits_{n\in N}r(e_{n})\setminus\bigcup\limits_{m\in M}r(e_{m})$$
is finite. We want to show
$$\prod\limits_{n\in N}S_{e_{n}^{*}}S_{e_{n}}\prod\limits_{m\in M}(1-S_{e_{m}^{*}}S_{e_{m}})=\sum\limits_{v\in r(N,M)}P_{v}.$$
We will use induction on $|N\cup M|=k$. To that end, for $k=1$,  we have $r(N,M)=r(e_{n})$ and 
$$\prod\limits_{n\in N}S_{e_{n}^{*}}S_{e_{n}}\prod\limits_{m\in M}(1-S_{e_{m}^{*}}S_{e_{m}})=S_{e_{n}^{*}}S_{e_{n}}.$$
Note, by the definitions of $\Delta$ and $X(e_{n})$, $|r(e_{n})|<\infty\implies X_{\Delta,n}=\emptyset\implies r(e_{n})=G_{n}^{0}$. Thus,
$$S_{e_{n}^{*}}S_{e_{n}}=\sum\limits_{x\in X(e_{n})}t_{\alpha_{x}}t_{\alpha_{x}^{*}}=\sum\limits_{v\in G_{n}^{0}}t_{\alpha_{v}}t_{\alpha_{v}^{*}}+\sum\limits_{\omega\in X_{\Delta,n}}t_{\alpha_{\omega}}t_{\alpha_{\omega}^{*}}=\sum\limits_{v\in G_{n}^{0}}t_{\alpha_{v}}t_{\alpha_{v}^{*}}=\sum\limits_{v\in r(e_{n})}P_{v},$$
completing the base case.

Now, suppose our assumption holds true for $|N\cup M|\leq  k$. Let, $N',\ M',$ be such that $|N'\cup M'|=k+1$. There are two cases to consider.\\

\textit{Case 1}: $N'=\{n'\}\cup N$. Note $|r(N,M)|<\infty \implies |r(N,M')|<\infty \text{ or } |r(e_{n'})|<\infty$. If $|r(N,M')|<\infty$, then by our inductive hypothesis, we have 
$$\prod\limits_{n\in N}S_{e_{n}^{*}}S_{e_{n}}\prod\limits_{m\in M'}(1-S_{e_{m}^{*}}S_{e_{m}})=\sum\limits_{v\in r(N,M')}P_{v}.$$
Using \textbf{EL3}, we can work out 
\begin{align*}
\prod\limits_{n\in N'}S_{e_{n}^{*}}S_{e_{n}}\prod\limits_{m\in M'}(1-&S_{e_{m}^{*}}S_{e_{m}})=S_{e_{n'}^{*}}S_{e_{n'}}\prod\limits_{n\in N}S_{e_{n}^{*}}S_{e_{n}}\prod\limits_{m\in M'}(1-S_{e_{m}^{*}}S_{e_{m}})\\
&=S_{e_{n}^{*}}S_{e_{n}}\sum\limits_{v\in r(N,M')}P_{v}=\sum\limits_{v\in r(N,M')\cap r(e_{n'})}P_{v}\\
&=\sum\limits_{v\in r(N',M')}P_{v}.
\end{align*}
On the other hand, if $|r(e_{n'})|<\infty$, then 
\begin{align*}
\prod\limits_{n\in N'}S_{e_{n}^{*}}S_{e_{n}}\prod\limits_{m\in M'}(1&-S_{e_{m}^{*}}S_{e_{m}})=S_{e_{n'}^{*}}S_{e_{n'}}\prod\limits_{n\in N}S_{e_{n}^{*}}S_{e_{n}}\prod\limits_{m\in M'}(1-S_{e_{m}^{*}}S_{e_{m}})\\
&=\Bigg(\sum\limits_{v\in r(e_{n'})}P_{v}\Bigg)\prod\limits_{n\in N}S_{e_{n}^{*}}S_{e_{n}}\prod\limits_{m\in M'}(1-S_{e_{m}^{*}}S_{e_{m}})\\
&=\Bigg(\sum\limits_{v\in r(e_{n'})\cap\big(\bigcap\limits_{n\in N}r(e_{n})\big)}P_{v}\Bigg)\prod\limits_{m\in M'}(1-S_{e_{m}^{*}}S_{e_{m}})\\
&=\sum\limits_{v\in r(e_{n'})\cap\big(\bigcap\limits_{n\in N}r(e_{n})\big)}P_{v}\ \ -\sum\limits_{v\in r(e_{n'})\cap\big(\bigcap\limits_{n\in N}r(e_{n})\ \cap\bigcap\limits_{m\in M'}r(e_{m})\big)}P_{v}\\
&=\sum\limits_{v\in r(N',M')}P_{v}.
\end{align*}

\textit{Case 2}: $M'=\{m'\}\cup M$. If $|r(N',M)|<\infty$,  then by our inductive hypothesis,
\begin{align*}
\prod\limits_{n\in N'}S_{e_{n}^{*}}S_{e_{n}}\prod\limits_{m\in M'}(1-S_{e_{m}^{*}}S_{e_{m}})&=\Bigg(\prod\limits_{n\in N'}S_{e_{n}^{*}}S_{e_{n}}\prod\limits_{m\in M}(1-S_{e_{m}^{*}}S_{e_{m}})\Bigg)(1-S_{e_{m'}^{*}}S_{e_{m'}})\\
&=\Bigg(\sum\limits_{v\in r(N',M)}P_{v}\Bigg)(1-S_{e_{m'}^{*}}S_{e_{m'}})\\
&=\sum\limits_{v\in r(N',M)}P_{v}\ \ \ -\sum\limits_{v\in r(N',M)\cap r(e_{m'})}P_{v}\\
&=\sum\limits_{v\in r(N',M)\setminus r(e_{m'})}P_{v}=\sum\limits_{v\in r(N',M')}P_{v}.
\end{align*}
On the other hand, if $|r(N',M)|=\infty$, then the fact $|r(N',M')|<\infty$ means $|r(N',\{m'\})|<\infty$. Applying our inductive hypothesis, we have
$$\prod\limits_{n\in N'}S_{e_{n}^{*}}S_{e_{n}}(1-S_{e_{m'}^{*}}S_{e_{m'}})=\sum\limits_{v\in r(N',\{m'\})}P_{v}.$$
And so, by also exploiting \textbf{EL2} and \textbf{EL3},
\begin{align*}
\prod\limits_{n\in N'}S_{e_{n}^{*}}S_{e_{n}}&\prod\limits_{m\in M'}(1-S_{e_{m}^{*}}S_{e_{m}})=\Bigg(\prod\limits_{n\in N'}S_{e_{n}^{*}}S_{e_{n}}\prod\limits_{m\in M}(1-S_{e_{m}^{*}}S_{e_{m}})\Bigg)(1-S_{e_{m'}^{*}}S_{e_{m'}})\\
&=\Bigg(\prod\limits_{n\in N'}S_{e_{n}^{*}}S_{e_{n}}(1-S_{e_{m'}^{*}}S_{e_{m'}})\Bigg)\prod\limits_{m\in M}(1-S_{e_{m}^{*}}S_{e_{m}})\\
&=\Bigg(\sum\limits_{v\in r(N',\{m'\})}P_{v}\Bigg)\prod\limits_{m\in M}(1-S_{e_{m}^{*}}S_{e_{m}}),\\
&\text{note: }v\in r(e_{m})\implies P_{v}(1-S_{e_{m}^{*}}S_{e_{m}})=0\implies P_{v}\prod\limits_{m\in M}(1-S_{e_{m}^{*}}S_{e_{m}})=0,\\
&\text{alternatively, }v\in\bigcap\limits_{m\in M}r(e_{m})^{c}\implies P_{v}\prod\limits_{m\in M}(1-S_{e_{m}^{*}}S_{e_{m}})=P_{v}, \text{ and so, }\\
&=\sum\limits_{v\in r(N',\{m'\})}P_{v}\ \ \ -\sum\limits_{v\in r(N',\{m'\})\cap\big(\bigcap\limits_{m\in M}r(e_{m})^{c}\big)}P_{v}\\
&=\sum\limits_{v\in r(N',M')}P_{v}.
\end{align*}
Thus, our lemma is proved.
\end{proof}

The following lemma closely follows \cite[Lemma 5.11]{Takeshi-Katsura-Paul-Muhly-Aidan-Sims--Mark-Tomforde:2010aa}.

\begin{lemma}\label{lglg}
Let $\mathcal{G}$ be an ultragraph with no singular vertices. For each $\omega\in\Delta_{k}\subseteq E_{\mathcal{G}}^{0}$, define 
$$N,M\subseteq\{1,2,\dots,k\}$$
such that for each $n\in N$, $\omega_{n}=1$, and for each $m\in M$, $\omega_{m}=0$. Then,
$$\prod\limits_{n\in N} S_{e_{n}^{*}}S_{e_{n}}\prod\limits_{m\in M}(1-S_{e_{m}^{*}}S_{e_{m}})=t_{\alpha_{\omega}}t_{\alpha_{\omega}^{*}}+\sum_{\substack{v\in r(N,M)\\ |\sigma(v)|<k}}P_{v}.$$
Note that $M=\{1,2,\dots,k\}\setminus N$.
\end{lemma}

\begin{proof}
We will show our result by induction on $k$. For $k=1$, we have $\omega=(1)$ and $r(N,M)=r(e_{1})$, and so
\begin{align*}
\prod\limits_{n\in N} S_{e_{n}^{*}}S_{e_{n}}\prod\limits_{m\in M}(&1-S_{e_{m}^{*}}S_{e_{m}})=S_{e_{1}^{*}}S_{e_{1}}\\
&=\sum\limits_{x\in X(e_{1})}t_{\alpha_{x}}t_{\alpha_{x}^{*}}=\sum\limits_{v\in G_{1}^{0}}t_{\alpha_{v}}t_{\alpha_{v}^{*}}+\sum\limits_{\omega\in X_{\Delta,1}}t_{\alpha_{\omega}}t_{\alpha_{\omega}^{*}},\\
&\text{since $|r(e_{1})|=\infty\implies G_{1}^{0}=\emptyset$, and $X_{\Delta,1}=\{\omega=(1)\}$},\\
&=t_{\alpha_{\omega}}t_{\alpha_{\omega}^{*}},
\end{align*}
completing our base case.
Now, let $\omega\in\Delta_{k}$. There are three cases to consider. For the first two cases, we will set $\omega':=\omega|_{k-1}\neq(0^{k-1})$. Further, since $r(\omega)\subseteq r(\omega')$ (in particular, $r(\omega)=r(\omega')\cap r(e_{k}) \text{ or }r(\omega)=r(\omega')\setminus r(e_{k})$), $\omega\in\Delta_{k}\implies \omega'\in\Delta_{k-1}$.

\textit{Case 1}: $\omega=(\omega',1)$. Let us set $N':=N\setminus\{k\}$, and note $N',M\subseteq\{1,2,\dots,k-1\}$. With that, we have
\begin{align*}
\prod\limits_{n\in N} S_{e_{n}^{*}}S_{e_{n}}\prod\limits_{m\in M}(1&-S_{e_{m}^{*}}S_{e_{m}})=S_{e_{k}^{*}}S_{e_{k}}\Bigg(\prod\limits_{n\in N'} S_{e_{n}^{*}}S_{e_{n}}\prod\limits_{m\in M}(1-S_{e_{m}^{*}}S_{e_{m}})\Bigg),\\
&\text{which, by our inductive hypothesis},\\
&=S_{e_{k}^{*}}S_{e_{k}}\Bigg(t_{\alpha_{\omega'}}t_{\alpha_{\omega'}^{*}}+\sum_{\substack{v\in r(N',M)\\ |\sigma(v)|<k-1}}P_{v}\Bigg)\\
&=S_{e_{k}^{*}}S_{e_{k}}t_{\alpha_{\omega'}}t_{\alpha_{\omega'}^{*}}+\sum_{\substack{v\in r(N',M)\cap r(e_{k})\\ |\sigma(v)|<k-1}}P_{v}\\
&=S_{e_{k}^{*}}S_{e_{k}}t_{\alpha_{\omega'}}t_{\alpha_{\omega'}^{*}}+\sum_{\substack{v\in r(N,M)\\ |\sigma(v)|<k-1}}P_{v}.
\end{align*}
Further,
\begin{align*}
S_{e_{k}^{*}}S_{e_{k}}t_{\alpha_{\omega'}}t_{\alpha_{\omega'}^{*}}&=\sum\limits_{x\in X(e_{k})}t_{\alpha_{x}}t_{\alpha_{x}^{*}}t_{\alpha_{\omega'}}t_{\alpha_{\omega'}^{*}}\\
&=\sum\limits_{v\in G_{k}^{0}}t_{\alpha_{v}}t_{\alpha_{v}^{*}}t_{\alpha_{\omega'}}t_{\alpha_{\omega'}^{*}}+\sum\limits_{\omega''\in X_{\Delta,k}}t_{\alpha_{\omega''}}t_{\alpha_{\omega''}^{*}}t_{\alpha_{\omega'}}t_{\alpha_{\omega'}^{*}}.
\end{align*}
Recall that $t_{\alpha_{x}}t_{\alpha_{x}^{*}}t_{\alpha_{y}}t_{\alpha_{y}^{*}}\neq0$ if and only if $\alpha_{x}$ extends $\alpha_{y}$, or vice versa. Thus, by \cite[Lemma 4.6 (2)]{Takeshi-Katsura-Paul-Muhly-Aidan-Sims--Mark-Tomforde:2010aa},  we have
$$\sum\limits_{\omega''\in X_{\Delta,k}}t_{\alpha_{\omega''}}t_{\alpha_{\omega''}^{*}}t_{\alpha_{\omega'}}t_{\alpha_{\omega'}^{*}}=t_{\alpha_{\omega}}t_{\alpha_{\omega}^{*}}.$$
Moreover, by \cite[Lemma 4.6 (1) \& (3)]{Takeshi-Katsura-Paul-Muhly-Aidan-Sims--Mark-Tomforde:2010aa}, we have
$$\sum\limits_{v\in G_{k}^{0}}t_{\alpha_{v}}t_{\alpha_{v}^{*}}t_{\alpha_{\omega'}}t_{\alpha_{\omega'}^{*}}=\sum_{\substack{\{v\in G_{k}^{0}:\ |\sigma(v)|\geq k-1,\\ \sigma(v)|_{k-1}=\omega'\}}}P_{v}.$$
Also, $v\in G_{k}^{0}\implies v\in r(e_{k})\implies|\sigma(v)|<k$; combining this with the fact $|\sigma(v)|\geq k-1$ and $\sigma(v)|_{k-1}=\omega'$, we have $\sigma(v)=\omega'$. By \cite[Lemma 3.7]{Takeshi-Katsura-Paul-Muhly-Aidan-Sims--Mark-Tomforde:2010aa}, 
$$\sigma(v)=\omega'\implies v\in r(\omega')= r(N',M),$$
which in turn means $v\in r(N,M)=r(e_{k})\cap r(N',M)$. On the other hand, suppose $v\in r(N,M)$ and $|\sigma(v)|=k-1$. Well, $v\in r(N,M)\implies v\in r(\omega').$ By\\
\noindent \cite[Lemma 3.7]{Takeshi-Katsura-Paul-Muhly-Aidan-Sims--Mark-Tomforde:2010aa}, we also have $v\in r(\sigma(v))$. For $|\omega'|=|\sigma(v)|=k-1$, we have $\sigma(v)=\omega'$. To see this, suppose $\sigma(v)\neq\omega'$. Since $|\omega'|=|\sigma(v)|=k-1$, $\sigma(v)\neq\omega'$ implies there exists $j\in\{1,2,\dots, k-1\}$ such that $\omega'_{j}\neq\sigma(v)_{j}$. WLOG, suppose $\omega'_{j}=1$ and $\sigma(v)_{j}=0$; this means $r(\omega')\subseteq r(e_{j})$ and $r(\sigma(v))\cap r(e_{j})=\emptyset$ (i.e., $r(\omega')\cap r(\sigma(v))=\emptyset$).$\Rightarrow\!\Leftarrow$Thus, 
$$v\in r(N,M) \text{ and } |\sigma(v)|=k-1\implies v\in \{v\in G_{k}^{0}:\ |\sigma(v)|\geq k-1, \sigma(v)|_{k-1}=\omega'\},$$
which in turn means
$$\sum_{\substack{\{v\in G_{k}^{0}:\ |\sigma(v)|\geq k-1,\\ \sigma(v)|_{k-1}=\omega'\}}}P_{v}=\sum_{\substack{v\in r(N,M)\\ |\sigma(v)|=k-1}}P_{v}.$$
Putting it all together, we have 
$$\prod\limits_{n\in N} S_{e_{n}^{*}}S_{e_{n}}\prod\limits_{m\in M}(1-S_{e_{m}^{*}}S_{e_{m}})=t_{\alpha_{\omega}}t_{\alpha_{\omega}^{*}}+\sum_{\substack{v\in r(N,M)\\ |\sigma(v)|<k}}P_{v}.$$

\textit{Case 2}: $\omega=(\omega',0)$. In this case, for $M':=M\setminus\{k\}$, 
\begin{align*}
\prod\limits_{n\in N} S_{e_{n}^{*}}S_{e_{n}}&\prod\limits_{m\in M}(1-S_{e_{m}^{*}}S_{e_{m}})=(1-S_{e_{k}^{*}}S_{e_{k}})\Bigg(\prod\limits_{n\in N} S_{e_{n}^{*}}S_{e_{n}}\prod\limits_{m\in M'}(1-S_{e_{m}^{*}}S_{e_{m}})\Bigg)\\
&=\prod\limits_{n\in N} S_{e_{n}^{*}}S_{e_{n}}\prod\limits_{m\in M'}(1-S_{e_{m}^{*}}S_{e_{m}})-S_{e_{k}^{*}}S_{e_{k}}\prod\limits_{n\in N} S_{e_{n}^{*}}S_{e_{n}}\prod\limits_{m\in M'}(1-S_{e_{m}^{*}}S_{e_{m}}).
\end{align*}
 By our inductive hypothesis, we have 
$$\prod\limits_{n\in N} S_{e_{n}^{*}}S_{e_{n}}\prod\limits_{m\in M'}(1-S_{e_{m}^{*}}S_{e_{m}})=t_{\alpha_{\omega'}}t_{\alpha_{\omega'}^{*}}+ \sum_{\substack{v\in r(N,M')\\ |\sigma(v)|<k-1}}P_{v}.$$
Let us set $\omega'':=(\omega',1)$ and $N'':=N\cup\{k\}$. Further, let us suppose $\omega''\in\Delta_{k}$. By case 1, we have
\begin{align*}
S_{e_{k}^{*}}S_{e_{k}}\prod\limits_{n\in N} S_{e_{n}^{*}}S_{e_{n}}&\prod\limits_{m\in M'}(1-S_{e_{m}^{*}}S_{e_{m}})=\prod\limits_{n\in N''} S_{e_{n}^{*}}S_{e_{n}}\prod\limits_{m\in M'}(1-S_{e_{m}^{*}}S_{e_{m}})\\
&=t_{\alpha_{\omega''}}t_{\alpha_{\omega''}^{*}}+ \sum_{\substack{v\in r(N'',M')\\ |\sigma(v)|<k}}P_{v}.
\end{align*}
Thus,
\begin{align*}
\prod\limits_{n\in N} S_{e_{n}^{*}}S_{e_{n}}\prod\limits_{m\in M}(1&-S_{e_{m}^{*}}S_{e_{m}})=\Bigg(t_{\alpha_{\omega'}}t_{\alpha_{\omega'}^{*}}+ \sum_{\substack{v\in r(N,M')\\ |\sigma(v)|<k-1}}P_{v}\Bigg)-\Bigg(t_{\alpha_{\omega''}}t_{\alpha_{\omega''}^{*}}+ \sum_{\substack{v\in r(N'',M')\\ |\sigma(v)|<k}}P_{v}\Bigg)\\
&=\Bigg(t_{\alpha_{\omega'}}t_{\alpha_{\omega'}^{*}}-t_{\alpha_{\omega''}}t_{\alpha_{\omega''}^{*}}\Bigg)+ \Bigg(\sum_{\substack{v\in r(N,M')\\ |\sigma(v)|<k-1}}P_{v}-\sum_{\substack{v\in r(e_{k})\cap r(N,M')\\ |\sigma(v)|<k}}P_{v}\Bigg)
\end{align*}
From \cite[Proposition 3.14]{Takeshi-Katsura-Paul-Muhly-Aidan-Sims--Mark-Tomforde:2010aa}, and \textbf{LP4}, we can deduce
$$q_{\omega'}=t_{\omega}t_{\omega^{*}}+t_{\omega''}t_{\omega''^{*}}+\sum_{\substack{\{v\in G^{0}:\\ \sigma(v)=\omega'\}}}P_{v}.$$
We can use the same argument as in case 1) to show 
$$\{v\in G^{0}:\sigma(v)=\omega'\}=\{v\in r(\omega'):|\sigma(v)|=k-1\}.$$
And so, for $r(\omega')=r(N,M')$, 
$$q_{\omega'}=t_{\omega}t_{\omega^{*}}+t_{\omega''}t_{\omega''^{*}}+\sum_{\substack{v\in r(N,M') \\ |\sigma(v)|=k-1}}P_{v}.$$
Applying \cite[Lemma 4.6]{Takeshi-Katsura-Paul-Muhly-Aidan-Sims--Mark-Tomforde:2010aa}, we have
\begin{align*}
t_{\alpha_{\omega'}}t_{\alpha_{\omega'}^{*}}&=t_{\alpha_{\omega'}}q_{\omega'}t_{\alpha_{\omega'}^{*}}\\
&=t_{\alpha_{\omega'}}\Bigg(t_{\omega}t_{\omega^{*}}+t_{\omega''}t_{\omega''^{*}}+\sum_{\substack{v\in r(N,M') \\ |\sigma(v)|=k-1}}P_{v}\Bigg)t_{\alpha_{\omega'}^{*}}\\
&=t_{\alpha_{\omega'}\omega}t_{(\alpha_{\omega'}\omega)^{*}}+t_{\alpha_{\omega'}\omega''}t_{(\alpha_{\omega'}\omega'')^{*}}+\sum_{\substack{v\in r(N,M') \\ |\sigma(v)|=k-1}}P_{v}\\
&\text{which, by the uniqueness of $\alpha_{x}$},\\
&=t_{\alpha_{\omega}}t_{\alpha_{\omega}^{*}}+t_{\alpha_{\omega''}}t_{\alpha_{\omega''}^{*}}+\sum_{\substack{v\in r(N,M') \\ |\sigma(v)|=k-1}}P_{v}.
\end{align*}
In particular, this means
{\footnotesize
\begin{align*}
\prod\limits_{n\in N} S_{e_{n}^{*}}S_{e_{n}}\prod\limits_{m\in M}(1-S_{e_{m}^{*}}S_{e_{m}})&=\Bigg(t_{\alpha_{\omega}}t_{\alpha_{\omega}^{*}}+\sum_{\substack{v\in r(N,M') \\ |\sigma(v)|=k-1}}P_{v}\Bigg)+ \Bigg(\sum_{\substack{v\in r(N,M')\\ |\sigma(v)|<k-1}}P_{v}-\sum_{\substack{v\in r(e_{k})\cap r(N,M')\\ |\sigma(v)|<k}}P_{v}\Bigg)\\
&=t_{\alpha_{\omega}}t_{\alpha_{\omega}^{*}}+\sum_{\substack{v\in r(N,M')\setminus r(e_{k})\\ |\sigma(v)|<k}}P_{v}\\
&=t_{\alpha_{\omega}}t_{\alpha_{\omega}^{*}}+\sum_{\substack{v\in r(N,M)\\ |\sigma(v)|<k}}P_{v}.
\end{align*}\par}
On the other hand, suppose $|r(\omega'')|=|r(N'',M')|<\infty$. By Lemma \ref{lex3}, we have
$$\prod\limits_{n\in N''} S_{e_{n}^{*}}S_{e_{n}}\prod\limits_{m\in M'}(1-S_{e_{m}^{*}}S_{e_{m}})=\sum\limits_{v\in r(N'',M')}P_{v}.$$
Moreover, by \cite[Lemma 4.3]{Takeshi-Katsura-Paul-Muhly-Aidan-Sims--Mark-Tomforde:2010aa},
$$r'(\omega''):=\{v\in r(N'',M'): |\sigma(v)|\geq k\}=\emptyset;$$
and so, 
$$\prod\limits_{n\in N''} S_{e_{n}^{*}}S_{e_{n}}\prod\limits_{m\in M'}(1-S_{e_{m}^{*}}S_{e_{m}})=\sum_{\substack{v\in r(N'',M')\\ |\sigma(v)|<k}}P_{v}.$$
Also note that in this case 
$$q_{\omega'}=t_{\omega}t_{\omega^{*}}+\sum_{\substack{v\in r(N,M') \\ |\sigma(v)|=k-1}}P_{v}\implies t_{\alpha_{\omega'}}t_{\alpha_{\omega'}^{*}}=t_{\alpha_{\omega}}t_{\alpha_{\omega}^{*}}+\sum_{\substack{\substack{v\in r(N,M') \\ |\sigma(v)|=k-1}}}P_{v}.$$
Putting it all together, we have
\begin{align*}
\prod\limits_{n\in N} S_{e_{n}^{*}}S_{e_{n}}\prod\limits_{m\in M}(1-S_{e_{m}^{*}}S_{e_{m}})&=\Bigg(t_{\alpha_{\omega'}}t_{\alpha_{\omega'}^{*}}+ \sum_{\substack{v\in r(N,M')\\ |\sigma(v)|<k-1}}P_{v}\Bigg)-\Bigg(\sum_{\substack{v\in r(N'',M')\\ |\sigma(v)|<k}}P_{v}\Bigg)\\
&=\Bigg(t_{\alpha_{\omega}}t_{\alpha_{\omega}^{*}}+\sum_{\substack{\substack{v\in r(N,M') \\ |\sigma(v)|=k-1}}}P_{v}\Bigg)+\sum_{\substack{v\in r(N,M')\\ |\sigma(v)|<k-1}}P_{v}-\sum_{\substack{v\in r(N'',M')\\ |\sigma(v)|<k}}P_{v}\\
&=t_{\alpha_{\omega}}t_{\alpha_{\omega}^{*}}+\sum_{\substack{v\in r(N,M)\\ |\sigma(v)|<k}}P_{v}.
\end{align*}

\textit{Case 3}: $\omega=(0^{k-1},1)$. By Lemma \ref{lci}, 
$$1=\sum\limits_{N'\subseteq\{1,2,\dots,k-1\}}\Bigg(\prod\limits_{n\in N'}S_{e_{n}^{*}}S_{e_{n}}\prod\limits_{m\in M'}(1-S_{e_{m}^{*}}S_{e_{m}}) \Bigg),$$
in the unitization of $L_{R}(E_{\mathcal{G}})$ (note, $M'=\{1,2,\dots,k-1\}\setminus N'$.); and so
{\footnotesize
\begin{align*}
&S_{e_{k}^{*}}S_{e_{k}}=S_{e_{k}^{*}}S_{e_{k}}\Bigg( \sum\limits_{N'\subseteq\{1,2,\dots,k-1\}}\Bigg(\prod\limits_{n\in N'}S_{e_{n}^{*}}S_{e_{n}}\prod\limits_{m\in M'}(1-S_{e_{m}^{*}}S_{e_{m}}) \Bigg)\Bigg)\\
&=S_{e_{k}^{*}}S_{e_{k}}\Bigg( \sum\limits_{\emptyset\neq N'\subseteq\{1,2,\dots,k-1\}}\Bigg(\prod\limits_{n\in N'}S_{e_{n}^{*}}S_{e_{n}}\prod\limits_{m\in M'}(1-S_{e_{m}^{*}}S_{e_{m}})\Bigg)+\prod\limits_{m\in\{1,2,\dots,k-1\}}(1-S_{e_{m}^{*}}S_{e_{m}})\Bigg).
\end{align*}\par}
Taking $N=\{k\}$, this means
{\footnotesize
\begin{align*}
\prod\limits_{n\in N} S_{e_{n}^{*}}S_{e_{n}}\prod\limits_{m\in M}(1-&S_{e_{m}^{*}}S_{e_{m}})=S_{e_{k}^{*}}S_{e_{k}}- \sum\limits_{\emptyset\neq N'\subseteq\{1,2,\dots,k-1\}}S_{e_{k}^{*}}S_{e_{k}}\prod\limits_{n\in N'}S_{e_{n}^{*}}S_{e_{n}}\prod\limits_{m\in M'}(1-S_{e_{m}^{*}}S_{e_{m}})\\
&=S_{e_{k}^{*}}S_{e_{k}}-\sum_{\substack{N''=N'\cup\{k\}: \\ \emptyset\neq N'\subseteq\{1,2,\dots,k-1\}}}\prod\limits_{n\in N''}S_{e_{n}^{*}}S_{e_{n}}\prod\limits_{m\in M'}(1-S_{e_{m}^{*}}S_{e_{m}}),\\
&\text{where $M'=\{1,2,\dots,k\}\setminus N''$.}
\end{align*}\par}
Observe that
$$S_{e_{k}^{*}}S_{e_{k}}=\sum\limits_{v\in G_{k}^{0}}t_{\alpha_{v}}t_{\alpha_{v}^{*}}+\sum\limits_{\omega''\in X_{\Delta,k}}t_{\alpha_{\omega''}}t_{\alpha_{\omega''}^{*}}=\sum_{\substack{v\in r(e_{k}):\\ |\sigma(v)|<k}}P_{v}+\sum_{\substack{\omega''\in\Delta_{k}:\\ \omega''_{k}=1}}t_{\alpha_{\omega''}}t_{\alpha_{\omega''}^{*}}.$$
On the other hand,
\begin{align*}
\sum_{\substack{N''=N'\cup\{k\}: \\ \emptyset\neq N'\subseteq\{1,2,\dots,k-1\}}}\prod\limits_{n\in N''}S_{e_{n}^{*}}S_{e_{n}}&\prod\limits_{m\in M'}(1-S_{e_{m}^{*}}S_{e_{m}})\\
&=\Bigg(\sum_{\substack{N''=N'\cup\{k\}: \\ \emptyset\neq N'\subseteq\{1,2,\dots,k-1\}\\ |r(N'',M')|=\infty}}\prod\limits_{n\in N''}S_{e_{n}^{*}}S_{e_{n}}\prod\limits_{m\in M'}(1-S_{e_{m}^{*}}S_{e_{m}})\Bigg)\\
&+\Bigg(\sum_{\substack{N''=N'\cup\{k\}: \\ \emptyset\neq N'\subseteq\{1,2,\dots,k-1\}\\ |r(N'',M')|<\infty}}\prod\limits_{n\in N''}S_{e_{n}^{*}}S_{e_{n}}\prod\limits_{m\in M'}(1-S_{e_{m}^{*}}S_{e_{m}})\Bigg).
\end{align*}
It is worth noting the elements of $\{r(N'',M')\}_{N''\subseteq\{1,2,\dots,k\}}$ are pairwise disjoint sets. Now, consider the equation
$$\sum_{\substack{N''=N'\cup\{k\}: \\ \emptyset\neq N'\subseteq\{1,2,\dots,k-1\}\\ |r(N'',M')|=\infty}}\prod\limits_{n\in N''}S_{e_{n}^{*}}S_{e_{n}}\prod\limits_{m\in M'}(1-S_{e_{m}^{*}}S_{e_{m}}).$$
Taking $\omega''\in\{0,1\}^{k}\setminus\{0^{k}\}$ to be such that $\omega''_{n}=1$ for $n\in N''$, and $\omega''_{m}=0$ for $m\in M'$, note that $|r(N'',M')|=\infty\implies \omega''=(\omega',1)\in \Delta_{k}$, where $\omega'\neq0^{k-1}$. More importantly, by case 1, we have
{\footnotesize
\begin{align*}
\sum_{\substack{N''=N'\cup\{k\}: \\ \emptyset\neq N'\subseteq\{1,2,\dots,k-1\}\\ |r(N'',M')|=\infty}}\prod\limits_{n\in N''}S_{e_{i}^{*}}S_{e_{i}}&\prod\limits_{m\in M'}(1-S_{e_{m}^{*}}S_{e_{m}})=\sum_{\substack{\omega''\in\Delta_{k}:\\ \omega''=(\omega',1), \omega'\in\Delta_{k-1}}}\Bigg(t_{\alpha_{\omega''}}t_{\alpha_{\omega''}^{*}}+\sum_{\substack{v\in r(N'',M'):\\ |\sigma(v)|<k}}P_{v}\Bigg)\\
&=\sum_{\substack{\omega''\in\Delta_{k}:\\ \omega''=(\omega',1), \omega'\in\Delta_{k-1}}}t_{\alpha_{\omega''}}t_{\alpha_{\omega''}^{*}}+\sum_{\substack{N''=N'\cup\{k\}: \\ \emptyset\neq N'\subseteq\{1,2,\dots,k-1\}\\ |r(N'',M')|=\infty}}\sum_{\substack{v\in r(N'',M'):\\ |\sigma(v)|<k}}P_{v}.
\end{align*}\par}
Applying \cite[Lemma 4.3]{Takeshi-Katsura-Paul-Muhly-Aidan-Sims--Mark-Tomforde:2010aa} and Lemma \ref{lex3}, we can also deduce
$$\sum_{\substack{N''=N'\cup\{k\}: \\ \emptyset\neq N'\subseteq\{1,2,\dots,k-1\}\\ |r(N'',M')|<\infty}}\prod\limits_{n\in N''}S_{e_{i}^{*}}S_{e_{i}}\prod\limits_{m\in M'}(1-S_{e_{m}^{*}}S_{e_{m}})=\sum_{\substack{N''=N'\cup\{k\}: \\ \emptyset\neq N'\subseteq\{1,2,\dots,k-1\}\\ |r(N'',M')|<\infty}}\sum_{\substack{v\in r(N'',M'):\\ |\sigma(v)|<k}}P_{v}.$$
So,
\begin{align*}
\sum_{\substack{N''=N'\cup\{k\}: \\ \emptyset\neq N'\subseteq\{1,2,\dots,k-1\}}}\prod\limits_{n\in N''}S_{e_{i}^{*}}S_{e_{i}}&\prod\limits_{m\in M'}(1-S_{e_{m}^{*}}S_{e_{m}})\\
&=\sum_{\substack{\omega''\in\Delta_{k}:\\ \omega''=(\omega',1), \omega'\in\Delta_{k-1}}}t_{\alpha_{\omega''}}t_{\alpha_{\omega''}^{*}}+\sum_{\substack{N''=N'\cup\{k\}: \\ \emptyset\neq N'\subseteq\{1,2,\dots,k-1\}\\ |r(N'',M')|=\infty}}\sum_{\substack{v\in r(N'',M'):\\ |\sigma(v)|<k}}P_{v}\\
&\hspace{1.5cm}+\sum_{\substack{N''=N'\cup\{k\}: \\ \emptyset\neq N'\subseteq\{1,2,\dots,k-1\}\\ |r(N'',M')|<\infty}}\sum_{\substack{v\in r(N'',M'):\\ |\sigma(v)|<k}}P_{v}\\
&=\sum_{\substack{\omega''\in\Delta_{k}:\\ \omega''=(\omega',1), \omega'\in\Delta_{k-1}}}t_{\alpha_{\omega''}}t_{\alpha_{\omega''}^{*}}+\sum_{\substack{N''=N'\cup\{k\}: \\ \emptyset\neq N'\subseteq\{1,2,\dots,k-1\}}}\sum_{\substack{v\in r(N'',M'):\\ |\sigma(v)|<k}}P_{v}
\end{align*}
Lastly, for $v'\in\bigg(\bigcup\limits_{i=1}^{k-1}r(e_{i})\bigg)\cap r(e_{k})$ such that $|\sigma(v')|<k$, let $N'':=\{n\in\{1,2,\cdots,k\}:v'\in r(e_{n})\}$; then, we can see $P_{v'}$ is a summand in 
$$\sum_{\substack{N''=N'\cup\{k\}: \\ \emptyset\neq N'\subseteq\{1,2,\dots,k-1\}}}\sum_{\substack{v\in r(N'',M'):\\ |\sigma(v)|<k}}P_{v}.$$
Thus,
\begin{align*}
\prod\limits_{n\in N} S_{e_{n}^{*}}S_{e_{n}}&\prod\limits_{m\in M}(1-S_{e_{m}^{*}}S_{e_{m}})=S_{e_{k}^{*}}S_{e_{k}}-\sum_{\substack{N''=N'\cup\{k\}: \\ \emptyset\neq N'\subseteq\{1,2,\dots,k-1\}}}\prod\limits_{n\in N''}S_{e_{i}^{*}}S_{e_{i}}\prod\limits_{m\in M'}(1-S_{e_{m}^{*}}S_{e_{m}})\\
&=\Bigg(\sum_{\substack{v\in r(e_{k}):\\ |\sigma(v)|<k}}P_{v}+\sum_{\substack{\omega''\in\Delta_{k}:\\ \omega''_{k}=1}}t_{\alpha_{\omega''}}t_{\alpha_{\omega''}^{*}}\Bigg)\\
&\hspace{1cm}-\Bigg(\sum_{\substack{\omega''\in\Delta_{k}:\\ \omega''=(\omega',1), \omega'\in\Delta_{k-1}}}t_{\alpha_{\omega''}}t_{\alpha_{\omega''}^{*}}+\sum_{\substack{N''=N'\cup\{k\}: \\ \emptyset\neq N'\subseteq\{1,2,\dots,k-1\}}}\sum_{\substack{v\in r(N'',M'):\\ |\sigma(v)|<k}}P_{v}\Bigg)\\
&=t_{\alpha_{\omega}}t_{\alpha_{\omega}^{*}}+\sum_{\substack{v\in r(N,M):\\ |\sigma(v)|<k}}P_{v};
\end{align*}
this completes case 3, thereby proving our lemma.
\end{proof}

The following lemma is essentially identical to \cite[Lemma 5.19]{Takeshi-Katsura-Paul-Muhly-Aidan-Sims--Mark-Tomforde:2010aa} in its statement and proof. Nonetheless, we will give it here.

\begin{lemma}\label{lglg2}
Let $\gamma$ be a path in $E_{\mathcal{G}}$ such that $s_{E}(\gamma)\in W_{0}\sqcup\Gamma_{0}$ and, 
$$\gamma=g_{0}(e_{n_{1}},x_{1})g_{1}(e_{n_{2}},x_{2})g_{2}\dots(e_{n_{k}},x_{k})g_{k},$$
as in \cite[Lemma 4.1]{Takeshi-Katsura-Paul-Muhly-Aidan-Sims--Mark-Tomforde:2010aa}. Then,
$$t_{\gamma}=S_{e_{n_{1}}}S_{e_{n_{2}}}\dots S_{e_{n_{k}}}t_{\alpha_{r_{E}(\gamma)}}.$$
It's worth noting the $g_{i}$'s can be elements in $E_{\mathcal{G}}^{0}$.
\end{lemma}

\begin{proof} We will use induction on $k$. To that end, suppose $k=0$; which is to say, $\gamma=g_{0}$. For $s_{E}(g_{0})\in W_{0}\sqcup\Gamma_{0}$, and $g_{0}$ a path in $F$, the uniqueness of $\alpha_{r_{E}(g_{0})}$ implies $g_{0}=\alpha_{r_{E}(g_{0})}$. Thus, $t_{g_{0}}=t_{\alpha_{r_{E}(g_{0})}},$ completing our base case.

Now, for 
$$\gamma=g_{0}(e_{n_{1}},x_{1})g_{1}(e_{n_{2}},x_{2})g_{2}\dots(e_{n_{k}},x_{i})g_{k}=\gamma'(e_{n_{k}},x_{k})g_{k},$$
we have by our inductive hypothesis,
$$t_{\gamma}=S_{e_{n_{1}}}S_{e_{n_{2}}}\dots S_{e_{n_{k-1}}}t_{\alpha_{r_{E}(\gamma')}}t_{(e_{n_{k}},x_{k})}t_{g_{k}}.$$
For $r_{E}(\gamma)=r_{E}(g_{k})$, the uniqueness of $\alpha_{r_{E}(\gamma)}$ implies $\alpha_{r_{E}(\gamma)}=\alpha_{x_{k}}g_{k}$, which in turn implies $t_{\alpha_{r_{E}(\gamma)}}=t_{\alpha_{x_{k}}}t_{g_{k}}.$ And so
\begin{align*}
S_{e_{n_{k}}}t_{\alpha_{r_{E}(\gamma)}}=S_{e_{n_{k}}}t_{\alpha_{x_{k}}}t_{g_{k}}&=\sum\limits_{x\in X(e_{n_{k}})}t_{\alpha_{s(e_{n_{k}})}}t_{(e_{n_{k}},x)}t_{\alpha_{x}^{*}}t_{\alpha_{x_{k}}}t_{g_{k}}, \text{which, by Claim \ref{corth},}\\
&=t_{\alpha_{s(e_{n_{k}})}}t_{(e_{n_{k}},x_{k})}t_{g_{k}}, \text{and since $r_{E}(\gamma')=s(e_{n_{k}})$,}\\
&=t_{\alpha_{r_{E}(\gamma')}}t_{(e_{n_{k}},x_{k})}t_{g_{k}}.
\end{align*}
Hence, $t_{\gamma}=S_{e_{n_{1}}}S_{e_{n_{2}}}\dots S_{e_{n_{k}}}t_{\alpha_{r_{E}(\gamma)}}$.
\end{proof}

We are now equipped to prove one of our main results. 

\begin{theorem}\label{tlgis}
Let $\mathcal{G}$ be an ultragraph with no singular vertices, then 
$$L_{R}(\mathcal{G})\cong QL_{R}(E_{\mathcal{G}})Q.$$
\end{theorem}

\begin{proof}
By Theorem \ref{texlg}, it suffices to show $\mathcal{EL}_{R}(\mathcal{G})\cong QL_{R}(\mathcal{G})Q$. To that end, we apply Lemma \ref{lex3} to get an Exel-Laca $\mathcal{G}$-family $\{P_{v}, S_{e}, S_{e^{*}}\}$ in $L_{R}(E_{\mathcal{G}})$. We know from Lemma \ref{gspan} and Proposition \ref{gzgrade}.
$$L_{R}(E_{\mathcal{G}})=\text{span}_{R}\{t_{\alpha}t_{\beta^{*}}:\ r_{E}(\alpha)=r_{E}(\beta)\},$$
and has a $\mathbb{Z}$-graded structure with
$$L_{R}(E_{\mathcal{G}})_{i}=\text{span}_{R}\{t_{\alpha}t_{\beta^{*}}:\ r_{E}(\alpha)=r_{E}(\beta),\ |\alpha|-|\beta|=i\}.$$
Given $Q\in \mathcal{M}(L_{R}(E_{\mathcal{G}}))$, it is straightforward to see
$$QL_{R}(\mathcal{G})Q=\text{span}_{R}\{t_{\alpha}t_{\beta^{*}}:\ r_{E}(\alpha)=r_{E}(\beta)\}\ \text{with}\ s_{E}(\alpha),\ s_{E}(\beta)\in W_{0}\sqcup\Gamma_{0}\}.$$
Further, by the definition of $\{P_{v}, S_{e}, S_{e^{*}}\}$, we can see $\{P_{v}, S_{e}, S_{e^{*}}\}\subseteq QL_{R}(\mathcal{G})Q$. And so, by Lemma \ref{lex2} and Theorem \ref{thm:keythm}, we have an injective map $\phi:\mathcal{EL}_{R}(\mathcal{G})\to(L_{R}(E_{\mathcal{G}}))$ such that $\text{im}\phi\subseteq QL_{R}(\mathcal{G})Q$. Finally, by \cite[Lemma 4.1]{Takeshi-Katsura-Paul-Muhly-Aidan-Sims--Mark-Tomforde:2010aa}, every path $\alpha$ in $E_{\mathcal{G}}$ can be expressed uniquely as
$$\alpha=g_{0}(e_{n_{1}},x_{1})g_{1}(e_{n_{2}},x_{2})g_{2}\dots(e_{n_{k}},x_{k})g_{k},$$
where each $e_{n_{i}}$ is in $\mathcal{G}^{1}$, and each $g_{i}$ is a path in $F$ (possibly an element of $E_{\mathcal{G}^{0}}$). Then, by Lemmas \ref{lglg} and \ref{lglg2}, we have $QL_{R}(\mathcal{G})Q\subseteq\text{im}\phi$; and so $\phi$ is surjective as well. Thus, $\mathcal{EL}_{R}(\mathcal{G})\cong QL_{R}(E_{\mathcal{G}})Q.$
\end{proof}

\begin{theorem}\label{imptheo}
Let $\mathcal{G}$ be an ultragraph with no singular vertices. Then,
$$L_{R}(\mathcal{G}) \text{ is Morita equivalent to } L_{R}(E_{\mathcal{G}}).$$
\end{theorem}

\begin{proof}
This follows directly from Theorem \ref{tlgis} and Corollary \ref{meqcor}.
\end{proof}

\section{Desingularization and Morita Equivalence}

The goal of this section is mostly to extend Theorem \ref{imptheo} to ultragraphs with singular vertices. That is, even in the case where $\mathcal{G}$ may contain singular vertices, $L_{R}(\mathcal{G})$ is Morita equivalent to the Leavitt path algebra of a graph. The techniques and ideas used here are from \cite{G.-Abrams:2008aa}, but modified to work for ultragraphs. We begin with the notion of \textit{desingularization}, which was introduced for ultragraphs by Tomforde in \cite{Tomforde:2003aa}. To that end, let $\mathcal{G}$ be an ultragraph containing singular vertices. If $v_{0}\in G^{0}$ is a sink, we add a \textit{tail} at $v_{0}$ by introducing edges $\{f_{i}\}_{i=1}^{\infty}$, and vertices  $\{v_{i}\}_{i=1}^{\infty}$, such that $s(f_{i})=v_{i-1}$ and $r(f_{i})=v_{i}$. 

\begin{figure}[h!]
\begin{center}
\begin{tikzpicture}
\tikzset{vertex/.style = {shape=circle, draw=black!100,fill=black!100, thick, inner sep=0pt, minimum size=2 mm}}
\tikzset{edge/.style = {->, line width=1pt}}
\tikzset{v/.style = {shape=rectangle, dashed, draw, inner sep=0pt, minimum size=2em, minimum width=3em}}
    \node[vertex] (a) at (-9,1) [label=above:$v_{0}$]{};
    \node[vertex] (b) at (-6,1) [label=above:$v_{1}$]{};
    \node[vertex] (c) at (-3,1) [label=above:$v_{2}$]{};
    \node[vertex] (d) at (2,1) [label=above:$v_{i-1}$]{};
       
   \path (a) edge [->, >=latex, line width=.75pt, shorten <= 2pt, shorten >= 2pt, right] node[above]{$f_{1}$} (b);
   
   \path (b) edge [->, >=latex, line width=.75pt, shorten <= 2pt, shorten >= 2pt, right] node[above]{$f_{2}$} (c);

    \node [shape=circle,minimum size=1.5em] (d1) at (0,1) {};
    
     \path (c) edge [->, >=latex, line width=.75pt, shorten <= 2pt, shorten >= 2pt, right] (d1);

   \path (d1) to node {\dots} (d);
   
   \node [shape=circle,minimum size=1.5em] (d2) at (5,1) {};
  
   \path (d) edge [->, >=latex, line width=.75pt, shorten <= 2pt, shorten >= 2pt, right] node[above]{$f_{i}$} (d2);

     \node [shape=circle,minimum size=1.5em] (d3) at (6,1) {};
     
     \path (d2) to node {\dots} (d3);

   \end{tikzpicture}
   \caption{}
   \label{im7}
\end{center}
\end{figure}

If $v_{0}\in G^{0}$ is an infinite emitter, let $\{e_{i}\}_{i=1}^{\infty}$ be an enumeration of $s^{-1}(v)$. We add a tail starting at $v_{0}$ by adding edges $\{f_{i},g_{i}\}_{i=1}^{\infty}$, and vertices $\{v_{i}\}_{i=1}^{\infty}$, such that $s(g_{i})=s(f_{i})=v_{i-1}$, $r(f_{i})=v_{i}$, and $r(g_{i})=r(e_{i})$.

\begin{figure}[h!]
\begin{center}
\begin{adjustbox}{max width=1\textwidth}
\begin{tikzpicture}
\tikzset{vertex/.style = {shape=circle, draw=black!100,fill=black!100, thick, inner sep=0pt, minimum size=2 mm}}
\tikzset{edge/.style = {->, line width=1pt}}
\tikzset{v/.style = {shape=rectangle, dashed, draw, inner sep=0pt, minimum size=2em, minimum width=6.5em}}
    \node[vertex] (a) at (-9,1) [label=above:$v_{0}$]{};
    \node[vertex] (b) at (-6,1) [label=above:$v_{1}$]{};
    \node[vertex] (c) at (-3,1) [label=above:$v_{2}$]{};
    \node[vertex] (d) at (2,1) [label=above:$v_{i-1}$]{};
    \node[v] (e) at (-9,-1){$r(g_{1})=r(e_{1})$};
    \node[v] (f) at (-6,-1){$r(g_{2})=r(e_{2})$};
    \node[v] (g) at (-3,-1){};
    \node[v] (h) at (2,-1){$r(g_{i})=r(e_{i})$};

   \path (a) edge [->, >=latex, line width=.75pt, shorten <= 2pt, shorten >= 2pt, right] node[above]{$f_{1}$} (b);
   
   \path (b) edge [->, >=latex, line width=.75pt, shorten <= 2pt, shorten >= 2pt, right] node[above]{$f_{2}$} (c);

    \node [shape=circle,minimum size=1.5em] (d1) at (0,1) {};
    
     \path (c) edge [->, >=latex, line width=.75pt, shorten <= 2pt, shorten >= 2pt, right] (d1);

   \path (d1) to node {\dots} (d);
   
   \node [shape=circle,minimum size=1.5em] (d2) at (5,1) {};
  
   \path (d) edge [->, >=latex, line width=.75pt, shorten <= 2pt, shorten >= 2pt, right] node[above]{$f_{i}$} (d2);

     \node [shape=circle,minimum size=1.5em] (d3) at (6,1) {};
     
     \path (d2) to node {\dots} (d3);

     \node [shape=circle,minimum size=1.5em] (d5) at (2,-1) {}; 
     
         \path (g) to node {\dots} (d5);     

    \path (a) edge [->, >=latex, line width=.75pt, shorten <= 2pt, shorten >= 2pt, right] node[sloped, above]{$g_{1}$} (e);
   
     \path (b) edge [->, >=latex, line width=.75pt, shorten <= 2pt, shorten >= 2pt, right] node[sloped, above]{$g_{2}$} (f); 
     
     \path (c) edge [->, >=latex, line width=.75pt, shorten <= 2pt, shorten >= 2pt, right] node[sloped, above]{} (g); 
      
     \path (d) edge [->, >=latex, line width=.75pt, shorten <= 2pt, shorten >= 2pt, right] node[sloped, above]{$g_{i}$} (h);

   \end{tikzpicture}
   \end{adjustbox}
   \caption{}
   \label{im8}
\end{center}
\end{figure}

The ultragraph, $\mathcal{F}$, obtained by adding the appropriate tail to each singular vertex of $\mathcal{G}$ is called the \textit{desingularization of $\mathcal{G}$}. Our task now is to show $L_{R}(\mathcal{G})$ is Morita equivalent to $L_{R}(\mathcal{F})$. In the case of a graph $E$, it is straightforward to show 
$$L_{R}(E)\cong\bigoplus\limits_{v\in E^{0}}L_{R}(E)q_{v}$$
as left $L_{R}(E)$-modules. In order for us to make use of the techniques available to us to prove our result, we need a similar fact to hold for ultragraphs. One can quickly check
$$L_{R}(\mathcal{G})\ncong\bigoplus\limits_{v\in G^{0}}L_{R}(\mathcal{G})p_{v} \text{, and } L_{R}(\mathcal{G})\ncong\bigoplus\limits_{A\in \mathcal{G}^{0}}L_{R}(\mathcal{G})p_{A},$$
as left $L_{R}(\mathcal{G})$-modules. And so we need something a little different. 

\begin{lemma}\label{final:lemma1}
Let $\mathcal{G}$ be an ultragraph. Then,\\
1) $\mathcal{G}^{0}$ is countable,\\
2) there exists a collection of sets $\{\mathtt{A}_{i}\}_{i\in\mathbb{N}}\subseteq\mathcal{G}^{0}$ such that $\mathtt{A}_{i}\cap \mathtt{A}_{j}=\emptyset$ when $i\neq j$, and each $A\in\mathcal{G}^{0}$ is contained in the union of finitely many $\mathtt{A}_{i}$'s.
\end{lemma}

\begin{proof}
 1) Let $\{A'_{i}\}_{i\in\mathbb{N}}$ be an enumeration of the set $\{v\in G^{0}\}\cup\{r(e):\ e\in\mathcal{G}^{1}\}$. The smallest lattice $\mathcal{L}$ of $\mathcal{P}(G^{0})$ generated by $\{A'_{i}\}_{i\in\mathbb{N}}$ is given by 
$$\mathcal{L}=\bigg\{\bigcap_{i\in X_{1}}A'_{i}\cup\bigcap_{i\in X_{2}}A'_{i}\ \cup\dots\cup\bigcap_{i\in X_{n}}A'_{i}: X_{1},\dots, X_{n} \text{ are finite subsets of }\mathbb{N}\bigg\}$$
\cite[Lemma 2.12]{Tomforde:2003aa}.

For a given set $X$, let $\mathcal{P}_{f}(X)$ denote the subset of $\mathcal{P}(X)$ consisting of finite subsets of $X$. It is an established fact that if $X$ is countable, $\mathcal{P}_{f}(X)$ is countable as well. This in turn means $\mathcal{P}_{f}(\mathcal{P}_{f}(X))$ is countable, and so on. Since there is a surjective map $\mathcal{P}_{f}(\mathcal{P}_{f}(\mathbb{N}))\to\mathcal{L}$ given by
$$\{X_{1},X_{2},\dots,X_{n}\}\mapsto\bigcap_{i\in X_{1}}A'_{i}\cup\bigcap_{i\in X_{2}}A'_{i}\ \cup\dots\cup\bigcap_{i\in X_{n}}A'_{i},$$
we can conclude $\mathcal{L}$ is countable; as such, let $\textbf{A}_{1},\textbf{A}_{2},\dots$ be an enumeration of $\mathcal{L}$.

Now, consider the set
$$\bigg\{\bigcup_{k=1}^{n} (\textbf{A}_{i}\big\backslash\textbf{A}_{j})_{k}:\ i,j,n\in\mathbb{N}\bigg\}.$$ 
Using elementary properties of finite unions and intersections, relative complements, and the fact $\mathcal{L}$ is a lattice, we can show the set given above is an algebra within $\mathcal{P}(G^{0})$; in fact, it is a subset of $\mathcal{G}^{0}$. Since, by definition, $\mathcal{G}^{0}$ is the smallest algebra containing $\{v\in G^{0}\}\cup\{r(e):\ e\in\mathcal{G}^{1}\}$, it must be that
$$\mathcal{G}^{0}=\bigg\{\bigcup_{k=1}^{n} (\textbf{A}_{i}\big\backslash\textbf{A}_{j})_{k}:\ i,j,n\in\mathbb{N}\bigg\}.$$
Since there is a surjective map $\mathcal{P}_{f}(\mathbb{N}\times\mathbb{N})\to\big\{\bigcup_{k=1}^{n} (\textbf{A}_{i}\big\backslash\textbf{A}_{j})_{k}:\ i,j,n\in\mathbb{N}\big\}$ given by 
$$\{(i,j)_{1},(i,j)_{2},\dots,(i,j)_{n}\}\mapsto\bigcup_{k=1}^{n} (\textbf{A}_{i}\big\backslash\textbf{A}_{j})_{k},$$
$\mathcal{G}^{0}$ is countable.

2) Now that we have established $\mathcal{G}^{0}$ is countable, let $B_{1},B_{2},\dots$ be an enumeration of $\mathcal{G}^{0}$. Let

\begin{tabbing}
\hspace{5cm} \=$\mathtt{A}_{1}=B_{1}$\\
\>$\mathtt{A}_{2}=B_{2}\setminus B_{1}$\\
\> \hspace{1cm} \vdots\\
\>$\mathtt{A}_{k}=B_{k}\setminus\Big(\bigcup\limits_{i=1}^{k-1}B_{i}\Big)$\\
\> \hspace{1cm} \vdots\\
\end{tabbing}
For $i\neq j$, we have $\mathtt{A}_{i}\cap \mathtt{A}_{j}=\emptyset$; moreover, for each $k$, $B_{k}\subseteq \Big( \bigcup\limits_{i=1}^{k}\mathtt{A}_{i}\Big)=\Big( \bigcup\limits_{i=1}^{k}B_{i}\Big)$.
\end{proof}

What we now want to show is $L_{R}(\mathcal{G})\cong\bigoplus\limits_{A\in \mathcal{G}^{0}}L_{R}(\mathcal{G})p_{A}$ as left $L_{R}(\mathcal{G})$-modules. We will do this by first showing $\{\mathtt{A}_{i}\}_{i\in\mathbb{N}}$ can be used to construct a $\sigma$-unit for $L_{R}(\mathcal{G})$.

\begin{lemma}\label{final:lemma2}
Let $\{\mathtt{A}_{i}\}_{i\in\mathbb{N}}$ be as in Lemma \ref{final:lemma1}. For $k\in\mathbb{N}$, let $t_{k}:=\sum\limits_{i\leq k}p_{\mathtt{A}_{i}}$. Then, $\{t_{k}\}_{k=1}^{\infty}$ is a $\sigma$-unit for $L_{R}(\mathcal{G})$. Hence, $L_{R}(\mathcal{G})$ is a $\sigma$-unital ring.
\end{lemma}

\begin{proof}
 We will start by showing $t_{k}$ is an idempotent for each $k$. To that end, by \textbf{uLP1} and the fact $\mathtt{A}_{i}\cap\mathtt{A}_{j}=\emptyset$ for $i\neq j$, we have 
$$t_{k}^{2}=\Big(\sum\limits_{i\leq k}p_{\mathtt{A}_{i}}\Big)\Big(\sum\limits_{j\leq k}p_{\mathtt{A}_{j}}\Big)=\sum\limits_{i\leq k}p_{\mathtt{A}_{i}}^{2}=\sum\limits_{i\leq k}p_{\mathtt{A}_{i}}=t_{k}.$$
 Now, suppose $l,k\in\mathbb{N}$ with $l\leq k$. Then,
 $$t_{k}t_{l}=\Big(\sum\limits_{i\leq k}p_{\mathtt{A}_{i}}\Big)\Big(\sum\limits_{j\leq l}p_{\mathtt{A}_{j}}\Big)=\sum\limits_{j\leq l}p_{\mathtt{A}_{j}}^{2}=\sum\limits_{j\leq l}p_{\mathtt{A}_{j}}=t_{l}.$$
 A similar argument shows $t_{l}t_{k}=t_{l}$; and so it is certainly true $t_{k}t_{k+1}=t_{k+1}t_{k}=t_{k}$ for all $k$. Finally, to show $\{t_{k}\}_{k=1}^{\infty}$ is a $\sigma$-unit, we need to show
$$L_{R}(\mathcal{G})=\bigcup\limits_{k=1}^{\infty}t_{k}L_{R}(\mathcal{G})t_{k}.$$
So, let $x=\sum\limits_{j=1}^{n}r_{j}s_{\alpha_{j}}p_{A_{j}}s_{\beta_{j}^{*}}\in L_{R}(\mathcal{G})$, and let
 $$V=\{B\in \mathcal{G}^{0}:B=s(\alpha_{j}),\ \text{or}\ B=s(\beta_{j}),\ \text{for}\ 1\leq j\leq n\}.$$
 Take care to note $B\in V$ can be a set consisting of multiple vertices. Since $V$ is a finite set, $A=\bigcup\limits_{B\in V}B$ is an element of $\mathcal{G}^{0}$ as well. By Lemma \ref{final:lemma1}, $A$ is contained in the union of finitely many $\mathtt{A}_{i}$'s. This in turn means there exists a $k$ such that $A\subseteq\bigcup\limits_{i\leq k}\mathtt{A}_{i}$. Since the $\mathtt{A}_{i}$'s are pairwise disjoint, \textbf{uLP1} gives us $t_{k}=\sum\limits_{i\leq k}p_{\mathtt{A}_{i}}=p_{\scriptstyle\bigcup\limits_{i\leq k}\mathtt{A}_{i}}$; then, \textbf{uLP1} and \textbf{uLP2} imply
$$t_{k}x=\Big(p_{\scriptstyle\bigcup\limits_{i\leq k}\mathtt{A}_{i}}\Big)\Big(\sum\limits_{j=1}^{n}r_{j}s_{\alpha_{j}}p_{A_{j}}s_{\beta_{j}^{*}}\Big)=\sum\limits_{j=1}^{n}r_{j}p_{\scriptstyle\bigcup\limits_{i\leq k}\mathtt{A}_{i}}s_{\alpha_{j}}p_{A_{j}}s_{\beta_{j}^{*}}=\sum\limits_{j=1}^{n}r_{j}s_{\alpha_{j}}p_{A_{j}}s_{\beta_{j}^{*}}=x.$$
Similarly, 
$$xt_{k}=\sum\limits_{j=1}^{n}r_{j}s_{\alpha_{j}}p_{A_{j}}s_{\beta_{j}^{*}}p_{\scriptstyle\bigcup\limits_{i\leq k}\mathtt{A}_{i}}=\sum\limits_{j=1}^{n}r_{j}s_{\alpha_{j}}p_{A_{j}}s_{\beta_{j}^{*}}=x.$$
This gives us $x=t_{k}xt_{k}$. Thus, $L_{R}(\mathcal{G})=\bigcup\limits_{k=1}^{\infty}t_{k}L_{R}(\mathcal{G})t_{k};$ meaning $\{t_{k}\}_{k=1}^{\infty}$ is a $\sigma$-unit of $L_{R}(\mathcal{G})$.
\end{proof}
 
We have the following useful corollary.

\begin{corr}\label{final:corollary1}
Let $\{\mathtt{A}_{i}\}_{i\in\mathbb{N}}$ be as in Lemma \ref{final:lemma1}. Then, $L_{R}(\mathcal{G})\cong\bigoplus\limits_{i\in\mathbb{N}}L_{R}(\mathcal{G})p_{\mathtt{A}_{i}}$ as left $L_{R}(\mathcal{G})$ modules.
\end{corr}

\begin{proof}
Since the $\mathtt{A}_{i}$'s are pairwise disjoint, we have
$$L_{R}(\mathcal{G})p_{\mathtt{A}_{i}}\cap L_{R}(\mathcal{G})p_{\mathtt{A}_{j}}=\{0\}$$
for $i\neq j$. By Lemma \ref{final:lemma2}, for any $x\in L_{R}(\mathcal{G})$, there exists a $k$ such that $x=\sum\limits_{i\leq k}xp_{\mathtt{A}_{i}}.$ Thus, 
$$L_{R}(\mathcal{G})\cong\bigoplus\limits_{i\in\mathbb{N}}L_{R}(\mathcal{G})p_{\mathtt{A}_{i}}$$ 
as a left $L_{R}(\mathcal{G})$-module.
\end{proof}

To make our intention explicit, the goal of this section is to prove $L_{R}(\mathcal{G})$ is Morita equivalent to $L_{R}(\mathcal{F})$. Which, by Theorem \ref{imptheo}, would mean $L_{R}(\mathcal{G})$ is Morita equivalent to $L_{R}(E_{\mathcal{F}})$, thereby establishing our desired result. We will prove the Morita equivalence for $L_{R}(\mathcal{G})$ and $L_{R}(\mathcal{F})$ using the only tool available to us: Theorem \ref{ring:theorem1}; a key component we'll need to use this theorem is the existence of an injective $R$-algebra homomorphism 
$$\phi:L_{R}(\mathcal{G})\to L_{R}(\mathcal{F})$$
which commutes nicely with inclusion maps within the setting of direct limits---that is what we will do in Proposition \ref{final:proposition2}. In order to establish the existence of such a map $\phi$, we will need to first show that for any $0\neq x\in L_{R}(\mathcal{G})$, there exist $a,b\in L_{R}(\mathcal{G})$ such that:\\
1) $axb=rp_{v}$, with $r\in R\setminus\{0\}$, or\\
2) $axb=\sum\limits_{i=1}^{n} r_{i}s_{\alpha}^{i}$, for some cycle $\alpha\in\mathcal{G}^{*}$.\\
Note, case 2) is stating $axb$ is a polynomial in $s_{\alpha}$ over $R$ for some cycle $\alpha$; we will prove this assertion in Lemma \ref{final:lemma3}. Finally, in order to prove Lemma \ref{final:lemma3}, we will need to first show, given any $0\neq x\in L_{R}(\mathcal{G})$, we can find $v\in G^{0}$ such that $xp_{v}\neq 0$---this bit is straightforward in the case of graphs, not so much in the case of ultragraphs. 

\begin{prop}\label{final:proposition1}
Let $x\in L_{R}(\mathcal{G})$ such that $x\neq0$. Then, there exists $v\in G^{0}$ such that $xp_{v}\neq0$.
\end{prop}
 
\begin{proof}
 Suppose $x\neq0$. We will break up our proof into three cases. Before we start, recall $x$ can be expressed as $x=\sum\limits_{i=1}^{n}r_{i}s_{\alpha_{i}}p_{A_{i}}s_{\beta_{i}^{*}}$ (see Theorem \ref{uggrade}). For our first case, we will consider the possibility where $|s(\beta_{i})|<\infty$ for each $i$. Again, we will reiterate the fact $s(\beta_{i})$ need not be a vertex when $|\beta_{i}|=0$ ($|\beta_{i}|=0\implies \beta_{i}=A\in\mathcal{G}^{0}$, with $s_{\beta_{i}^{*}}=p_{A}$, in which case $s(\beta_{i})=A$).
 
 \textit{Case 1:} Let $x=\sum\limits_{i=1}^{n}r_{i}s_{\alpha_{i}}p_{A_{i}}s_{\beta_{i}^{*}}$ with $s(\beta_{i})$ finite for each $i$. Set
$$V=\{v\in G^{0}: v\in s(\beta_{i}) \text{ for } 1\leq i\leq n\}.$$
 Since each $s(\beta_{i})$ is finite, $V$ must be finite as well. Suppose  $xp_{v}=0$ for every $v\in G^{0}$; because $xp_{V}=\sum\limits_{v\in V}xp_{v}$, it must be $xp_{V}=0$. But,
$$xp_{V}=\sum\limits_{i=1}^{n}r_{i}s_{\alpha_{i}}p_{A_{i}}s_{\beta_{i}^{*}}p_{V}=\sum\limits_{i=1}^{n}r_{i}s_{\alpha_{i}}p_{A_{i}}s_{\beta_{i}^{*}}=x,$$
which means $x=0$. $\Rightarrow\!\Leftarrow$ Thus, there exists $v\in G^{0}$ such that $xp_{v}\neq0$.

Now, we must consider the possibility that not every $s(\beta_{i})$ is finite. Said more specifically, we must consider the possibility that, for some $i$, $\beta_{i}=A\in\mathcal{G}^{0}$, a countably infinite subset of $G^{0}$. We will prove it in the next case; to make things easier we will assume $|\beta_{i}|=0$ for all $1\leq i\leq n$. This means, by exploiting \textbf{uLP1}, we can express $x$ as $x=\sum\limits_{i=1}^{n}r_{i}s_{\alpha_{i}}p_{A_{i}}.$

\textit{Case 2:} Let $x=\sum\limits_{i=1}^{n}r_{i}s_{\alpha_{i}}p_{A_{i}}$ with $A_{i}$ possibly infinite. Pick a vertex $v_{1}\in r(\alpha_{1})\cap A_{1}$. If $xp_{v_{1}}=0$, we have, by applying \textbf{uLP1} and \textbf{uLP2}, 
$$xp_{v_{1}}=\sum\limits_{i=1}^{n}r_{i}s_{\alpha_{i}}p_{A_{i}\cap v_{1}}=\sum\limits_{\{i:v_{1}\in r(\alpha_{i})\cap A_{i}\}}r_{i}s_{\alpha_{i}}p_{v_{1}}=0.$$
 The sum on the right-hand side has at least one term since $v_{1}\in r(\alpha_{1})\cap A_{1}$. Rewrite $xp_{v_{1}}$ as $\sum\limits_{k}\sum\limits_{|\alpha_{i}|=k}r_{i}s_{\alpha_{i}}p_{v_{1}}.$ By the $\mathbb{Z}$-grading on $L_{R}(\mathcal{G})$ (see Theorem \ref{uggrade}), 
$$\sum\limits_{k}\sum\limits_{|\alpha_{i}|=k}r_{i}s_{\alpha_{i}}p_{v_{1}}=0\iff \sum\limits_{|\alpha_{i}|=k}r_{i}s_{\alpha_{i}}p_{v_{1}}=0\text{ for each }k.$$
And so, for what follows, we may as well assume all the $\alpha_{i}$'s have the same length in the sum
$$\sum\limits_{\{i:v_{1}\in r(\alpha_{i})\cap A_{i}\}}r_{i}s_{\alpha_{i}}p_{v_{1}}.$$
Then, by \textbf{uLP3}, we have
$$0=s_{\alpha_{1}^{*}}xp_{v_{1}}= \sum\limits_{\{i:v_{1}\in r(\alpha_{i})\}}r_{i}s_{\alpha_{1}^{*}}s_{\alpha_{i}}p_{v_{1}}=r_{1}p_{v_{1}};$$
 but we know $r_{1}p_{v_{1}}\neq0$ (recall that $rp_{A}\neq 0$ for $A\in\mathcal{G}^{0}$ and $r\in R\setminus\{0\}$, see paragraph following Definition \ref{uglpa}). $\Rightarrow\!\Leftarrow$ Thus, $xp_{v_{1}}\neq 0$.
 
For our final, and most general, case, we will consider the possibility that not all $s(\beta_{i})$'s are finite, and not all $\beta_{i}$'s have length 0. By grouping terms accordingly, we can express $x$ as
$$x=\sum\limits_{i=1}^{n}r_{i}s_{\alpha_{i}}p_{A_{i}}+\sum\limits_{j=1}^{m}r'_{j}s_{\alpha_{j}'}p_{A'_{j}}s_{\beta_{j}'^{*}}$$
where the $A_{i}$'s are infinite, and the $s(\beta_{j}')$'s are finite. Moreover, by \textbf{uLP1} and \textbf{uLP2}, we have 
$$\sum\limits_{i=1}^{n}r_{i}s_{\alpha_{i}}p_{A_{i}}=\sum\limits_{i=1}^{n}r_{i}s_{\alpha_{i}}p_{r(\alpha_{i})}p_{A_{i}}=\sum\limits_{i=1}^{n}r_{i}s_{\alpha_{i}}p_{r(\alpha_{i})\cap A_{i}}.$$
Now, if $r(\alpha_{i})\cap A_{i}$ is finite for each $i$, then by replacing each $A_{i}$ by $r(\alpha_{i})\cap A_{i}$ in our expression of $x$, and by our assumption each $s(\beta_{j}')$ is finite, we reduce to case 1. Thus, for our final case, we may as well assume $r(\alpha_{i})\cap A_{i}$ is infinite for each $i$.
 
\textit{Case 3:} Let $x=\sum\limits_{i=1}^{n}r_{i}s_{\alpha_{i}}p_{A_{i}}+\sum\limits_{j=1}^{m}r'_{j}s_{\alpha_{j}'}p_{A'_{j}}s_{\beta_{j}'^{*}}$ with $r(\alpha_{i})\cap A_{i}$ infinite for each $i$ and $s(\beta_{j}')$ finite for each $j$. Let
 $$V=\{v\in G^{0}: v\in s(\beta_{j}') \text{ for } 1\leq j\leq m\}.$$
$V$ is a finite set, and since $A_{1}$ is countably infinite, this means $A_{1}\setminus V\neq \emptyset$. Further, since $s(\beta_{j}')\cap (A_{1}\setminus V)=\emptyset$ for each $j$, we have
$$xp_{A_{1}\setminus V}=\sum\limits_{i=1}^{n}r_{i}s_{\alpha_{i}}p_{A_{i}\cap (A_{1}\setminus V)}=\sum\limits_{\{i:\ r(\alpha_{i})\cap A_{i}\cap (A_{1}\setminus V)\neq \emptyset\}}r_{i}s_{\alpha_{i}}p_{A_{i}\cap (A_{1}\setminus V)}.$$
Since, $r(\alpha_{1})\cap A_{1}\cap (A_{1}\setminus V)\neq \emptyset$ ($V$ is finite, $r(\alpha_{1})\cap A_{1}$ is not), the sum on the right has at least one term. Now, suppose
$$\sum\limits_{\{i:\ r(\alpha_{i})\cap A_{i}\cap (A_{1}\setminus V)\neq \emptyset\}}r_{i}s_{\alpha_{i}}p_{A_{i}\cap (A_{1}\setminus V)}=0.$$
By the $\mathbb{Z}$-grading on $L_{R}(\mathcal{G})$, we may assume the $\alpha_{i}$'s have the same length for each $i$ in $\{i:r(\alpha_{i})\cap A_{i}\cap (A_{1}\setminus V)\neq \emptyset\}$. This means 
$$0=s_{\alpha_{1}^{*}}xp_{A_{1}\setminus V}=r_{1}p_{r(\alpha_{1})\cap A_{1}\cap (A_{1}\setminus V)}\ \Rightarrow\!\Leftarrow.$$
So, 
$$xp_{A_{1}\setminus V}=\sum\limits_{\{i:\ r(\alpha_{i})\cap A_{i}\cap (A_{1}\setminus V)\neq \emptyset\}}r_{i}s_{\alpha_{i}}p_{A_{i}\cap (A_{1}\setminus V)}\neq 0,$$
reducing us to case 2; for which we know there exists a $v$ such that $xp_{v}\neq0$.

Since every $x\in L_{R}(\mathcal{G})$ fits into one of the three cases, we have our desired result.
\end{proof}

With our previous result in hand, we are now ready to set the next stepping stone. We invite the reader to compare the following lemma with \cite[Proposition 3.1]{G.-Aranda-Pino:2008aa}.

\begin{lemma}\label{final:lemma3}
Let $x\in L_{R}(\mathcal{G})$ such that $x\neq0$. Then, there exist $a,b\in L_{R}(\mathcal{G})$ such that $axb=rp_{v}$, for some $r\in R\setminus\{0\}$, or $axb=\sum\limits_{i=0}^{n} r_{i}s_{\alpha}^{i}$ for some cycle $\alpha\in\mathcal{G}^{*}$.
\end{lemma}
 
 \begin{proof}
 Let $x\in L_{R}(\mathcal{G})$ such that $x\neq 0$. Our first goal is to show that there exists $\mu\in\mathcal{G}^{*}$ such that $xs_{\mu}$ can be expressed entirely in real edges and is nonzero; put more succinctly, $xs_{\mu}$ can be expressed as
 $$xs_{\mu}=\sum\limits_{i=1}^{n}r_{i}s_{\alpha_{i}}p_{A_{i}}\neq0.$$
 To that end, let $v\in G^{0}$ such that $xp_{v}\neq0$; by Proposition \ref{final:proposition1}, such a $v$ must exist. By grouping terms based on the presence of ghost edges, we can write $xp_{v}$ as 
$$xp_{v}=\sum\limits_{i=1}^{n}r_{i}s_{\alpha_{i}}p_{A_{i}}s_{\beta_{i}^{*}}+\sum\limits_{j=1}^{m}r'_{j}s_{\zeta_{j}}p_{B_{j}}$$
with $|\beta_{i}|\geq 1$ for each $i$; further, we may also assume the degree of $\sum\limits_{i=1}^{n}r_{i}s_{\alpha_{i}}p_{A_{i}}s_{\beta_{i}^{*}}$ in ghost edges is minimal. That is, given any other expression of $xp_{v}$,
 $$xp_{v}=\sum\limits_{k}r''_{k}s_{\alpha''_{k}}p_{A_{k}}s_{\beta_{k}''^{*}},$$
 there exists some $k$ such that $|\beta''_{k}|\geq\text{max}\{|\beta_{i}|\}$. Note that 
$$xp_{v}=xp_{v}^{2}=\sum\limits_{i=1}^{n}r_{i}s_{\alpha_{i}}p_{A_{i}}s_{\beta_{i}^{*}}p_{v}+\sum\limits_{j=1}^{m}r'_{j}s_{\zeta_{j}}p_{B_{j}}p_{v}.$$
Since, by Lemma \ref{ulve}, 
 \[
r_{i}s_{\alpha_{i}}p_{A_{i}}s_{\beta_{i}^{*}}p_{v} =\left\{
   \begin{array}{ll}
   r_{i}s_{\alpha_{i}}p_{A_{i}}s_{\beta_{i}^{*}} \hspace{.5cm}\text{if}\ v=s(\beta_{i}),\\
     0 \hspace{2.2cm} \text{otherwise},
  \end{array}
  \right.
\]
and
\[
   r'_{j}s_{\zeta_{j}}p_{B_{j}}p_{v} = \left\{
   \begin{array}{ll}
   r'_{j}s_{\zeta_{j}}p_{v}\hspace{1cm}\text{if}\ v\in r(\zeta_{j})\cap B_{j},\\
     0 \hspace{2cm} \text{otherwise},
   \end{array}
  \right.
\] 
we may express $xp_{v}$ as 
$$xp_{v}=\sum\limits_{i=1}^{n}r_{i}s_{\alpha_{i}}p_{A_{i}}s_{\beta_{i}^{*}}+\sum\limits_{j=1}^{m}r'_{j}s_{\zeta_{j}}p_{v}$$
where $s(\beta_{i})=v$ and $v\in r(s_{\zeta_{j}})$ for each $i,j$; because multiplying an expression by $p_{v}$ can't increase the degree in ghost edges, the expression $xp_{v}=\sum\limits_{i=1}^{n}r_{i}s_{\alpha_{i}}p_{A_{i}}s_{\beta_{i}^{*}}+\sum\limits_{j=1}^{m}r'_{j}s_{\zeta_{j}}p_{v}$ is in minimal ghost edge degree as well. Letting $e_{1}^{i}$ denote the first edge of $\beta_{i}$, and $\beta'_{i}$ the path such that $\beta_{i}=e_{1}^{i}\beta'_{i}$ (it's possible $\beta_{i}=e_{1}^{i}$, in which case we can take $\beta_{i}'=p_{r(e_{1}^{i})}$), set $Y_{i}=r_{i}s_{\alpha_{i}}p_{A_{i}}s_{\beta_{i}^{'*}}$ and $Y=\sum\limits_{j=1}^{m}r'_{j}s_{\zeta_{j}}p_{v}$, and so 
$$xp_{v}=\sum\limits_{i=1}^{n}Y_{i}s_{e_{1}^{i^{*}}}+Y.$$
If $xp_{v}s_{e_{1}^{i}}=0$ for each $i$, we have $Y_{i}=-Ys_{e_{1}^{i}}$. In which case,
$$xp_{v}=\sum\limits_{i=1}^{n}-Ys_{e_{1}^{i}}s_{e_{1}^{i^{*}}}+Y=Y\bigg(\sum\limits_{i=1}^{n}-s_{e_{1}^{i}}s_{e_{1}^{i^{*}}}+p_{v}\bigg).$$
Since $xp_{v}\neq0$, $\sum\limits_{i=1}^{n}-s_{e_{1}^{i}}s_{e_{1}^{i^{*}}}+p_{v}\neq0$. By \textbf{uLP4}, there exists $f\in\mathcal{G}^{1}$ such that $s(f)=v$ and $f\neq e_{1}^{i}$ for any $i$. It then follows 
$$xp_{v}s_{f}=xs_{f}=\sum\limits_{i=1}^{n}-Ys_{e_{1}^{i}}s_{e_{1}^{i^{*}}}s_{f}+Ys_{f}=Ys_{f}.$$
Suppose
$$Ys_{f}=\sum\limits_{j=1}^{m}r'_{j}s_{\zeta_{j}}p_{v}s_{f}=\sum\limits_{j=1}^{m}r'_{j}s_{\zeta_{j}}s_{f}=\sum\limits_{j=1}^{m}r'_{j}s_{\zeta_{j}f}=0.$$
By the $\mathbb{Z}$-grading on $L_{R}(\mathcal{G})$ (see case 2 in proof of Proposition \ref{final:proposition1}), we may assume $|\zeta_{j}f|$ is the same for each $j$. We then have 
$$r'_{j_{0}}p_{r(f)}=r'_{j_{0}}s_{(\zeta_{j_{0}}f)^{*}}s_{\zeta_{j_{0}}f}=s_{(\zeta_{j_{0}}f)^{*}}\bigg(\sum\limits_{j=1}^{m}r'_{j}s_{\zeta_{j}f}\bigg)=0$$
for each fixed $j_{0}$, which is only possible if $r'_{j_{0}}=0$. This would then mean $r'_{j}=0$ for each $j$, implying $Y=0$, which in turn would mean $xp_{v}=0.\ \ \Rightarrow\!\Leftarrow$ Thus, $Ys_{f}\neq0$, and setting $\mu=s_{f}$, we have our desired result in the case where $xp_{v}s_{e_{1}^{i}}=0$ for all $i$. On the other hand, there is the possibility $xp_{v}s_{e_{1}^{k}}\neq0$ for some $1\leq k\leq n$. Which is to say,
$$xp_{v}s_{e_{1}^{k}}=xs_{e_{1}^{k}}=\sum\limits_{i=1}^{n}Y_{i}s_{e_{1}^{i^{*}}}s_{e_{1}^{k}}+Ys_{e_{1}^{k}}=Y_{k}+Ys_{e_{1}^{k}}\neq0.$$
Notice that $xs_{e_{1}^{k}}$ has a strictly less degree in ghost edges than $xp_{v}$. We can apply the same process to $xs_{e_{1}^{k}}$ as we did to $x$; getting either an edge $f$ such that $xs_{e_{1}^{k}}s_{f}\neq0$ and can be expressed using only real edges, or an edge $e_{1}^{k''}$ such that $xs_{e_{1}^{k}}s_{e_{1}^{k''}}\neq 0$ and has a degree in ghost edges which is strictly less than that of $xs_{e_{1}^{k}}$. Repeating this process a finite number of times, we will arrive at a $\mu\in\mathcal{G}^{*}$ such that $xs_{\mu}$ can be expressed using only real edges.
 
With that, we will show there exist $a,b\in L_{R}(\mathcal{G})$ such that $axb=rp_{v}$ for some $v\in G^{0}$ and $r\neq0$, or $axb=\sum\limits_{i=0}^{n} r_{i}s_{\alpha}^{i}$ for some cycle $\alpha\in\mathcal{G}^{*}$. To that end, let $x\in L_{R}(\mathcal{G})$. If $x$ can be expressed in only real edges we are fine. If not, as we have shown above, $x\neq0$, there exists $\mu\in\mathcal{G}^{*}$ such that $xs_{\mu}$ can be expressed using only real edges. Replacing $x$ by $xs_{\mu}$ if necessary, we may assume $x$ can be expressed in solely in real edges. Which is to say, $x$ can be expressed as
$$x=\sum\limits_{i=1}^{n}r_{i}s_{\alpha_{i}}p_{A_{i}}.$$
Also, by Proposition \ref{final:proposition1}, there exists $v\in G^{0}$ such that $xp_{v}\neq0$. Replacing $x$ by $xp_{v}$ if necessary, we may also assume
$$x=\sum\limits_{i=1}^{n}r_{i}s_{\alpha_{i}}p_{v}$$
for some $v\in G^{0}$. We will prove our assertion by induction on $n$. For $n=1$,  
$$x=r_{1}s_{\alpha_{1}}p_{v}.$$
In which case, taking $a=s_{\alpha_{1}^{*}}$, we have, by \textbf{uLP3}, 
$$ax=r_{1}s_{\alpha_{1}^{*}}s_{\alpha_{1}}p_{v}=r_{1}p_{r(\alpha_{1})}p_{v}=r_{1}p_{v};$$
setting $b=p_{v}$, we have concluded our base case.

Now, assuming our hypothesis holds for all numbers less than $n$, let
$$x=\sum\limits_{i=1}^{n}r_{i}s_{\alpha_{i}}p_{v}$$
with the summands arranged such that $|\alpha_{j}|\leq|\alpha_{k}|$ for all $j,k$ with $1\leq j\leq k\leq n$. Muliplying $x$ on the left by $s_{\alpha_{1}^{*}}$, we have 
$$s_{\alpha_{1}^{*}}x=\sum\limits_{i=1}^{n}r_{i}s_{\alpha_{1}^{*}}s_{\alpha_{i}}p_{v}=r_{1}p_{v}+\sum\limits_{i=2}^{n}r_{i}s_{\alpha_{1}^{*}}s_{\alpha_{i}}p_{v},$$
which is nonzero since. To see why, note that if $|\alpha_{j}|=|\alpha_{k}|$, for all $j,k$, \textbf{uLP3} implies $s_{\alpha_{1}^{*}}x=r_{1}p_{v}\neq0$. Otherwise, by \textbf{uLP3}, there is some $2\leq k\leq n$ such that 
$$s_{\alpha_{1}^{*}}x=r_{1}p_{v}+\sum\limits_{i=k}^{n}r_{i}s_{\alpha_{1}^{*}}s_{\alpha_{i}}p_{v}$$
with $|\alpha_{1}^{*}\alpha_{i}|\geq1$ for all $k\leq i\leq n$; meaning, given $k\leq i$, $s_{\alpha_{1}^{*}}s_{\alpha_{i}}=0$, or $s_{\alpha_{1}^{*}}s_{\alpha_{i}}\in L_{R}(\mathcal{G})_{m}$ for some $m>0$. And so,  by the $\mathbb{Z}$-grading on $L_{R}(\mathcal{G})$, we again have $s_{\alpha_{1}^{*}}x\neq0$ (see case 2 in proof of Proposition \ref{final:proposition1}). Now, if $s_{\alpha_{1}^{*}}s_{\alpha_{k}}=0$ for some $2\leq k\leq n$, we reduce to the case of less than $n$ summands, and by our inductive hypothesis, we can find $a,b$ such that $axb$ is in one of the two desired forms. The only other possibility is $s_{\alpha_{1}^{*}}s_{\alpha_{i}}\neq0$ for all $2\leq i \leq n$. In which case, there exists a path $\beta_{i}$ for each $2\leq i\leq n$, with $|\beta_{i}|\geq 1$, such that $\alpha_{i}=\alpha_{1}\beta_{i}$. We then have 
$$s_{\alpha_{1}^{*}}x=\sum\limits_{i=1}^{n}r_{i}s_{\alpha_{1}^{*}}s_{\alpha_{i}}p_{v}=r_{1}p_{v}+\sum\limits_{i=2}^{n}r_{i}s_{\beta_{i}}p_{v};$$
multiplying $s_{\alpha_{1}^{*}}$ by $p_{v}$ on the left, we have
$$p_{v}s_{\alpha_{1}^{*}}x=r_{1}p_{v}^{2}+\sum\limits_{i=2}^{n}r_{i}p_{v}s_{\beta_{i}}p_{v}=r_{1}p_{v}+\sum\limits_{i=2}^{n}r_{i}p_{v}s_{\beta_{i}}p_{v}\neq 0.$$
If $p_{v}s_{\beta_{k}}p_{v}=0$ for some $2\leq k \leq n$, we again reduce to less than $n$ summands and we are done. Otherwise, we are left with the case
$$s_{\alpha_{1}^{*}}x=r_{1}p_{v}+\sum\limits_{i=2}^{n}r_{i}s_{\beta_{i}}p_{v}$$
where $s(\beta_{i})=v$ for each $i$; which is to say, $\beta_{i}$ is a cycle for each $i$. Up to this point, having started with an arbitrary $x\in L_{R}(\mathcal{G})$, we have have whittled our way down to showing our assertion holds for $x$, or 
$$s_{\alpha_{1}^{*}}x=r_{1}p_{v}+\sum\limits_{i=2}^{n}r_{i}s_{\beta_{i}}p_{v}$$
with each $\beta_{i}$ a cycle. Now, for $s_{\alpha_{1}^{*}}x=r_{1}p_{v}+\sum\limits_{i=2}^{n}r_{i}s_{\beta_{i}}p_{v}$, suppose there is a $\nu\in\mathcal{G}^{*}$ such that $s_{\nu^{*}}s_{\beta_{k}}=0$ for some $2\leq k\leq n$ but not all such $k$. This means $s(\nu)=v$ (otherwise $s_{\nu^{*}}s_{\beta_{i}}=0$ for all $i$). In turn, this implies 
$$s_{\nu^{*}}r_{1}p_{v}s_{\nu}=r_{1}p_{r(\nu)}\neq0.$$
Note, 
$$s_{\nu^{*}}s_{\alpha_{1}^{*}}xs_{\nu}=r_{1}p_{r(\nu)}+\sum\limits_{i=2}^{n}r_{i}s_{\nu^{*}}s_{\beta_{i}}s_{\nu}.$$
And, by our assumption, there is at least one $k$ such that $s_{\nu^{*}}s_{\beta_{k}}s_{\nu}\neq0$, for which it must be $\nu^{*}\beta_{k}\nu\in\mathcal{G}^{*}$. Since $|\nu^{*}\beta_{k}\nu|=|\beta_{k}|\geq 1$, by the $\mathbb{Z}$-grading of $L_{R}(\mathcal{G})$, we have
$$s_{\nu^{*}}s_{\alpha_{1}^{*}}xs_{\nu}=r_{1}p_{r(\nu)}+\sum\limits_{i=2}^{n}r_{i}s_{\nu^{*}}s_{\beta_{i}}s_{\nu}\neq0.$$
Because we are assuming $s_{\nu^{*}}s_{\beta_{k}}=0$ for some, but not all, $2\leq k\leq n$ (and so some, but not all, the summands in the above expression are 0), our inductive hypothesis tells us there exist $a',b'$ such that $a's_{\nu^{*}}s_{\alpha_{1}^{*}}xs_{\nu}b'$ satisfies one of the specified forms. Setting $a=a's_{\nu^{*}}s_{\alpha_{1}^{*}}$ and $b=s_{\nu}b'$, we prove our assertion for $x$ in the case $s_{\nu^{*}}s_{\beta_{k}}=0$ for some, but not all, $2\leq k\leq n$. On the other hand, if $\nu\in\mathcal{G}^{*}$ is such that $s_{\nu^{*}}s_{\beta_{k}}=0$ for some $k$, then it must be $s_{\nu^{*}}s_{\beta_{i}}=0$ for all $i$. In particular, this means $s_{\beta_{j}^{*}}s_{\beta_{k}}\neq0$ for all $j,k$ (since $s_{\beta_{j}^{*}}s_{\beta_{j}}\neq0$), from which we can deduce $|\beta_{j}|<|\beta_{k}|$ for $j<k$, since $|\beta_{j}|=|\beta_{k}|$ and $\beta_{j}\neq\beta_{k}$ implies $s_{\beta_{j}^{*}}s_{\beta_{k}}=0$. Moreover, since all the $\beta_{i}$'s are cycles, this means for $j<k$ there exists a cycle $\tau$, with $|\tau|\geq1$, such that $\beta_{k}=\beta_{j}\tau$.
 
There is an important fact we should note about cycles: for $\beta\in\mathcal{G}^{*}$ a cycle, we have
$$\beta=\gamma_{1}...\gamma_{m}$$
where each $\gamma_{i}=e^{i}_{1}...e^{i}_{k}$ is a cycle with $s(e^{i}_{1})\notin r(e^{i}_{j})$ for $j<k$; we will call such cycles \textit{simple cycles}. To see this, let $\beta=e_{1}...e_{n}$ be a cycle and let $i$ be the smallest number such that $s(e_{1})\in r(e_{i})$. If $i=n$, we are done. Otherwise, letting $\gamma_{1}=e_{1}..e_{i}$ , we have $\beta=\gamma_{1}\beta'$ where $\beta'$ is also a cycle with $|\beta'|<|\beta|$. Taking $\beta'$ and applying the same process a finite number of times, we will eventually have $\beta=\gamma_{1}...\gamma_{m}$ where each $\gamma_{i}$ has the desired property.
 
So, by the observation made in the previous paragraph, and the fact $\beta_{k}=\beta_{j}\tau$, for some cycle $\tau$, when $j<k$, we can express the $\beta_{i}$'s, for $2\leq i\leq n$, as follows. We first express $\beta_{2}$ as 
$$\beta_{2}=\gamma^{2}_{1}...\gamma^{2}_{m_{2}}$$
where each of the $\gamma$'s are simple cycles. To clear up any confusion, the superscripts in the expression above are purely for indexing purposes and not to suggest exponentiation. Now, since $\beta_{3}=\beta_{2}\tau$ for some cycle $\tau$, and $\tau$ in turn can be expressed as $\tau=\gamma^{3}_{1}...\gamma^{3}_{m_{3}}$, we have
$$\beta_{3}=\gamma^{2}_{1}...\gamma^{2}_{m_{2}}\gamma^{3}_{1}...\gamma^{3}_{m_{3}}$$
where the $\gamma$'s are again simple cycles. Continuing in this manner, we may express $\beta_{n}$ as
$$\beta_{n}=\gamma^{2}_{1}...\gamma^{2}_{m_{2}}\gamma^{3}_{1}...\gamma^{3}_{m_{3}}\cdot\cdot\cdot\gamma^{n}_{1}...\gamma^{n}_{m_{n}}$$
in simple cycles. Thus, 
\begin{tabbing}
\hspace{1cm} $s_{\alpha_{1}^{*}}x\ $\=$=r_{1}p_{v}+\sum\limits_{i=2}^{n}r_{i}s_{\beta_{i}}p_{v}$\\
\>$=r_{1}p_{v}+r_{2}s_{\gamma^{2}_{1}...\gamma^{2}_{m_{2}}}p_{v}+r_{3}s_{\gamma^{2}_{1}...\gamma^{2}_{m_{2}}\gamma^{3}_{1}...\gamma^{3}_{m_{3}}}p_{v}+\cdot\cdot\cdot+r_{n}s_{\gamma^{2}_{1}...\gamma^{2}_{m_{2}}\gamma^{3}_{1}...\gamma^{3}_{m_{3}}\cdot\cdot\cdot\gamma^{n}_{1}...\gamma^{n}_{m_{n}}}p_{v}$.
\end{tabbing}
Now, if the $\gamma$'s aren't all the same cycle, it must be $\gamma^{2}_{1}\neq\gamma^{j}_{k}$ for some $j,k$. Since all the $\gamma$'s are simple cycles starting at $v$, it must then be $s_{\gamma^{j*}_{k}}s_{\gamma^{2}_{1}}=0$ and so  
$$s_{\gamma_{j_{k}}^{*}}s_{\alpha_{1}^{*}}xs_{\gamma_{j_{k}}}=r_{1}p_{v};$$
setting $a=s_{\gamma^{j*}_{k}}s_{\alpha_{1}^{*}}$ and $b=s_{\gamma^{j}_{k}}$, we have $axb=r_{1}p_{v}$. If, on the other hand, the $\gamma$'s are the same simple cycle, setting $\alpha=\gamma^{2}_{1}$, $a=s_{\alpha_{1}^{*}}$, and $b=p_{v}$, we have an expression of $axb$ as a polynomial in $s_{\alpha}$. 
\end{proof}

Armed with Lemma \ref{final:lemma3}, we are now ready to establish the last key piece we need in order to prove the main result of this section. From here on out we will denote the source and range maps of $\mathcal{G}$ by $s_{G}$ and $r_{G}$, and of $\mathcal{F}$ by $s_{F}$ and $r_{F}$, when necessary. Otherwise, we will simply use $r$ and $s$ to avoid cluttered notation and hope it's clear which is which from context.

\begin{prop}\label{final:proposition2}
Let $\mathcal{G}$ be an ultragraph and $\mathcal{F}$ its desingularization. There exists an injective R-algebra homomorphism $\phi:L_{R}(\mathcal{G})\to L_{R}(\mathcal{F})$. \textup{(See \cite[Proposition 5.5]{G.-Abrams:2008aa})}
\end{prop}

\begin{proof}
We will prove our assertion by constructing a Leavitt $\mathcal{G}$-family, $\{P_{A},S_{e},S_{e}^{*}\}$, in $L_{R}(\mathcal{F})$. Then, by the universal mapping property of $L_{R}(\mathcal{G})$, we will have an $R$-algebra homomorphism 
$$\phi:L_{R}(\mathcal{G})\to L_{R}(\mathcal{F}).$$
Note, if $e\in\mathcal{G}^{1}$ is such that $s(e)$ is an infinite emitter, we will write ``$e_{i}$'' for $e$ to indicate it is the $i$-th edge in the enumeration of $s^{-1}(s(e))$. To that end, for $L_{R}(\mathcal{F})=L_{R}(\{q,t\})$ (i.e., $L_{R}(\mathcal{F})$ is generated by a universal Leavitt $\mathcal{F}$-family $\{q_{A},t_{e},t_{e}^{*}\}$), define $\{P_{A},S_{e},S_{e}^{*}\}\subseteq L_{R}(\mathcal{F})$ by:

\begin{tabbing}
\hspace{1.25cm}\=$P_{A}:=q_{A},\hspace{2.45cm}\text{for}\ A\in\mathcal{G}^{0}$,\\
    \>$S_{e}\ :=t_{e}, S_{e^{*}}:=t_{e^{*}}, \hspace{.65cm} \text{for}\ e\in\mathcal{G}^{1}\ \text{such that}\ s(e)\ \text{is not an infinite emitter}$,\\
     \>$S_{e_{i}} :=t_{\alpha},S_{e_{i}^{*}}:=t_{\alpha^{*}}, \hspace{.55cm} \text{for}\ s(e_{i})\ \text{an infinite emitter, where } \alpha=f_{1}f_{2}...f_{i-1}g_{i}$\\
     \> \hspace{4.09cm} is a portion of the tail added to $s(e_{i})$ (see Figure \ref{im8}).
\end{tabbing}
We will show  $\{P_{A},S_{e},S_{e}^{*}\}$ is indeed a Leavitt $\mathcal{G}$-family in $L_{R}(\mathcal{F})$.

(\textbf{uLP1.}) The fact $\{P_{A},S_{e},S_{e}^{*}\}$ satisfies \textbf{uLP1} follows immediately from the fact $\{q_{A},t_{e},t_{e}^{*}\}$ satisfies \textbf{uLP1}.

(\textbf{uLP2.}) Suppose $e\in\mathcal{G}^{1}$ such that $|s^{-1}(s(e))|<\infty$. Then, the fact
$$P_{s(e)}S_{e}=S_{e}P_{r(e)}=S_{e} \text{ and }P_{r(e)}S_{e^{*}}=S_{e^{*}}=S_{e^{*}}$$
follows directly from the fact $\{q_{A},t_{e},t_{e}^{*}\}$ satisfies \textbf{uLP2}. On the other hand, suppose $e_{i}\in\mathcal{G}^{1}$ with $|s^{-1}(s(e_{i}))|=\infty$. For $S_{e_{i}}:=t_{\alpha}$, and $S_{e_{i}^{*}}:=t_{\alpha^{*}}$, where $\alpha=f_{1}f_{2}...f_{i-1}g_{i}$ is a portion of the tail added to $s(e_{i})$, we have, by again exploiting the fact $\{q_{A},t_{e},t_{e}^{*}\}$ satisfies \textbf{uLP2}, 
$$P_{s(e_{i})}S_{e_{i}}=q_{s(e_{i})}t_{\alpha}=q_{s(f_{1})}t_{\alpha}=t_{\alpha}=S_{e_{i}}$$
and 
$$S_{e_{i}}P_{r(e_{i})}=t_{\alpha}q_{r(e_{i})}=t_{\alpha}q_{r(g_{i})}=t_{\alpha}=S_{e_{i}}.$$
We can similarly show, $P_{r(e_{i})}S_{e_{i}^{*}}=S_{e_{i}^{*}}P_{s(e_{i})}=S_{e_{i}^{*}}$.

(\textbf{uLP3.}) First, suppose $e,e'\in\mathcal{G}^{1}$ such that $|s^{-1}(s(e))|,|s^{-1}(s(e'))|<\infty$. Then, for $S_{e}:=t_{e}$ and $S_{e'}:=t_{e'}$ the fact $S_{e^{*}}S_{e}=\delta_{e,e'}P_{r(e)}$ follows directly from the fact $\{q_{A},t_{e},t_{e}^{*}\}$ satisfies \textbf{uLP3}. Second, suppose $e_{i},e \in\mathcal{G}^{1}$ are such that $|s^{-1}(s(e))|<\infty$ and $|s^{-1}(s(e_{i}))=\infty$. Since it is then the case $e_{i}\neq e$, it suffices to show $S_{e^{*}}S_{e_{i}}=S_{e_{i}^{*}}S_{e}=0$. To that end, for $S_{e_{i}^{*}}:=t_{\alpha^{*}}=t_{g_{i}^{*}}t_{f_{i-1}^{*}}...t_{f_{1}^{*}}$, we have
$$S_{e_{i}^{*}}S_{e}=t_{g_{i}^{*}}t_{f_{i-1}^{*}}...t_{f_{1}^{*}}t_{e}.$$
Since $f_{1}\neq e$, and $\{q_{A},t_{e},t_{e}^{*}\}$ satisfies \textbf{uLP3}, $t_{f_{1}^{*}}t_{e}=0$; and so $S_{e_{i}^{*}}S_{e}=0$. A similar argument shows $S_{e^{*}}S_{e_{i}}=0$. Lastly, suppose $e_{i},e_{j}\in\mathcal{G}^{1}$ such that $|s^{-1}(s(e_{i}))|=|s^{-1}(s(e_{j}))|=\infty$; and so $S_{e_{i}}:=t_{\alpha}$, $S_{e_{j}}:=t_{\alpha'}$, with $\alpha=f_{1}f_{2}\dots f_{i-1}g_{i}$ and $\alpha'=f'_{1}f'_{2}\dots f'_{j-1}g'_{j}$. Now, if $e_{i}=e_{j}$, it must be, by construction, $\alpha=\alpha'$. In which case, we have 
$$S_{e_{i}^{*}}S_{e_{i}}=t_{\alpha^{*}}t_{\alpha}=q_{r(g_{i})}=q_{r(e_{i})}=P_{r(e_{i})}.$$
It remains to show $S_{e_{i}^{*}}S_{e_{j}}=S_{e_{j}^{*}}S_{e_{i}}=0$ when $e_{i}\neq e_{j}$. To that end, suppose $e_{i}\neq e_{j}$. Consider first the case when $s(e_{i})\neq s(e_{j})$. Then $f_{1}\neq f'_{1}$, and so we have 
$$S_{e_{i}^{*}}S_{e_{j}}=t_{\alpha^{*}}t_{\alpha'}=t_{g_{i}^{*}}t_{f_{i-1}^{*}}...\big(t_{f_{1}^{*}}t_{f'_{1}}\big)...t_{f'_{j-1}}t_{g'_{j}}=t_{g_{i}^{*}}t_{f_{i-1}^{*}}...t_{f_{2}^{*}}0t_{f'_{2}}...t_{f'_{j-1}}t_{g'_{j}}=0.$$
On the other hand, should $s(e_{i})=s(e_{j})$, $e_{i}\neq e_{j}$ implies $i\neq j$. WLOG, let $j<i$. In which case we have $f_{k}=f'_{k}$ for $k\in\{1,\dots,j-1\}$, but $g'_{j}\neq f_{j}$, and so $S_{e_{i}^{*}}S_{e_{j}}=t_{\alpha^{*}}t_{\alpha'}=0$. These same arguments also show $S_{e_{j}^{*}}S_{e_{i}}=0$. Thus, $S_{e_{i}^{*}}S_{e_{j}}=S_{e_{j}^{*}}S_{e_{i}}=0$ when $e_{i}\neq e_{j}$. Putting it all together, we have $\{P_{A},S_{e},S_{e}^{*}\}$ satisfies \textbf{uLP3}.

(\textbf{uLP4.}) Suppose $v\in G^{0}$ such that $|s^{-1}(v)|<\infty$. Then, for each $e\in s^{-1}(v)$, we have $S_{e}=t_{e}$ and $S_{e^{*}}=t_{e^{*}}$; also note that in this case $s_{G}^{-1}(v)=s_{F}^{-1}(v)$ ($s_{G}^{-1}(v)\neq s_{F}^{-1}(v)$ if and only if $v$ is an infinite emitter in $\mathcal{G}$). Since $\{q_{A},t_{e},t_{e}^{*}\}$ satisfies \textbf{uLP4}, we have
$$P_{v}:=q_{v}=\sum\limits_{e\in s_{F}^{-1}(v)}t_{e}t_{e^{*}}=\sum\limits_{e\in s_{G}^{-1}(v)}t_{e}t_{e^{*}}=\sum\limits_{e\in s_{G}^{-1}(v)}S_{e}S_{e^{*}}.$$
And so $\{P_{A},S_{e},S_{e}^{*}\}$ satisfies \textbf{uLP4} as well.
 
Now that we have established $\{P_{A},S_{e},S_{e}^{*}\}$ is a Leavitt $\mathcal{G}$-family in $L_{R}(\mathcal{F})$, for $L_{R}(\mathcal{G})=L_{R}(\{p,s\})$, the universal mapping property of $L_{R}(\mathcal{G})$ give an $R$-algebra homomorphism 
$$\phi: L_{R}(\mathcal{G})\to L_{R}(\mathcal{F})$$ 
such that $\phi(p_{A})=P_{A}$, $\phi(s_{e})=S_{e}$, and $\phi(s_{e^{*}})=S_{e^{*}}$ for all $A\in\mathcal{G}^{0}$ and $e\in\mathcal{G}^{1}$. Our goal now is to show $\phi$ is  injective. Up to this point, we have exclusively relied on Theorem \ref{uggrdhomom} to show a map out of a Leavitt path algebra is injective. However, for $s(e_{i})$ an infinite emitter in $\mathcal{G}$, the fact $S_{e_{i}}=t_{\alpha}$ with $|\alpha|=i$ means $\phi$ is not a $\mathbb{Z}$-graded homomorphism. This means we will have to show injectivity by a different manner. To that end, suppose $0\neq x\in \text{ker }\phi$. By Lemma \ref{final:lemma3}, there exist $a,b\in L_{R}(\mathcal{G})$ such that\\
1) $0\neq axb=rp_{v}$ for some $v\in G^{0}$ and $r\in R\setminus\{0\}$, or\\
2) $0\neq axb=\sum\limits_{i=0}^{n}r_{i}s_{\alpha}^{i}$ for some cycle $\alpha\in\mathcal{G}^{*}$.\\
And, certainly, $x\in\text{ker}\phi\implies axb\in\text{ker}\phi$. Now, if $axb=rp_{v}$, we have
$$0=\phi(axb)=\phi(rp_{v})=rP_{v}=rq_{v}.$$
But, as we have previously seen, $rq_{v}\neq 0$ for $r\in R\setminus\{0\}$. $\Rightarrow\!\Leftarrow$ The remaining alternative is
$$0=\phi(axb)=\phi\big(\sum\limits_{i=0}^{n}r_{i}s_{\alpha}^{i}\big)=\sum\limits_{i=0}^{n}r_{i}\phi(s_{\alpha})^{i}=\sum\limits_{i=0}^{n}r_{i}S_{\alpha}^{i},$$
for some cycle $\alpha\in\mathcal{G}^{*}$. At this point, we call the reader's attention to the fact, due to how $\{P_{A},S_{e},S_{e^{*}}\}$ is defined,
$$S_{\alpha}=t_{\alpha'} \text{ for some cycle }\alpha'\in\mathcal{F}^{*};$$
 and, since $L_{R}(\mathcal{F})=L_{R}(\{q,t\})$, $rt_{\alpha'}=0\iff r=0$ (to see this, simply note $rt_{\alpha'^{*}}t_{\alpha'}=rq_{r(\alpha')}$). So, we have
$$0=\sum\limits_{i=0}^{n}r_{i}S_{\alpha}^{i}=\sum\limits_{i=0}^{n}r_{i}t_{\alpha'}^{i}.$$
Since $axb\neq 0$, all the $r_{i}$'s can't be zero in the expression $\sum\limits_{i=0}^{n}r_{i}t_{\alpha'}^{i}$; moreover, a summand $r_{i}t_{\alpha'}^{i}$ with $r_{i}\neq0$ uniquely belongs in $L_{R}(\mathcal{F})_{i\cdot|\alpha'|}$ (with respect to the $\mathbb{Z}$-grading on $L_{R}(\mathcal{F})$). But then, as we have seen before, the $\mathbb{Z}$-grading on $L_{R}(\mathcal{F})$ implies 
$$\sum\limits_{i=0}^{n}r_{i}t_{\alpha'}^{i}\neq 0. \Rightarrow\!\Leftarrow$$
Thus, there can't exist $0\neq x\in\text{ker}\phi$, so $\phi$ is injective.
\end{proof}

We are now ready to prove the main result of this section. As the proof is long, we will outline the main idea before proceeding. As mentioned before, we wish to leverage Theorem \ref{ring:theorem1}. Toward that goal, we'll take $\{\mathtt{A}_{i}\}_{i\in\mathbb{N}}\subseteq\mathcal{G}^{0}\subseteq\mathcal{F}^{0}$ to be as in Lemma \ref{final:lemma1}. Then, for $L_{R}(\mathcal{G})=L_{R}(\{p,s\})$, $L_{R}(\mathcal{F})=L_{R}(\{q,t\})$, and $k\in\mathbb{N}$ we set
$$t_{k}:=\sum\limits_{i\leq k}p_{\mathtt{A}_{i}} \text{ and } t'_{k}:=\sum\limits_{i\leq k}q_{\mathtt{A}_{i}}.$$
We will first show the injective map from Proposition \ref{final:proposition2} restricts to an isomorphism from $t_{k}L_{R}(\mathcal{G})t_{k}$ onto $t'_{k}L_{R}(\mathcal{F})t'_{k}$ for each $k$, a fact we will use to show
$$ \varinjlim\limits_{k\in\mathbb{N}}t'_{k}L_{R}(\mathcal{F})t'_{k}\cong \varinjlim\limits_{k\in\mathbb{N}}t_{k}L_{R}(\mathcal{G})t_{k}.$$
Having established the isomorphism of the direct limits above, we will show
$$L_{R}(\mathcal{G})\cong \varinjlim\limits_{k\in\mathbb{N}}t_{k}L_{R}(\mathcal{G})t_{k}.$$
Finally, for $L_{R}(\mathcal{F})t'_{k}$ (a finitely generated, projective, left $L_{R}(\mathcal{F})$-module) we will show the existence of a compatible set $\{L_{R}(\mathcal{F})t'_{k},\varphi,\psi,\mathbb{N}\}$ such that $\varinjlim\limits_{k\in\mathbb{N}}L_{R}(\mathcal{F})t'_{k}$ is a generator of $L_{R}(\mathcal{F})-MOD$ and
$$\varinjlim\limits_{k\in\mathbb{N}}\Big(\text{End}_{L_{R}(\mathcal{F})}(L_{R}(\mathcal{F})t'_{k})\Big)^{op}\cong \varinjlim\limits_{k\in\mathbb{N}}t'_{k}L_{R}(\mathcal{F})t'_{k}.$$
And so, for
$$\varinjlim\limits_{k\in\mathbb{N}}\Big(\text{End}_{L_{R}(\mathcal{F})}(L_{R}(\mathcal{F})t'_{k})\Big)^{op}\cong \varinjlim\limits_{k\in\mathbb{N}}t'_{k}L_{R}(\mathcal{F})t'_{k}\cong\varinjlim\limits_{k\in\mathbb{N}}t_{k}L_{R}(\mathcal{G})t_{k}\cong L_{R}(\mathcal{G}),$$
Theorem \ref{ring:theorem1} establishes the Morita equivalence of $L_{R}(\mathcal{F})$ and $L_{R}(\mathcal{G})$. Now, on to showing what we have outlined. The proof is essentially that of \cite[Theorem 5.6]{G.-Abrams:2008aa}, but modified to work in the case of ultragraph Leavitt path algebras.

\begin{theorem}\label{final:theorem1}
Let $\mathcal{G}$ be an ultragraph and $\mathcal{F}$ its desingularization. Then, $L_{R}(\mathcal{G})$ and $L_{R}(\mathcal{F})$ are Morita equivalent. 
\end{theorem}

\begin{proof}
Let $\{\mathtt{A}_{i}\}_{i\in\mathbb{N}}\subseteq\mathcal{G}^{0}\subseteq\mathcal{F}^{0}$ be as in Lemma \ref{final:lemma1}, and set
$$t_{k}:=\sum\limits_{i\leq k}p_{\mathtt{A}_{i}} \text{ and } t'_{k}:=\sum\limits_{i\leq k}q_{\mathtt{A}_{i}}.$$
The first crucial step in our proof is establishing the fact the map $\phi$ from\\
\noindent Proposition \ref{final:proposition2} restricts to an isomorphism from $t_{k}L_{R}(\mathcal{G})t_{k}$ onto $t'_{k}L_{R}(\mathcal{F})t'_{k}$ for each $k$. Since $\phi(t_{k})=t'_{k}$, it follows
$$\text{im}\phi|_{t_{k}L_{R}(\mathcal{G})t_{k}}\subseteq t'_{k}L_{R}(\mathcal{F})t'_{k}.$$
That $\phi$ is injective has already been established. What we need to show now is that it is surjective. To that end, we have
{\footnotesize 
$$t'_{k}L_{R}(\mathcal{F})t'_{k}=\text{span}_{R}\{t_{\alpha}q_{A}t_{\beta^{*}}:\alpha,\beta\in\mathcal{F}^{*},\ r_{F}(\alpha)\cap A\cap r_{F}(\beta)\neq\emptyset,\ \text{and}\ s_{F}(\alpha), s_{F}(\beta)\subseteq\bigcup_{i\leq k}\mathtt{A}_{i}\in \mathcal{G}^{0}\},$$\par} 
\noindent which follows primarily from Theorem \ref{uggrade}. And so, to establish the surjectivity of $\phi|_{t_{k}L_{R}(\mathcal{G})t_{k}}$, we need only show
$$t_{\alpha}q_{A}t_{\beta^{*}}\in\text{im}\phi|_{t_{k}L_{R}(\mathcal{G})t_{k}} \text{ for } \alpha,\beta\in\mathcal{F}^{*} \text{ with } s_{F}(\alpha),s_{F}(\beta)\subseteq\bigcup_{i\leq k}\mathtt{A}_{i}.$$
 As one might guess, understanding what $\alpha\in\mathcal{F}^{*}$, with $s_{F}(\alpha)\in\mathcal{G}^{0}$, looks like is important to our endeavor. We will show such a path $\alpha$ takes one of two forms:\\
 1) $\alpha=\alpha_{1}\dots\alpha_{n}$, where, for each $i$, $\alpha_{i}\in\mathcal{G}^{*}$, or $\alpha_{i}=f_{1}...f_{j-1}g_{j}$ in some tail added to an infinite emitter in $\mathcal{G}$, or\\
 2) $\alpha=\alpha_{1}\dots\alpha_{n}f_{1}\dots f_{m}$, where each $\alpha_{i}$ satisfies one of the conditions listed above, and $f_{1}\dots f_{m}$ is an initial segment in a tail added to a singular vertex (either a sink or an infinite emitter) in $\mathcal{G}$.\\
To see this, notice that the construction of $\mathcal{F}$ makes it so that, for $\alpha\in\mathcal{F}^{*}$, 
$$r_{F}(\alpha)\in\mathcal{G}^{0}, \text{ or } r_{F}(\alpha)=r_{F}(f)\in F^{0}\setminus G^{0},$$
where $f$ is an edge along a tail added to some singular vertex. For $\alpha$ such that $s_{F}(\alpha)\in\mathcal{G}^{0}$, consider first the case where $r_{F}(\alpha)\in\mathcal{G}^{0}$. Now, should $\alpha\in\mathcal{G}^{*}$, we are done. Otherwise, $\alpha$ must contain an edge $f$ in a tail added to a singular vertex, since $s_{F}(\alpha)\in\mathcal{G}^{0}$, this in turn means it contains an edge whose source is a singular vertex in $\mathcal{G}$. And so we can express $\alpha$ as 
$$\alpha=\alpha_{1}\alpha'$$
where $\alpha_{1}$ is a path in $\mathcal{G}$ (possibly an element of $\mathcal{G}^{0}$), and $\alpha'\in\mathcal{F}^{*}$ such that $s_{F}(\alpha')$ is a singular vertex in $\mathcal{G}$. But then the fact $r_{F}(\alpha')=r_{F}(\alpha)\in\mathcal{G}^{0}$ forces $\alpha'$ to have an initial segment of the form $f_{1}...f_{j-1}g_{j}$; i.e., $\alpha'$ can be expressed as
$$\alpha'=\alpha_{2}\alpha''$$
where $\alpha_{2}=f_{1}...f_{j-1}g_{j}$ in a tail added to $s_{F}(\alpha')$. Notice $\alpha''\in\mathcal{F}^{*}$ is such that $s_{F}(\alpha''),r_{F}(\alpha'')\in\mathcal{G}^{0}$, just like $\alpha$. Applying the same argument to $\alpha''$ as we did with $\alpha$, and so on, we have 
$$\alpha=\alpha_{1}\dots\alpha_{n}$$
such that, for each $i$, $\alpha_{i}\in\mathcal{G}^{*}$, or $\alpha_{i}=f_{1}...f_{j-1}g_{j}$ in some tail added to an infinite emitter in $\mathcal{G}$. Alternatively, suppose $r_{F}(\alpha)=r_{F}(f)=v\in F^{0}\setminus G^{0}$ for some edge $f$ in a tail added to singular vertex. Since $s_{F}(\alpha)\in \mathcal{G}^{0}$, we can express $\alpha$ as
$$\alpha=\alpha'f_{1}...f_{m}$$
where $f_{m}=f$ and $\alpha'$ is a path (possibly an element of $\mathcal{G}^{0}$) such that $s_{F}(\alpha')\in \mathcal{G}^{0}$ and $s_{F}(f_{1})\in r_{F}(\alpha')\in\mathcal{G}^{0}$. But then, by what we have previously seen, 
$$\alpha=\alpha'f_{1}\dots f_{m}=\alpha_{1}\dots\alpha_{n} f_{1}\dots f_{m}$$
where each $\alpha_{i}$ satisfies one of the given two conditions. In light of what we have just seen, consider again
{\footnotesize
$$t'_{k}L_{R}(\mathcal{F})t'_{k}=\text{span}_{R}\{t_{\alpha}q_{A}t_{\beta^{*}}:\alpha,\beta\in\mathcal{F}^{*},\ r_{F}(\alpha)\cap A\cap r_{F}(\beta)\neq\emptyset,\ \text{and}\ s_{F}(\alpha), s_{F}(\beta)\subseteq\bigcup_{i\leq k}\mathtt{A}_{i}\in \mathcal{G}^{0}\}.$$\par}
\noindent For $\alpha, \beta\in\mathcal{F}^{*}$ such that $t_{\alpha}q_{A}t_{\beta^{*}}\in t'_{k}L_{R}(\mathcal{F})t'_{k}$, it must be $\alpha=\alpha_{1}..\alpha_{n}$, or $\alpha=\alpha_{1}..\alpha_{n}f_{1}...f_{m}$, and $\beta=\beta_{1}...\beta_{k}$, or $\beta=\beta_{1}...\beta_{k}f'_{1}...f'_{l}$ as described above. There's more we can say, notice if $\alpha=\alpha_{1}..\alpha_{n}f_{1}...f_{m}$ and $\beta=\beta_{1}...\beta_{k}$ (similarly $\alpha=\alpha_{1}..\alpha_{n}$ and $\beta=\beta_{1}...\beta_{k}f'_{1}...f'_{l}$), $t_{\alpha}q_{A}t_{\beta^{*}}=0$ since $r_{F}(\alpha)\cap r_{F}(\beta)=\emptyset$. Therefore, we only need to consider $t_{\alpha}q_{A}t_{\beta^{*}}$ where $\alpha=\alpha_{1}..\alpha_{n}$ and $\beta=\beta_{1}...\beta_{k}$, or $\alpha=\alpha_{1}..\alpha_{n}f_{1}...f_{m}$ and $\beta=\beta_{1}...\beta_{k}f'_{1}...f'_{l}$.

In the case where $\alpha$ and $\beta$ are such that $\alpha=\alpha_{1}..\alpha_{n}$ and $\beta=\beta_{1}...\beta_{k}$, we can see 
$$t_{\alpha}q_{A}t_{\beta^{*}}\in\text{im}\phi|_{t_{k}L_{R}(\mathcal{G})t_{k}}$$
based on how $\phi$ is defined (see Proposition \ref{final:proposition2}). What's left is to show the same holds when $\alpha=\alpha_{1}..\alpha_{n}f_{1}...f_{m}$ and $\beta=\beta_{1}...\beta_{k}f'_{1}...f'_{l}$. To that end, it suffices to show
$$t_{f_{1}...f_{m}}q_{A}t_{(f'_{1}...f'_{l})^{*}}\in\text{im}\phi|_{t_{k}L_{R}(\mathcal{G})t_{k}}.$$
Based on the construction of $\mathcal{F}$, and \textbf{uLP2},
$$t_{f_{m}}q_{A}t_{(f'_{l})^{*}}=t_{f_{m}}q_{r_{F}(f_{m})}t_{(f'_{l})^{*}}\neq0\iff f_{m}=f'_{l}.$$
This in turn means
$$t_{f_{1}...f_{m}}q_{A}t_{(f'_{1}...f'_{l})^{*}}\neq0\iff f_{1}...f_{m}=f'_{1}...f'_{l},$$
since $f_{m}=f'_{l}$ implies $m=l$, and the fact $t_{f_{1}...f_{m}}q_{A}t_{(f'_{1}...f'_{l})^{*}}\neq0$ means $f_{i}=f'_{i}$ for each $1\leq i\leq m$. For $t_{f_{1}...f_{m}}q_{A}t_{(f_{1}...f_{m})^{*}}\in t'_{k}L_{R}(\mathcal{F})t'_{k}$, there are two cases to consider: $s_{F}(f_{1})$ is a sink, or $s_{F}(f_{1})$ is an infinite emitter in $\mathcal{G}$. In the case $s_{F}(f_{1})$ is a sink, a quick calculation shows
$$t_{f_{1}...f_{m}}q_{A}t_{(f_{1}...f_{m})^{*}}=q_{s_{F}(f_{1})}=\phi(p_{s_{F}(f_{1})})\in\text{im}\phi|_{t_{k}L_{R}(\mathcal{G})t_{k}}.$$
On the other hand, suppose $s_{F}(f_{1})$ is an infinite emitter and let $\{e_{i}\}_{i\in\mathbb{N}}$ and an enumeration of $s_{F}^{-1}(s_{F}(f_{1}))$. By \textbf{uLP4}, we have $t_{f_{m}}t_{f_{m}^{*}}=q_{s_{F}(f_{m})}-t_{g_{m}}t_{g_{m}^{*}}$, and so
\begin{align*}
t_{f_{1}...f_{m}}q_{A}t_{(f_{1}...f_{m})^{*}}&=t_{f_{1}..f_{m-1}}\big(q_{s_{F}(f_{m})}-t_{g_{m}}t_{g_{m}^{*}}\big)t_{(f_{1}...f_{m-1})^{*}}\\
&=t_{f_{1}...f_{m-2}}\big(t_{f_{m-1}}t_{f_{m-1}^{*}}\big)t_{(f_{1}...f_{m-2})^{*}}-t_{f_{1}..f_{m-1}g_{m}}t_{(f_{1}...f_{m-1}g_{m})^{*}},\\
&\text{    applying \textbf{uLP4} to }t_{f_{m-1}}t_{f_{m-1}^{*}},\\
&=t_{f_{1}...f_{m-2}}\big(q_{s_{F}(f_{m-1})}-t_{g_{m-1}}t_{g_{m-1}^{*}}\big)t_{(f_{1}...f_{m-2})^{*}}-t_{f_{1}..f_{m-1}g_{m}}t_{(f_{1}...f_{m-1}g_{m})^{*}}\\
&=t_{f_{1}...f_{m-2}}t_{(f_{1}...f_{m-2})^{*}}-\big(t_{f_{1}..f_{m-2}g_{m-1}}t_{(f_{1}...f_{m-2}g_{m-1})^{*}}\\
&+t_{f_{1}..f_{m-1}g_{m}}t_{(f_{1}...f_{m-1}g_{m})^{*}}\big),\\
&\text{   applying \textbf{uLP4} to }t_{f_{m-2}}t_{f_{m-2}^{*}}\text{ and so on, we have}\\
&=t_{f_{1}}t_{f_{1}^{*}}-\Bigg(\sum\limits_{i=1}^{m}t_{f_{1}...f_{i-1}g_{i}}t_{(f_{1}...f_{i-1}g_{i})^{*}}\Bigg)\\
&=\big(q_{s_{F}(f_{1})}-t_{g_{1}}t_{g_{1}^{*}}\big)-\Bigg(\sum\limits_{i=1}^{m}t_{f_{1}...f_{i-1}g_{i}}t_{(f_{1}...f_{i-1}g_{i})^{*}}\Bigg)\\
&=\phi(p_{s_{F}(f_{1})})-\phi(t_{e_{1}}t_{e_{1}^{*}})-\Bigg(\sum\limits_{i=1}^{m}\phi(t_{e_{i}}t_{e_{i}^{*}})\Bigg)\in\text{im}\phi|_{t_{k}L_{R}(\mathcal{G})t_{k}}.
\end{align*}
Thus, $t'_{k}L_{R}(\mathcal{F})t'_{k}\subseteq\text{im}\phi|_{t_{k}L_{R}(\mathcal{G})t_{k}}$, and so, as desired, $\phi|_{t_{k}L_{R}(\mathcal{G})t_{k}}$ is surjective. Now, for $k\leq l$, let $\varphi_{k,l}:t_{k}L_{R}(\mathcal{G})t_{k}\to t_{l}L_{R}(\mathcal{G})t_{l} \text{ and }\varphi'_{k,l}:t'_{k}L_{R}(\mathcal{F})t'_{k}\to t'_{l}L_{R}(\mathcal{F})t'_{l}$ be the inclusion maps; one can then easily check $\big\langle t_{k}L_{R}(\mathcal{G})t_{k}, \varphi_{k,l}\big\rangle$ and $\big\langle t'_{k}L_{R}(\mathcal{G})t'_{k}, \varphi'_{k,l}\big\rangle$ are direct systems over $\mathbb{N}$ (note, we are taking the direct limit in the category of rings). Moreover, we have the following commuting diagram:

\begin{figure}[h!]
\begin{center}
\begin{tikzpicture}

 \node [shape=circle,minimum size=1.5em] (d1) at (0,1) {$t_{k}L_{R}(\mathcal{G})t_{k}$};
 \node [shape=circle,minimum size=1.5em] (d2) at (6,1) {$t'_{k}L_{R}(\mathcal{F})t'_{k}$};
 \node [shape=circle,minimum size=1.5em] (d3) at (.4,-1) {$t_{k+1}L_{R}(\mathcal{G})t_{k+1}$};
 \node [shape=circle,minimum size=1.5em] (d4) at (6.4,-1) {$t'_{k+1}L_{R}(\mathcal{F})t'_{k+1}$};
 
\path (d1) edge [->, >=latex, shorten <= 2pt, shorten >= 2pt, right] node[above]{$\phi|_{t_{k}L_{R}(\mathcal{G})t_{k}}$} (d2);

\path (d3) edge [->, >=latex, shorten <= 2pt, shorten >= 2pt, right] node[above]{$\phi|_{t_{k+1}L_{R}(\mathcal{G})t_{k+1}}$} (d4);

 \node [shape=circle,minimum size=1.5em] (d6) at (0,-1) {};

 \path (d1) edge [->, >=latex, shorten <= -17pt, right] node[pos=-0.1]{$\varphi_{k,k+1}$} (d6);
 
 \node [shape=circle,minimum size=1.5em] (d7) at (6,-1) {};
 
  \path (d2) edge [->, >=latex, shorten <= -17pt, right] node[pos=-0.1]{$\varphi'_{k,k+1}$} (d7);

\end{tikzpicture}
\end{center}
\end{figure}
\noindent Since $\phi|_{t_{k}L_{R}(\mathcal{G})t_{k}}$ is an isomorphism for each $k$, we have
$$\varinjlim\limits_{k\in\mathbb{N}}t_{k}L_{R}(\mathcal{G})t_{k} \cong \varinjlim\limits_{k\in\mathbb{N}}t'_{k}L_{R}(\mathcal{F})t'_{k}$$
as rings. For verification of the stated isomorphism above see 24.4 in \cite{Wisbauer:1991aa} where it is proved within the context of modules, but the argument remains unchanged for any category which admits direct limits. Further, let $i_{k}:t_{k}L_{R}(\mathcal{G})t_{k}\to L_{R}(\mathcal{G})$ be the inclusion map for each $k$. For $l\geq k$, we have $i_{k}=i_{l}\circ\varphi_{k,l}.$ By the universal mapping property of $\varinjlim\limits_{k\in\mathbb{N}}t_{k}L_{R}(\mathcal{G})t_{k}$, there exists a ring homomorphism
$$\Phi:\varinjlim\limits_{k\in\mathbb{N}}t_{k}L_{R}(\mathcal{G})t_{k}\to L_{R}(\mathcal{G}).$$
Since $\{t_{k}\}_{k\in\mathbb{N}}$ is a set of local units for $L_{R}(\mathcal{G})$, $\Phi$ is a surjective. Further, suppose $\Phi(x)=0$. Since $x=\eta_{k}(y)$, where 
$$\eta_{k}:t_{k}L_{R}(\mathcal{G})t_{k}\to\varinjlim\limits_{k\in\mathbb{N}}t_{k}L_{R}(\mathcal{G})t_{k}$$
is as in the definition of a direct limit, we have $\Phi\circ\eta_{k}(y)=0$. But, by the properties of a direct limit, $i_{k}=\Phi\circ\eta_{k}$, meaning 
$$\Phi\circ\eta_{k}(y)=i_{k}(y)=0.$$
Well, $i_{k}$ is the inclusion map, so it must be $y=0$, and so $x=\eta_{k}(y)=0$. Thus, $\Phi$ is injective. All in all, this means
$$\varinjlim\limits_{k\in\mathbb{N}}t_{k}L_{R}(\mathcal{G})t_{k}\cong L_{R}(\mathcal{G}).$$

We have established, up to this point
$$\varinjlim\limits_{k\in\mathbb{N}}t'_{k}L_{R}(\mathcal{F})t'_{k}\cong\varinjlim\limits_{k\in\mathbb{N}}t_{k}L_{R}(\mathcal{G})t_{k}\cong L_{R}(\mathcal{G}).$$
To complete our proof, it remains to show $L_{R}(\mathcal{F})t'_{k})$ is a finitely generated, projective, left $L_{R}(\mathcal{F})$-module for each $k$, and that there exists a compatible set $\{L_{R}(\mathcal{F})t'_{k},\varphi,\psi,\mathbb{N}\}$ such that $\varinjlim\limits_{k\in\mathbb{N}}L_{R}(\mathcal{F})t'_{k}$ is a generator of $L_{R}(\mathcal{F})-MOD$ and
$$\varinjlim\limits_{k\in\mathbb{N}}\Big(\text{End}_{L_{R}(\mathcal{F})}(L_{R}(\mathcal{F})t'_{k})\Big)^{op}\cong\varinjlim\limits_{k\in\mathbb{N}}t'_{k}L_{R}(\mathcal{F})t'_{k}.$$
To that end, first note that, for each $k$, $L_{R}(\mathcal{F})t'_{k}$ is generated by $t'_{k}$ as an $L_{R}(\mathcal{F})$-module; thus, it is a finitely generated module. Further, by Claim \ref{ring:remark2}, we have $L_{R}(\mathcal{F})t'_{k}$ is projective. We now want to show $\varinjlim\limits_{k\in\mathbb{N}}L_{R}(\mathcal{F})t'_{k}$ is a generator of $L_{R}(\mathcal{F})-MOD$. As showing this fact is a bit convoluted, we will first sketch out the steps. We will first show $L_{R}(\mathcal{F})$ is a generator of $L_{R}(\mathcal{F})-MOD$; a fact which we will in turn leverage to show $\bigoplus\limits_{i\in\mathbb{N}}L_{R}(\mathcal{F})q_{\mathtt{A}_{i}}$ is a generator of $L_{R}(\mathcal{F})-MOD$. And, finally, we will establish
$$\bigoplus\limits_{i\in\mathbb{N}}L_{R}(\mathcal{F})q_{\mathtt{A}_{i}}\cong\varinjlim\limits_{k\in\mathbb{N}}L_{R}(\mathcal{F})t'_{k},$$
thereby showing $\varinjlim\limits_{k\in\mathbb{N}}L_{R}(\mathcal{F})t'_{k}$ is a generator for $L_{R}(\mathcal{F})-MOD$.

By Claim \ref{defabequiv}, to show $L_{R}(\mathcal{F})$ is a generator of $L_{R}(\mathcal{F})-MOD$, it suffices to show, for every non-zero $L_{R}(\mathcal{F})$-module homomorpshim $f:M\to N,$ there exists an $L_{R}(\mathcal{F})$-module homomorpshim $h:L_{R}(\mathcal{F})\to M$ such that $f\circ h\neq 0$. To that end, let $0\neq f:M\to N$ be an $L_{R}(\mathcal{F})$-module homomorpshim. This means there exists $m\in M$ such that $f(m)\neq0$, and since $M=L_{R}(\mathcal{F})M$, there exists $m'$ such that $m=xm'$ for some $x\in L_{R}(\mathcal{F})$. Defining $h:L_{R}(\mathcal{F})\to M$ by $x\mapsto xm'$, we have have $f\circ h\neq0$. Meaning, $L_{R}(\mathcal{F})$ is a generator of $L_{R}(\mathcal{F})-MOD$.

Now, let $\{\mathtt{A}_{i}\}_{i\in\mathbb{N}}\subseteq\mathcal{G}^{0}\subseteq\mathcal{F}^{0}$ be as in Lemma \ref{final:lemma1}. For any $A\in\mathcal{F}^{0}$, one can see $A=A'\cup\{v_{j}\}_{j=1}^{k}$ for some $A'\in\mathcal{G}^{0}$ and $\{v_{j}\}_{j=1}^{k}\subseteq F^{0}\setminus G^{0}$. Since $A'$ is contained in finitely many of the $\mathtt{A}_{i}$'s, if after enumerating $F^{0}\setminus G^{0}$ we set
$$\mathtt{B}_{2i}:=\mathtt{A}_{i} \text{ and } \mathtt{B}_{2i-1}:=\{v_{i}\}, \text{ for } v_{i}\in F^{0}\setminus G^{0},$$
we can check $\{\mathtt{B}_{i}\}_{i\in\mathbb{N}}$ satisfies the hypothesis stated in Lemma \ref{final:lemma1}. Thus, by Corollary \ref{final:corollary1},  
$$L_{R}(\mathcal{F})\cong\bigoplus\limits_{i\in\mathbb{N}}L_{R}(\mathcal{F})q_{\mathtt{B}_{i}}\cong \bigg(\bigoplus\limits_{i\in\mathbb{N}}L_{R}(\mathcal{F})q_{\mathtt{A_{i}}}\bigg)\bigoplus\bigg(\bigoplus\limits_{v\in F^{0}\setminus G^{0}}L_{R}(\mathcal{F})q_{v}\bigg).$$

Our next task is to use the fact established above to show, for some set $S$, there is a surjective module homomorphism
$$\rho:\bigoplus\limits_{s\in S}\bigg(\bigoplus\limits_{i\in\mathbb{N}}L_{R}(\mathcal{F})q_{\mathtt{A}_{i}}\bigg)\to L_{R}(\mathcal{F}),$$
thereby establishing $\bigoplus\limits_{i\in\mathbb{N}}L_{R}(\mathcal{F})q_{\mathtt{A}_{i}}$ is a generator of $L_{R}(\mathcal{F})-MOD$ as well. With this in mind, suppose $v_{0}$ is a singular vertex in $\mathcal{G}$, and $v_{j}\in F^{0}\setminus G^{0}$ is a vertex along a portion of a tail, $\alpha=f_{1}...f_{j}$, added at $v_{0}$ (see figure below). 
\begin{figure}[h!]
\begin{center}
\begin{tikzpicture}
\tikzset{vertex/.style = {shape=circle, draw=black!100,fill=black!100, thick, inner sep=0pt, minimum size=2 mm}}
\tikzset{edge/.style = {->, line width=1pt}}
\tikzset{v/.style = {shape=rectangle, dashed, draw, inner sep=0pt, minimum size=2em, minimum width=3em}}
    \node[vertex] (a) at (-9,1) [label=above:$v_{0}$]{};
    \node[vertex] (b) at (-6,1) [label=above:$v_{1}$]{};
    \node[vertex] (c) at (-3,1) [label=above:$v_{2}$]{};
    \node[vertex] (d) at (2,1) [label=above:$v_{j-1}$]{};
     \node[vertex] (e) at (5,1) [label=above:$v_{j}$]{};   
       
   \path (a) edge [->, >=latex, line width=.75pt, shorten <= 2pt, shorten >= 2pt, right] node[above]{$f_{1}$} (b);
   
   \path (b) edge [->, >=latex, line width=.75pt, shorten <= 2pt, shorten >= 2pt, right] node[above]{$f_{2}$} (c);

    \node [shape=circle,minimum size=1.5em] (d1) at (0,1) {};
    
     \path (c) edge [->, >=latex, line width=.75pt, shorten <= 2pt, shorten >= 2pt, right] (d1);

   \path (d1) to node {\dots} (d);

   \path (d) edge [->, >=latex, line width=.75pt, shorten <= 2pt, shorten >= 2pt, right] node[above]{$f_{j}$} (e);

   \end{tikzpicture}
\caption{}
\label{im9}
\end{center}
\end{figure}

\noindent We can define a module homomorphism,
$$\rho^{*}_{j}:L_{R}(\mathcal{F})q_{v_{j}}\to L_{R}(\mathcal{F})q_{v_{0}},$$
by $\rho^{*}_{j}(y)=y\alpha^{*}$; similarly, we can define
$$\rho_{j}:L_{R}(\mathcal{F})q_{v_{0}}\to L_{R}(\mathcal{F})q_{v_{j}}$$
by $\rho_{j}(x)=x\alpha$. Since $\alpha^{*}\alpha=v_{j}$, we can conclude $\rho_{j}\circ\rho^{*}_{j}=\text{Id}_{L_{R}(\mathcal{F})q_{v_{j}}}$, which in turn means $\rho_{j}$ is surjective. Suppose now $v_{0}\in\mathtt{A}_{i_{0}}$. For
$$L_{R}(\mathcal{F})q_{\mathtt{A}_{i_{0}}}\cong L_{R}(\mathcal{F})q_{\mathtt{A}_{i_{0}}\setminus\{v_{0}\}}\bigoplus L_{R}(\mathcal{F})q_{v_{0}},$$
we can take the projection onto $L_{R}(\mathcal{F})q_{v_{0}}$, and by composing with $\rho_{j}$, we have a surjective homomorphsim from $L_{R}(\mathcal{F})q_{\mathtt{A}_{i_{0}}}$ onto $L_{R}(\mathcal{F})q_{v_{j}}$. In turn, by composing with the projection of $\bigoplus\limits_{i\in \mathbb{N}}L_{R}(\mathcal{F})q_{\mathtt{A}_{i}}$ onto $L_{R}(\mathcal{F})q_{\mathtt{A}_{i_{0}}}$, we can conclude $L_{R}(\mathcal{F})q_{v_{j}}$ is the homomorphic image of $\bigoplus\limits_{i\in \mathbb{N}}L_{R}(\mathcal{F})q_{\mathtt{A}_{i}}$ for each $v_{j}\in F^{0}\setminus G^{0}$. Thus, defining homomorphisms corrdinate wise, we get a surjective homomorphism
$$\bigoplus\limits_{v_{j}\in F^{0}\setminus G^{0}}\bigg(\bigoplus\limits_{i\in\mathbb{N}}L_{R}(\mathcal{F})q_{\mathtt{A}_{i}}\bigg)\to\bigoplus\limits_{v_{j}\in F^{0}\setminus G^{0}}L_{R}(\mathcal{F})q_{v_{j}},$$
giving us a surjective homomorphism
{\footnotesize
$$\bigg(\bigoplus\limits_{i\in\mathbb{N}}L_{R}(\mathcal{F})q_{\mathtt{A}_{i}}\bigg)\oplus\bigg(\bigoplus\limits_{v_{j}\in F^{0}\setminus G^{0}}\bigg(\bigoplus\limits_{i\in\mathbb{N}}L_{R}(\mathcal{F})q_{\mathtt{A}_{i}}\bigg) \bigg)\to\bigg(\bigoplus\limits_{i\in\mathbb{N}}L_{R}(\mathcal{F})q_{\mathtt{A}_{i}}\bigg)\oplus\bigg(\bigoplus\limits_{v_{j}\in F^{0}\setminus G^{0}}L_{R}(\mathcal{F})q_{v_{j}} \bigg).$$\par}
\noindent Thus, for $L_{R}(\mathcal{F})\cong\bigg(\bigoplus\limits_{i\in\mathbb{N}}L_{R}(\mathcal{F})q_{\mathtt{A_{i}}}\bigg)\bigoplus\bigg(\bigoplus\limits_{v\in F^{0}\setminus G^{0}}L_{R}(\mathcal{F})q_{v}\bigg)$, we have a set $S$ and a surjective homomorphism
$$\rho:\bigoplus\limits_{s\in S}\bigg(\bigoplus\limits_{i\in\mathbb{N}}L_{R}(\mathcal{F})q_{\mathtt{A}_{i}}\bigg)\to L_{R}(\mathcal{F}),$$
and since $ L_{R}(\mathcal{F})$ is a generator of $L_{R}(\mathcal{F})-MOD$, this means $\bigoplus\limits_{i\in \mathbb{N}}L_{R}(\mathcal{F})q_{\mathtt{A}_{i}}$ is a generator of $L_{R}(\mathcal{F})-MOD$ as well.

For $k\leq l$, let $\varphi'_{k,l}:L_{R}(\mathcal{F})t'_{k}\to L_{R}(\mathcal{F})t'_{l}$ be the inclusion map and consider the direct system of $L_{R}(\mathcal{F})$-modules, $\Big\langle L_{R}(\mathcal{F})t'_{k}, \varphi'_{k,l}\Big\rangle$, over $\mathbb{N}$. First, notice the map given by
$$x\mapsto (xq_{\mathtt{A}_{1}},...,xq_{\mathtt{A}_{k}})$$
defines an isomorphism from $L_{R}(\mathcal{F})t'_{k}$ onto $\bigoplus\limits_{i\leq k}L_{R}(\mathcal{F})q_{\mathtt{A}_{i}}$; then, by composing with the inclusion map from $\bigoplus\limits_{i\leq k}L_{R}(\mathcal{F})q_{\mathtt{A}_{i}}$ into $\bigoplus\limits_{i\in\mathbb{N}}L_{R}(\mathcal{F})q_{\mathtt{A}_{i}}$, we get an injective map,
$$i_{k}:L_{R}(\mathcal{F})t'_{k}\to\bigoplus\limits_{i\in\mathbb{N}}L_{R}(\mathcal{F})q_{\mathtt{A}_{i}},$$
such that $i_{k}=i_{l}\circ\varphi'_{k,l}$ for each $k\leq l$. And so, by the universal mapping property of direct limits, we have an $L_{R}(\mathcal{F})$-module homomorphism
$$\theta: \varinjlim\limits_{k\in\mathbb{N}}L_{R}(\mathcal{F})t'_{k} \to \bigoplus\limits_{i\in\mathbb{N}}L_{R}(\mathcal{F})q_{\mathtt{A}_{i}}$$
such that $i_{k}=\theta\circ i_{k}$ for each $k$. Since for each $y\in \bigoplus\limits_{i\in\mathbb{N}}L_{R}(\mathcal{F})q_{\mathtt{A}_{i}}$, there exists a $k$ and $x\in L_{R}(\mathcal{F})t'_{k}$ such that $i_{k}(x)=y$, $\theta$ is surjective. Moreover, since, being the inclusion map, $\varphi'_{k,l}$ is injective for each $k\leq l$, $\theta$ is injective as well (see \cite[24.3, 1) and 2)]{Wisbauer:1991aa}). Thus,
$$\varinjlim\limits_{k\in\mathbb{N}}L_{R}(\mathcal{F})t'_{k} \cong \bigoplus\limits_{i\in\mathbb{N}}L_{R}(\mathcal{F})q_{\mathtt{A}_{i}},$$
meaning $\varinjlim\limits_{k\in\mathbb{N}}L_{R}(\mathcal{F})t'_{k}$ is a generator of $L_{R}(\mathcal{F})-MOD$ as well. Finally, to complete our proof, we need to show  
$$\varinjlim\limits_{k\in\mathbb{N}}\Big(\text{End}_{L_{R}(\mathcal{F})}(L_{R}(\mathcal{F})t'_{k})\Big)^{op}\cong\varinjlim\limits_{k\in\mathbb{N}}t'_{k}L_{R}(\mathcal{F})t'_{k}.$$
To that end, let $\phi\in\text{End}_{L_{R}(\mathcal{F})}(L_{R}(\mathcal{F})t'_{k})$. Since $\phi$ is an $L_{R}(\mathcal{F})$-module homomorphism, note $\phi(xt'_{k})=x\phi(t'_{k})$. Moreover, $t'_{k}\phi(t'_{k})=\phi((t'_{k})^{2})=\phi(t'_{k})$, and since $\phi(t'_{k})$ is already an element of $L_{R}(\mathcal{F})t'_{k}$, this means $\phi(t'_{k})\in t'_{k}L_{R}(\mathcal{F})t'_{k}$. Thus, we can conclude $\phi$ is given by right multiplication by an element of $t'_{k}L_{R}(\mathcal{F})t'_{k}$. Define, then, a map
$$\Phi_{k}: t'_{k}L_{R}(\mathcal{F})t'_{k}\to \text{End}_{L_{R}(\mathcal{F})}(L_{R}(\mathcal{F})t'_{k})$$
where $\Phi_{k}(x)\in  \text{End}_{L_{R}(\mathcal{F})}(L_{R}(\mathcal{F})t'_{k})$ is right multiplication by $x$; since every element of $\text{End}_{L_{R}(\mathcal{F})}(L_{R}(\mathcal{F})t'_{k})$ arises this way, $\Phi_{k}$ is surjective. Also, $L_{R}(\mathcal{F})$ is a ring with local units (meaning $xy=0$ for all $x$ if and only if $y=0$), and so $\Phi_{k}$ is injective as well. Lastly, we can easily check $\Phi_{k}(xy)=\Phi_{k}(y)\circ\Phi_{k}(x)$, which means $\Phi_{k}$ is an anti-isomorphism; put differently,
$$t'_{k}L_{R}(\mathcal{F})t'_{k}\cong\Big(\text{End}_{L_{R}(\mathcal{F})}(L_{R}(\mathcal{F})t'_{k})\Big)^{op}.$$
By Proposition \ref{ring:prop2}, 
$$\Big\langle \Big(\text{End}_{L_{R}(\mathcal{F})}(L_{R}(\mathcal{F})t'_{k})\Big)^{op}, \overline{\varphi'}_{k,l}\Big\rangle$$
is a direct system of rings over $\mathbb{N}$. For $k\leq l$, one can check the following diagram commutes:

\clearpage

\begin{figure}[h!]
\begin{center}
\begin{tikzpicture}

 \node [shape=circle,minimum size=1.5em] (d1) at (0,1) {$t'_{k}L_{R}(\mathcal{F})t'_{k}$};
 \node [shape=circle,minimum size=1.5em] (d2) at (6,1) {$\Big(\text{End}_{L_{R}(\mathcal{F})}(L_{R}(\mathcal{F})t'_{k})\Big)^{op}$};
 \node [shape=circle,minimum size=1.5em] (d3) at (0,-2) {$t'_{l}L_{R}(\mathcal{F})t'_{l}$};
 \node [shape=circle,minimum size=1.5em] (d4) at (6,-2) {$\Big(\text{End}_{L_{R}(\mathcal{F})}(L_{R}(\mathcal{F})t'_{l})\Big)^{op}$};
 
\path (d1) edge [->, >=latex, shorten <= 2pt, shorten >= 2pt, right] node[above]{$\Phi_{k}$} (d2);

\path (d3) edge [->, >=latex, shorten <= 2pt, shorten >= 2pt, right] node[above]{$\Phi_{l}$} (d4);

 \node [shape=circle,minimum size=1.5em] (d6) at (0,-2) {};

 \path (d1) edge [->, >=latex, shorten <= -17pt, right] node[pos=0.3]{$\varphi'_{k,l}$} (d6);
 
 \node [shape=circle,minimum size=1.5em] (d7) at (6,-2) {};
 
  \path (d2) edge [->, >=latex, shorten <= -50pt, right] node[pos=-2]{$\overline{\varphi'}_{k,l}$} (d7);

\end{tikzpicture}
\end{center}
\end{figure}

\noindent Since each $\Phi_{k}$ is an isomorphism, we have
$$\varinjlim\limits_{k\in\mathbb{N}}\Big(\text{End}_{L_{R}(\mathcal{F})}(L_{R}(\mathcal{F})t'_{k})\Big)^{op}\cong\varinjlim\limits_{k\in\mathbb{N}}t'_{k}L_{R}(\mathcal{F})t'_{k}\cong L_{R}(\mathcal{G}).$$
Thus, by Theorem \ref{ring:theorem1}, $L_{R}(\mathcal{G})$ and $L_{R}(\mathcal{F})$ are Morita equivalent. 
\end{proof}

Theorems \ref{final:theorem1} and \ref{imptheo} allow us to concisely state the main result of this thesis:

\begin{theorem}\label{dstraight}
Let $\mathcal{G}$ be any ultragraph and let $R$ be any commutative unital ring. Then, there exists a graph $E$ such that $L_{R}(\mathcal{G}) \text{ is Morita equivalent to }L_{R}(E).$
\end{theorem}

\section{Simplicity conditions for $L_{R}(\mathcal{G})$}

In \cite{Gene-Abrams-and-Gonzalo-Aranda-Pino:2005aa}, Abrams and Pino give conditions on a row-finite graph $E$ which will guarantee $L_{K}(E)$ is a simple algebra, where $K$ is a field---a graph is row-finite if $\{v\in E^{0}: |s_{E}^{-1}(v)|=\infty\}=\emptyset.$ We will extend their result to ultragraph Leavitt path algebras using Morita equivalence. It's worth noting the result we will prove has already been shown in \cite{Daniel-Goncalves-and-Danilo-Royer:2017aa}; however, the authors achieve their result relying on entirely different techniques. Moreover, we will show, unlike the case with graphs, we need not impose the row-finite condition. Tomforde proves an analogous result for ultragraph $C^{*}$-algebras in \cite{Tomforde:2003aaa}. 

To begin, recall that a cycle in an ultragraph $\mathcal{G}$ is a path $\alpha=e_{1}\dots e_{n}$ such that $s(\alpha)\in r(\alpha)$; an edge $e$ is an \textit{exit} for $\alpha$ if there exists an $i$ such that $s(e)=s(e_{i})$ with $e\neq e_{i}$. An ultragraph $\mathcal{G}$ satisfies \textit{Condition (L)} if every cycle $\alpha=e_{1}\dots e_{n}$ in $\mathcal{G}$ has an exit, or there is an $i$ such that $r(e_{i})$ contains a sink. Similarly, for a graph $E$, a cycle is a path $\alpha=e_{1}\dots e_{n}$ in $E$ such that $r_{E}(\alpha)=s_{E}(\alpha)$; an edge $e$ is an \textit{exit} for $\alpha$ if there exists an $i$ such that $s(e)=s(e_{i})$ with $e\neq e_{i}$. A graph $E$ satisfies \textit{Condition (L)} if every cycle in $E$ has an exit. \textit{Condition (L)} is a necessary condition, for graphs and ultragraphs, in order for their respective Leavitt path algebras over $K$ to be simple. Beside \textit{Condition (L)}, there is one more necessary condition in order to guarantee simplicity. 

\begin{definition}\label{hereditarysat}
Let $\mathcal{G}$ be an ultragraph. $H\subseteq\mathcal{G}^{0}$ is \textit{hereditary} if:\\
1) for each $e\in \mathcal{G}^{1}$, $s(e)\in H \implies r(e)\in H$.\\
2) for all $A,B\in H$, $A\cup B\in H$.\\
3) given $A\in H$ and $B\in \mathcal{G}^{0}$, $B\subseteq A\implies B\in H$.\\
Also, $H\subseteq\mathcal{G}^{0}$ is \textit{saturated} if for any non-singular vertex $v\in G^{0}$, $\{r(e)\}_{e\in\mathcal{G}^{1}: s(e)=v}\subseteq H\implies \{v\}\in H.$
\end{definition}

In the case of a graph $E$, $H\subseteq E^{0}$ is \textit{hereditary} if, for each $e\in H$, $s_{E}(e)\in H\implies r_{E}(e)\in H;$ it is \textit{saturated} if, for every non-singular vertex $v\in E^{0}$, $\{r(e)\}_{e\in E^{1}: s_{E}(e)=v}\subseteq H\implies v\in H.$ Abrams and Pino show in \cite{Gene-Abrams-and-Gonzalo-Aranda-Pino:2005aa}, assuming $E$ is row-finite, $L_{K}(E)$ is simple if and only if $E$ is such that:\\
1) $E$ satisfies \textit{Condition (L)}, and\\
2) the only saturated hereditary subsets of $E^{0}$ are $E^{0}$ and $\emptyset$.

We will show that for any ultragraph $\mathcal{G}$, $L_{K}(\mathcal{G})$ is simple if and only if $\mathcal{G}$ is such that:\\
1) $\mathcal{G}$ satisfies \textit{Condition (L)}, and\\
2) the only saturated hereditary subsets of $\mathcal{G}^{0}$ are $\mathcal{G}^{0}$ and $\emptyset$. We will start by first proving it for an ultragraph $\mathcal{G}$ with no singular vertices.

\begin{prop}\label{singcoLhersat1}
Let $\mathcal{G}$ be an ultragraph with no singular vertices and $E_{\mathcal{G}}$ its associated graph. Then, $\mathcal{G}^{0}$ and $\emptyset$ are the only saturated hereditary subsets of $\mathcal{G}^{0}$ if and only if $E_{\mathcal{G}}^{0}$ and $\emptyset$ are the only saturated hereditary subsets of $E_{\mathcal{G}}^{0}$.
\end{prop}

\begin{proof}
Since $\mathcal{G}$ doesn't contain any singular vertices, \cite[Theorem 6.12]{Takeshi-Katsura-Paul-Muhly-Aidan-Sims--Mark-Tomforde:2008aa} implies there is a one-to-one correspondence between the gauge-invariant ideals of $C^{*}(\mathcal{G})$ and the saturated hereditary subsets of $\mathcal{G}^{0}$. By \cite[Proposition 3.14]{Takeshi-Katsura-Paul-Muhly-Aidan-Sims--Mark-Tomforde:2010aa}, $E_{\mathcal{G}}$ doesn't contain singular vertices if and only if $\mathcal{G}$ doesn't contain any singular vertices; which, by \cite[Theorem 3.6]{Teresa-Bates-Jeong-Hee-Hong-Iain-Raeburn--Wojciech-Szymanski:2002aa}, means there is a one-to-one correspondence between the saturated hereditary subsets of $E_{\mathcal{G}}^{0}$ and the gauge-invariant ideals of $C^{*}(E_{\mathcal{G}})$. Finally, in section 6 of \cite{Takeshi-Katsura-Paul-Muhly-Aidan-Sims--Mark-Tomforde:2010aa}, it's shown there is a bijection between the gauge-invariant ideals of $C^{*}(\mathcal{G})$ and the gauge-invariant ideals of $C^{*}(E_{\mathcal{G}})$. Thus, for $\mathcal{G}$ with no singular vertices, there is a bijection between the saturated hereditary subsets of $\mathcal{G}^{0}$ and the saturated hereditary subsets of $E_{\mathcal{G}}^{0}$.
\end{proof}

\begin{prop}\label{singcoLhersat2}
Let $\mathcal{G}$ be an ultragraph and $E_{\mathcal{G}}$ its associated graph. $\mathcal{G}$ satisfies \textit{Condition (L)} if and only if $E_{\mathcal{G}}$ satisfies \textit{Condition (L)}.
\end{prop}

\begin{proof}
$(\Rightarrow)$ Suppose $\mathcal{G}$ satisfies \textit{Condition (L)}. Let $\alpha$ be a cycle in $E_{\mathcal{G}}$. The graph $F$ from Remark \ref{FsubE} doesn't contain any cycles. Thus, $\alpha$ must contain an edge of the form $(e_{n},x)$; which in turn means $\alpha$ must pass through a vertex $v\in G^{0}$. We will take $v$ to be the base point of $\alpha$. Recall that $\alpha$ can be expressed as 
$$\alpha=g_{0}(e_{n_{1}},x_{1})g_{1}(e_{n_{2}},x_{2})g_{2}\dots(e_{n_{k}},x_{k})g_{k}$$
where, for each $i$, $e_{n_{i}}\in\mathcal{G}^{1}$, $x_{i}\in X(e_{n_{i}})$, and $g_{i}\in F^{*}$; because we are assuming $v$ to be the base point, we can safely assume $s_{E}(g_{0})=v=r_{E}(g_{k})$. Then, \cite[Lemma 4.9]{Takeshi-Katsura-Paul-Muhly-Aidan-Sims--Mark-Tomforde:2010aa} implies $\gamma=e_{n_{1}}e_{n_{2}}\dots e_{n_{k}}$ is a cycle in $\mathcal{G}$ with $v=s(e_{n_{1}})\in r(e_{n_{k}})$. Since we are assuming $\mathcal{G}$ satisfies \textit{Condition (L)}, one of two things must hold.

\textit{Case 1:} $\gamma$ has an exit $e$. In this case, there is an edge, $(e,x)$, which is an exit for $\alpha$.
 
 \textit{Case 2:} $\gamma$ contains an edge $e_{n_{i}}$, for some $i$, such that there exists a sink $w\in r(e_{n_{i}})$. There are two possibilities here. First, if $w\in X(e_{n_{i}})$, then $(e_{n_{i}},w)$ is an edge in $E_{\mathcal{G}}$. Further, by \cite[Proposition 3.14]{Takeshi-Katsura-Paul-Muhly-Aidan-Sims--Mark-Tomforde:2010aa}, $w$ is a sink in $E_{\mathcal{G}}$ as well, which means $(e_{n_{i}},w)\neq (e_{n_{i}},x_{i})$ since $x_{i}=s_{E}(g_{i})$. And so $(e_{n_{i}},w)$ is an exit for $\alpha$. If, on the other hand, $w\notin X(e_{n_{i}})$, then we have, by \cite[Lemma 4.4]{Takeshi-Katsura-Paul-Muhly-Aidan-Sims--Mark-Tomforde:2010aa}, $\omega\in X(e_{n_{i}})\cap \Delta$ such that $w\in r'(\omega)$. Should $\omega\neq x_{i}$, we would have $(e_{n_{i}},\omega)$ as an exit for $\alpha$. So, suppose $\omega=x_{i}$. By \cite[Lemma 4.6 (3)]{Takeshi-Katsura-Paul-Muhly-Aidan-Sims--Mark-Tomforde:2010aa}, there exists a path $g$ in $F$ such that $s_{E}(g)=\omega$ and $r_{E}(g)=w$. Now, note $s_{E}(g_{i})=\omega$ and $r_{E}(g_{i})=s(e_{n_{i+1}})\neq w$. This means $g\neq g_{i}$; however, since $s_{E}(g)=s_{E}(g_{i})$, this implies $\alpha$ must have an exit.
 
Taking the two cases together, we have that $E_{\mathcal{G}}$ must satisfy \textit{Condition (L)} as well.
 
$(\Leftarrow)$ Suppose $E_{\mathcal{G}}$ satisfies \textit{Condition (L)}. Let $\gamma=e_{n_{1}}e_{n_{2}}\dots e_{n_{k}}$ be a cycle in $\mathcal{G}$. Suppose there is an $i$ such that $|r(e_{n_{i}})|>1$. If $r(e_{n_{i}})$ contains a sink, we are done. If not, it must be that there exists an edge $e$ such that $s(e)\in r(e_{n_{i}})$ and $s(e)\neq s(e_{n_{i+1}})$; in which case, $e$ is an exist for $\gamma$. If, on the other hand, $|r(e_{n_{i}})|=1$ for each $i$, it must be $r(e_{n_{i}})=X(e_{n_{i}})=\{s(e_{n_{i+1}})\}$ for $1\leq i \leq k-1$ and $r(e_{n_{k}})=\{s(e_{n_{1}})\}$. And so we get a cycle 
$$\alpha=(e_{n_{1}},s(e_{n_{2}}))(e_{n_{2}},s(e_{n_{3}}))\dots(e_{n_{k}},s(e_{n_{1}}))$$
in $E_{\mathcal{G}}$. Since $E_{\mathcal{G}}$ satisfies \textit{Condition (L)}, $\alpha$ has an exit; further, by \cite[Lemma 4.6 (1)]{Takeshi-Katsura-Paul-Muhly-Aidan-Sims--Mark-Tomforde:2010aa}, the exit must be of the form $(e,x)$. Moreover, for $s(e)=s(e_{n_{i}})$, it must be that $e\neq e_{n_{i}}$. This is because $e_{n_{i}}= e$, and $(e,x)\neq (e_{n_{i}},s(e_{n_{i+1}}))$, imply $x\neq s(e_{n_{i+1}})$, but then this contradicts the fact $X(e_{n_{i}})=\{s(e_{n_{i+1}})\}$. And so, for $(e,x)$ an exit for $\alpha$, we have that $e$ is an exit for $\gamma$. Thus, $\gamma$ has an exist, or there exists an $i$ such that $r(e_{n_{i}})$ contains a sink; meaning $\mathcal{G}$ satisfies \textit{Condition (L)}.
\end{proof}

\begin{theorem}\label{singsim}
Let $\mathcal{G}$ be an ultragraph with no singular vertices, and $K$ a field. Then, $L_{K}(\mathcal{G})$ is simple if and only if $\mathcal{G}$ satisfies \textit{Condition (L)} and $\mathcal{G}^{0}$ and $\emptyset$ are the only saturated hereditary subsets of $\mathcal{G}^{0}$.
\end{theorem}

\begin{proof}
 Since $L_{K}(\mathcal{G})$ and $L_{K}(E_{\mathcal{G}})$ are Morita equivalent, Proposition \ref{Morid} then implies $L_{K}(\mathcal{G})$ and $L_{K}(E_{\mathcal{G}})$ have isomorphic lattice of ideals. Since $\mathcal{G}$ doesn't have any singular vertices, $E_{\mathcal{G}}$ doesn't have any singular vertices either \cite[Proposition 3.14]{Takeshi-Katsura-Paul-Muhly-Aidan-Sims--Mark-Tomforde:2010aa}. Which, by the work of Abrams and Pino (see \cite{Gene-Abrams-and-Gonzalo-Aranda-Pino:2005aa}), means $L_{K}(E_{\mathcal{G}})$ is  simple if and only if $E_{\mathcal{G}}$ satisfies \textit{Condition (L)} and the only saturated hereditary subsets of $E_{\mathcal{G}}^{0}$ are $E_{\mathcal{G}}^{0}$ and $\emptyset$. But, by propositions \ref{singcoLhersat1}  and \ref{singcoLhersat2}, $E_{\mathcal{G}}$ satisfies \textit{Condition (L)}, with $E_{\mathcal{G}}^{0}$ and $\emptyset$ being the only saturated hereditary subsets of $E_{\mathcal{G}}^{0}$, if and only if  $\mathcal{G}$ satisfies \textit{Condition (L)} and $\mathcal{G}^{0}$ and $\emptyset$ are the only saturated hereditary subsets of $\mathcal{G}^{0}$. Thus, $L_{K}(\mathcal{G})$ is simple if and only if $\mathcal{G}$ satisfies \textit{Condition (L)} and $\mathcal{G}^{0}$ and $\emptyset$ are the only saturated hereditary subsets of $\mathcal{G}^{0}$.
\end{proof}

We now want to extend the previous theorem to the case where $\mathcal{G}$ may contain singular vertices. In order to do that, we need the following proposition. 

\begin{prop}\label{desinconl}
Let $\mathcal{G}$ be an ultragraph, and $\mathcal{F}$ its desingularization. Then,\\
1) $\mathcal{G}^{0}$ and $\emptyset$ are the only saturated hereditary subsets of $\mathcal{G}^{0}$ if and only if $\mathcal{F}^{0}$ and $\emptyset$ are the only saturated hereditary subsets of $\mathcal{F}^{0}$.\\ 
2) $\mathcal{G}$ satisfies \textit{Condition (L)} if and only if $\mathcal{F}$ satisfies \textit{Condition (L)}.
\end{prop}

\begin{proof}
1) Suppose $\mathcal{G}^{0}$ and $\emptyset$ are the only hereditary subsets of $\mathcal{G}^{0}$. By \cite[Theorem 3.10]{Tomforde:2003aaa}, $C^{*}(\mathcal{G})$ is a simple $C^{*}$-algebra. Further, by \cite[Proposition 6.6]{Tomforde:2003aa}, $C^{*}(\mathcal{G})$ and $C^{*}(\mathcal{F})$ are Morita equivalent as $C^{*}$-algebras. This means, by \cite[Theorem 3.22]{Raeburn:1998aa}, there is a lattice isomorphism between the closed ideals of $C^{*}(\mathcal{G})$ and the closed ideals of $C^{*}(\mathcal{F})$; which in turn means $C^{*}(\mathcal{F})$ is also a simple $C^{*}$-algebra. Applying \cite[Theorem 3.10]{Tomforde:2003aaa} again, we have $\mathcal{F}^{0}$ and $\emptyset$ are the only saturated hereditary subsets of $\mathcal{F}^{0}$. The same argument establishes the opposite direction. Thus, $\mathcal{G}^{0}$ and $\emptyset$ are the only saturated hereditary subsets of $\mathcal{G}^{0}$ if and only if $\mathcal{F}^{0}$ and $\emptyset$ are the only saturated hereditary subsets of $\mathcal{F}^{0}$. 

2) $(\Rightarrow)$ Suppose $\mathcal{G}$ satisfies \textit{Condition (L)}, and let $\alpha$ be a cycle in $\mathcal{F}$. If $\alpha$ doesn't pass through an infinite emitter in $\mathcal{G}$, then $\alpha$ is also a cycle in $\mathcal{G}$ and we are done. Otherwise, $\alpha$ must pass through an infinite emitter, $v_{0}$, in $\mathcal{G}$. By the construction of $\mathcal{F}$, there are edges $f_{1}$ and $g_{1}$ such that $s_{F}(f_{1})=s_{F}(g_{1})=v_{0}$ with $f_{1}\neq g_{1}$; thus, $\alpha$ must have an exit. And so, all in all, if $\mathcal{G}$ satisfies \textit{Condition (L)}, $\mathcal{F}$ must satisfy \textit{Condition (L)} as well.

$(\Leftarrow)$ Suppose $\mathcal{F}$ satisfies \textit{Condition (L)}, and let $\alpha$ be a cycle in $\mathcal{G}$. If $\alpha$ doesn't pass through an infinite emitter, then $\alpha$ is a cycle in $\mathcal{F}$ and we are done. If, on the other hand, $\alpha$ does pass through an infinite emitter, then it certainly has an exit. Thus, $\mathcal{F}$ satisfying \textit{Condition (L)} implies $\mathcal{G}$ satisfies \textit{Condition (L)}.
\end{proof}

\begin{theorem}\label{singsim1}
Let $\mathcal{G}$ be an ultragraph and $K$ a field. Then, $L_{K}(\mathcal{G})$ is simple if and only if $\mathcal{G}$ satisfies \textit{Condition (L)} and $\mathcal{G}^{0}$ and $\emptyset$ are the only saturated hereditary subsets of $\mathcal{G}^{0}$.
\end{theorem}

\begin{proof}
Let $\mathcal{G}$ be an ultragraph, and $\mathcal{F}$ its desingularization. By Theorem \ref{final:theorem1} and Proposition  \ref{Morid}, $L_{K}(\mathcal{G})$ is simple if and only if $L_{K}(\mathcal{F})$ is simple. But, by\\
\noindent Theorem \ref{singsim}, $L_{K}(\mathcal{F})$ is simple if and only if $\mathcal{F}$ satisfies \textit{Condition (L)} and $\mathcal{F}^{0}$ and $\emptyset$ are the only saturated hereditary subsets of $\mathcal{F}^{0}$. Finally, applying Proposition \ref{desinconl}, we can conclude $L_{K}(\mathcal{G})$ is simple if and only if $\mathcal{G}$ satisfies \textit{Condition (L)} and $\mathcal{G}^{0}$ and $\emptyset$ are the only saturated hereditary subsets of $\mathcal{G}^{0}$.
\end{proof} 

Now, one might ask what sort of simplicity conditions exist for $L_{R}(\mathcal{G})$, where $R$ is any unital commutative ring. In general, studying the ideal structure of $L_{R}(\mathcal{G})$ purely in terms of the properties of $\mathcal{G}$ is difficult, if not impossible. The main reason being, if $R$ is not a field, it can have a rich ideal structure of its own, which in turn influences the ideal structure of $L_{R}(\mathcal{G})$; obviously, the ideal structure of $R$ has nothing to do with the properties of $\mathcal{G}$. There are, however, types of ideals we can study in terms of the properties of $\mathcal{G}$. 

\begin{definition}\label{bideal}
Let $\mathcal{G}$ be an ultragraph. $I \lhd L_{R}(\mathcal{G})$ is called a \textit{basic ideal} if for any $r\in R\setminus\{0\}$ and any $A\in \mathcal{G}^{0}$, $rp_{A}\in I\implies p_{A}\in I$. Further, $L_{R}(\mathcal{G})$ is \textit{basically simple} if $\{0\}$ and $L_{R}(\mathcal{G})$ are its only basic ideals.
\end{definition}

In the case of a graph $E$, an ideal $I$ in $L_{R}(E)$ is basic if $rq_{v}\in I$, for any $r\in R\setminus\{0\}$, implies $q_{v}\in I$. Similarly, we say $L_{R}(E)$ is \textit{basically simple} if $\{0\}$ and $L_{R}(E)$ are the only basic ideals of $L_{R}(E)$. Let $H\subseteq \mathcal{G}^{0}$ be saturated and hereditary. We set
$$B_{H}:=\Big\{v\in G^{0}: |s^{-1}(v)|=\infty \text{ with } 0<|s^{-1}(v)\cap\{e:r(e)\notin H\}|<\infty\Big\}.$$
The elements of $B_{H}$ are called the \textit{breaking vertices} of $H$. An \textit{admissible pair in $\mathcal{G}$} is a pair $(H,S)$ where $H$ is a saturated hereditary subset of $\mathcal{G}^{0}$ and $S\subseteq B_{H}$. Interestingly, there is a bijection between the set of admissible pairs of $\mathcal{G}$ and the graded basic ideals of $L_{R}(\mathcal{G})$ \cite[Theorem 4.4 (2)]{M.-Imanfar:2017aa}. In the case of a graph $E$, this reduces to a bijection between the set of saturated hereditary subsets of $E^{0}$ and the graded basic ideals of $L_{R}(E)$ (see \cite[Theorem 7.9]{Tomforde:2011aa}).   

Tomforde shows in \cite{Tomforde:2011aa}, given a row-finite graph $E$, $L_{R}(E)$ is basically simple if and only if $E$ satisfies \textit{Condition (L)} and $\emptyset$ and $E^{0}$ are the only saturated hereditary subsets  of $E^{0}$. One can easily be led to believe a similar result should hold for $L_{R}(\mathcal{G})$. One might further be tempted to use Morita equivalence, as previously done, to establish such a result. However, a challenge particular to this method is establishing the lattice isomorphism of ideals given in Proposition  \ref{Morid} restricts to an isomorphism between basic ideals; the author has not yet been able to establish this. Perhaps the better approach is to work directly using methods similar to the ones found in \cite{Tomforde:2011aa}.

\bibliographystyle{amsplain}
\bibliography{thesisbib}
\end{document}